\documentclass[reqno]{amsart}
\usepackage[margin=1in]{geometry}
\usepackage{bbm}
\usepackage{amsfonts}
\usepackage{latexsym, amssymb, amsmath, amscd, amsthm, amsxtra}
\usepackage{mathtools}
\usepackage{enumerate}
\usepackage[all]{xy}
\usepackage{mathrsfs}
\usepackage{fancyhdr}
\usepackage{listings}
\usepackage{hyperref}
\usepackage{cleveref}
\usepackage{soul}
\usepackage{xcolor}
\usepackage{dsfont}
\usepackage{stmaryrd}
\usepackage{comment}

\allowdisplaybreaks

\DeclareMathOperator{\re}{{\mathrm{Re}}}
\DeclareMathOperator{\im}{{\mathrm{Im}}}
\newcommand{\Pol}{{\mathbf{M}}}

\hypersetup{
     colorlinks=true
}
\def\P{\mathbb{P}}

\def\E{\mathbb{E}}
\newcommand{\gE}[1]{\left\langle #1 \right \rangle}
\def\Z{\mathbb{Z}}
\def\R{\mathbb{R}}
\def\N{\mathbb{N}}

\def\11{{\mathbf{1}}}

\newcommand{\AO}{\mathrm{AO}}
\newcommand{\BA}{\mathrm{BA}}
\newcommand{\WO}{\mathrm{WO}}
\newcommand{\deq}{\mathrel{\mathop:}=}
\newcommand{\pscale}{\text{pseudo-scaling size}}
\newcommand{\rB}{\mathscr B}
\newcommand{\wE}{{\mathcal E}}
\newcommand{\bp}{\mathbf p}
\newcommand{\bq}{\mathbf q}
\newcommand{\bk}{\mathbf k}
\newcommand{\T}{\mathbb T}
\newcommand{\dist}{\mathrm{dist}}

\newcommand{\bE}{\mathbf{E}}
\newcommand{\fd}{{\mathfrak d}}
\newcommand{\fD}{{\mathfrak D}}

\newcommand{\bx}{{\bf{x}}}

\newcommand{\bu}{{\bf{u}}}
\newcommand{\bv}{{\bf{v}}}
\newcommand{\bw}{{\bf{w}}}

\newcommand{\inv}{\mathrm{inv}}
\newcommand{\bigGamma}{\Gamma}

\newcommand{\size}{{\mathsf{size}}}
\newcommand{\psize}{{\mathsf{psize}}}
\newcommand{\wsize}{{\mathsf{wsize}}}
\newcommand{\gh}{{\rm{gh}}}

\newcommand{\C}{\mathbb C}
\newcommand{\Iso}{{\mathcal I}}
\renewcommand{\bar}{\overline}
\newcommand{\wt}{\widetilde}
\newcommand{\wh}{\widehat}

\newcommand{\al}{\alpha}

\newcommand{\nonuni}{{\text{complete $T$-equation}}}

\newcommand{\qq}[1]{\langle{#1}\rangle}
\newcommand{\qqq}[1]{\llbracket{#1}\rrbracket}
\newcommand{\KM}{M^0_{L\to n}}
\newcommand{\KMp}{M^+_{L\to n}}

\newcommand{\LK}{{L\to n}}

\newcommand{\infinf}{{\infty\to \infty}}
\newcommand{\projn}{\pi_{L\to n}}

\newcommand{\Zn}{\widetilde{\mathbb Z}_n^d}
\newcommand{\Gc}{{\mathring G}}

\newcommand{\Sp}{\mathcal S^+}

\newcommand{\zT}{\mathring T}
\newcommand{\cT}{\mathcal T}
\newcommand{\czT}{\mathring{\mathcal T}}
\newcommand{\zTheta}{\mathring \Theta}

\newcommand{\dashed}{\text{diffusive}}
\newcommand{\Dashed}{\text{Diffusive}}
\newcommand{\self}{\text{self-energy}}
\newcommand{\selfs}{\text{self-energies}}
\newcommand{\Self}{\text{Self-energy}}

\newcommand{\pself}{\text{pseudo-self-energy}}
\newcommand{\pselfs}{\text{pseudo-self-energies}}

\newcommand{\pTexp}{\text{pseudo-$T$-expansion}}
\newcommand{\PTexp}{\text{Pseudo-$T$-expansion}}
\newcommand{\pTeq}{\text{pseudo-$T$-equation}}

\newcommand{\incomp}{{\text{$T$-equation}}}
\newcommand{\fa}{{\mathfrak a}}
\newcommand{\fb}{{\mathfrak b}}
\newcommand{\fc}{{\mathfrak c}}

\newcommand{\fC}{{\mathfrak C}}

\newcommand{\QGn}{\mathcal Q^{(n)}}
\newcommand{\Err}{{\mathcal Err}}
\newcommand{\J}{J}

\newcommand{\PT}{\mathcal R}

\newcommand{\QT}{\mathcal Q}

\newcommand{\AT}{\mathcal A}
\newcommand{\ATn}{\mathcal A^{(>C)}}

\DeclareMathOperator{\OO}{O}
\DeclareMathOperator{\oo}{o}

\newcommand{\WT}{\mathcal W}

\newcommand{\Sdeltak}{{\Sigma}_{T,k}}

\newcommand{\Sele}{{\mathcal E}}

\newcommand{\Selek}{{\mathcal E}}

\newcommand{\sizeself}{\boldsymbol{\psi}}

\newcommand{\cof}{\mathrm{cof}}
\newcommand{\qll}[1]{[\![{#1}]\!]}
\newcommand{\rep}[1]{({#1})}

\newcommand{\conc}{\mathfrak g}
\newcommand{\soe}{d_\eta}
\newcommand{\etas}{\lambda^{-2}W^{-4} L^{5-d}}

\newcommand{\blam}{\beta(\lambda)}

\newcommand{\Wlambda}{\frac{W}{\lambda^2}}
\newcommand{\errL}{\epsilon_0}
\newcommand{\errw}{\epsilon_1}
\newcommand{\erre}{\mathfrak d}
\newcommand{\serr}{\epsilon_0}
\newcommand{\heta}{\mathfrak h_\lambda}
\newcommand{\peta}{{\mathfrak p}_\eta}
\newcommand{\pheta}{\widetilde{\mathfrak p}_\eta}

\DeclareMathOperator{\diag}{diag}
\DeclareMathOperator{\tr}{Tr}
\DeclareMathOperator{\var}{Var}

\newcommand{\be}{\begin{equation}}
\newcommand{\ee}{\end{equation}}
\newcommand{\ii}{\mathrm{i}}
\newcommand{\dd}{\mathrm{d}}
\newcommand{\ff}{\mathrm{f}}
\newcommand{\e}{{\varepsilon}}
\newcommand{\cal}{\mathcal}

\newcommand{\cor}{\color{red}}

\newcommand{\sN}{{\mathsf N}}
\newcommand{\sT}{{\mathsf T}}
\newcommand{\sH}{{\mathsf H}}
\newcommand{\sG}{{\mathsf G}}

\newcommand{\sL}{{\mathsf L}}

\newcommand{\thn}{\vartheta}
\newcommand{\zthn}{\mathring\vartheta}
\newcommand{\zcL}{\mathring{\mathcal L}}

\theoremstyle{plain} 
\newtheorem{theorem}{Theorem}[section]
\newtheorem*{theorem*}{Theorem}
\newtheorem{lemma}[theorem]{Lemma}

\newtheorem*{lemma*}{Lemma}
\newtheorem{corollary}[theorem]{Corollary}
\newtheorem*{corollary*}{Corollary}
\newtheorem{proposition}[theorem]{Proposition}
\newtheorem*{proposition*}{Proposition}

\newtheorem*{claim*}{Claim}

\newtheorem{definition}[theorem]{Definition}
\newtheorem*{definition*}{Definition}
\theoremstyle{remark}
\newtheorem{example}[theorem]{Example}
\newtheorem*{example*}{Example}
\newtheorem{remark}[theorem]{Remark}

\newtheorem*{remark*}{Remark}
\newtheorem*{remarks*}{Remarks}
\newtheorem{strategy}[theorem]{Strategy}

\allowdisplaybreaks
\def\@setthanks{\vspace{-\baselineskip}\def\thanks##1{\@par##1\@addpunct.}\thankses}

\numberwithin{equation}{section}
\usepackage{environ}

\title{Delocalization of a general class of random block Schr{\"o}dinger operators}
 
\author{Fan Yang$^\star$}
\author{Jun Yin$^\ddagger$}

\thanks{$\star$Yau Mathematical Sciences Center, Tsinghua University, and Beijing Institute of Mathematical Sciences and Applications, \href{mailto:fyangmath@mail.tsinghua.edu.cn}{fyangmath@mail.tsinghua.edu.cn}. Supported in part by the National Key R\&D Program of China (No. 2023YFA1010400). 
}
\thanks{$\ddagger$Department of Mathematics, University of California, Los Angeles, \href{mailto:jyin@math.ucla.edu}{jyin@math.ucla.edu}. Supported in part by the Simons Fellows in Mathematics Award 85515. 
}

\begin{document}

\begin{abstract}

We consider a natural class of extensions of the Anderson model on $\mathbb Z^d$, called \emph{random block Schr{\"o}dinger operators} (RBSOs), defined on the $d$-dimensional torus $(\mathbb Z/L\mathbb Z)^d$. 
These operators take the form $H=V + \lambda \Psi$, where $V$ is a diagonal block matrix whose diagonal blocks are i.i.d.~$W^d\times W^d$ GUE, representing a random block potential, $\Psi$ describes interactions between neighboring blocks, and $0<\lambda\ll 1$ is a small coupling parameter (making $H$ a perturbation of $V$). We focus on three specific RBSOs: (1) the \emph{block Anderson model}, where $\Psi$ is the discrete Laplacian on $(\mathbb Z/L\mathbb Z)^d$; (2) the \emph{Anderson orbital model}, where $\Psi$ is a block Laplacian operator; (3) the \emph{Wegner orbital model}, where the nearest-neighbor blocks of $\Psi$ are themselves random matrices. 
Assuming $d\ge 7$ and $W\ge L^\varepsilon$ for a small constant $\varepsilon>0$, and under a certain lower bound on $\lambda$, we establish delocalization and quantum unique ergodicity for bulk eigenvectors, along with quantum diffusion estimates for the Green's function. Combined with the localization results of \cite{Wegner}, our results rigorously demonstrate the existence of an Anderson localization-delocalization transition for RBSOs as $\lambda$ varies.

Our proof is based on the $T$-expansion method and the concept of \emph{self-energy renormalization}, originally developed in the study of random band matrices \cite{BandI,BandII,BandIII}. 
In addition, we introduce a conceptually novel idea---\emph{coupling renormalization}---which extends the notion of self-energy renormalization. While this phenomenon is well-known in quantum field theory, it is identified here for the first time in the context of random Schr{\"o}dinger operators. We expect that our methods can be extended to models with real or non-Gaussian block potentials, as well as more general forms of interactions.



\end{abstract}

\maketitle

\vspace{-5pt}

 {
 \hypersetup{linkcolor=black}
 \tableofcontents
 }



\section{Introduction}

The random Schr{\"o}dinger operator or, more specifically, the Anderson model \cite{Anderson}, is a significant framework for describing the spectral and transport properties of disordered media, such as the behavior of moving electrons in a metal with impurities. 
Mathematically, the $d$-dimensional Anderson model is defined by a random Hamiltonian on $\ell^2(\Z^d)$: \be\label{Anderson_orig}H=-\Delta + \delta V\ee
where $\Delta$ is the Laplacian operator on $\Z^d$, $V$ is a random potential with i.i.d.~random diagonal entries, and $\delta>0$ is a coupling parameter describing the strength of the disorder. The Anderson model exhibits a localization-delocalization transition that depends on the energy, dimension, and disorder strength, capturing the physical phenomenon of metal-insulator transition in disordered quantum systems. 
In the strong disorder regime, where $\delta$ is large, the eigenvectors of the Anderson model are expected to be exponentially localized. Conversely, in the weak disorder regime, it is conjectured that for dimensions $d\ge 3$, there exist mobility edges that separate the localized and delocalized phases; particularly, within the bulk of the spectrum, the eigenvectors are expected to be completely delocalized.

The localization phenomenon of the 1D Anderson model has long been well understood; see e.g., \cite{GMP,KunzSou,Carmona1982_Duke,Damanik2002} among many other references. Studying the Anderson model in higher dimensions ($d\ge 2$) is significantly more challenging.
The first rigorous localization result was established by  Fr{\"o}hlich and Spencer \cite{FroSpen_1983} using multi-scale analysis. A simpler alternative proof was later provided by  Aizenman and Molchanov \cite{Aizenman1993} using a fractional moment method. In dimension $d=2$, Anderson localization is expected to occur at all energies when $\delta>0$ \cite{PRL_Anderson}.  
Numerous remarkable results have been proven concerning the localization of the Anderson model in dimensions $d\ge 2$ (see, e.g., \cite{FroSpen_1985,Carmona1987,SimonWolff,Aizenman1994,Bourgain2005,Germinet2013,DingSmart2020,LiZhang2019}), but the above conjecture remains open. 
For more extensive reviews and related references, we direct readers to \cite{Kirsch2007,Stolz2011,Spencer_Anderson,Aizenman_book,Hundertmark_book}. 

On the other hand, the understanding of delocalization in the Anderson model for dimensions $d\ge 3$ is even more limited---almost nonexistent. To date, the existence of delocalized phase and mobility
edges has only been proved for the Anderson model on Bethe lattice \cite{Bethe_PRL,Bethe_JEMS,Bethe-Anderson}, but not for any finite-dimensional integer lattice $\Z^d$. In this paper, we make an effort toward addressing the Anderson delocalization conjecture by considering a natural variant of the random Schr{\"o}dinger operator---the \emph{random block Schr{\"o}dinger operator} (RBSO), which is formed by replacing the i.i.d.~potential with i.i.d.~block potential. More precisely, we define RBSO, denoted by $H$, on a $d$-dimensional lattice $\Z_L^d:=\{1,2, \cdots, L\}^d$ of linear size $L$. We decompose $\Z_L^d$ into $n^d$ disjoint boxes of side length $W$ (with $L=nW$). Correspondingly, we define $V$ as a diagonal block matrix: $V=\diag(V_1,\ldots, V_{n^d})$, where $V_i$ (for $i\in\{1,\ldots, n^d\}$) are i.i.d.~$W^d\times W^d$ random matrices, with indices labeled by the vertices in these boxes. For simplicity, we assume these random matrices are drawn from GUE (Gaussian unitary ensemble), but this assumption is not essential for us.

From a physical perspective, our new RBSO is related to the Anderson model on the lattice $\Z_n^d$ through a coarse graining transformation, and we refer to it as the ``block Anderson model" in this paper. 
Motivated by the work of Wegner \cite{Wegner1} and its continuation in  \cite{Wegner2,Wegner3}, we also consider another type of RBSO that models the motion of a quantum particle with multiple internal degrees of freedom (referred to as orbits or spin) in a disordered medium. Specifically, the quantum particle moves according to the Anderson model on $\Z_n^d$, while its spin either remains unchanged or rotates randomly as the particle hops between sites on $\Z_n^d$. In the former scenario, we call it the ``Anderson orbital model", while in the latter, it is known as the celebrated ``Wegner orbital model". 
In this paper, we replace the interaction term $-\Delta$ in \eqref{Anderson_orig} with a more general (deterministic or random) matrix $\Psi$, which models interactions between neighboring blocks, and study RBSO $H=\Psi + \delta V$ that represents the three models mentioned above. In other words, we will consider $\Psi=\Psi^{\BA},\ \Psi^{\AO},$ or $\Psi^{\WO}$, corresponding to the block Anderson, Anderson orbital, and Wegner orbital models, respectively. We have chosen to focus on these three models for the sake of clarity in presentation; however, the specific forms of $\Psi$ are not crucial to our results, as noted in \Cref{rmk:extension} below. 

The localized regime of the Anderson/Wegner orbital models has been analyzed in various settings under strong disorder \cite{Sch2009,Wegner,MaceraSodin:CMP,Shapiro_JMP,CPSS1_4,CS1_4}. In this paper, we focus on their delocalization under weak disorder. A major advantage of RBSO over the Anderson model is that it contains significantly more random entries. While approaching the delocalized regime of the Anderson model is challenging with only $N=L^d$ random entries in $V$, we will show that $N^{1+\e}$ random entries in the block potential of our RBSO are already sufficient for proving Anderson delocalization in high dimensions. More precisely, suppose $d\ge 7$ and $W\ge L^\e$ for an arbitrarily small constant $\e>0$. We prove that if $\delta$ is sufficiently ``small" (specifically, for the Anderson/Wegner orbital models, we require $\delta \ll W^{d/4}$), then the bulk eigenvectors of $H$ are delocalized and satisfy a quantum unique ergodicity (QUE) estimate. Moreover, we show that the evolution of the particle behaves diffusively up to the Thouless time (defined in \eqref{eq:tTh} below). 
It has been established in \cite{Wegner} that the Anderson/Wegner orbital models are localized when $\delta\gg W^{d/2}$. Thus, our results rigorously establish, for the first time in the literature, the existence of an Anderson localization-delocalization transition for a certain class of finite-dimensional random Schr{\"o}dinger operators as the interaction strength varies. 
We believe that if the condition $W\ge L^\e$ can be relaxed to 
$W\ge C$ for a large but finite constant $C>0$, the delocalization conjecture for the Anderson model would be (almost) within reach.

RBSOs serve as a natural interpolation between random Schr{\"o}dinger operators and Wigner random matrices \cite{Wigner} as $W$ increases from 1 to $L$. Another well-known interpolation is the celebrated \emph{random band matrix} (RBM) ensemble \cite{ConJ-Ref1, ConJ-Ref2,fy} defined on $\Z_L^d$. The RBM is a Wigner-type random matrix $(H_{xy})_{x,y\in \Z_L^d}$ such that $H_{xy}$ is non-negligible only when $|x-y|\le W$, with $W$ being the band width. 
Over the past decade, significant progress has been made in understanding both the localization and delocalization of RBMs in all dimensions. We refer the reader to \cite{PB_review, PartI, CPSS1_4, BandI} for a brief review of relevant references.
A recent breakthrough \cite{Band1D} established the delocalization of 1D random band matrices under the sharp condition $W\gg L^{1/2}$ on band width. So far, the best delocalization result for higher-dimensional band matrices was achieved in a series of papers \cite{BandI,BandII,BandIII}, which proved that if $d\ge 7$ and $W\ge L^\e$ for an arbitrarily small constant $\e>0$, the bulk eigenvectors of RBM are delocalized.  

The proof in this paper builds on the ideas developed in \cite{BandI,BandII,BandIII}. In fact, the RBM can be regarded as a variant of the Wegner orbital model, where $\delta$ is of order 1 and the nearest-neighbor blocks of $\Psi$ are random upper/lower triangular matrices. 
However, compared to RBM, RBSO lacks translation symmetry, which is a key element in the proofs presented in \cite{BandI,BandII,BandIII}. Additionally, the deterministic limit of the Green's function for RBSO takes a more complex form, which is no longer a scalar matrix (see \eqref{def_G0} below). The graphical tools and expansions associated with RBSO are also more intricate than those used for RBM. 
Beyond these technical complications, there is a more crucial distinction between RBSOs and RBMs that renders the proofs in \cite{BandI,BandII,BandIII} conceptually invalid. While RBM is a special case of RBSO with $\delta\asymp 1$, the interesting parameter regime for the delocalization of RBSOs is $\delta\gg 1$. 
Consequently, the Green's function of RBSO diffuses much more slowly than RBM. 
Within the context of RBM, addressing this amounts to extending the proofs in \cite{BandI,BandII,BandIII} to lower dimensions with $d<7$. More precisely, if we were to apply the methods developed there naively, some of our graphs would include additional powers of $\delta$. This would lead to a significantly worse continuity estimate, ultimately rendering our proof invalid. 

To address the aforementioned issue, we discover a new mechanism called \emph{coupling (or vertex) renormalization}. We have borrowed this term from Quantum Field Theory (QFT).
Roughly speaking, coupling renormalization is a common phenomenon in QFT that describes a cancellation mechanism occurring when calculating the interactions of several propagators using Feynman diagrams.
In the setting of RBSO, where propagators correspond to entries of Green's function, we observe a similar cancellation mechanism when calculating the product of Green's function entries using graph expansions. More precisely, in graph expansions, the leading graphs will involve Green's function entries coupled at a specific ``vertex". Upon summing these graphs, we find that they cancel each other remarkably, resulting in a vanishing factor. For a more detailed discussion of this phenomenon and some basic ideas behind its proof, we refer readers to \Cref{sec:idea} below.
Finally, we note that the proof of coupling renormalization in this paper inspires the $\mathcal K$-loop sum zero property discovered in the joint paper by the second author and Horng-Tzer Yau \cite{Band1D} (see Section 3.3), which ultimately leads to the resolution of the delocalization conjecture for 1D random band matrices.

\subsection{The model}

For definiteness, throughout this paper, we assume that $L=nW$ for some $n,W\in 2\N+1$. Then, we choose the center of the lattice as $0$. However, our results still hold for even $n$ or $W$, as long as we choose a different center for the lattice. Consider a cube of linear size \(L\) in \(\mathbb{Z}^{d}\), i.e., 
\be\label{ZLd}
\Z_L^d:=\qll{ -(L-1)/2 , (L-1)/2} ^d. 
\ee
Hereafter, for any $a,b\in \R$, we denote $\llbracket a, b\rrbracket: = [a,b]\cap \Z$. 
We will view $\Z_L^d$ as a torus and denote by $\rep{x-y}_L$ the representative of $x-y$ in $\Z_L^d$, i.e.,  
\be\label{representativeL}\rep{x-y}_L:= \left((x-y)+L\Z^d\right)\cap \Z_L^d.\ee
Now, we impose a block structure on $\Z_L^d$ with blocks of side length $W$.

\begin{definition}[Lattice of blocks] \label{def: BM2}
Fix any $d\in \N$. Suppose $L=nW$ for some integers $n, W\in 2\N+1$. We divide $\mathbb Z_L^d$ into $n^d$ cubic blocks of side length $W$ such that the central one is $\qll{ -(W-1)/2, (W-1)/2}^d $. Given any $x\in \Z_L^d$, denote the block containing $x$ by $[x]$. Denote the lattice of blocks $[x]$ by $\wt \Z_n^d$. We will view $\wt\Z_n^d$ as a torus and denote by $\rep{[x]-[y]}_n$ the representative of $[x]-[y]$ in $\Zn$. For convenience, we will regard $[x]$ both as a vertex of the lattice $\Zn$ and a subset of vertices on the lattice $\Z_L^d$. Denote by $\{x\}$ the representative of $x$ in the cubic block $[0]$ containing the origin, i.e.,
$\{x\}:=(x+W\Z^d)\cap [0] = x - W[x].$ 
For any $x\in \Z_L^d$, we define the projection $\projn:\Z_L^d\to \Zn$ 
such that $ \projn(x)=[x]$. 
\end{definition}

Any $x\in \Z_L^d$ can be labelled as $([x],\{x\})$. Correspondingly, we define the tensor product of two vectors $\bu$ and $\bv$ with entries indexed by the vertices of $\wt \Z_n^d$ and $[0]$, respectively, as 
\be\label{eq:tensor} \bu\otimes \bv (x):=\bu([x])\bv(\{x\}),\quad x\in \Z_L^d.
\ee
Then, the tensor product of matrices $A_n$ and $A_W$ defined on $\wt\Z_n^d$ and $[0]$, respectively, is defined through 
\be\label{eq:tensor2} A_n\otimes A_W (\bu\otimes \bv) :=(A_n \bu)\otimes (A_W\bv).\ee
Clearly, $\|x-y\|_L:=\| \rep{x-y}_L \|$ is a {periodic} distance on $\Z_L^d$ for any norm $\|\cdot\|$ on $\Z^d$. For definiteness, we use the $\ell^1$-norm in this paper, i.e., $\|x-y\|_L:=\|\rep{x-y}_L\|_1$, which is also the graph distance on $\Z_L^d$. Similarly, we also define the periodic $\ell^1$-distance $\|\cdot\|_n$ on $\Zn$. For simplicity, throughout this paper, we will abbreviate
\begin{align}\label{Japanesebracket} |x-y|\equiv \|x-y\|_L,\quad &\langle x-y \rangle \equiv \|x-y\|_L + W,\quad \text{for} \ \ x,y \in \Z_L^d, \\
\label{Japanesebracket2} |[x]-[y]|\equiv \|[x]-[y]\|_n,\quad &\langle [x]-[y] \rangle \equiv \|[x]-[y]\|_n + 1, \quad \text{for} \ \  x,y \in \wt\Z_n^d.
\end{align}
We use $x\sim y$ to mean that $x$ and $y$ are neighbors on $\Z_L^d$, i.e., $|x-y|=1$. Similarly, $[x]\sim [y]$ means that $[x]$ and $[y]$ are neighbors on $\wt\Z_n^d$. 

For the convenience of presentation, we rescale \eqref{Anderson_orig} by $\lambda:=\delta^{-1}$ (so a smaller $\lambda$ indicates stronger disorder and an increased tendency to localize). Additionally, as previously mentioned, we replace the interaction matrix $-\Delta$ with a more general matrix $\Psi$ that models interactions between neighboring blocks, as stated in the following definition. 


\begin{definition}[Random block Schr{\"o}dinger operators] \label{def: BM}
We define an $N\times N$ ($N=L^d$) complex Hermitian random block matrix $ V$, whose entries are independent Gaussian random variables up to the Hermitian condition $V_{xy}=\overline V_{yx}$. Specifically, the off-diagonal entries of $V$ are complex Gaussian random variables:
\be\label{bandcw0}
V_{xy}\sim {\cal N }_{\C}(0, s_{xy})\quad \text{with}\quad  s_{xy}:=W^{-d} {\bf 1}\left( [x] =[y] \right), \quad \text{for} \ \  x\ne y,
\ee
while the diagonal entries of $V$ are real Gaussian random variables distributed as ${\cal N }_{\R}(0, W^{-d})$. In other words, $V$ is a diagonal block matrix with i.i.d.~GUE blocks. 
We will call $V$ a $d$-dimensional ``block potential". Then, we define a general class of random Hamiltonian as  
\be\label{eq:H_blocka}
H\equiv H({\lambda}):=\lambda \Psi + V,
\ee
where $\lambda>0$ is a deterministic parameter and $\Psi$ is the interaction Hamiltonian introducing hopping between different blocks. For definiteness, we consider the following three types of $\Psi$ in this paper.  
\begin{itemize}
    \item[(i)] {\bf Block Anderson ($\BA$) model}. $\Psi^{\BA}:=2dI_L-\Delta_L$, where $\Delta_L$ denotes the discrete Laplacian on $\Z_L^d$ and $I_L$ denotes the $N\times N$ identity matrix. In other words, we let
    $\Psi^{\BA}_{xy}=\mathbf 1(x\sim y)$ for $x,y \in \Z_L^d$.

\item[(ii)] {\bf Anderson orbital ($\AO$) model}. $\Psi^{\AO}:=(2dI_n-\Delta_n)\otimes I_W$, where $\Delta_n$ denotes the discrete Laplacian on \smash{$\wt\Z_n^d$} and $I_W$ denotes the $W^d\times W^d$ identity matrix.

\item[(iii)] {\bf Wegner orbital ($\WO$) model}. The neighboring blocks of $ \Psi^{\WO}$ are independent blocks of $W^d\times W^d$ complex Ginibre matrices up to the Hermitian symmetry $\Psi^\WO=(\Psi^\WO)^\dag$. In other words, the entries of $\Psi$ are independent complex Gaussian random variables up to the Hermitian symmetry:
\be\label{bandcw1}
\Psi_{xy}\sim {\cal N }_{\C}(0, s'_{xy}),\quad \text{with}\quad  s'_{xy}:=W^{-d} {\bf 1}\left( [x] \sim [y] \right).
\ee

\end{itemize}
We will call $H$ a $d$-dimensional block Anderson/Anderson orbital/Wegner orbital model with linear size $L$, block size $W$, and coupling parameter $\lambda$. 
For simplicity of presentation, with a slight abuse of notations, we will often use consistent notations for some quantities (e.g., $\Psi$) that are used in all three models. 
When we want to make the specific model we refer to clear, we will add the super-index $\BA$, $\AO$, or $\WO$. 
\end{definition}

\begin{remark}\label{rmk:extension}
	We remark that in \cite{Wegner}, the Anderson orbital model in (ii) is called the ``block Anderson model". In this paper, however, we have adopted the name ``Anderson orbital" to distinguish it from our block Anderson model in (i).  
In this paper, we have chosen three of the most classical representatives in the physics literature. In particular, they represent some general families of RBSOs that exhibit the following features: deterministic interactions that are translation-invariant (block Anderson model), deterministic interactions that are translation-invariant on the block level (Anderson orbital model), and random interactions whose \emph{distribution} is translation-invariant on the block level (Wegner orbital model). 
The results and proofs of this paper can be extended to more general random block models that exhibit the metal-insulator transition. Specifically, we expect our method can be applied to deal with almost arbitrary deterministic or random interactions that exhibit block translation symmetry. However, due to length constraints, we do not pursue such directions in the current paper. 
\end{remark}

\subsection{Overview of the main results}
In this subsection, we provide a brief overview of our main results. First, we state the localization of our block models. We define the Green's function (or resolvent) of the Hamiltonian $H$ as
\be\label{def_Green}
G(z):=(H-z)^{-1} = \left(\lambda \Psi + V-z\right)^{-1},  \quad z\in \C.
\ee
Note that since the entries of $H$ have continuous density, $G(z)$ is well-defined almost surely even when $z\in \R$. 
Utilizing the fractional moment method \cite{Aizenman1993}, it was proved in \cite{Wegner} that the fractional moments of the $G$ entries are exponentially localized for small enough $\lambda$. To state the result, we introduce the quantity 
$$\Lambda_{\Psi}:=\left(\E\|\Psi_{[x][y]}\|_{HS}^2\right)^{1/2},$$
for two blocks $[x]\sim[y]$ on $\wt\Z_n^d$, which describes the interaction strength between two neighboring blocks. Here, $\Psi_{[x][y]}=(\Psi_{ab}:a\in [x],b\in [y])$ denotes the $W^d\times W^d$ block of $\Psi$ with the row and column indices belonging to $[x]$ and $[y]$, respectively. 
Due to the block translation symmetry of the three models, $\Lambda_{\Psi} $ does not depend on the choice of the blocks $[x]$ and $[y]$. Through a simple calculation, we find that 
$$\Lambda_{\Psi}=\begin{cases}
 W^{(d-1)/2}, & \text{for} \ \Psi=\Psi^{\BA}\\
    W^{d/2}, & \text{for} \ \Psi\in\{\Psi^{\AO},\Psi^{\WO}\}
\end{cases}$$

\begin{theorem}[Localization]\label{thm_localization}
Consider the models $H$ in \Cref{def: BM}. Fix any $d\ge 1$ and $s\in (0,1)$, there exists a constant $c_{s,d}>0$ such that the following holds. When $\lambda \le c_{s,d}/\Lambda_{\Psi},$ 
there exists a constant $C_{s,d}>0$ such that for all $W, L\ge 1$ and $z\in \R$,
\be\label{fraction-decay}
\mathbb E |G_{xy}(z)|^s\le  C_{s,d} \Lambda_{\Psi}^s \left( \lambda C_{s,d} \Lambda_{\Psi}\right)^{s|[x]-[y]|}.
\ee
\end{theorem} 
\begin{proof}
The bound \eqref{fraction-decay} has been established in \cite{Wegner} for the Anderson orbital and Wegner orbital models. The same argument extends to the block Anderson model by using the fact that the off-diagonal blocks of $\Psi^{\BA}$ are of rank $W^{d-1}$ (in contrast to $W^{d}$ for $\Psi^{\AO}$ and $\Psi^{\WO}$). We omit the details for brevity.
\end{proof}

By \Cref{thm_localization}, if we take $\lambda\le c/\Lambda_\Psi$ for a constant $0<c<c_{s,d}\wedge C_{s,d}^{-1}$, then the estimate \eqref{fraction-decay} gives the exponential decay of $\mathbb E |G_{xy}(z)|^s$, from which we readily derive the localization of eigenvectors by \cite[Theorem A.1]{ASFH2001}.
It is possible to relax the Gaussian assumption on the entries of $V$ to more general distributions with densities and finite high moments, but we do not pursue this direction in the current paper. On the other hand, if the entries of $V$ are discrete random variables, then proving Anderson localization seems to be much more challenging (see e.g., the Anderson-Bernoulli model considered in \cite{Bourgain2005,DingSmart2020,LiZhang2019}). 

One main result of this paper gives a counterpart of the above localization result. By extending the methods developed in the recent series \cite{BandIII,BandI,BandII}, we will establish the following results when $d\ge 7$ and $W\ge L^\e$ for a small constant $\e>0$. Roughly speaking, for the models in \Cref{def: BM}, assuming that
\be\label{eq:condW_BA}
\lambda\gg W^{d/4}/\Lambda_{\Psi} 
,\ee
we will prove that: 
\begin{itemize}
\item {\bf Local law} (\Cref{thm_locallaw}). Within the bulk of the spectrum, we establish a sharp local law for the Green's function $G(z)$ for $\im z$ down to the scale 
\be\label{eq:defeta*}
\eta_*:=\frac{1}{\blam}\frac{W^{d-5}}{L^{d-5}},
\ee
where the parameter $\blam$ is defined as
\be\label{eq:blam}\blam:=\frac{W^d}{\lambda^2 \Lambda_\Psi^2}=\begin{cases}
 W/\lambda^2 , & \text{for} \ \Psi=\Psi^{\BA}\\
    \lambda^{-2}, & \text{for} \ \Psi\in\{\Psi^{\AO},\Psi^{\WO}\}
\end{cases}.\ee

\item {\bf Delocalization} (\Cref{thm:supu}). For any small constant $\e>0$, most bulk eigenvectors of $H$ have localization length $\ge L^{1-\e}$ with probability $1-\oo(1)$.  

\item {\bf Quantum unique ergodicity} (QUE, \Cref{thm:QUE}). Under certain conditions, there exists a constant $\e>0$ such that with probability $1-\oo(1)$, every bulk eigenvector is almost ``flat" on \emph{all scales $\Omega(L^{1-\e})$}. As a consequence, it implies that the localization length of each bulk eigenvector is indeed equal to $L$. 

\item {\bf Quantum diffusion} (\Cref{thm_diffu}). The evolution of the particle follows a quantum diffusion up to the Thouless time. Here, the \emph{Thouless time} \cite{Edwards_1972,Thouless_1977, Spencer2} is defined to be the typical time scale for a particle to reach the boundary of the system. For our models, it is given by   \be\label{eq:tTh}
t_{Th}=\blam {L^2}/{W^2}.
\ee
\end{itemize}
We refer readers to \Cref{sec:main} below for more precise statements of these results. Together with \Cref{thm_localization}, our results imply the Anderson metal-insulator transition of the RBSOs in \Cref{def: BM} as $\lambda$ decreases from $W^{d/4}\Lambda_{\Psi}^{-1}$ to $\Lambda_{\Psi}^{-1}$. We conjecture that $\lambda_c= \Lambda_{\Psi}^{-1}$ represents the correct threshold for the Anderson transition of our RBSOs (see the discussions below \eqref{eq_rewriteE} for heuristic reasoning). 

To facilitate the presentation, we introduce some necessary notations that will be used throughout this paper. We will use the set of natural numbers $\N=\{1,2,3,\ldots\}$ and the upper half complex plane $\C_+:=\{z\in \C:\im z>0\}$.  
We use superscripts `$-$' and `$\dag$' to denote the complex conjugate and Hermitian conjugate of matrices, respectively, i.e., $A^-_{xy}:=\overline A_{xy}$ and $A^\dag_{xy}:=\overline A_{yx}.$ As a convention, we denote $A^\emptyset =A$.  
In this paper, we are interested in the asymptotic regime with $L,W\to \infty$. When we refer to a constant, it will not depend on $L$, $W$, or $\lambda$. Unless otherwise noted, we will use $C$, $D$ etc.~to denote large positive constants, whose values may change from line to line. Similarly, we will use $\epsilon$, $\tau$, $c$, $\fc$, $\fd$ etc.~to denote small positive constants. 
For any two (possibly complex) sequences $a_L$ and $b_L$ depending on $L$, $a_L = \OO(b_L)$, $b_L=\Omega(a_L)$, or $a_L \lesssim b_L$ means that $|a_L| \le C|b_L|$ for some constant $C>0$, whereas $a_L=\oo(b_L)$ or $|a_L|\ll |b_L|$ means that $|a_L| /|b_L| \to 0$ as $L\to \infty$. 
We say that $a_L \asymp b_L$ if $a_L = \OO(b_L)$ and $b_L = \OO(a_L)$. For any $a,b\in\R$, $a\le b$, we denote $\llbracket a, b\rrbracket: = [a,b]\cap \Z$, $\qqq{a}:=\qqq{1,a}$, $a\vee b:=\max\{a, b\}$, and $a\wedge b:=\min\{a, b\}$. For an event $\Xi$, we let $\mathbf 1_\Xi$ or $\mathbf 1(\Xi)$ denote its indicator function.  
For any graph (or lattice), we use $x\sim y$ to mean two vertices $x,y$ are neighbors. 
Given a vector $\mathbf v$, $|\mathbf v|\equiv \|\mathbf v\|_2$ denotes the Euclidean norm and $\|\mathbf v\|_p$ denotes the $\ell^p$-norm. 
Given a matrix $\cal A = (\cal A_{ij})$, $\|\cal A\|$ and $\|\cal A\|_{\max}:=\max_{i,j}|\cal A_{ij}|$ denote the operator norm and maximum norm, respectively. We will use $\cal A_{ij}$ and $ \cal A(i,j)$ interchangeably in this paper. 

For simplicity of notation, throughout this paper, we will use the following convenient notion of stochastic domination introduced in \cite{EKY_Average}. 

\begin{definition}[Stochastic domination and high probability event]\label{stoch_domination}
	{\rm{(i)}} Let
	\[\xi=\left(\xi^{(W)}(u):W\in\mathbb N, u\in U^{(W)}\right),\hskip 10pt \zeta=\left(\zeta^{(W)}(u):W\in\mathbb N, u\in U^{(W)}\right),\]
	be two families of non-negative random variables, where $U^{(W)}$ is a possibly $W$-dependent parameter set. We say $\xi$ is stochastically dominated by $\zeta$, uniformly in $u$, if for any fixed (small) $\tau>0$ and (large) $D>0$, 
	\[\mathbb P\bigg[\bigcup_{u\in U^{(W)}}\left\{\xi^{(W)}(u)>W^\tau\zeta^{(W)}(u)\right\}\bigg]\le W^{-D}\]
	for large enough $W\ge W_0(\tau, D)$, and we will use the notation $\xi\prec\zeta$. 
	If for some complex family $\xi$ we have $|\xi|\prec\zeta$, then we will also write $\xi \prec \zeta$ or $\xi=\OO_\prec(\zeta)$. 
	
	\vspace{5pt}
	\noindent {\rm{(ii)}} As a convention, for two deterministic non-negative quantities $\xi$ and $\zeta$, we will write $\xi\prec\zeta$ if and only if $\xi\le W^\tau \zeta$ for any constant $\tau>0$. 
	
	
	\vspace{5pt}
	\noindent {\rm{(iii)}} We say that an event $\Xi$ holds with high probability (w.h.p.) if for any constant $D>0$, $\mathbb P(\Xi)\ge 1- W^{-D}$ for large enough $W$. More generally, we say that an event $\Omega$ holds $w.h.p.$ in $\Xi$ if for any constant $D>0$,
	$\P( \Xi\setminus \Omega)\le W^{-D}$ for large enough $W$.
\end{definition}

The following classical Ward's identity, which follows from a simple algebraic calculation, will be used tacitly  throughout this paper.
\begin{lemma}[Ward's identity]\label{lem-Ward}
Let $\cal A$ be a Hermitian matrix. Define its resolvent as $R(z):=(\cal A-z)^{-1}$ for any $z= E+ \ii \eta\in \C_+$. Then, we have 
    \be\label{eq_Ward0}
    \begin{split}
\sum_x \overline {R_{xy'}}  R_{xy} = \frac{R_{y'y}-\overline{R_{yy'}}}{2\ii \eta},\quad
\sum_x \overline {R_{y'x}}  R_{yx} = \frac{R_{yy'}-\overline{R_{y'y}}}{2\ii \eta}.
\end{split}
\ee
As a special case, if $y=y'$, we have
\be\label{eq_Ward}
\sum_x |R_{xy}( z)|^2 =\sum_x |R_{yx}( z)|^2 = \frac{\im R_{yy}(z) }{ \eta}.
\ee
\end{lemma}

\subsection{Coupling renormalization} \label{sec:idea}

In this subsection, we briefly describe some basic ideas regarding the coupling (or vertex) renormalization mechanism utilized in our proof. When evaluating the propagator $|G_{xy}|^2$ for RBMs in \cite{BandI}, we identify a $\self$ renormalization mechanism, which also applies to the RBSOs. Roughly speaking, we will derive an expansion of $|G_{xy}|^2$ (up to some small errors) as follows: 
\be\nonumber
\parbox[c]{14cm}{\includegraphics[width=\linewidth]{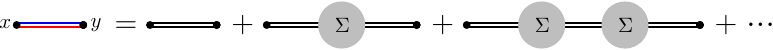}}
\ee
In this picture, the blue and red edges correspond to the $G_{xy}$ and $\bar G_{xy}$ factors, respectively. The black double-line edges represent entries of the diffusive matrix defined in \eqref{def:Theta} below, while the gray disk denotes a deterministic matrix $\Sigma$, referred to as the \emph{$\self$} in the context of QFT. The $\self$ arises from ``self-interactions" of $|G_{xy}|^2$, which is known as the dressed propagator in QFT (with the double-line edge denoting a bare propagator).    

On the other hand, coupling renormalization occurs when we evaluate the product of several Green's function entries, as illustrated in the following picture: 
\be\label{eq:diff_renorm}
\parbox[c]{14cm}{\includegraphics[width=\linewidth]{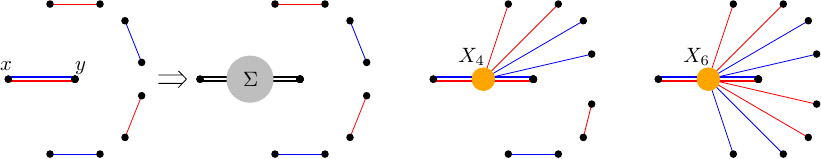}}
\ee
The left-hand side (LHS) represents the product of $|G_{xy}|^2$ with four other $G$ edges. Expanding the propagator $|G_{xy}|^2$ may lead to three possible scenarios on the right-hand side (RHS): (1) only self-interactions occur between $x$ and $y$, without involving other $G$ edges; (2) $|G_{xy}|^2$ and two other $G$ edges couple at a ``vertex" $X_4$ (the ``4" indicates that the vertex involves 4 edges in the original graph); (3) $|G_{xy}|^2$ and four other $G$ edges couple at a ``vertex" $X_6$. We remark that $X_4$ and $X_6$ are not true vertices in our graphs; rather, each represents a subgraph (often referred to as a ``vertex function" in QFT) with several vertices that vary on the local scale $W$. We will refer to $X_4$ and $X_6$ as ``molecules" in this paper, as defined in \Cref{def_poly} below. We find that summing over the graphs in scenario (2) (resp.~scenario (3)) results in the cancellation of the subgraphs within molecule $X_4$ (resp.~$X_6$). This gives the desired coupling/vertex renormalization.

We note that while coupling renormalization may seem like a strictly stronger generalization of $\self$ renormalization, this is actually not the case. Coupling renormalization occurs at the level of a single molecule, whereas $\self$ involves deterministic subgraphs consisting of multiple molecules that vary on the global scale $L$. Hence, coupling renormalization indeed constitutes a cancellation of vertex functions at molecules, while $\self$ renormalization involves a more intricate sum zero property (see \Cref{collection elements} below). 
It is quite possible that a more general ``coupling renormalization" mechanism involving multiple molecules could also apply to our model. However, while exploring this possibility is intriguing in its own right, it does not lead to an improvement in our results or a relaxation of the assumptions.   

To understand why coupling renormalization occurs, we consider an $N\times N$ Wigner matrix, which can be viewed as a special case of our RBSO with $W=L$. (In fact, our proof of coupling renormalization for RBSO is based on a comparison between the ``local behaviors" of RBSO and Wigner matrices.) We examine the following loop graph consisting of four edges representing the Green's function entries of a Wigner matrix.    
\be\nonumber
\parbox[c]{6cm}{\includegraphics[width=\linewidth]{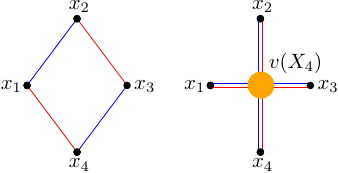}}
\ee
Consider an arbitrary $z=E+\ii \eta$ in the bulk of the spectrum, with $|E|\le 2-\kappa$ for a constant $\kappa>0$ and $\eta\gg N^{-1}$. Through a systematic expansion of the loop graph, we get a sum of graphs. Among them, a typical leading term is illustrated in the right picture, where the four $G$ edges now interact through a center molecule $X_4$. By using Ward's identity and the local law for Green's function, we can bound the sum of the left graph over the four vertices $x_i\in \qqq{N}$, $i\in\{1,2,3,4\}$, by $\OO(N\eta^{-3})$ with high probability. Conversely, applying Ward's identity, the sum of the right graph over the four external vertices and the internal molecule is of order $\Omega(N\eta^{-4})\cdot v(X_4)$ with high probability, where $v(X_4)$ denotes the vertex function associated with $X_4$. By comparing these results, we see that $v(X_4)$ must be of order $\OO(\eta)$, which is significantly smaller than its naive order of $\OO(1)$. This indicates a nontrivial cancellation for the vertex function.

However, there are significant technical challenges in implementing the above ideas. For instance, we need to show that the vertex function depends solely on the number of $G$ edges involved, but not the detailed structure of graphs; otherwise, the cancellation achieved cannot be extended to more general graphs. Additionally, different interactions may occur at different molecules, making it highly nontrivial to establish that we can achieve cancellation at each of them. Furthermore, complicated errors arise in the expansions, and we must show that these errors are negligible in proving the coupling renormalization. For full details of the proof, we refer readers to \Cref{sec:Vexpansion}. 



\medskip

\noindent{\bf Organization of the remaining text}. In \Cref{sec:main}, we state the main results of this paper: the local law, \Cref{thm_locallaw}, quantum unique ergodicity, \Cref{thm:QUE}, and quantum diffusion, \Cref{thm_diffu}. \Cref{sec:preliminary} introduces preliminary notations and results that will be utilized in proving the main results. In \Cref{sec:Texp}, we present the core object of our proof---the $T$-expansion. Using this framework, we provide the proof of local law for the Wegner orbital model in \Cref{sec:pflocal}. 
This proof is based on two key components: the construction of the $T$-expansion (\Cref{completeTexp}) and the continuity estimate (\Cref{gvalue_continuity}). We detail the proofs of these two results 
in Sections \ref{sec:constr}--\ref{sec:continuity}. In \Cref{sec:constr}, we prove \Cref{completeTexp}, assuming a sum zero property for the $\selfs$ (stated as \Cref{cancellation property}), with its proof postponed to \Cref{sec:sumzero}. 
The proof of \Cref{gvalue_continuity} will be provided in \Cref{sec:continuity}. This proof relies on a new tool called \emph{$V$-expansion} and the key coupling/vertex renormalization mechanism described earlier, which are studied in detail in \Cref{sec:Vexpansion}.


In the appendix, the extension of the proof of local law (as presented in Sections \ref{sec:constr}--\ref{sec:continuity} and \Cref{sec:sumzero}) 
to the block Anderson and Anderson orbital models will be discussed in \Cref{sec:ext}. 
The proofs of quantum diffusion (\Cref{thm_diffu}) and QUE (\Cref{thm:QUE}) will be presented in \Cref{sec:pf_Qdiff,sec:QUE}, respectively.  
Appendices \ref{appd_det}--\ref{sec:examples} include auxiliary proofs for the convenience of readers, as well as examples related to coupling and $\self$ renormalization, which may aid in understanding some of our technical proofs.

\subsection*{Acknowledgement}  
We would like to thank Changji Xu and Horng-Tzer Yau for fruitful discussions.

\section{Main results}\label{sec:main}

Since the delocalization of the model is in principle ``weaker" as $\lambda$ becomes smaller, for simplicity of notations and without loss of generality, we always assume that $\lambda\ll 1$ in this paper, i.e.,
\be\label{eq:small+lambda}
\lambda=W^{-\xi}
\ee
for a fixed constant $\xi>0$. All our proofs of should still hold when $\lambda$ is larger, but we do not pursue this direction in this paper. Our study of the delocalization is also based on the Green's function of $H$ defined in \eqref{def_Green}. In previous works (see e.g., \cite{He2018,AEK_PTRF,EKS_Forum}), it has been shown that if $W\to \infty$, $G(z)$ converges to a deterministic matrix limit $M(z)$ (defined as follows) in the sense of local laws if $\im z\gg W^{-d}$. 


 \begin{definition}[Matrix limit of $G$]\label{defn_Mm}
For $\Psi\in \{\Psi^{\BA},\Psi^{\AO}\}$, define $m(z)\equiv m_N(z)$ as the unique solution to 
 \be\label{self_m}
\frac{1}{N}\tr   \frac{1}{\lambda \Psi -z- m(z)} = m(z)
\ee
 such that $\im m(z)>0$ for $z\in \C_+$. Then, we define the matrix $M(z)\equiv M_N(z)$ as 
\be\label{def_G0}
M(z):= \frac{1}{\lambda \Psi -z- m(z)}.
 \ee
For $\Psi=\Psi^{\WO}$, we define $m(z)$ and $M(z)$ as
\be\label{self_mWO}
m(z)={\big(1+2d\lambda^2\big)^{-1/2}}m_{sc}\big({z}/\sqrt{1+2d\lambda^2}\big),\quad M(z):=m(z)I_N,
\ee
where $m_{sc}$ denotes the Stieltjes transform of the Wigner semicircle law:
$$ m_{sc}(z)=\frac{-z+\sqrt{z^2-4}}{2}.$$
 \end{definition}

Note that for the block Anderson and Anderson orbital models, $M(z)$ is translationally invariant due to the translation symmetry of $\Psi$, which implies that $M_{xx}=m$ for all $x\in \Z_L^d$. 
It is known that $m(z)$ is the Stieltjes transform of a probability measure $\mu_{N,\lambda}$, called the free convolution of the semicircle law and the empirical measure of $\lambda\Psi$. 
Moreover, the probability density $\rho_{N,\lambda}$ of $\mu_{N,\lambda}$ is determined from $m(z)$ by 
 $$\rho_{N,\lambda}(x)=\frac{1}{\pi}\lim_{\eta\downarrow 0}\im m(x+\ii \eta).$$
Under \eqref{eq:small+lambda}, this density has one single component $[-e_\lambda, e_\lambda]$ with $\pm e_\lambda\equiv \pm e_\lambda(N)$ denoting the left and right spectral edges, respectively. 
Since the semicircle law has edges $\pm 2$ and the empirical measure of $\Psi$ has compact support, we have $|e_\lambda - 2|=\OO(\lambda)$. As a consequence, we can consider our bulk eigenvalue spectrum to be $[-2+\kappa,2-\kappa]$ for a small constant $\kappa>0$. 
From \eqref{self_m} and \eqref{self_mWO}, we readily see that $m(z)=m_{sc}(z)+\oo(1)$. Furthermore, we can derive $m(z)$ as a series expansion in terms of $\lambda$:
\be\label{eq:expandm} 
m(z,\lambda)=m_{sc}(z) + m_1(z)\lambda + m_2(z)\lambda^2 +\cdots.
\ee
For \eqref{self_mWO}, this is simply the Taylor expansion of $m(z)$ around $\lambda=0$. For $\Psi\in \{\Psi^{\BA},\Psi^{\AO}\}$, we can express $m_k$ in terms of $\Psi$. For example, the first two terms are found to be 
\be \label{eq_m1}
m_1(z)=\frac{-m_{sc}^2}{1-m_{sc}^2}\cdot \frac{1}{N}\tr \Psi=0,\quad m_{2}(z)=\frac{m_{sc}^3}{1-m_{sc}^2}\cdot \frac{1}{N}\tr (\Psi^2).
\ee
In fact, by the definitions of $\Psi^{\BA}$ and $\Psi^{\AO}$, it is easy to check that all odd terms vanish: $m_{2k+1}(z)=0$ since $\tr(\Psi^{l})=0$ for odd $l$.  
Let $z=E+\ii \eta$. When $\Psi\in \{\Psi^{\BA},\Psi^{\AO}\}$, using the definition of $M$ in \eqref{def_G0} and the Ward's identity in \Cref{lem-Ward}, we obtain that 
\be\label{L2M} 
 \sum_{y}|M_{xy}|^2 =\frac{\im M_{xx}}{\eta + \im m}=\frac{\im m}{\eta + \im m}.
\ee
For the Wegner orbital model, using the identity $|m_{sc}|^2/(1-|m_{sc}|^2)=\im m_{sc}/\eta$ for $m_{sc}(z)$, we obtain that 
\be\label{L2M2} 
|m(z)|^2 = \frac{\im m(z)}{\eta + (1+2d\lambda^2)\im m(z)} .
\ee

\subsection{Local law and quantum unique ergodicity}

In this subsection, we state the main results that indicate the appearance of the delocalized phase for our RBSOs as summarized below equation \eqref{eq:condW_BA}. Define the variance matrix 
$$ S= (S_{xy})_{x,y\in \Z_L^d}\quad \text{with}\quad  
S_{xy}=\var(H_{xy}).$$ 
In other words, under \Cref{def: BM}, we have 
$S_{xy}=s_{xy}$ for the block Anderson and Anderson orbital models, and $S_{xy}=s_{xy}+\lambda^2 s'_{xy}$ for the Wegner orbital model. For the block Anderson and Anderson orbital models, it is more convenient to state the local laws in terms of the $T$-variables:
\be\label{eq:TT}T_{xy}(z):=\sum_{\al}S_{x\al}|G_{\al y}(z)|^2, \quad \wt T_{xy}(z):=\sum_{\al}|G_{x\al}(z)|^2S_{\al y}, \quad x,y \in \Z_L^d.
\ee
The reason is that the behavior of $G_{xy}$ can be complicated when $x$ or $y$ vary on small scales of order 1 (due to the specific choice of $\Psi$ and the behavior of $M$). However, after taking a local average over a $W$ scale to form the $T$-variables, the details on scales of order $\OO(1)$ become irrelevant. Consequently, the local law estimates for the $T$-variables take a simpler form. 

\begin{theorem}[Local law]\label{thm_locallaw}
Consider the RBSOs in \Cref{def: BM}. Let $\kappa,\erre,\delta\in (0,1)$ be arbitrary small constants. Fix any $d\ge 7$, assume that $W\ge L^\delta$ and $\lambda$ in \eqref{eq:small+lambda} satisfies that  
\be\label{eq:cond-lambda2}
\lambda \ge W^{d/4+\fd/4}/\Lambda_{\Psi}.
\ee
Then, for any constants $\tau,D>0$, the following local law estimates hold for $z=E+\ii \eta$ and any $x,y \in \Z_{L}^d$: 
\begin{equation}
	\label{locallaw}
\P\bigg(\sup_{|E|\le 2- \kappa}\sup_{W^{\fd}\eta_*\le \eta\le 1} T_{xy} (z)  \le  W^\tau  B_{xy}+ \frac{W^\tau  }{N\eta}\bigg) \ge 1- L^{-D}, 
\end{equation}
\begin{equation}
	\label{locallawmax}
\P\bigg(\sup_{|E|\le 2- \kappa}\sup_{W^{\fd}\eta_*\le \eta\le 1} \max_{x,y\in \Z_L^d}|G_{xy} (z) - M_{xy}(z) | \le  W^\tau \left(\frac{\blam}{W^{d}}+\frac{1}{N\eta}\right)^{1/2}\bigg) \ge 1- L^{-D} , %
\end{equation}
for large enough $L$, where $\eta_*$ is defined in \eqref{eq:defeta*}
and $B$ is a matrix defined as
\be\label{defnBxy}
B_{xy}:=\frac{\blam}{W^2\left\langle x-y\right\rangle^{d-2}}.
\ee 
\end{theorem}

Since $H$ is Hermitian, we have $G_{yx}(\overline z)=\overline {G_{xy}(z)}$, which gives $T_{xy}(z)=\wt T_{yx}(\overline z)$. Thus, the estimate \eqref{locallaw} also holds for \smash{$\wt T_{yx}(z)$} by using the $z\to \overline z$ symmetry.
We believe that the local laws \eqref{locallaw} and \eqref{locallawmax} are sharp and should hold for all $\eta\gg L^{-d}$. Note for $\eta\gg \eta_*$, the term $(N\eta)^{-1}$ in \eqref{locallawmax} is dominated by $\blam/W^d$. 

An immediate corollary of \eqref{locallaw} and \eqref{locallawmax} is the Anderson delocalization of the bulk eigenvectors of $H$. Denote the eigenvalues and eigenvectors of $H$ by $\{\lambda_k\}$ and $\{\bu_k\}$. 
For any constants $K>1$, $0<\gamma \le 1 $, and \emph{localization length} $W\ll \ell \ll L$, define the random subset of indices 
$$ {\mathcal B}_{\gamma, K,\ell} :=\left\{k: \lambda_k \in [-2+\kappa, 2-\kappa] \ \text{ so that  } \min_{x_0\in \Z_L^d} \sum_x |  \bu_k(x)|^2 \exp\left[\left(\frac{|x-x_0|}{\ell}\right)^\gamma\right] \le K \right\}, $$
which contains all indices associated with the bulk eigenvectors that are exponentially localized in a ball of radius $\OO(\ell)$. 

\begin{corollary}[Weak delocalization of bulk eigenvectors]\label{thm:supu}
Under the assumptions of Theorem \ref{thm_locallaw}, suppose $W\le  \ell \le L^{1 - \e}$ for a constant $\e>0$.
Then, for any constants $\tau, D>0$, we have that 
\begin{align}
	\P\bigg(\sup_{k: |\lambda_k | \leq 2 - \kappa} \|\bu_k\|_\infty^2 \leq W^\tau \eta_* \bigg) &\ge 1- L^{-D} ,\label{eq:delocalmax}\\
\mathbb P \left[ \frac{|\mathcal B_{\gamma,K,\ell}|}{N} \le W^\tau \left(\frac{\ell^2}{L^2} + \frac{\blam^{1/2}}{W^{d/2}}\right)\right] &\ge 1-L^{-D} ,\label{uinf_locallength}  
	\end{align}
for large enough $L$.
\end{corollary}
\begin{proof}
We have the bound $|\bu_k(x)|^2 \le \eta\im G_{xx}(\lambda_k + \ii \eta)$ for any $\eta>0$. Then, taking $\eta=W^\tau\eta_*$ and using the local law \eqref{locallawmax}, we conclude \eqref{eq:delocalmax}. For any $y\in \Z_L^d$, $W\le \ell \le L^{1-\e}$, and $z=E+\ii \eta$ with $|E|\le 2-\kappa$ and $\eta = W^{2+\fd}/(\blam L^2)$, using the local law \eqref{locallaw}, we obtain the following estimate with probability $1-\OO(L^{-D})$ for any constants $\tau, D>0$: 
$$ \eta \sum_{x:|x-y|\le \ell}|G_{xy}|^2 \le  \eta \sum_{x:\dist(y,[x]) \le \ell}|G_{xy}|^2 \le \eta\blam \frac{W^\tau \ell^2}{W^2} +\frac{W^\tau \ell^d}{N} \le 2W^{\tau+\fd}\frac{\ell^2}{L^2}.$$
With this estimate and the local law \eqref{locallawmax} at $z=E+\ii \eta$, we can obtain \eqref{uinf_locallength} using the argument in the proof of \cite[Proposition 7.1]{delocal}. 
\end{proof}

The estimate \eqref{uinf_locallength} asserts that for the models in \Cref{def: BM} with block size essentially of order one ($L^{\delta}$ for an arbitrarily small constant $\delta>0$), the majority of bulk eigenvectors have localization lengths essentially of the size of the system (in the sense that they are larger than $L^{1 - \e}$ for any small constant $\e > 0$). 
If we know that the local law \eqref{locallawmax} holds for $\eta\gg L^{-d}$, then we should have the following complete delocalization of bulk eigenvectors:
$$\mathbb P \bigg(\sup_{k: |\lambda_k | \leq 2 - \kappa} \|\bu_k\|_\infty^2 \le L^{-d +\tau}\bigg)\ge  1-L^{-D} .
$$
The estimate \eqref{eq:delocalmax}, although not sharp, has the correct leading dependence in $L^{-d}$. As a consequence, we can derive the QUE for our models, which will imply that the localization lengths of the bulk eigenvectors are indeed equal to $L$.

\begin{theorem}[Quantum unique ergodicity]\label{thm:QUE}
In the setting of \Cref{thm_locallaw}, the following results hold. 
\begin{itemize}
\item[(i)] Given any subset $A_n\in \Z_n^d$, denote $A_L:=\projn^{-1}(A_n)$, where $\projn$ was defined in \Cref{def: BM2}. Suppose $\ell:= |A_n|^{1/d}$ satisfies the following condition for a constant $c>0$: 
\be\label{eq:cond_ell1}
(W\ell)^{d-2}  \geq  {\blam}^{-1}{L^{10}W^{2d-12+c}}. \ee 
Then, for each $\alpha$ such that $ |\lambda_\alpha| \leq 2 - \kappa$, the following event occurs with probability tending to 1:
	\begin{equation}
	\label{eq:que}
 \frac{1}{|A_N|}\sum_{x \in A_N} (N|u_\alpha(x)|^2 -1)\to 0 . 
\end{equation}

\item[(ii)] Given any $\ell>0$ satisfying 
\be\label{eq:cond_ell2}
(W\ell)^{2d-2}\ge  L^{d+5} W^{d-7+c} \ee
for a constant $c>0$, the following event occurs with probability tending to 1 for any constant $\e >0$:
\begin{equation}
	\label{eq:weakque}
   \frac{1}{N}\,\left|\left\{ \alpha: |\lambda_\alpha|<2 - \kappa,\bigg|\frac{1}{W^d|I|}\sum_{[y] \in I}\sum_{x\in [y]} (N|u_\alpha(x)|^2 -1)\bigg| \geq \e  \text{ for some $I \in \mathcal I$}\right\}\right| \to 0\,.
\end{equation}
Here, $\mathcal I: = \left\{I_{\mathbf k,\ell}: \mathbf k\in \Z^d, -n/\ell \le k_i \le n/\ell \right\}$ is a collection of boxes that covers $\wt\Z^d_n$, where 
$$I_{\mathbf k,\ell}:=\{(y)_n: y=(y_1,\ldots, y_d) \ \text{and}\  y_i \in [(k_i-1)\ell/2  ,(k_i+1)\ell/2) \cap \Z^d, i=1,\ldots, d\},$$ 
and $(y)_n$ is defined as in notation \eqref{representativeL}, with $L$ and $\Z_L^d$ replaced by $n$ and $\wt\Z_n^d$, respectively.
\end{itemize}
\end{theorem}

When $d\ge 12$, part (i) of \Cref{thm:QUE} essentially says that as long as $\blam^{-1}L^{12-d}W^{2d-12}\le L^{-\e}$ for a constant $\e>0$, there exists a constant $\tau>0$ such that the $\ell^2$-mass of every bulk eigenvector is approximately evenly distributed on all scales $\ell=\Omega(L^{1-\tau})$. 
When $d> 7$, part (ii) of \Cref{thm:QUE} says that as long as $L^{7-d}W^{d-7}\le W^{-\e}$ for a constant $\e>0$, then there exists a constant $\tau>0$ such that the $\ell^2$-mass of most bulk eigenvector is approximately evenly distributed on a pre-chosen $\ell$-covering of $\Z_L^d$ on all scales $\ell=\Omega(L^{1-\tau})$.  

\begin{remark} \label{rem:smalleta}
Our main results, including \Cref{thm_locallaw}, \Cref{thm:supu}, and \Cref{thm_diffu} below can be extended readily to smaller $\eta$ with  
\be\label{etacirc}\eta\ge W^\fd\eta_{\circ},\quad \text{where} \quad \eta_\circ:=\frac{\blam}{W^5L^{d-5}}.\ee
As a consequence, the QUE estimate \eqref{eq:que} can be proved under the weaker condition 
$$
(W\ell)^{d-2}  \geq  \blam^3  {L^{10}}{W^{-12+c}}, $$
while the estimate \eqref{eq:weakque} can be proved  under the weaker condition 
$$
(W\ell)^{2d-2}\ge  
    \blam^2 {L^{d+5}}W^{-7+c}.$$
We refer to \Cref{rem:Improveeta} below for more explanation of this extension to smaller $\eta$. While the relevant proof is straightforward, considering the length of this paper, we will not present the details here.
\end{remark}

\subsection{Quantum diffusion}
Similar to random band matrices \cite{BandI}, our RBSOs also satisfy the \emph{quantum diffusion conjecture}. To state it, we first define the $\dashed$ matrices 
\be \label{def:Theta}
\thn : =S\frac{1}{1-M^0 S}=\sum_{k=0}^\infty  S\left({M}^0 S\right)^k,\quad \zthn:=P^\perp \thn P^\perp . 
\ee
Here, the matrix ${M}^0$ is defined by ${M}^0_{xy}:= |M_{xy}|^2$ and $P^\perp:=I_N-\mathbf e\mathbf e^\top$, where $\mathbf e:=N^{-1/2}(1,\cdots, 1)^\top$ is the Perron–Frobenius eigenvector of $S$. 

\begin{theorem}[Quantum diffusion]\label{thm_diffu}
Suppose the assumptions of \Cref{thm_locallaw} hold. Fix any large constant $D>0$. There exist constants $\fc_1,\fc_2,\fc_3,C>0$ and a deterministic matrix $\Sele=\Sele(z)$ called \emph{$\self$} such that the following expansion holds for all $x,y \in \Z_{L}^d$ and $z=E+\ii \eta$ with $|E|\le 2-\kappa$ and $\eta \in [ W^\fd \eta_* , 1]$:
\begin{equation}\label{eq:Qdiff}
		\E T_{xy}(z)= \frac{\im m(z) +\OO(\eta+W^{-\fc_1})}{N\eta}+ \sum_{\al} \zthn_{x\al}(\Selek)\left(M^0_{\al y}
	+ \cal G_{\al y}\right) + \wt{\cal G}_{xy}+\OO(W^{-D}),
\end{equation}
where $\cal G$ and $\wt{\cal G}$ are deterministic matrices satisfying that  
\begin{align}\label{eq:boundGaly}
\left|\cal G_{x y}\right|&\le W^{-\fc_2} \left[W^{-d}\exp\left(-\frac{|x-y|}{CW}\right) + \langle x - y\rangle^{-d}\right],\\
|\wt{\cal G}_{x y}|&\le W^{-\fc_2} \sum_\al S_{x\al}|M_{\al y}|,\label{eq:boundGalywt} 
\end{align} 
and $\zthn(\Sele)$ denotes a \emph{renormalized $\dashed$ matrix} defined as  
\be\label{theta_renormal1}
\zthn(\Selek):= (1-\zthn \Sele )^{-1}\zthn .
\ee
Moreover, $\Sele=\Sele(z)$ satisfies the following properties 
for all $x,y\in \Z_L^d$ and $ [a]\in  \Zn$:
 \be\label{self_sym}
	\Sele(x+W[a], y + W[a]) = \Sele(x,y), \quad   \Sele(x, y) = \Sele(-x,-y), 
\ee 
 \be\label{4th_property}
 |\Sele(x,y)| \le  \frac{W^{-\fc_3}}{\blam} \frac{W^{2}}{\langle x-y\rangle^{d+2}} , 
\ee
\be\label{3rd_property}
\Big|W^{-d}\sum_{y\in [a]}\sum_{x\in \Z_L^d} \Sele(x,y)\Big|  \le  W^{-\fc_3}  \left(  \eta + t_{Th}^{-1} \right).
\ee
\end{theorem}

The above theorem essentially means that the non-local and long-time behavior of the quantum evolution of the particle exhibits a diffusive feature for $t$ less than the Thouless time $t_{Th}$ (recall \eqref{eq:tTh}).  
Roughly speaking, in the context of Green's function, this amounts to saying that $\mathbb E|G_{xy}|^2$ (or equivalently, $\E T_{xy}$) can be approximated by the Green's function of a classical random walk on $\mathbb \Z_L^d$ for $|x-y|\gg W$ and $\eta=t^{-1}\gg t_{Th}^{-1}$. This random walk essentially has transition matrix $S$ (up to a normalization factor $1+2d\lambda^2$) for the Wegner orbital model and transition matrix $SM^0S$ (up to a normalization factor $\sum_x M^0_{xy}$) for the block Anderson and Anderson orbital models, but we need to introduce a proper \emph{self-energy renormalization} to it. The Thouless time is just the typical time scale required for the random walk to hit the boundary starting from the center of the lattice. After the Thouless time, the random walk will come back due to the periodic boundary condition, so the behavior of $\mathbb E|G_{xy}|^2$ should be seen as a superposition of several independent random walks at different times. 

Now, we discuss how our result, \Cref{thm_diffu}, describes the quantum diffusion picture outlined above. To facilitate this discussion and the subsequent proofs, we introduce the following ``projection" operation for matrices defined on \smash{$\Z_L^d$}.

\begin{definition}\label{def:projlift}
Define the $W^d\times W^d$ matrix $\bE$ with $\bE_{ij}\equiv W^{-d}.$ Then, for the block Anderson and Anderson orbital models, we can express the variance matrix $S$ as 
\be\label{eq:SBA}
S^{\BA}=S^{\AO} = S(0): =I_n \otimes \bE,
\ee
while for the Wegner orbital model, we have 
\be\label{eq:SWO}
S^{\WO}\equiv S(\lambda): = S(0) + \lambda^2(2dI_n - \Delta_n)\otimes \bE.
\ee
Recall that $I_n$ and $\Delta_n$ are respectively the identity and Laplacian matrices defined on $\Zn$. Now, given an $N\times N$ matrix $\cal A$ defined on $ \Z_L^d$, we define its ``projection" $\cal A^{L\to n}$ to \smash{$ \wt \Z_n^d$} as
\be\label{eq:defLton}
\cal A^{L\to n}_{[x][y]}:=  W^{-d}\, \sum_{x'\in [x]}\sum_{y'\in [y]}\cal A_{x'y'}= W^{d}\left(S(0) \cal A S(0)\right)_{xy}.
\ee
Under this definition, we have $S^{\LK}(0)=I_n$, $ S^{\LK}(\lambda)=I_n + \lambda^2(2dI_n - \Delta_n)$, and $S(0)\cal AS(0)= \cal A^{L\to n}\otimes\bE$.
In this paper, we also put $L\to n$ on sub-indices sometimes. 
\end{definition}

In a detailed discussion below, we will see that $\zthn_{xy}(\Sele)$ has a typical diffusive behavior as the Fourier transform of a function of the form $f(\mathbf p)\asymp (\eta+a(\mathbf p))^{-1}$, where $a(\mathbf p)$ is quadratic in $\mathbf p$, the Fourier variable (or called the ``momentum"). It is known that the Fourier transform of $f(\mathbf p)$ behaves like $B_{xy}$ (in the sense of order) when $|x-y|\gg W$ (see \Cref{lem redundantagain} below).  Hence, when \smash{$\eta=t^{-1}\gg t_{Th}^{-1}$}, the first term on the RHS of \eqref{eq:Qdiff} satisfies
$$\frac{\im m(z) +\OO(\eta+W^{-\fc_1})}{N\eta} \ll \frac{\blam}{W^2L^{d-2}} \le B_{xy},$$
and hence does not affect the diffusive behavior of $\E T_{xy}$. However, when $\eta \le t_{Th}^{-1}$, this term is non-negligible anymore for \smash{$|x-y|\ge L(\eta t_{Th})^{1/(d-2)}$}. This is due to random walks that traverse the torus more than once.   

Next, the following lemma shows that the entries $M_{xy}$ are non-negligible only in the local regime $|x-y|=\OO(W)$. As a consequence, the term \smash{$\wt{\cal G}_{x y}$} in \eqref{eq:Qdiff} is irrelevant for the non-local behavior of $\E T_{xy}$. 

\begin{lemma} [Properties of $M(z)$]\label{lem:propM}
Given any small constant $\kappa>0$, for $z=E+\ii \eta$ with $|E|\le 2-\kappa$ and $\eta\ge 0$, there exists a constant $C>0$ such that  
\be
\label{Mbound}
|M_{xy} (z)|\le (C\lambda)^{|x-y|},\quad \forall x\ne y \in \Z_L^d,
\ee
for the block Anderson model (i.e., $M=M^{\BA}$). For the Anderson orbital model (i.e., $M=M^{\AO}$), we have that $M=M^{\LK}\otimes I_W$, where $M^{\LK}$ satisfies that
\be
\label{Mbound_AO}
|M_{[x][y]}^{\LK} (z)|\le (C\lambda)^{|[x]-[y]|},\quad \forall [x]\ne [y] \in \wt\Z_n^d.
\ee
 \end{lemma}
\begin{proof}
By \eqref{eq:expandm}, we have $\im m \ge \im m_{sc}+\oo(1)\gtrsim 1$, which gives $|m+z|\ge \im(m+z)\gtrsim 1$. Then, \eqref{Mbound} and \eqref{Mbound_AO} are easy consequences of the following expansion,
\be\label{eq:expandM}  M_{xy}=-\sum_{k=1}^\infty \frac{(\lambda \Psi)^{k}_{xy}}{(z+m)^{k+1}},\quad x\ne y,
\ee
and the facts that $\Psi^{\BA}_{xy}=0$ for $|x-y|>1$ and $\Psi^{\AO}_{xy}=0$ for $|[x]-[y]|>1$.
\end{proof}

Finally, for the term $\zthn(\Sele)\cal G$ in \eqref{eq:Qdiff}, the estimate \eqref{eq:boundGaly} shows that the rows of $\cal G$ are summable and their row sums are of small order \smash{$\OO(W^{-\fc_2}\log N)$}. Therefore, the Fourier transform of \smash{$[\zthn(\Sele)\cal G]_{xy}$} with respect to $y$ is of order $\oo(f(\mathbf p))$. 
Therefore, the term $\zthn(\Sele)\cal G$ also gives a small error in the momentum space.

The above discussion shows that when $t_{Th}^{-1}\ll \eta \ll 1$ and $|x-y|\gg W$, the diffusive behavior of $\E T_{xy}$ is dominated by the term \smash{$[\zthn(\Sele)M^0]_{xy}$}. In the following discussion, we focus on explaining why \smash{$\zthn$ and $\zthn(\Sele)$} are called ``diffusive matrix" and ``renormalized diffusive matrix", respectively. 
The diffusive behaviors of $\thn^{\WO}$ and \smash{$\zthn^{\WO}$} has been discussed in \cite{BandI}. In the following discussion, we focus on the block Anderson and Anderson orbital models.  
Using the simple identity $S(0)^2=S(0)$ and adopting the notations in \Cref{def:projlift}, we can write $\thn$ as 
\be\label{exp_Theta2}
\thn= S(0)[1-S(0)M^0S(0)]^{-1}S(0) =(1-\KM)^{-1} \otimes \bE. 
\ee
$\KM$ is an $n^d\times n^d$ matrix with positive entries. It has a Perron–Frobenius eigenvector $\wt {\mathbf e}=n^{-d/2}(1,\ldots,1)^\top$ with the corresponding eigenvalue given by
\be\label{eq:a_sum} a = \sum_{[y]\in \wt \Z_n^d} (\KM)_{[x][y]} =\frac{1}{W^d}\sum_{w\in[x]}  \sum_{y\in \Z_L^d}|M_{w y}|^2 = \frac{\im m}{ \im m + \eta } \ee
due to \eqref{L2M}. 
Then, from \eqref{exp_Theta2}, we get that 
\be\label{exp_Theta3}
\zthn= \wt P^\perp (1-\KM)^{-1} \wt P^\perp\otimes\bE,\quad \text{where}\quad \wt P^\perp:= P^\perp_{\LK}= I_n -\wt{\mathbf e}\wt{\mathbf e}^\top.
\ee
Now, to understand the behavior of $\zthn$, we need to study the matrix $(1-\KM)^{-1}$.

First, notice that $\KM$ is translation invariant: $\KM([x]+[a],[y]+[a])=\KM([x],[y])$, which follows from the block translation symmetry of $M$ itself. Next, by \eqref{Mbound} and \eqref{Mbound_AO}, there exists a constant $C>0$ such that the following estimates hold: for the block Anderson model, 
\be\label{eq:KM}(\KM)_{[x][y]} \lesssim \begin{cases}\left( C{\lambda^2}/{W}\right)^{|[x]-[y]|},& \ \text{if}\ 1\le |[x]-[y]| \le d\\
 \lambda^{W|[x]-[y]|/C },& \ \text{if}\ |[x]-[y]|>d
\end{cases};
\ee
for the Anderson orbital model,
\be\label{eq:KM2}(\KM)_{[x][y]} \le (C\lambda^2)^{|[x]-[y]|}, \quad \text{if}\ |[x]-[y]|\ge 1. 
\ee
Then, combining \eqref{eq:KM} and \eqref{eq:KM2} with \eqref{eq:a_sum}, we can derive that
\be\label{eq:KM0} (\KM)_{[x][x]}= \frac{\im m}{\im m + \eta} - \sum_{[y]:[y]\ne [x]} (\KM)_{[x][y]} 
=\frac{\im m}{\im m + \eta} -\OO\left(\lambda^2\Lambda_\Psi^2/W^d\right).
\ee
Furthermore, with \eqref{eq:expandM}, we can get the following lower bound:
\be\label{eq:KM1}
(\KM)_{[x][y]} \gtrsim   {\lambda^2}\Lambda_\Psi^2/{W^d} ,\quad \text{when} \ \ [x]\sim [y] .
\ee
The above observations show that $(1-\KM)^{-1}$ is essentially a Laplacian operator, which exhibits diffusive behavior. Its Fourier transform takes the following form when $|\mathbf p|\ll 1$:
\be\label{eq:diffu_cons}\left(\eta+\blam^{-1}\mathbf p^\top A\mathbf p\right)^{-1},\quad \mathbf p \in \left(\frac{2\pi}{n}\wt \Z_n\right)^d,\ee 
where $A$ is a positive definite matrix of diffusion coefficients. 

Similar to \eqref{exp_Theta3}, using $S(0)^2=S(0)$, we can rewrite $\zthn(\Sele)$ as 
\be\label{eq_rewriteE}
\zthn(\Selek) = (1-\zthn S(0)\Sele S(0) )^{-1}\zthn =\wt P^\perp\left(1-M^0_{\LK}-\Sele_{\LK}\right)^{-1}\wt P^\perp  \otimes \bE.\ee
We now look at the properties of $\Sele_{\LK}$. The first condition in \eqref{self_sym} means that $\Sele$ is invariant under block translations, implying that $\Sele_{\LK}$ is translation variant on \smash{$\Zn$}. The second condition in \eqref{self_sym} implies that $\Sele_{\LK}$ is a symmetric matrix on \smash{$\Zn$}: $\Sele_{\LK}(0,[x])=\Sele_{\LK}(0,-[x])=\Sele_{\LK}([x],0)$. The condition \eqref{4th_property} shows that $\Sele_{\LK}(0,[x])\le W^{-\fc_3}\blam^{-1}\langle [x]\rangle^{-(d+2)}$, so its Fourier transform is well-defined and twice differentiable. Finally, \eqref{3rd_property} gives a crucial cancellation of $\Sele_{\LK}(0,[x])$ when summing over $[x]$, which is referred to as the \emph{sum zero property} in \cite{BandI}. As seen in \cite{BandI}, this sum zero property is key to the quantum diffusion of random band matrices in dimensions $d\ge 7$; the same phenomenon also appears for our RBSOs. Combining properties \eqref{self_sym}--\eqref{3rd_property}, we see that the Fourier transform of $\Sele$ takes the form $\oo(\eta+|\mathbf p|^2/\blam)$.
Hence, adding $\Sele_n$ to $\KM$ only renormalizes the diffusion coefficients in $A$ by a small perturbation, which does not affect the diffusive behavior of \smash{$\zthn(\Sele)$}. 

Following the above discussions, with the Fourier transform of $\zthn(\Sele)$, we will show that $\zthn_{xy}(\Sele)$ is typically of order $B_{xy}$ in \Cref{lem redundantagain}. Furthermore, by \Cref{lem G<T} below, $\|G-M\|_{\max}$ is controlled by the size of $\|T\|_{\max}$. 
Hence, the local laws \eqref{locallaw} and \eqref{locallawmax} should be sharp. In particular, the error bound in \eqref{locallawmax} is small only when $\eta\ge N^{-1+\e}$ and $\lambda\ge W^{\e}/\Lambda_\Psi$. Hence, we conjecture that the critical threshold of $\lambda$ for the Anderson transition is $\Lambda_\Psi^{-1}$. 
Another heuristic reasoning stems from our paper \cite{QC_YY} on the quantum chaos transition of a random block matrix model. This model can be viewed as a generalized version of our RBSO, where the nearest-neighbor interactions in $\Psi$ can take almost arbitrary forms. In \cite{QC_YY}, we analyze a 1D model with $W\sim L$ and establish a localization-delocalization transition when $\lambda\Lambda_\Psi$ crosses 1. 
We anticipate a similar phenomenon occurring for our RBSO in higher dimensions, even when $W\ll L$.

On the other hand, the condition \eqref{eq:cond-lambda2} arises from the requirement that $\sum_x B_{xy_1}B_{xy_2}$ produces a small factor for all \smash{$y_1,y_2\in \Z_L^d$}. Without this requirement, the resolvent expressions in our proof may grow increasingly large as we expand them, and the expressions in our $\self$ cannot be bounded properly and may diverge as $W\to \infty$. Relaxing the condition \eqref{eq:cond-lambda2} within our proof framework appears quite challenging; it is essentially equivalent to extending the proofs for RBM in \cite{BandI, BandII, BandIII} to dimensions $d<4$. 



\begin{remark}\label{rmk:key}
For simplicity of presentation, we have only considered the complex Gaussian case and trivial variance profile within each block. As has been explained in \cite{BandI,BandII,BandIII}, our proof can be readily adapted to non-Gaussian RBSO after some technical modifications. More precisely, in the proof, we will use Gaussian integration by parts in expanding resolvent entries, but this can be replaced by certain cumulant expansion formulas (see e.g., \cite[Proposition 3.1]{Cumulant1} and \cite[Section II]{Cumulant2}) for general distributions. 
In addition, the variance profile $\mathbf E$ within blocks can be replaced by a more general one. For example, our proof would work almost verbatim if $\mathbf E$ is replaced by a variance matrix that can be written as $\mathbf (\mathbf E')^2$ for some doubly stochastic $\mathbf E'$ whose entries are of order $W^{-d}$. We believe that there are no essential difficulties in extending it to a general block variance matrix, as long as we maintain the mean-field nature of the variances within each block. 
The two most essential features that make our proofs possible are \emph{block translation symmetry} and \emph{enough randomness} (i.e., $N^{1+\e}$ random entries in $H$ v.s.~$N$ random entries in the original Anderson model). 
\end{remark}

\section{Preliminaries}\label{sec:preliminary}


\subsection{Some deterministic estimates}\label{sec_diffusive}

In this subsection, we present some basic estimates regarding the deterministic matrices that will be used in the main proof. The proofs of these estimates are based on standard techniques involving the analysis of the Fourier series representations of the deterministic matrices. Readers can find the details of all the proofs in Appendix \ref{appd_det}.

The most important deterministic matrices for our proof are the diffusive matrices $\thn$, $\zthn$, and $\zthn(\Sele)$ defined in \eqref{def:Theta} and \eqref{theta_renormal1}. The following two lemmas describe their basic behaviors. 


\begin{lemma}[Behavior of the diffusive matrices]\label{lem theta}
Assume that $d\ge 7$ and $W\ge 1$ for the RBSOs in \Cref{def: BM}. Define $\ell_{\lambda,\eta}:=[\blam\eta]^{-1/2} + 1$. Let $\kappa\in (0,1)$ be an arbitrarily small constant. If $z=E+\ii \eta$ with $|E| \le 2-\kappa$ and $\eta> 0$, the following estimates hold for any constants $\tau,D>0$ and all $x,y \in \Z_L^d$:
\begin{align}
&|\zthn_{xy}(z)| \lesssim  B_{xy} \log L;\label{thetaxy}\\
&|\thn_{xy}(z)| \lesssim  \langle x-y\rangle^{-D},\ \  \text{if}\ \  |x-y|> W^{1+\tau} \ell_{\lambda,\eta}. \label{thetaxy2}
\end{align}
The bound \eqref{thetaxy} also holds for $|\thn_{xy}(z)|$ when $\eta\ge t_{Th}^{-1}$ (recall \eqref{eq:tTh}). 
\end{lemma}

When $t\ge t_{Th}^{-1}$, using \eqref{L2M} and \eqref{L2M2}, we obtain that for the block Anderson and Anderson orbital models,  
\be\label{eq:diffzthn-thn1} 
\thn_{xy}(z)-\zthn_{xy}(z) = \frac{1}{N}\frac{1}{1- \sum_y|M_{xy}(z)|^2} \lesssim \frac{1}{N\eta} \lesssim B_{xy},
\ee
and for the Wegner orbital model,  
\be\label{eq:diffzthn-thn2} 
\thn_{xy}(z)-\zthn_{xy}(z) = \frac{1}{N}\frac{1+2d\lambda^2}{1-(1+2d\lambda^2)|m(z)|^2} \lesssim \frac{1}{N\eta}\lesssim B_{xy}.
\ee
Hence, the bound $|\thn_{xy}|\lesssim B_{xy}\log L$ when $\eta\ge t_{Th}^{-1}$ follows directly from \eqref{thetaxy} for $|\zthn_{xy}(z)|$.

\begin{lemma}\label{lem:label_diffusive}
Suppose the assumptions of \Cref{lem theta} hold and $z=E+\ii \eta$ with $|E|\le 2-\kappa$ and $\eta>0$. Suppose ${\cal E}$ is a deterministic $\Z_L^d\times \Z_L^d$ matrix satisfying the following properties: there exists a deterministic parameter $0<\psi<1$ (which may depend on $W$ and $L$) such that 
	\be\label{self_sym1}
	\Sele_{\LK}([x], [x]+[a]) = \Sele_{\LK}(0,[a]), \quad   \Sele_{\LK}(0, [a]) = \Sele_{\LK}(0,-[a]), \quad \forall \ [x],[a]\in \wt\Z_n^d,
	\ee 
	\be\label{self_decay}
	  |\Sele_{\LK}(0,[x])| \le \frac{\psi}{\langle [x]\rangle^{d+2}} , \quad \forall \ [x]\in \wt\Z_n^d.
	\ee
Then, we have 
\be\label{thetaxy_renorm0}
(\zthn^{\LK} \cal E_{\LK})_{[x][y]} \prec \frac{\psi\blam}{\qq{[x]-[y]}^{d-2}}  ,
\quad \forall\ [x],[y] \in \wt\Z_n^d.
\ee
In addition, if there exists a deterministic parameter  $0<\psi_0<1$ such that  
\be\label{self_zero}
	 \Big|\sum_{[x]\in \Zn} \Sele_{\LK}(0,[x]) \Big|  \le \psi_0  ,
	\ee
then we have a better bound
\be\label{thetaxy_renorm}
 (\zthn^{\LK} \cal E_{\LK})_{[x][y]} \prec \frac{1}{\langle [x]-[y] \rangle^{d}}\left( \frac{\psi}{\langle [x]-[y]\rangle^{2}} + \psi_0 \right)\min\left(\eta^{-1},\blam\langle [x]-[y]\rangle^{2}\right),\quad \forall\ [x] ,[y]\in \wt\Z_n^d.
\ee
The bounds \eqref{thetaxy_renorm0} and \eqref{thetaxy_renorm} also hold if $\zthn$ is replaced by $\thn$ when $\eta\ge t_{Th}^{-1}$. 
\end{lemma}
\begin{remark}
By \eqref{self_decay}, we have $\sum_{[x]} {\cal E}_{\LK}(0,[x])\lesssim \psi,$ so we can always take $\psi_0=\OO(\psi)$, in which case \eqref{thetaxy_renorm} is always stronger than \eqref{thetaxy_renorm0}. In addition, if $\Sele$ satisfies \eqref{self_sym}--\eqref{3rd_property} for a constant $\fc_3>0$, then as discussed below \eqref{eq_rewriteE}, $\Sele_{\LK}$ satisfies \eqref{self_sym1}, \eqref{self_decay}, and \eqref{self_zero} with 
$$\psi=W^{-\fc_3}\blam^{-1},\quad \psi_0=W^{-\fc_3}  \left(  \eta + t_{Th}^{-1}  \right).$$
Then, \eqref{thetaxy_renorm} gives that $ (\zthn^{\LK} \cal E_{\LK})_{[x][y]} \prec W^{-\fc_3}/\langle [x]-[y]\rangle^{d}$, which will be used in \Cref{lem redundantagain} below to bound $\zthn(\Sele)$.
\end{remark}

%


In our proof, we will encounter another two deterministic matrices $S^\pm$. Denote the matrix ${M}^+$ as ${M}^+_{xy}:= (M_{xy})^2$. We then define $S^\pm$ as  
\be \label{def:S+}
  S^+\equiv S^+(z,\lambda) := S[1-M^+(z,\lambda) S]^{-1}=\sum_{k=0}^\infty S\left({M}^+(z,\lambda)S\right)^k, \quad S^-:= ( S^+ )^\dagger.
\ee
These matrices satisfy the following estimates. 

\begin{lemma} \label{lem deter}
Under the assumptions of \Cref{lem theta}, for any $z=E+\ii \eta$ with $|E| \le 2-\kappa$ and $\eta> 0$, there exists a constant $C_0>0$ such that the following estimates hold:
\begin{align}\label{S+xy}
|S^\pm_{xy}(z)| &\le   C_0 W^{-d}\exp\left({-\frac{|x-y|}{C_0W}}\right), \\
|S^\pm_{xy}(z,\lambda)-S^\pm_{xy}(z,0)| &\le   C_0\blam^{-1} W^{-d}\exp\left({-\frac{|x-y|}{C_0W}}\right). \label{S+xy-0}
\end{align} 
\end{lemma}

To prove the sum-zero property for the $\selfs$, we will need to use the \emph{infinite space limits} of the deterministic matrices discussed above. 

\begin{definition}[Infinite space limits]\label{def infspace0}
Fix any $d\ge 7$ and $|E|\le 2-\kappa,$ we define the following infinite space limits of $M$, $S$, $S^\pm$, and $\zthn$ obtained by keeping $W$ (and hence $\lambda$) fixed and taking $L \to \infty$ and $\eta \to 0$: for any $x,y\in \Z^d$,
\begin{equation}
    \begin{split}
        M^{(\infty)}_{xy}(E): = \lim_{L\to \infty}M^{(L)}_{xy}(E+\ii 0_+),\quad & S^{(W,\infty)}_{xy}:= \lim_{L\to \infty} S^{(W,L)}_{xy},\\
         [S_{(W,\infty)}^{\pm}(E)]_{xy}:= \lim_{L\to \infty} [S_{(W,L)}^{\pm}(E+\ii 0_+)]_{xy} , \quad &\thn^{(W,\infty)}_{xy}:= \lim_{L\to \infty}\zthn^{(W,L)}_{xy}(E+\ii 0_+). 
    \end{split}
\end{equation}
Here, the superscripts/subscripts $(L)$ and $(W,L)$ indicate the dependence on $W$ or $L$. The diagonal entries of $M^{(L)}$ and $M^{(\infty)}$ are denoted by $m^{(L)}$ and $m^{(\infty)}$, respectively, and we  also define $M_{(\infty)}^0$ and $M_{(\infty)}^+$ as:
$$[M_{(\infty)}^0(E)]_{xy} :=|M^{(\infty)}_{xy}(E)|^2,\quad [M_{(\infty)}^+(E)]_{xy}:= [M^{(\infty)}_{xy}(E)]^2.$$ 
The operator $\thn^{(W,\infty)}$ can also be obtained from the infinite space limit of  $\thn^{(W,L)}$:
$$\thn^{(W,\infty)}_{xy}= \lim_{L\to \infty}\thn^{(W,L)}_{xy}(E+\ii t_{Th}^{-1}).$$
We use $\Delta$ to denote the Laplacian operator on $\Z^d$, which is the infinite space limit of $\Delta_L$.  
Finally, we denote the infinite space limit of the renormalized lattice as \smash{$\wt \Z^d$}. Then, given any linear operator $\cal A:\ell^2(\Z^d)\to \ell^2(\Z^d)$, we define its `projection' $\cal A^{\infinf}$ to  $\wt \Z^d$ as in \eqref{eq:defLton} for $[x],[y]\in \wt\Z^d$.
\end{definition}

With the tools of Fourier transforms, we can easily bound the differences between these deterministic matrices and their infinite space limits. As long as $W$ is sufficiently large, all the following results hold uniformly in $L\ge W$ without assuming the condition $W\ge L^\delta$. 
First, the following lemma bounds the difference between $M^{(L)}$ and $M^{(\infty)}$. 

\begin{lemma}\label{lem:estM}
Suppose \eqref{eq:small+lambda} and the assumptions of \Cref{lem theta} hold. Let $\al(W):=0$ for the Wegner orbital model, $\al(W):=1$ for the block Anderson model, and $\al(W):=W^2$ for the Anderson orbital model. For any $z=E+\ii \eta$ with $|E|\le 2-\kappa$ and $0<\eta\le 1$, 
there exists a sufficiently large $W_0\in \N$ and a constant $C>0$ (independent of $W$ or $L$) such that the following statements hold uniformly for all $L\ge W\ge W_0$: 
\be\label{eq:mL_minf}
\big|m^{(L)}(z)-m^{(\infty)}(E)\big|\le C\left(\eta+ \lambda^2 \al(W)/L^2\right); 
\ee
for the block Anderson model, we have that for all $x\ne y\in \Z_L^d$, 
\be\label{eq:ML-Minf}
|M^{(L)}_{xy}(z)-M^{(\infty)}_{xy}(E)|\lesssim \left(\eta + \lambda^2 \al(W)/ L^{2}\right)(C\lambda)^{|x-y|} +(C\lambda)^{L/C};
\ee
for the Anderson orbital model, we have that for all $[x]\ne [y]\in \wt\Z_n^d$, 
\be\label{eq:ML-MinfAO}
|M^{(L)}_{xy}(z)-M^{(\infty)}_{xy}(E)|\lesssim \left(\eta + \lambda^2 \al(W)/L^{2}\right)(C\lambda)^{|[x]-[y]|} +(C\lambda)^{n/C}.
\ee
\end{lemma}

Next, we present two lemmas that control the difference between \smash{$S^{\pm}_{(W,L)}(z)$ and $S_{(W,\infty)}^{\pm}(E)$}, as well as the difference between \smash{$\zthn^{(W,L)}(z)$ and $\thn^{(W,\infty)}(E)$}. 

\begin{lemma}\label{lem:estS+}
In the setting of \Cref{lem:estM}, there exists a constant $C>0$ such that for all $x, y\in \Z_L^d$,
    \be\label{eq:S+-Sinf} 
    \left|[S^+_{(W,L)}(z)]_{xy}-[S_{(W,\infty)}^{+}(E)]_{xy}\right|\lesssim \frac{\eta+\lambda^2 \al(W)/L^{2}}{W^{d}}e^{-{|[x]-[y]|}/{C}} + e^{-n/C}.
    \ee
\end{lemma}

\begin{lemma}\label{lem esthatTheta}
In the setting of \Cref{lem:estM}, we have that for all $x\in [0]$ and $y\in \Z^d$,
\be\label{Theta-wh1}
|\thn^{(W,\infty)}_{xy}(E)|\lesssim   \frac{\blam \log (|y|+W)}{W^2 \left(|y|+W\right)^{ d-2}},
\ee
and for all $x\in [0]$ and $y\in \Z_L^d$, 
\be\label{Theta-wh}
\left|\zthn^{(W,L)}_{xy}(z)-\thn^{(W,\infty)}_{xy}(E)\right|\lesssim \blam \left(\eta+ t_{Th}^{-1}\right)\cdot \frac{\blam \log L }{ W^4(\|y\|_L+W)^{d-4}} .  
\ee
When $\eta\ge t_{Th}^{-1}$, by \eqref{eq:diffzthn-thn1} and \eqref{eq:diffzthn-thn2}, the bound \eqref{Theta-wh} also holds if we replace $\zthn^{(W,L)}_{xy}(z)$ with $\thn^{(W,L)}_{xy}(z)$. 
\end{lemma}

Finally, we prove a similar result as \Cref{lem:label_diffusive} in the infinite space limit, under the assumption that the $\self$ satisfies an exact sum zero property.

\begin{lemma}\label{lem redundantagain2}
In the setting of \Cref{lem:estM}, suppose $\Sele:\ell^2(\wt\Z^d)\to \ell^2(\wt\Z^d)$ is a deterministic linear operator satisfying \eqref{self_sym1} and \eqref{self_decay} (with $[x],[a]\in \wt \Z^d$), and 
\begin{equation}\label{weaz}
\sum_{[x]\in \wt\Z^d} \Sele_{[0][x]}(E)=0.
\end{equation} 
Then, we have that for all $[x],[y] \in \wt\Z^d,$
\begin{align}
&\left(\thn^{(W,\infty)}_{\infinf} \Sele\right)_{[x][y]} \lesssim \blam\frac{ \psi \log (|[x]-[y]|+2)}{(|[x]-[y]|+1)^{d}}. \label{redundant again2}
\end{align}
 \end{lemma}

\subsection{Graph, scaling sizes and doubly connected property}\label{sec:graph_sizes}

Similar to \cite{BandI,BandII,BandIII}, we will organize our proofs using graphs. In this subsection, we introduce the basic concepts of graph, atomic graph, molecular graph, scaling size, and the doubly connected property. Our graphs will consist of matrix indices as vertices and various types of edges representing matrix entries. In particular, our graph will contain the entries of the following matrices:  $
I, \ \Psi^{\BA} \text{ or } \Psi^{\AO}, \ M, \  S, \ S^\pm,\ \thn,\ \zthn,\ G,$ and $\Gc$, where we abbreviate \[\Gc:=G-M.\]
We denote by $\E_x$ the partial expectation with respect to the $x$-th row and column of $H$: $\E_x(\cdot)= \E(\cdot|H^{(x)}),$ where $H^{(x)}$ is the $(N-1)\times (N-1)$ minor of $H$ obtained by removing the row and column of $H$ indexed by $x$. Following \cite{BandI}, we will use the simplified notations $P_x :=\E_x$ and $Q_x := 1-\E_x$ in this paper. 
We remark that while most graphical notations are defined similarly as in \cite{BandI,BandII,BandIII}, some other notations may take slightly different meanings (e.g., the notions of atoms and various types of edges). We first introduce the most basic concept of (vertex-level) graphs. 

\begin{definition}[Graphs] \label{def_graph1} 
Given a graph with vertices and edges, we assign the following structures and call it a vertex-level graph.  
	
	\begin{itemize}
				
		\item {\bf Vertices:} Vertices represent matrix indices in our expressions. Every graph has some external or internal vertices: external vertices represent indices whose values are fixed, while internal vertices represent summation indices that will be summed over. 
				

		\item {\bf Edges:} We have the following types of edges. 
		
		\begin{enumerate}

			\item{\bf  Solid edges and weights:} 
\begin{itemize}
\item A blue (resp.~red) oriented solid edge from $x$ to $y$ represents a $G_{xy}$ (resp.~$\bar G_{xy}$) factor. 
    
\item A blue (resp.~red) oriented solid edge with a circle ($\circ$) from $x$ to $y$ represents a $\Gc_{xy}$ (resp.~$\Gc^-_{xy}$) factor.

\item $G_{xx}$, $\bar G_{xx}$, $\Gc_{xx}$, and $\Gc^-_{xx}$ factors will be represented by self-loops on the vertex $x$. Following the convention in \cite{BandI}, we will also call $G_{xx}$ and $\bar G_{xx}$ as blue and red {\bf weights} and call $\Gc_{xx}$ and $\Gc^-_{xx}$ as blue and red {\bf light weights}.


\end{itemize}

\item {\bf Dotted edges:} 
\begin{itemize}
    \item A black dotted edge between $x$ and $y$ represents a factor $\mathbf 1_{x=y}$, and a $\times$-dotted edge represents a factor $\mathbf 1_{x\ne y}$. There is at most one dotted or $\times$-dotted edge between every pair of vertices. 

    \item For the \emph{block Anderson model}, a black dotted edge with a triangle ($\Delta$, which stands for ``Laplacian") between $x$ and $y$ represents a $\Psi^{\BA}_{xy}$ factor.

    \item For the \emph{block Anderson model}, a blue (resp. red) dotted edge between $x$ and $y$ represents an $M_{xy}$ (resp. $\bar M_{xy}$) factor.
\end{itemize}

\item {\bf Waved edges:}
			\begin{itemize}
				\item A black waved edge between $x$ and $y$ represents an $S_{xy}$ factor. 
				
				\item A blue (resp.~red) waved edge between $x$ and $y$ represents an $S^+_{xy}$ (resp.~$S^-_{xy}$) factor.
				
				\item For the \emph{Anderson orbital model}, a blue (resp.~red) waved edge between $x$ and $y$ with label $M$ represents an $M_{xy}$ (resp.~$\bar M_{xy}$) factor.

    \item For the \emph{Anderson orbital model}, a black waved edge with a triangle ($\Delta$) between $x$ and $y$ represents a $\Psi^{\AO}_{xy}$ factor.
			\end{itemize}

\item {\bf $\Dashed$ edges:} A $\dashed$ edge between $x$ and $y$ represents a ${\thn}_{xy}$ factor, and we draw it as a black double-line edge between $x$ and $y$. A $\dashed$ edge between $x$ and $y$ with a $\circ$ represents a \smash{${\zthn}_{xy}$} factor, and we draw it as a black double-line edge with a $\circ$ between $x$ and $y$. 
			
\item {\bf Free edges:} A black dashed edge between $x$ and $y$ represents an $(N \eta)^{-1}$ factor. We will call it a \emph{free edge}. 
   

		
			
\end{enumerate}
Edges between internal vertices are called {internal edges}, while edges with at least one end at an external vertex are called {external edges}. The orientations of non-solid edges do not matter because they all represent entries of symmetric matrices. Furthermore, we will say that a black edge has a neutral (zero) charge, while blue and red edges have $+$ and $-$ charges.

		\item{\bf $P$ and $Q$ labels:} A solid edge may have a label $P_x$ or $Q_x$ for some vertex $x$ in the graph. Moreover, every solid edge has at most one $P$ or $Q$ label.

		\item{\bf Coefficients:} There is a coefficient associated with every graph.

	\end{itemize}
	
\end{definition}

During the proof, we will introduce more types of edges. Next, we introduce the concept of subgraphs. 
\begin{definition}[Subgraphs]\label{def_sub}
	A graph $\cal G_1$ is said to be a subgraph of $\cal G_2$, denoted by $\cal G_1\subset \cal G_2$, if every graphical component (except the coefficient) of $\cal G_1$ is also in $\cal G_2$. Moreover, $\cal G_1 $ is a proper subgraph of $\cal G_2$ if  $\cal G_1\subset \cal G_2$ and $\cal G_1\ne \cal G_2$. Given a subset $\cal S$ of vertices in a graph $\cal G$, the subgraph $\cal G|_{\cal S}$ induced on $\cal S$ refers to the subgraph of $\cal G$ consisting of vertices in $\cal S$ and the edges between them. Given a subgraph $\cal G$, its closure $\overline{\cal G}$ refers to $\cal G$ plus its external edges. 
\end{definition}

To each graph or subgraph, we assign a \emph{value} as follows. 

\begin{definition}[Values of graphs]\label{ValG} 
Given a graph $\mathcal G$, we define its value as an expression obtained as follows. We first take the product of all the edges and the coefficient of the graph $\cal G$. Then, for the edges with the same $P_x$ or $Q_x$ label, we apply $P_x$ or $Q_x$ to their product. Finally, we sum over all the internal indices represented by the internal vertices, while the external indices are fixed by their given values. The value of a linear combination of graphs  $\sum_i c_i \cal G_i$ is naturally defined as the linear combination of the graph values of $\cal G_i$.
\end{definition}

For simplicity, throughout this paper, we will always abuse the notation by identifying a graph (a geometric object) with its value (an analytic expression). In this sense, noticing that a free edge represents an $(N\eta)^{-1}$ factor without indices, two graphs are equivalent if they have the same number of free edges and all other graph components are the same. In other words, we can move a free edge freely to another place without changing the graph.


Our graphs will have a three-level structure: some microscopic structures varying on scales of order 1, called \emph{atoms}, which are the equivalence classes of vertices connected through dotted edges; some mesoscopic structures varying on scales of order $W$, called \emph{molecules}, which are the equivalence classes of atoms connected through waved edges; a global (macroscopic) structure varying on scales up to $L$. Given a graph defined in \Cref{def_graph1}, if we ignore its inner structure within each atom, we will get an atomic graph with vertices being atoms; if we ignore the inner structure within each molecule, we will get a molecular graph  with vertices being molecules. The atomic and molecular graphs will be very useful in organizing the graphs and understanding their mesoscopic and macroscopic structures. We now give their formal definitions.

\begin{definition}[Atoms and atomic graphs]\label{def_atom}

We partition the set of all vertices into a union of disjoint subsets called atoms. Two vertices belong to the same atom if and only if they are connected by a path of dotted edges (i.e., dotted edges, dotted edges with $\Delta$, and blue/red dotted edges). Every atom containing at least one external vertex is called an external atom; otherwise, it is an internal atom.  

 Given a graph $\cal G$, we obtain its atomic graph as follows:
	\begin{itemize}
		\item merge all vertices in the same atom and represent them by a vertex;
		
		\item keep solid, waved, $\dashed$, and free edges between different atoms;
		
		\item discard all the other components in $\cal G$ (including $\times$-dotted edges, edges between vertices in the same atom, coefficients, and $P/Q$ labels).
	\end{itemize}

	
\end{definition}

Note that for the Anderson orbital and Wegner orbital models, there is only one type of dotted edge, i.e., those representing the $\mathbf 1_{x=y}$ factors. In particular, after merging vertices connected by dotted edges as we have done for random band matrices, the atoms are just vertices of our vertex-level graphs. Therefore, following \cite{BandI,BandII,BandIII}, we can use the notions of atoms and vertices interchangeably for the Anderson orbital and Wegner orbital models. 
On the other hand, for the block Anderson model, the notion of atoms above takes a different meaning from that in \cite{BandIII,BandI,BandII} (the notion of molecules in \Cref{def_poly} below, however, is unchanged) due to the two new types of dotted edges. 
The current notion of atoms is convenient and has a great advantage: the atomic graphs in the current paper indeed have the same structures as the atomic graphs in \cite{BandIII,BandI,BandII}. 
Hence, many graphical tools and results developed in these works apply verbatim to all our three settings of RBSO.

By the definitions of dotted edges, the definition of $\Psi^{\BA}$, and \eqref{Mbound} for the block Anderson model, we see that up to an error (in the sense of graph values) of order $\OO(W^{-D})$, 
\be\label{scaleatom}
 \hbox{$x$, $y$ are in the same atom} \implies  \|x-y\|\le C_{D,\xi},
\ee
for a constant $C_{D,\xi}$ depending on $D$ and $\xi$ in \eqref{eq:small+lambda}.


\begin{definition}[Molecules and molecular graphs]\label{def_poly}
We partition the set of all atoms into a union of disjoint subsets called molecules. Two atoms belong to the same molecule if and only if they are connected by a path of waved edges. Every molecule containing at least one external atom is called an external molecule; otherwise, it is an internal molecule. With a slight abuse of notation, when we say a vertex $x$ belongs to a molecule $\cal M$, we mean that $x$ belongs to an atom in $\cal M$. 


Given a graph $\cal G$, we define its molecular graph as follows. 
	\begin{itemize}
		\item merge all vertices in the same molecule and represent them by a vertex;
		
		\item keep solid, $\dashed$, and free edges between molecules;

		\item discard all the other components in $\cal G$ (including $\times$-dotted edges, edges between vertices in the same molecule, coefficients, and $P/Q$ labels).
	\end{itemize}
\end{definition}

By the estimates on the waved edges---derived from the definitions of $S$ and $\Psi^{\BA}$, the bound \eqref{Mbound_AO} for the Anderson orbital model, and the estimate \eqref{S+xy}---we see that up to an error of order $\OO(W^{-D})$,  
\be\label{scalemole}
 \hbox{$x$, $y$ are in the same molecule} \implies  \|x-y\|\le C_{D}W\log W,
\ee
for a constant $C_D$ depending on $D$.
In this paper, atomic and molecular graphs are used solely to help with the analysis of graph structures, with all graph expansions applied exclusively to vertex-level graphs. Following the notion \cite{BandIII,BandI,BandII}, we shall refer to the structure of the molecular graph as the {\it global structure}.  

\begin{definition}[Normal graphs]  \label{defnlvl0}  
	We say a graph $\cal G$ is \emph{normal} if it contains at most $\OO(1)$ many vertices and edges, and all its internal vertices are connected together or to external vertices through paths of dotted, waved, or $\dashed$ edges.
\end{definition}

All graphs appearing in our proof are normal. Given a normal graph, we can define its scaling size as in \Cref{def scaling} below. For the Anderson orbital model, to define the scaling size,  we introduce one more definition, called ``\emph{big atoms}", which is an intermediate structure between atom and molecule. 

\begin{definition}[Big atoms for AO]
We partition the set of all vertices into a union of disjoint subsets called big atoms. Two vertices belong to the same big atom if and only if they are connected by a path of dotted edges, $M$-waved edges, and $\Delta$-waved edges. Every big atom containing at least one external vertex is called an external big atom; otherwise, it is an internal big atom.
\end{definition}
 



\begin{definition}[Scaling size]\label{def scaling} 
    Given a normal graph $\cal G$, suppose every solid edge in the graph has a $\circ$ (i.e., it represents a $\Gc$ factor). Then, we define the scaling size of $\cal G$ as follows for the three models:
    \begin{itemize}

\item For the block Anderson and Wegner orbital models,
 \begin{align}
		\size(\cal G): =&|\cof(\cal G)|\times \left( \frac{1}{N\eta}\frac{L^2}{W^2}\right)^{\#\{\text{free edges}\}}\heta^{\#\{\text{diffusive edges}\}+\frac{1}{2}\#\{\text{solid edges}\}} \nonumber\\
  &\times  W^{-d(\#\{ \text{waved edges}\}  -  \#\{\text{internal atoms}\})}.
  \label{eq_defsize}
	\end{align}


 \item For the Anderson orbital model, 
\begin{align}\
\size(\cal G): =&|\cof(\cal G)|\times \left( \frac{1}{N\eta}\frac{L^2}{W^2}\right)^{\#\{\text{free edges}\}}\heta^{\#\{\text{diffusive edges}\}+\frac{1}{2}\#\{\text{solid edges}\}} \nonumber\\
  &\times  W^{-d(\#\{ \text{normal waved edges}\}  -  \#\{\text{internal big atoms}\})}, \label{eq_defsize_AO}
	\end{align}
where \emph{normal waved edges} refer to waved edges without the labels $M$ or $\Delta$.

    \end{itemize}
Here, $\cof(\cal G)$ denotes the coefficient of $\cal G$, and for simplicity, we have adopted the notation  
 \be\label{eq:heta}
 \heta:={\blam}/{W^{d}}. 
 \ee
If $\cal G=\sum_k \cal G_k$ for a sequence of normal graphs $\cal G_k$ whose solid edges all have label $\circ$, then we define 
$$\size(\cal G):=\max_k \size(\cal G_k).$$ 
With this convention, given a normal graph $\cal G$ that has $G$ edges, $\size(\cal G)$ is defined by treating $\cal G$ as a sum of graphs by expanding each $G$ or $G^-$ edge as $\Gc+M$ or $\Gc^- + M^-$. 
\end{definition}

Assuming an important graphical property, the \emph{doubly connected property} (see \Cref{def 2net}), for a graph $\cal G$, we will see that $\size(\cal G)$ indeed accurately reflects the order of the value of $\cal G$ (as shown in \Cref{no dot} below) up to a $W^\tau$ factor. Take the block Anderson/Wegner orbital model as an example, the motivation behind the factors in the definition \eqref{eq_defsize} is as follows: 
\begin{itemize}
    \item every waved edge provides a factor $W^{-d}$ due to the definition of $S$ and \Cref{lem deter}; 

    \item every internal atom corresponds to an equivalence class of summation indices: by \eqref{scaleatom}, each atom contains only one free vertex, and all other vertices are constrained in its $\OO(1)$ neighborhood.

\end{itemize}    
In a doubly connected graph, every summation over an internal molecule actually involves the summation of at least one diffusive edge and at least one solid or free edge. (Here, the summation over an internal molecule is roughly understood as the summation over a vertex in the \emph{molecular graph}, which only involves diffusive, solid, and free edges.) We can bound the summation over a diffusive edge and a free edge by  
$$ \sum_{y}B_{xy}\frac{1}{N\eta} \prec W^d \cdot \frac{\blam}{W^d} \cdot \left(\frac{1}{N\eta}\frac{L^2}{W^2}\right).$$
If we know that \eqref{locallaw} holds, then we can bound the summation over a diffusive edge and a solid edge by  
\be\label{eq:diffu+solid}
\sum_{y}B_{xy}\left( B_{yw}+ \frac{1}{N\eta}\right)^{1/2} \prec W^d \cdot \frac{\blam}{W^d} \cdot \left(\frac{\blam}{W^d}+ \frac{1}{N\eta}\frac{L^4}{W^4}\right)^{1/2} \lesssim W^d \heta^{3/2}\quad \text{for}\quad \eta\ge W^\fd\eta_* .
\ee
As a consequence, for $\eta\ge W^\fd\eta_*$,
\begin{itemize}
\item every diffusive edge provides a factor of $\heta=\max_{x,y}B_{xy}$ due to \eqref{thetaxy}; 

\item every free edge provides a factor of $\frac{1}{N\eta}\frac{L^2}{W^2}$; 
    
\item every $\Gc$ edge contributes a factor of \smash{$ \heta^{1/2}$} due to the local law \eqref{locallawmax} and the estimate \eqref{eq:diffu+solid}.

\end{itemize}

We now introduce the crucial \emph{doubly connect property} introduced in \cite{BandI}. All graphs in our $T$-expansions will satisfy this property. 

\begin{definition} [Doubly connected property]\label{def 2net}  
An graph $\cal G$ without external molecules is said to be doubly connected if its molecular graph satisfies the following property. There exists a collection, say $\cal B_{black}$, of $\dashed$ edges, and another collection, say $\cal B_{blue}$, of blue solid, $\dashed$, and free edges such that 
\begin{itemize}
    \item[(a)] $\cal B_{black}\cap \cal B_{blue}=\emptyset$;
    \item[(b)] each of $\cal B_{black}$ and $\cal B_{blue}$ contains a spanning tree that connects all molecules. 
\end{itemize}
(Note that red solid edges are not used in either collection.) For simplicity of notations, we will call the edges in $\cal B_{black}$ as black edges and the edges in $\cal B_{blue}$ as blue edges. Correspondingly, $\cal B_{black}$ and $\cal B_{blue}$ are referred to as black and blue nets, respectively.

A graph $\cal G$ with external molecules is called doubly connected if its subgraph with all external molecules removed is doubly connected.
\end{definition}

Finally, we introduce the following two types of graphs that will appear in our $T$-expansions. 

\begin{definition}[Recollision and $Q$-graphs]\label{Def_recoll}
{(i)} Given a subset of vertices, say $\cal V$, we say a graph is a \emph{recollision graph} with respect to $\cal V$ if there is at least one dotted or waved edge connecting an internal vertex to a vertex in $\cal V$. The subset of vertices $\cal V$ we are referring to will be clear from the context.  

 	
\medskip
\noindent{(ii)} We say a graph is a \emph{$Q$-graph} if all $G$ edges and weights in the graph have the same $Q_x$ label for a specific vertex $x$. 
In other words, the value of a $Q$-graph can be written as $Q_x(\Gamma)$ for an external vertex $x$ or $\sum_x Q_x(\Gamma)$ for an internal vertex $x$, where $\Gamma$ is a graph without $P/Q$ labels. 
\end{definition}

\subsection{Self-energy}\label{sec_selfc}

We now introduce the graphical component to represent the $\self$ in \Cref{thm_diffu}. 

\begin{definition}[$\Self$]\label{collection elements}
Given $z=E+\ii\eta$ and some $\eta_0\in [\eta_*,1]$, a {\bf$\self$} $\cal E\equiv \cal E(z,W,L)$ with $\eta\ge \eta_0$ is a deterministic matrix satisfying the following properties: 

	
\begin{itemize}
\item[(i)] For any $x,y \in \Z_L^d$, $\Selek_{xy}$ is a sum of $\OO(1)$ many deterministic graphs with external vertices $x$ and $y$. These graphs consist of dotted edges (there may also be a dotted edge between $x$ and $y$), waved edges, and diffusive edges. Every graph in $\cal E$ is doubly connected, and each diffusive edge in it is redundant, that is, the graph obtained by removing the diffusive edge still obeys the doubly connected property.  
		
		
\item[(ii)] $\Sele$ satisfies \eqref{self_sym} and the following two properties for $z= E+ \ii \eta$ with $|E|\le  2- \kappa$ and $\eta \in [\eta_0, 1]$: 
 \be\label{4th_property0}
 |{\cal E}(x,y)| \le \frac{\sizeself(\Selek)}{\blam} \frac{W^2}{ \langle x-y\rangle^{d+2-c}} , \quad \forall \ x,y\in \Z_L^d,
\ee
    \be\label{3rd_property0}
	W^{-d}\sum_{y\in[0]}\sum_{x\in \Z_L^d}  {\cal E}(x,y)  \le   L^c \sizeself(\Selek)   \left(   \eta + t_{Th}^{-1}  \right),   
	\ee
for a deterministic parameter $0\le \sizeself(\Selek)\lesssim \blam^2 /W^d$ defined as 
\be\label{eq:psiSele_size}\psi(\Sele):=\blam W^d\size(\Sele),\ee
and for any small constant $c>0$. 
Here, we have abbreviated that $\size(\Sele)\equiv \size( \Sele_{xy})$ (since $\Sele$ consists of deterministic edges only, $\size(\Sele_{xy})$ does not depend on $x, y \in \Z_L^d$). 

\item[(iii)] For the Wegner orbital model, $\Sele$ can be written as 
$\Sele=\Sele_{\LK}\otimes \mathbf E.$
 \end{itemize} 
If $\cal E$ satisfies all the above properties except  \eqref{3rd_property0}, then we call it a {\bf$\pself$}. 
\end{definition}
\begin{remark}

We remark that the above properties \eqref{4th_property0} and \eqref{3rd_property0} are defined for all $L\ge W$ \emph{without} the constraint $W\ge L^\delta$. In particular, if we let $L\to \infty$ and $\eta\equiv \eta_L\to 0$, the property \eqref{3rd_property0} becomes 
\be\label{3rd_propertyinf}
W^{-d}\sum_{y\in[0]}\sum_{x\in \Z^d}  {\cal E}^\infty (x,y) =0 ,
\ee
where ${\cal E}^\infty$ denotes the infinite space limit of $\cal E$ (see \Cref{def infspace} below). Hence, we will refer to \eqref{3rd_property0} as the \emph{sum zero property}.
    
\end{remark}

Given a $\self$ $\Sele$, we have defined $\zthn(\Sele)$ in \eqref{theta_renormal1}. Similarly, we can also define $\thn(\Sele)$ as 
\be\label{theta_renormal10}
\thn(\Selek):= (1-\thn \Sele )^{-1}\thn .
\ee
We will also use another type of renormalized $\dashed$ matrices which appear from the Taylor expansion of $\zthn(\Sele)$ or $\thn(\Sele)$: $\zthn\Sele_1\zthn  \cdots \zthn \Sele_k \zthn$ or $\thn\Sele_1\thn \cdots \thn \Sele_k \thn$, where $\Sele_1,\ldots, \Sele_k$ is a sequence of $\selfs$.  
The following lemma demonstrates that these matrices satisfy a similar bound as in \eqref{thetaxy}. Its proof is a straightforward application of \Cref{lem:label_diffusive} and will be provided in Appendix \ref{sec:pf_redundantagain}. 

\begin{lemma}\label{lem redundantagain}
Under the assumptions of \Cref{thm_locallaw}, let $\Sele_0,\Sele_1,\ldots, \Sele_k$ be a sequence of $\selfs$ satisfying Definition \ref{collection elements}. Then, for $z=E+\ii \eta$ with $|E|\le  2- \kappa$ and $\eta \in [ \eta_0, 1]$, we have  
 	\be\label{BRB} 
\zthn_{xy}(\Sele_0)\prec  B_{xy}, \quad \left(\zthn\Sele_1\zthn\Sele_2\zthn \cdots \zthn \Sele_k \zthn\right)_{xy} \prec B_{xy}\prod_{i=1}^k \sizeself(\Sele_i), \quad \forall \ x,y \in \Z_L^d.
 \ee
If $\eta\ge t_{Th}^{-1}$, the first bound in \eqref{BRB} also holds for $\thn(\Sele)$, and the second one holds if we replace some $\zthn$ by $\thn$. 
\end{lemma}

Due to \eqref{BRB}, to simplify the graphical structures, we will regard the renormalized diffusive matrices as a new type of diffusive edges. 

\begin{definition}[Labeled diffusive edges] \label{def_graph comp} 
We represent $\thn_{xy}^s(\Selek)$, $s\in \{\emptyset,\circ\}$, by a diffusive edge between vertices $x$ and $y$ with label $(s,\Selek)$, and represent  
\be\label{eq:labeldiff}
\thn^{s_1}\Sele_1\thn^{s_2}\Sele_2\thn^{s_3} \cdots \thn^{s_k} \Sele_k \thn^{s_{k+1}},\quad (s_1,\ldots, s_{k+1})\in \{\emptyset,\circ\}^{k+1},\ee 
by a diffusive edge between vertices $x$ and $y$ with label $[s_1,\Sele_1,s_2,\Sele_2, \ldots, s_k,\Sele_k,s_{k+1}]$. Here, as a convention, $\thn^{\emptyset}$ and $\thn^{\circ}$ represent $\thn$ and \smash{$\zthn$}, respectively. A labeled $\dashed$ edge is drawn as a double-line edge with a label but without any internal structure. 
When calculating scaling sizes, an $(s,\Selek)$-labeled diffusive edge is counted as an edge of scaling size $\heta$, and an $[s_1,\Sele_1,s_2,\Sele_2, \ldots, s_k,\Sele_k,s_{k+1}]$-labeled diffusive edge is counted as an edge of scaling size $\heta \prod_{i=1}^k \sizeself(\Sele_i)$. 
Labeled diffusive edges are used as diffusive edges in the doubly connected property in \Cref{def 2net}, i.e., they can appear in both black and blue nets. 
\end{definition}

\begin{remark}
The diffusive edges in (i) of \Cref{collection elements} may also be labeled diffusive edges. More precisely, suppose we know that $\Sele_0,\Sele_1,\ldots, \Sele_k$ are $\selfs$ with $\eta\ge \eta_0$. Then, when constructing or defining a new $\self$ with $\eta\ge \eta_0$, such as $\Sele_{k+1}$, it may contain labeled diffusive edges of the forms shown in \eqref{BRB}. 
\end{remark}	

The properties (i) and (iii) in \Cref{collection elements} are due to our construction of the $\selfs$; see \Cref{sec:constr} below. The property \eqref{self_sym} is due to the parity and block translation symmetries of our models. The property \eqref{4th_property0} is a consequence of the following estimate on deterministic doubly connected graphs.

\begin{lemma}[Corollary 6.12 of \cite{BandI}]\label{dG-bd}
Let $\mathcal G$ be a deterministic doubly connected graph without external vertices. Pick any two vertices of $\mathcal G$ and fix their values as $x,y \in \Z^d_L$. The resulting graph $\mathcal G_{xy}$ satisfies that
\begin{equation}\label{det_double}
\begin{split}
|\mathcal G_{xy}|\le \mathcal G^{abs}_{xy} \le  W^{2d-4}\size(\cal G_{xy}) \cdot \langle x-y\rangle^{-(2d-4)+c}
\end{split}
\end{equation}
for any small constant $c>0$, where $ {\cal G}_{{\rm abs}}$ is obtained by replacing each component (i.e., edges and coefficients) in $\cal G$ with its absolute value. 
\end{lemma}

Hence, the major nontrivial property in \Cref{collection elements} is the sum zero property \eqref{3rd_property0}, which indicates a delicate cancellation occurring within the $\self$. The next lemma shows that exchanging $\thn$ and \smash{$\zthn$} edges in a $\self$ still results in a $\self$ that satisfies \Cref{collection elements} when $\eta\ge t_{Th}^{-1}$. Its proof will be provided in \Cref{subsect:pfthn-zthn}. 

\begin{lemma}\label{lem:thn-zthn}
In the setting of \Cref{thm_locallaw}, fix any $z=E+\ii \eta$ with $|E|\le 2-\kappa$ and $\eta\ge t_{Th}^{-1}$. Given a $\self$ $\Sele$, if we replace an arbitrary \smash{$\zthn$} edge in it with a $\thn$ edge or vice versa, the new term $\Sele^{new}$ is still a $\self$ in the sense of \Cref{collection elements}. 
\end{lemma}

\section{$T$-expansion}\label{sec:Texp}

To facilitate clarity of presentation, we focus in the remainder of the main text on the proof of the most important and technically involved result---the local law, \Cref{thm_locallaw}---for the Wegner orbital model. 
We choose this model as our primary focus because its proof involves simpler notation, and its various concepts most closely resemble those used for the RBM  \cite{BandI,BandII,BandIII}, 
while still capturing the main innovations of this work. 
Specifically, the proof for the Wegner orbital model already highlights the two core novelties of our approach: the proof of the continuity estimate (cf.~\Cref{sec:continuity}), and the coupling renormalization mechanism (cf.~\Cref{sec:Vexpansion}). 
The extension of our arguments to the block Anderson and Anderson orbital models is straightforward, requiring only minor modifications.
This extension will be discussed in detail in \Cref{sec:ext}. Furthermore, the quantum diffusion and QUE estimates follow as consequences of the local law and the technical tools developed in its proof; these will be presented in \Cref{sec:pf_Qdiff,sec:QUE}. 
Finally, we emphasize that unless a lemma explicitly refers to the Wegner orbital model in its statement, it should be understood to apply to all three models (although the corresponding proofs are deferred to the appendix).


Similar to \cite{BandI,BandII,BandIII}, our proof starts from the \emph{$T$-expansions} of $T$-variables, which are defined as 
\be\label{general_T}
\begin{split}
T_{x,yy'}:=\sum_{\al}S_{x\al}G_{\al y}G^-_{\al y'},\quad &T_{yy',x}:=\sum_{\al}G_{y\al}G^-_{y' \al}S_{\al x},\\
 \zT_{x,yy'}:=\sum_w P^\perp_{xw}T_{w,yy'},\quad &\zT_{yy',x}:=\sum_w T_{yy',w}P^\perp_{xw}. 
 \end{split}
\ee
Note that they extend the $T$-variables in \eqref{eq:TT} with $T_{x,yy}=T_{xy}$ and $T_{yy,x}=\wt T_{yx}$. These $T$-variables take the same form as the $T$-variables for random band matrices \cite{BandI,BandII,BandIII}, and we will use them in the proofs concerning the Wegner orbital model.

\subsection{Preliminary $T$-expansion: Wegner orbital model}\label{sec_second_WO}

We first present the preliminary $T$-expansion of the $T$-variables in \eqref{general_T}. We start with the following expansion that has been derived in \cite{BandIII,BandI} in the context of random band matrices. 

\begin{lemma}[Lemma 2.5 of \cite{BandI}]\label{2nd3rd T}
For the Wegner orbital model in \Cref{def: BM}, we have that
\begin{align}\label{seconduniversal}
& T_{\fa,\fb_1\fb_2}= m \thn_{\fa\fb_1} \overline G_{\fb_1\fb_2}  + \sum_x \thn_{\fa x}\left( \mathcal A^{(>2)}_{x,\fb_1\fb_2}  + \mathcal Q^{(2)}_{x,\fb_1\fb_2}  \right),
\end{align}
for any $\fa,\fb_1,\fb_2\in \Z_L^d$, where
 \begin{align}
		\mathcal A^{(>2)}_{x,\fb_1\fb_2}&:= m \sum_{y} S_{xy}\left[   (G_{yy}-m) G_{x \fb_1} \overline G_{x \fb_2} \nonumber+  ( \overline G_{xx} -\overline m)G_{y\fb_1}\overline G_{y\fb_2}\right]\,,\nonumber\\
		\mathcal Q^{(2)}_{x,\fb_1\fb_2}&:= Q_x \left(G_{x\fb_1} \overline G_{x\fb_2} \right)- |m|^2\sum_{y}  S_{xy}  Q_x\left( G_{y\fb_1}\overline G_{y\fb_2}  \right) - m  \delta_{x\fb_1}  Q_{\fb_1}\left(  \overline G_{\fb_1\fb_2} \right) \label{eq:Q2} \\
		&\quad\ - m\sum_{y}  S_{xy}  Q_x \left[   (G_{yy}-m) G_{x\fb_1} \overline G_{x\fb_2}  + ( \overline G_{xx}-\bar m) G_{y\fb_1}\overline G_{y\fb_2}  \right]\,.\nonumber
	\end{align}
\end{lemma}

\begin{lemma}[Lemma 2.1 of \cite{BandIII}]
	\label{2ndExp}
For the Wegner orbital model in \Cref{def: BM}, we have that
	\begin{equation}
		\label{eq:2nd}
		\begin{split}
			\zT_{\fa,\fb_1 \fb_2}&= m  \zthn_{\fa \fb_1}\overline G_{\fb_1\fb_2} + \sum_x \zthn_{\fa x}\left( \mathcal A^{(>2)}_{x,\fb_1\fb_2}  + \mathcal Q^{(2)}_{x,\fb_1\fb_2}  \right)\,,
		\end{split}
	\end{equation}
	for any $\fa,\fb_1,\fb_2\in \Z_L^d$. 
\end{lemma}

\begin{remark}
The difference between \eqref{eq:2nd} and \eqref{seconduniversal} is that we have picked out the zero mode, i.e., \be\label{eq:T-cT}T_{\fa,\fb_1\fb_2}-\zT_{\fa,\fb_1\fb_2}=N^{-1}\sum_{\fa}T_{\fa,\fb_1\fb_2}=\frac{ G_{\fb_2 \fb_1} - \overline G_{\fb_1 \fb_2}}{2\ii N\eta},\ee 
in the expansion \eqref{eq:2nd}. In particular, \eqref{seconduniversal} will be used to study the ``large $\eta$ case" with $\eta\ge t_{Th}^{-1}$, in which case the zero mode is negligible compared to the diffusive part, while \eqref{eq:2nd} aims to deal with the ``small $\eta$ case" with $\eta\le t_{Th}^{-1}$. 
\end{remark}

The above expansions \eqref{seconduniversal} and \eqref{eq:2nd} were used as the preliminary $T$-expansions in \cite{BandI} and \cite{BandIII}, respectively, and are referred to as the \emph{second order $T$-expansions} in those works. 
However, for our purpose, we will employ a different preliminary $T$-expansion obtained by performing further expansions with respect to the light weights (i.e., self-loops) $G_{yy}-m$ and $\bar G_{xx}-\bar m$ in \smash{$\mathcal A^{(>2)}_{x,\fb_1\fb_2}$}, which corresponds to the \emph{third order $T$-expansion} in \cite{BandI,BandIII}. 

\begin{lemma}[Preliminary $T$-expansion for $\WO$]
    \label{3rdExp}
For the Wegner orbital model, we have that
\begin{align}
		\label{eq:3rd0}
	T_{\fa,\fb_1 \fb_2}&= m  \thn_{\fa \fb_1}\overline G_{\fb_1\fb_2} + \sum_x \thn_{\fa x}\left( \mathcal A^{(>3)}_{x,\fb_1\fb_2}  + \mathcal Q^{(3)}_{x,\fb_1\fb_2}  \right)\, , \\
		\label{eq:3rd}
		\zT_{\fa,\fb_1 \fb_2}&= m  \zthn_{\fa \fb_1}\overline G_{\fb_1\fb_2} + \sum_x \zthn_{\fa x}\left( \mathcal A^{(>3)}_{x,\fb_1\fb_2}  + \mathcal Q^{(3)}_{x,\fb_1\fb_2}  \right)\,, 
	\end{align}
for any $\fa,\fb_1,\fb_2\in \Z_L^d$, where 
 \begin{align*}
\mathcal A^{(>3)}_{x,\fb_1\fb_2}:=&~ m^2\sum_{\al,\beta}  S^+_{x\al}S_{\al\beta} \left[\Gc_{\al\al} \Gc_{\beta\beta} G_{x \fb_1} \overline G_{x \fb_2} +   G_{x\beta}G_{\beta\al}G_{\al \fb_1} \overline G_{x \fb_2} +  G_{\beta\al}\overline G_{x \al}  G_{x \fb_1} \bar G_{\beta\fb_2}\right]\,\\
+&~|m|^2\sum_{y,\al} S_{xy}S_{x\al} \Gc^-_{xx}  \Gc^-_{\al\al} G_{y\fb_1}\overline G_{y\fb_2} + |m|^2\bar m^2 \sum_{y,\al,\beta} S_{xy}S^-_{x\al}S_{\al\beta} \Gc^-_{\al\al} \Gc^-_{\beta\beta} G_{y\fb_1}\overline G_{y\fb_2} \\
+&~|m|^2 \sum_{y,\al} S_{xy}S_{x\al}\left[ \bar G_{\al x} G_{yx} G_{\al \fb_1}\overline G_{y\fb_2} +   G_{y\fb_1}\overline G_{y\al}\bar G_{\al x} \bar G_{x\fb_2} \right]\\
+&~|m|^2\bar m^2 \sum_{y,\al,\beta} S_{xy}S^-_{x\al}S_{\al\beta} \bar G_{\beta\al}  G_{y\al}G_{\beta\fb_1}\overline G_{y\fb_2} + |m|^2\bar m^2 \sum_{y,\al,\beta} S_{xy}S^-_{x\al}S_{\al\beta}  G_{y\fb_1}\overline G_{y\beta}\bar G_{\beta\al}\bar G_{\al\fb_2} ,\\
\mathcal Q^{(3)}_{x,\fb_1\fb_2}:=&~\mathcal Q^{(2)}_{x,\fb_1\fb_2} + m \sum_{\al} S^+_{x\al}  Q_\al\left[ \Gc_{\al\al} G_{x \fb_1} \overline G_{x \fb_2}\right] - m^3 \sum_{\al,\beta} S^+_{x\al}S_{\al\beta} Q_\al\left[\Gc_{\beta\beta} G_{x \fb_1} \overline G_{x \fb_2}\right]\\
+&~ m \sum_{y} S_{xy} Q_x\left[\Gc^-_{xx} G_{y\fb_1}\overline G_{y\fb_2}\right] +  |m|^2\bar m \sum_{y,\al} S_{xy} S^-_{x\al}Q_x\left[\Gc^-_{\al\al} G_{y\fb_1}\overline G_{y\fb_2}\right]\\
-&~|m|^2 \bar m \sum_{y,\al} S_{xy}S_{x\al} Q_x\left[\Gc^-_{\al\al}  G_{y\fb_1}\overline G_{y\fb_2}\right]-|m|^2 \bar m^3 \sum_{y,\al,\beta} S_{xy}S^-_{x\al}S_{\al\beta} Q_\al\left[\Gc^-_{\beta\beta} G_{y\fb_1}\overline G_{y\fb_2}\right] \\
-&~m^2 \sum_{\al,\beta} S^+_{x\al}S_{\al\beta} Q_\al\left[\Gc_{\beta\beta}\Gc_{\al\al} G_{x \fb_1} \overline G_{x \fb_2}+G_{x\beta}G_{\beta\al} G_{\al\fb_1} \overline G_{x \fb_2}+G_{\beta\al}\overline G_{x\al} G_{x \fb_1} \bar G_{\beta\fb_2}\right] \\
-&~|m|^2  \sum_{y,\al} S_{xy}S_{x\al} Q_x\left[\Gc^-_{\al\al}  \Gc^-_{xx} G_{y\fb_1}\overline G_{y\fb_2}\right]-|m|^2 \bar m^2 \sum_{y,\al,\beta} S_{xy}S^-_{x\al}S_{\al\beta} Q_\al\left[\Gc^-_{\beta\beta} \Gc^-_{\al\al} G_{y\fb_1}\overline G_{y\fb_2}\right] \\
-&~|m|^2 \sum_{y,\al} S_{xy}S_{x\al}Q_x\left[ \bar G_{\al x} G_{yx} G_{\al \fb_1}\overline G_{y\fb_2}+G_{y\fb_1}\overline G_{y\al}\bar G_{\al x} \bar G_{x\fb_2}\right]\\
-&~|m|^2\bar m^2 \sum_{y,\al,\beta} S_{xy}S^-_{x\al}S_{\al\beta} Q_\al\left[\bar G_{\beta\al}  G_{y\al}G_{\beta\fb_1}\overline G_{y\fb_2}+G_{y\fb_1}\overline G_{y\beta}\bar G_{\beta\al}\bar G_{\al\fb_2}\right] .
\end{align*}

\end{lemma}
\begin{proof}
This is obtained by further expanding the light weights $G_{yy}-m$ and $\bar G_{xx}-\bar m$ in $\mathcal A^{(>2)}_{x,\fb_1\fb_2}$ using \Cref{ssl} below. We also used the identity $ S(1+m^2S^+)=S^+$ in the derivation.
\end{proof}

\subsection{$T$-expansion and $\pTexp$}

We are ready to define the concepts of $T$-expansions and $\pTexp$s for the Wegner orbital model, which are the main tools for our proof. 

\begin{definition} [$T$-expansion]\label{defn genuni}
In the setting of \Cref{thm_locallaw}, let $z=E+\ii \eta$ with $|E|\le 2-\kappa$ and $\eta\ge \eta_0$ for some $\eta_0\in [\eta_*,1]$. Suppose we have a sequence of $\selfs$ $\cal E_k$, $k=1, \ldots, k_0$, satisfying \Cref{collection elements} with $\eta\ge \eta_0$ and the following properties:
\begin{itemize}
\item[(i)] 
$ \sizeself(\Sele_k) \prec W^{-\fc_k} $ for a sequence of strictly increasing constants $0< \fc_k\le \fC$. (In our construction, our $\selfs$ actually satisfy that $c_k-c_{k-1}\ge \min\{d-2\log_W \blam,2\xi\}$ with the convention $c_0=0$.)

\item[(ii)] The diffusive edges in them may be labeled diffusive edges. Moreover, for each labeled diffusive edge in $\Sele_k$, its labels only involve $\selfs$ $\Sele_1, \ldots, \Sele_{k-1}$.


\item[(iii)] The (unlabeled) diffusive edges in $\Sele_k$, $k=1,\ldots, k_0$, are all $\zthn$ edges
(here, we consider each labeled diffusive edge as a subgraph comprising the components in \Cref{def_graph1} rather than a single double-line edge with a label). 
\end{itemize}
Let $\Sele=\sum_{k=1}^{k_0} \Sele_k$, and denote by $\Sele_k'$ the matrix obtained by replacing each $\zthn$ edge in $\Sele_k$ by a $\thn$ edge. (Note that by \Cref{lem:thn-zthn}, $\Sele_k'$ is still a $\self$ when $\eta\ge t_{Th}^{-1}$.) 
Then, for the Wegner orbital model, a $T$-expansion of \smash{$\zT_{\fa,\fb_1\fb_2}$} up to order $\fC$ (with $\eta\ge \eta_0$ and $D$-th order error) is an expression of the form 
\begin{equation}\label{mlevelTgdef}
\zT_{\fa,\fb_1 \fb_2}=  m  \zthn(\Sele)_{\fa \fb_1}\overline G_{\fb_1\fb_2}  + \sum_x \zthn(\Sele)_{\fa x} \left[\PT_{x,\fb_1 \fb_2} +  \AT_{x,\fb_1 \fb_2}  + \WT_{x,\fb_1 \fb_2}  + \QT_{x,\fb_1 \fb_2}  +  (\Err_{D})_{x,\fb_1 \fb_2}\right]\, ;
\end{equation}
a $T$-expansion of $T_{\fa,\fb_1\fb_2}$ up to order $\fC$ (with $\eta\ge \eta_0\vee t_{Th}^{-1}$ and $D$-th order error) is an expression of the form
\begin{equation}\label{mlevelTgdef_orig}
T_{\fa,\fb_1 \fb_2}=  m  \thn(\Sele')_{\fa \fb_1}\overline G_{\fb_1\fb_2} + \sum_x \thn(\Sele')_{\fa x} \left[\PT'_{x,\fb_1 \fb_2} +  \AT'_{x,\fb_1 \fb_2} + \QT'_{x,\fb_1 \fb_2}  +  (\Err_{D}')_{x,\fb_1 \fb_2}\right]\, .
\end{equation}
Here, $\PT',$ $  \AT'$, $ \QT'$, and $\Err_{D}'$ are obtained by replacing each $\zthn$ edge in $\PT,$ $  \AT$, $ \QT$, and $\Err_{D}$ by a $\thn$ edge. 
Moreover, we require the graphs in $\PT,$ $  \AT$, $\WT$, $ \QT$, and $\Err_{D}$ to satisfy the following properties. 
\begin{enumerate}

\item 
Each of 
$\PT_{x,\fb_1 \fb_2},$ $  \AT_{x,\fb_1 \fb_2}$, $\WT_{x,\fb_1\fb_2}$, $ \QT_{x,\fb_1 \fb_2}$, and $(\Err_{D})_{x,\fb_1 \fb_2}$ is a sum of $\OO(1)$ many normal graphs (recall Definition \ref{defnlvl0}) 
with external vertices $x,\fb_1,\fb_2$. Furthermore, in every graph,
\begin{itemize}
\item every (unlabeled) diffusive edge is a $\zthn$ edge; 

\item for each labeled diffusive edge in it, the labels only involve $\selfs$ $\Sele_1, \ldots, \Sele_{k_0-1}$;

\item there is an edge, blue solid, waved, $\dashed$, or dotted, connected to $\fb_1$;

\item there is an edge, red solid, waved, $\dashed$, or dotted, connected to $\fb_2$.
\end{itemize} 


		
\item $\PT_{x,\fb_1\fb_2}$ is a sum of $\{\fb_1,\fb_2\}$-recollision graphs (recall Definition \ref{Def_recoll}) without any $P/Q$ label or free edge, and 
\be\label{eq:smallR}
\size\Big(\sum_{x} \zthn(\Selek)_{\fa x}\PT_{x,\fb_1\fb_2}\Big) \lesssim [\heta^{1/2}\vee (\blam \heta)]\cdot  \heta \le W^{-\fd} \heta ,\quad \forall \eta\ge \eta_0.
\ee

\item  $\AT_{x,\fb_1\fb_2}$ is a sum of ``higher order graphs" without any $P/Q$ label or free edge, and 
\be\label{eq:smallA}
\size\Big(\sum_{x} \zthn(\Selek)_{\fa x}\AT_{x,\fb_1\fb_2}\Big) \lesssim W^{-\fC-\mathfrak c}\heta,\quad \forall \eta\ge \eta_0,
\ee
for a constant $\fc$ that does not depend on $\fC$ (in our construction, we can actually take $\fc:=\min\{d-2\log_W \blam,(d-\log_W \blam)/2,2\xi\}$).

\item  $\WT_{x,\fb_1\fb_2}$ is a sum of ``free graphs" without any $P/Q$ label, with exactly one free edge, and satisfying 
\be\label{eq:smallW}
\size\Big(\sum_{x} \zthn(\Selek)_{\fa x}\WT'_{x,\fb_1\fb_2}\Big) \lesssim  \blam \heta \le W^{-\fd}, \quad \forall \eta\ge \eta_0,
\ee
where $\WT'_{x,\fb_1\fb_2}$ refers to the expression obtained by removing the free edges from the graphs in $\WT_{x,\fb_1\fb_2}$. 

\item 
We denote the expression obtained by removing the following graphs from $\cal Q_{x,\fb_1\fb_2}$ as $\cal Q'_{x,\fb_1\fb_2}$:
 \begin{align}
&~\mathcal Q^{(2)}_{x,\fb_1\fb_2} + m \sum_{\al} S^+_{x\al}  Q_\al\left[ \Gc_{\al\al} G_{x \fb_1} \overline G_{x \fb_2}\right] - m^3 \sum_{\al,\beta} S^+_{x\al}S_{\al\beta} Q_\al\left[\Gc_{\beta\beta} G_{x \fb_1} \overline G_{x \fb_2}\right]\nonumber\\
+&~ m \sum_{y} S_{xy} Q_x\left[\Gc^-_{xx} G_{y\fb_1}\overline G_{y\fb_2}\right] +  |m|^2\bar m \sum_{y,\al} S_{xy} S^-_{x\al}Q_x\left[\Gc^-_{\al\al} G_{y\fb_1}\overline G_{y\fb_2}\right]\nonumber\\
-&~|m|^2 \bar m \sum_{y,\al} S_{xy}S_{x\al} Q_x\left[\Gc^-_{\al\al}  G_{y\fb_1}\overline G_{y\fb_2}\right]-|m|^2 \bar m^3 \sum_{y,\al,\beta} S_{xy}S^-_{x\al}S_{\al\beta} Q_\al\left[\Gc^-_{\beta\beta} G_{y\fb_1}\overline G_{y\fb_2}\right]  .\label{eq:largeQW}
\end{align}
Then, $\QT'_{\fa,\fb_1\fb_2}$ is a sum of $Q$-graphs 
without any free edge and satisfying 
\be\label{eq:smallQ} \size\Big(\sum_{x} \zthn(\Selek)_{\fa x}\QT'_{x,\fb_1\fb_2}\Big)\lesssim [\heta^{1/2}\vee (\blam \heta)]\cdot  \heta \le W^{-\fd} \heta ,\quad \forall \eta\ge \eta_0. 
\ee
 	
\item $(\Err_{D})_{z,\fb_1\fb_2}$ is a sum of error graphs such that 
\be\label{eq:smallRerr} \size\Big(\sum_x \zthn(\Selek)_{\fa x} (\Err_{D})_{x,\fb_1\fb_2} \Big)\le  W^{-D} \heta,\quad \forall \eta\ge \eta_0.
\ee
These graphs may contain $P/Q$ labels and hence are not included in $\AT_{x,\fb_1\fb_2}$.

\item The graphs in $\PT_{x,\fb_1\fb_2}$, $\AT_{x,\fb_1\fb_2}$, $\WT_{x,\fb_1\fb_2}$, $\QT_{x,\fb_1\fb_2}$ and $(\Err_{D})_{x,\fb_1\fb_2}$ are doubly connected in the sense of Definition \ref{def 2net} (with $x$ regarded as an internal vertex and $\fb_1,\fb_2$ as external vertices). Moreover, the free edge in every graph of $\WT_{x,\fb_1\fb_2}$ is redundant, that is, the graphs in $\WT_{x,\fb_1\fb_2}'$ obtained by removing the free edge still satisfy the doubly connected property.

\end{enumerate}
\end{definition}

To establish the local law, \Cref{thm_locallaw}, and the quantum diffusion, \Cref{thm_diffu}, we need to construct a sequence of $T$-expansions up to arbitrarily high order.  

\begin{theorem}[Construction of $T$-expansions]   \label{completeTexp} 
In the setting of \Cref{thm_locallaw}, let $ \eta_0=W^{\fd}\eta_*$. 
For any fixed constants $D>\fC>0$, we can construct $T$-expansions for \smash{$\zT_{\fa,\fb_1\fb_2}$} and $T_{\fa,\fb_1\fb_2}$ that satisfy Definition \ref{defn genuni} up to order $\fC$ (with $\eta\ge \eta_0$ and $D$-th order error).
\end{theorem}

In the proof, we will first construct a $T$-equation as outlined in \Cref{def incompgenuni} below, and then solve this $T$-equation to get a $T$-expansion.

\begin{definition}[$\incomp$]\label{def incompgenuni}
In the setting of \Cref{defn genuni}, a $T$-equation of $\zT_{\fa,\fb_1\fb_2}$ for the Wegner orbital model up to order $\fC$ (with $\eta\ge \eta_0$ and $D$-th order error) is an expression of the following form that corresponds to \eqref{mlevelTgdef}:
\begin{equation}
\label{mlevelT incomplete}
\begin{split}	
\zT_{\fa,\fb_1 \fb_2}&=  m  \zthn_{\fa \fb_1}\overline G_{\fb_1\fb_2} 
+ \sum_x (\zthn\Sele)_{\fa x}\zT_{x,\fb_1\fb_2} \\
&+ \sum_x \zthn_{\fa x} \left[\PT_{x,\fb_1 \fb_2} +  \AT_{x,\fb_1 \fb_2}  + \WT_{x,\fb_1 \fb_2}  + \QT_{x,\fb_1 \fb_2}  +  (\Err_{D})_{x,\fb_1 \fb_2}\right]\, ,
	\end{split}
\end{equation}
where $\Sele,\, \PT,\, \AT,\, \WT,\, \QT,\, \Err_{\fC,D}$ are the same expressions as in Definition \ref{defn genuni}. The $T$-equation of $T_{\fa,\fb_1\fb_2}$ corresponding to \eqref{mlevelTgdef_orig} can be defined in a similar way. 
\end{definition}




During the proof, there are steps where some $\Sele_k$ in the $\self$ $\Sele$ may not satisfy \eqref{3rd_property0}, i.e., they are $\pselfs$. In this case, we call the expansions in \Cref{defn genuni} and \Cref{def incompgenuni} as \emph{$\pTexp$s} and \emph{$\pTeq$s}, respectively.

\begin{definition} [$\PTexp$/equation]\label{defn pseudoT} 
Fix constants $\fC'\ge \fC>0$ and a large constant $D>\mathfrak C'$. In the setting of \Cref{defn genuni}, suppose we have a sequence of $\selfs$ $\cal E_k$, $k=1,\ldots, k_0$, and $\pselfs$ $\Sele_{k'}$, $k' =k_0+1, \ldots, k_1$, satisfying the following properties:
\begin{itemize}
\item[(i)] $ \sizeself(\Sele_k) \prec W^{-\fc_k} $ for a sequence of increasing constants $0<\fc_1< \cdots < \fc_{k_0}\le \fC < \fc_{k_0+1} < \cdots < \fc_{k_1}$.   

\item[(ii)] For each labeled diffusive edge in $\Sele_k$, its labels only involve $\selfs$ $\Sele_1, \ldots, \Sele_{(k-1)\wedge k_0}$.

\item[(iii)] The (unlabeled) diffusive edges in $\Sele_k$, $k=1,\ldots, k_1$, are all \smash{$\zthn$} edges. 
\end{itemize}
Let $\Sele=\sum_{k=1}^{k_1} \Sele_k$, and denote by $\Sele_k'$ the matrix obtained by replacing each $\zthn$ edge in $\Sele_k$ by a $\thn$ edge. 
For the Wegner orbital model, a $\pTexp$ of \smash{$\zT_{\fa,\fb_1\fb_2}$} (resp.~$T_{\fa,\fb_1\fb_2}$) with real order $\fC$, pseudo-order $\fC'$, and error order $D$ is still an expression of the form \eqref{mlevelTgdef} (resp.~\eqref{mlevelTgdef_orig}). 
The graphs in these expansions still satisfy all the properties in \Cref{defn genuni} except for the following change in the properties (1) and (3):   
$$\text{for each labeled diffusive edge in $\PT,$ $  \AT$, $\WT$, $ \QT$, and $\Err_{D}$, its labels only involve $\selfs$ $\Sele_1, \ldots, \Sele_{k_0}$},$$
and the higher order graphs satisfy that 
\be\label{eq:condhigher}
\size\Big(\sum_{x} \zthn(\Selek)_{\fa x}\AT_{x,\fb_1\fb_2}\Big) \lesssim W^{-\fC'-\mathfrak c}\heta,\quad \forall \eta\ge \eta_0.
\ee
Similarly, we can also define the $\pTeq$s of real order $\fC$, pseudo-order $\fC'$, and error order $D$ for the Wegner orbital model as in \Cref{def incompgenuni}, where some $\selfs$ become $\pselfs$. 
\end{definition}


\section{Proof of local law}\label{sec:pflocal}

In this section, we present the proof of the local law, \Cref{thm_locallaw}. The proof relies crucially on the $T$-expansion constructed in \Cref{completeTexp}, which will be proved specifically for the Wegner orbital model in \Cref{sec:constr}. 
Our proof is built upon several key tools that will be developed in Sections \ref{sec:Vexpansion}, \ref{sec:continuity}, and \ref{sec:sumzero}, including the $V$-expansions, coupling renormalization, and the sum zero property for $\selfs$. Extensions of the relevant arguments to the block Anderson and Anderson models will be discussed in \Cref{sec:ext}.



First, the following lemma shows that $\|G-M\|_{\max}$ is controlled by the size of $\|T\|_{\max}$, and hence explains the connection between the local laws \eqref{locallaw} and  \eqref{locallawmax}.  
Its proof is based on some standard arguments in the random matrix theory literature and will be presented in \Cref{appd:MDE} for the convenience of readers. 

\begin{lemma}\label{lem G<T}
Suppose for a constant $\delta_0>0$ and some deterministic parameter $W^{-d/2}\le \Phi\le W^{-\delta_0}$, 
\be\label{initialGT} 
	\|G(z)-M(z)\|_{\max}\prec W^{-\delta_0},\quad 
\max_{x,y}(T_{xy}+\wt T_{xy})  \prec \Phi^2 ,
\ee
	uniformly in $z\in \mathbf D$ for a subset $\mathbf D\subset \C_+$. Then, we have that uniformly in $z\in \mathbf D$,
	\be\label{diagG largedev} 
 \|G(z)-M(z)\|_{\max} \prec \Phi.
	\ee	
\end{lemma}


With \Cref{completeTexp}, \Cref{thm_locallaw} follows immediately from the following proposition, which states that given a $T$-expansion up to order $\fC$ and a sufficiently high error order $D$, the local laws \eqref{locallaw} and \eqref{locallawmax} hold as long as $L$ is not too large, depending on $\fC$. 
 

\begin{proposition}\label{locallaw-fix}
In the setting of \Cref{thm_locallaw}, suppose we have a $T$-expansion of order $\fC>0$ and a sufficiently high error order $D>\fC$ for $z= E+\ii\eta$, where $|E|\le  2-\kappa $ and $\eta \in [\eta_0,1]$ for some $W^{\fd}\eta_*\le \eta_0\le 1$. Additionally, assume there exists a constant $\errL>0$ that 
\begin{equation}		\label{Lcondition1}  
\frac{L^2}{W^2} \cdot W^{-\fC-\fc} \le W^{-\errL}, 
\end{equation}
where $\fc$ is the constant in \eqref{eq:smallA}. Then, the following local laws hold uniformly in all $z= E+\ii\eta$ with $|E|\le  2-\kappa $ and $\eta \in [\eta_0,1]$:
\begin{equation}\label{locallaw1}
T_{xy}(z) \prec B_{xy} + \frac{1}{N\eta},\quad \forall \ x,y \in \Z_L^d, 
\end{equation}
\begin{equation}\label{locallaw2}
\|G(z)-M(z)\|_{\max} \prec  \heta^{1/2} . 
\end{equation}

\end{proposition}


Similar to many previous proofs of local laws in the literature, we prove \Cref{locallaw-fix} through a multi-scale argument in $\eta$, that is, we gradually transfer the local law at a larger scale of $\eta$ to a multiplicative smaller scale of $\eta$. We first have an initial estimate at $\eta=\blam^{-1}$, which holds in all dimensions $d\ge 1$. 

\begin{lemma}[Initial estimate] \label{eta1case0} 
Under the assumptions of \Cref{locallaw-fix}, for any $z=E+\ii\eta$ with $|E|\le  2-\kappa$ and $\blam^{-1}\le \eta \le 1$, we have that 
\be\label{locallaw eta1}
\|G-M\|_{\max} \prec  \heta^{1/2},\quad \text{and}\quad T_{xy}\prec  B_{xy}  ,\quad \forall \ x,y \in \Z_L^d  .  
\ee
\end{lemma}
\begin{proof}
The proof of this lemma closely follows that of \cite[Lemma 7.2]{BandII} by using \Cref{lem G<T} and the second order $T$-expansion. This expansion is provided by \Cref{2ndExp} for the Wegner orbital model and by \Cref{2ndExp0} in the appendix for the block Anderson and Anderson orbital models. 
It can be regarded as a special case of \Cref{lemma ptree} below, but the proof is much easier since $\ell_{\lambda,\eta}\le 2$ for $\eta\ge \blam^{-1}$ and \Cref{lem theta} shows that the diffusive edges essentially vary on the local scale $W$. Hence, we omit the details of the proof.  

As a preliminary estimate for the proof, we first need to show that the first bound in \eqref{initialGT} holds for some $\delta_0>0$ at $\eta \in [\blam^{-1},1]$. For the Wegner orbital model, it follows from \cite[Theorem 2.3]{Semicircle} that
\be\label{eq:locallaw_weak} \|G(z)-m(z)I_N\|_{\max}\prec (W^d\eta)^{-1/2} \quad \text{for}\  \  W^{-d+\fd}\le \eta \le 1,\ee
which gives the desired estimate. For the block Anderson and Anderson orbital models, the local law \eqref{eq:locallaw_weak} can be established using methods from the literature, such as  \cite{AEK_PTRF,EKS_Forum,He2018}. That is, we first establish a local law at $\eta=C$ for a large enough constant $C>0$, and then combine \Cref{lem G<T} with a multi-scale argument in $\eta$ to conclude the proof of \eqref{eq:locallaw_weak}. Since the argument is standard, we omit the details. 
\end{proof}

Next, starting with a large $\eta$, Proposition \ref{lem: ini bound} below gives a key continuity estimate, which says that if the local law holds at $\eta$, then a weaker local law will hold at a multiplicative smaller scale than $\eta$.  To state it, we introduce the following definition. 

\begin{definition}\label{def_swnorms}
Given a $\Z_L^d\times \Z_L^d$ matrix $\mathcal A(z)$ with argument $z=E+\ii \eta$, we define its ``weak norm" as
\begin{align}\label{def_weakA}
	\| \mathcal A (z)\|_{w} & := \peta^{-1/2} \max_{x,y \in \Z_L^d} |\mathcal A_{xy}| +  \sup_{K\in [2W,L/2]}\max_{x,x_0 \in  \Z_L^d}\frac{1}{K^d\sqrt{\conc(K,\lambda, \eta)}}\sum_{y:|y-x_0| \leq K} (|\mathcal A_{xy}|+|\mathcal A_{yx}|),
\end{align}
where $\conc$ and $\peta$ are $\eta$-denepdent parameters defined as
\begin{align}
\label{eq:def-conc}
    \conc(K,\lambda,\eta)&:=\left[\left( \frac{\blam}{W^2 K^{d-2}}+ \frac{1}{N\eta}+ \frac{1}{K^{d/2}} \sqrt{\frac{ \blam}{N\eta}\frac{ L^2}{W^2} }\right)\left(  \frac{\blam}{W^4 K^{d-4}} + \frac{1}{N\eta}\frac{ L^2}{W^2} \right)\right]^{\frac{1}{2}}\, ,\\
    \peta&:= \frac{\blam}{W^{d}}+\frac{1}{N\eta}\frac{L^4}{W^4}\, .\label{eq:peta}
\end{align}
(Note that for $\eta\ge \eta_*$, $\peta$ is dominated by $\heta={\blam}/{W^{d}}$. Here, we have introduced it for our later discussion about the extension to the $\eta\le \eta_*$ setting.) We also define the ``strong norm":  
\begin{align}\label{def_strongA}
	\| \mathcal A(z) \|_{s} := \peta^{-1/2}\max_{x,y \in \Z_L^d} |\mathcal A_{xy}| +   \max_{x,y\in  \Z_L^d}\frac{1}{(B_{xy}+(N\eta)^{-1})^{1/2}} \frac{1}{W^d}\sum_{y':|y'-y| \leq 2W} (|\mathcal A_{xy'}|+|\mathcal A_{y'x}|)\, .
\end{align}
We remark that these ``norms" are mainly utilized to simplify notations, and we do not require them to satisfy the triangle inequality. 
\end{definition}

The above definition of the weak norm is inspired by the following \emph{continuity estimate}. 

\begin{proposition}[Continuity estimate]\label{lem: ini bound}
In the setting of Proposition \ref{locallaw-fix}, suppose for all $x,y\in \Z_L^d$,
\be\label{eq_cont_ini}
T_{xy}(\wt z)\prec B_{xy} +(N\wt\eta)^{-1},\quad \|G(\wt z)-M(\wt z)\|_{\max}\prec \heta^{1/2},
\ee
for $\widetilde z =E+ \ii \widetilde\eta$ with $|E|\le  2-\kappa$ and $\widetilde\eta\in [\eta_0,\blam^{-1}]$. Then, we have that  $$\|G (z) -M(z)\|_w \prec  \wt \eta/\eta ,$$
uniformly in $z=E+\ii \eta$ with $\max(\eta_0,W^{-\erre/20}\wt\eta) \le \eta \le \wt \eta$.
 \end{proposition}

\begin{proof} 
Fix any $x_0\in \Z_L^d$ and let $\mathcal I = \{y:|y-x_0| \leq K\}$. Using (5.25) of \cite{BandI} and \eqref{eq_cont_ini}, we get that
\be\label{cont_lem1}
\sum_{y \in \mathcal I} \left(|G_{yx}(z)|^2 + |G_{xy}(z)|^2\right) \prec \blam\frac{K^2}{W^2} + \frac{K^d}{N\eta}+ \left(\frac{\wt \eta}{\eta} \right)^2 \| \mathcal A_{\mathcal I} \|_{\ell^2 \to \ell^2},
\ee
where $\mathcal A_{\mathcal I}$ is the submatrix of $\im G:=\frac{1}{2\ii}(G- G^*)$ with row and column indices in ${\mathcal I}$. To control $\| \mathcal A_{\mathcal I} \|_{\ell^2 \to \ell^2}$, we establish the  following high-moment estimate.

\begin{lemma}\label{gvalue_continuity}
In the setting of \Cref{lem: ini bound}, for any fixed $p\in 2\N$, we have that 
\begin{equation}\label{eq_bound_Gx0}
\E \mathrm{Tr}(\mathcal A_{\mathcal I}^{p}) \prec |\cal I|^p\conc(K,\lambda,\wt\eta)^{{p-1}} \,,
\end{equation} 
\end{lemma}

Using Lemma \ref{gvalue_continuity}, we get that for any fixed $p\in 2\N$, 
$$ \E \| \mathcal A_{\mathcal I} \|_{\ell^2 \to \ell^2} ^{p} \leq \E \mathrm{Tr}(\mathcal A_{\mathcal I}^{p})  \prec (K^d)^p \conc(K,\lambda,\eta)^{p-1}\,.$$
Together with Markov's inequality, it yields that $\| \mathcal A_{\mathcal I}\|_{\ell^2 \to \ell^2} \prec K^d\conc(K,\lambda,\eta)$ since $p$ can be arbitrarily large. Plugging this estimate into \eqref{cont_lem1} gives that
\begin{equation}\label{eq:Tl2}
\max_{x,x_0}\frac{1}{K^d}\sum_{y:|y-x_0| \leq K} \left(|G_{yx}(z)|^2 + |G_{xy}(z)|^2\right) \prec \left(\frac{\wt \eta}{\eta} \right)^2 \conc(K,\lambda,\eta)\,.
\end{equation}
Combining \eqref{eq:Tl2} with the Cauchy-Schwarz inequality and using $M(z)=m(z)I_N$ for $\WO$, \eqref{Mbound} for $\BA$, or \eqref{Mbound_AO} for $\AO$, we obtain that  
$$
\max_{x,x_0}\frac{1}{K^d}\sum_{y:|y-x_0| \leq K} \left(|G_{yx}(z)| + |G_{xy}(z)| +|M_{yx}(z)|+|M_{xy}(z)|\right) \prec \frac{\wt \eta}{\eta} \sqrt{\conc(K,\lambda,\eta)}\, .
$$
This implies that 
$$\sup_{k\in [2W,L/2]}\max_{x,x_0 \in  \Z_L^d}\frac{1}{K^d\sqrt{\conc(K,\lambda, \eta)}}\sum_{y:|y-x_0| \leq K} \left(|(G-M)_{xy}|+|(G-M)_{yx}|\right) \prec \frac{\wt\eta}{\eta}.$$
It remains to prove that
$
\|G(z)-M(z)\|_{\max}\prec  (\wt \eta/\eta)\cdot  \heta^{1/2}. 
$
Note \eqref{eq:Tl2} implies that 
$$T_{xy}+\wt T_{yx}=\sum_{\al}S_{x\al}(|G_{\al y}(z)|^2+|G_{y\al}(z)|^2)\prec \frac{\wt \eta}{\eta} \sqrt{\conc(W,\lambda,\eta)} \lesssim \frac{\wt \eta}{\eta}  \heta^{1/2},\quad \forall x,y\in \Z_L^d.$$
Then, we can conclude the proof using \Cref{lem G<T} together with a standard $\e$-net and perturbation argument, see e.g., the proof of equation (5.8) in \cite{BandI}. We omit the details.
\end{proof}

Finally, with \Cref{lemma ptree}, we can improve the weak continuity estimate in \Cref{lem: ini bound} to the stronger local law.

\begin{lemma}[Entrywise bounds on $T$-variables]\label{lemma ptree} 
In the setting of Proposition \ref{locallaw-fix}, fix any $z=E+ \ii\eta$ with $|E|\le 2-\kappa$ and $\eta\in [\eta_0,\blam^{-1}]$. Suppose 
\begin{equation}
\label{eq:cond-ewb}
	\|G(z) - M(z) \|_{w} \prec W^{\errw}
\end{equation} 
for some constant $\errw>0$ sufficiently small depending on $d$, $\erre$, $\fC$ and $\errL$. Then,  we have that 
\begin{equation}\label{pth T}
 T_{xy}(z) \prec B_{xy} + \frac{1}{N\eta},\quad \forall \ x,y\in \Z_L^d.
\end{equation}
\end{lemma}

\begin{proof}[\bf Proof of Proposition \ref{locallaw-fix}]
First, Lemma \ref{eta1case0} shows that the local laws \eqref{locallaw1} and \eqref{locallaw2} hold for $z_0 = E +\ii \eta$ with $\eta_0 \le \eta \le 1$. Next, given $\eta\in [\eta_0,\blam^{-1}]$, we define a sequence of $z_k=E+\ii \eta_k$ with decreasing imaginary parts $\eta_k:= \max\left(W^{-k \errw/3}\blam^{-1}, \eta\right)$ for a sufficiently small constant $\errw>0$. If \eqref{locallaw1} and \eqref{locallaw2} hold for $z_k = E + \ii\eta_k$, then Proposition \ref{lem: ini bound} yields that $\|G(z) - M(z) \|_w \prec W^{\errw}$ uniformly in $z=E+\ii \eta$ with $\eta_{k+1}\le \eta\le \eta_k$. Therefore, the condition \eqref{eq:cond-ewb} holds, and we get \eqref{pth T} by \Cref{lemma ptree}. Combining \eqref{pth T} with Lemma \ref{lem G<T}, we see that the local laws \eqref{locallaw1} and \eqref{locallaw2} hold at $z_{k+1}$. By induction in $k$, the above arguments show that the local laws \eqref{locallaw1} and \eqref{locallaw2} hold at any fixed $z=E+\ii \eta$ with $\eta \in [ \eta_0,\blam^{-1}]$. The uniformity in $z$ follows from a standard $\e$-net and perturbation argument; see e.g., the proof of Theorem 2.16 in \cite{BandI}; we omit the relevant details. 
\end{proof}

We still need to prove \Cref{gvalue_continuity,lemma ptree}. 
\Cref{gvalue_continuity} is one of the most important and challenging results in this paper, with its proof presented in \Cref{sec:continuity}. In particular, the proof involves the key coupling renormalization mechanism mentioned in the introduction, which will be the focus of \Cref{sec:Vexpansion}. 
The proof of \Cref{lemma ptree} will be presented in \Cref{sec:ptree} after we introduce the $Q$-expansions.


\section{Construction of the $T$-expansion}\label{sec:constr} 

\subsection{Proof of \Cref{completeTexp}}

In this section, we construct a sequence of $T$-expansions satisfying Definition \ref{defn genuni} order by order, thereby completing the proof of \Cref{completeTexp}. First, it is straightforward to see that the $T$-expansion can be obtained by solving the corresponding $T$-equation defined in \Cref{def incompgenuni}. Second, given a $T$-expansion up to order $\fC$, we will follow the expansion strategy described in \cite{BandII} to construct a $\incomp$ of order strictly greater than $\fC$. 
For our proof, besides the properties in \Cref{collection elements}, we require more properties for our $\selfs$ and their \emph{infinite space limits}, defined as follows.  

\begin{definition}[Infinite space limits]\label{def infspace}
Given an arbitrary deterministic doubly connected graph $\cal G(z,W,L)$ with $ z=E+\ii \eta$, we define its infinite space (and zero $\eta$) limit $\cal G^\infty(E,W,\infty)$ as the graph obtained by letting $L\to \infty$ and $\eta\equiv \eta_L \to 0$. 
\end{definition}

Given a sequence of $\selfs$, $\Sele_1,\ldots, \Sele_k$, in the $T$-expansion/$T$-equation, their infinite space limits can be defined iteratively as follows. For a deterministic doubly connected graph \smash{$\cal G\equiv \cal G(m(z), S , S^{\pm}(z), \zthn(z))$} in $\Sele_1$ formed with waved edges, diffusive edges, and coefficients in terms of $m(z)$, we define $\cal G^\infty \equiv  \cal G^\infty \left( E,W,\infty\right)$ as follows. We first replace $m(z)$, \smash{$S\equiv S^{(W,L)}$, $ S^{\pm}(z)\equiv S_{(W,L)}^{\pm}(z)$, and $\zthn(z)\equiv  \zthn^{(W,L)}(z)$} in $\cal G $ with $m(E)$, $S^{(W,\infty)}$, $S_{(W,\infty)}^{\pm}(E)$, and $\thn^{(W,\infty)}(E)$, respectively (recall \Cref{def infspace0}). We denote the resulting graph by \be\nonumber 
\cal G^{(L)}(E,W,\infty)\equiv\cal G \left( m(E), S^{(W,\infty)} ,  S_{(W,\infty)}^{\pm}(E),  \thn^{(W,\infty)}(E)\right).\ee
Next, we let all internal vertices take values over $\Z^d$ and denote the resulting graph by $\cal G^\infty$. In this way, we obtain the infinite space limit $\Sele_1^\infty(E,W,\infty)$. Suppose we have defined the infinite space limits of $\Sele_1,\ldots, \Sele_j$. Then, given a graph \smash{$\cal G\equiv \cal G(m(z), S , S^{\pm}(z), \zthn(z), \{\Sele_l(z)\}_{l=1}^j)$} in $\Sele_{j+1}$, we first define 
\be\label{eq:semi_inf0} 
\cal G^{(L)}(E,W,\infty)\equiv \cal G \left( m(E), S^{(W,\infty)} ,  S_{(W,\infty)}^{\pm}(E),  \thn^{(W,\infty)}(E),\{\Sele_l^\infty(E)\}_{l=1}^j\right),\ee
and then let all internal vertices take values over $\Z^d$ to get $\cal G^\infty$.

By the definition of $S$, \Cref{lem:estS+}, and \Cref{lem esthatTheta}, we know that in the infinite space limit, the waved edges have exponential decay on the scale of $W$ and the diffusive edges decay like the $B_{xy}$ factor on $\Z^d$. Hence, our definitions of molecular graphs in \Cref{def_poly} and the doubly connected property in \Cref{def 2net} carry over to the infinite space limits of graphs. Furthermore, if the infinite space limits \smash{$\Sele_1^\infty,\ldots, \Sele_k^\infty$} of our $\selfs$ satisfy \eqref{4th_property0} and the sum zero property \eqref{3rd_propertyinf}, then using \eqref{redundant again2}, we can show that the bounds in \eqref{BRB} still hold in the infinite space limit. Specifically, for any $i,i_1,\ldots, i_l\in \qqq{k}$ and any constant $\tau>0$, we have 
 \be\label{BRB_inf} 
\thn^\infty_{xy}(\Sele_i^\infty)\lesssim \frac{\blam}{W^2(|x-y|+W)^{d-2-\tau}}, \quad \left(\thn^\infty\Sele_{i_1}^\infty\thn^\infty \cdots \thn^\infty \Sele_{i_k}^\infty \thn^\infty\right)_{xy} \lesssim \frac{\blam \prod_{j=1}^l \sizeself(\Sele_{i_j})}{W^2(|x-y|+W)^{d-2-\tau}},
 \ee
for large enough $W$. In other words, the labeled diffusive edges still behave like diffusive edges in the infinite space limit. Hence, as an induction hypothesis for our proof, we will assume that the $\selfs$ constructed in our $T$-expansions satisfy \eqref{4th_property0} and \eqref{3rd_propertyinf} in the infinite space limit. In addition, for the technical proof of \Cref{cancellation property}, we will impose further technical assumptions on the infinite space limits of the $\selfs$. 


\begin{proposition}[Construction of the pseudo-$\incomp$] \label{Teq}
Fix any large constants $D>\fC>0$ and $W^\fd \eta_*\le \eta_0 \le 1$. Suppose we have constructed a $T$-expansion up to order $\fC$ (with $\eta\ge \eta_0$ and $D$-th order error) that satisfies Definition \ref{defn genuni} for a sequence of $\selfs$ $\Sele_k$, $k \in \qqq{k_0}$, satisfying \Cref{collection elements} and the following additional properties:
\begin{itemize}
\item[(i)] $ \sizeself(\Sele_k) \lesssim W^{-\fc_k}$ for a sequence of strictly increasing constants $0< \fc_1 < \cdots < \fc_{k_0} = \fC$ such that 
\be\label{eq:increasing_const}
\fc_k - \fc_{k-1} \ge c_{\xi},
\ee
where $c_{\xi}:=\max\{d-2\log_W \blam, 2\xi\}$ is a positive constant and we adopt the convention that $\fc_0=0$. 

\item[(ii)] For each $k\in \qqq{k_0}$, the infinite space limit ${\cal E}_k^\infty$ can be written as 
\be\label{eq:inf_form_Ej}
{\cal E}_k^\infty = \frac{W^{-\fc_k}}{\blam} (\cal A_k \otimes \mathbf E),
\ee
where $\cal A_k(\cdot,\cdot):\ell^2(\wt\Z^d)\to \ell^2(\wt\Z^d)$ is a fixed operator that does not depend on $W$ and satisfies the following properties for a constant $C>0$ and any small constant $c>0$:
\be\label{eq:inf_form_symm}
\cal A_k([x]+[a],[y]+[a])=\cal A_k([x],[y]),\quad \cal A_k([x],[y])=\cal A_k(-[x],-[y]),\quad \forall [x],[y],[a]\in \wt\Z^d;
\ee
\be\label{eq:inf_form_decay}
\cal A_k([x],[y]) \le C(|[x]-[y]|+1)^{-(d+2)+c},\quad \forall [x],[y]\in \wt\Z^d;
\ee
\be\label{eq:inf_form_zero}
\sum_{[x]\in \wt\Z^d}\cal A_k([0],[x]) =0.
\ee

\item[(iii)]  For each $k\in \qqq{k_0}$, the following estimate holds for any small constant $c>0$:
\be\label{4th_property_substrac}
 \left|({\cal E}_k(z,W,L))_{xy}-({\cal E}_k^\infty(E,W,\infty))_{xy}\right| \le \left(\eta+t_{Th}^{-1}\right)\frac{\sizeself(\Sele_k)}{ \langle x-y\rangle^{d-c}}, \quad \forall \ x,y\in \Z_L^d.
\ee
\end{itemize}
Then, for the Wegner orbital model, we can construct pseudo-$T$-equations of $\zT_{\fa,\fb_1\fb_2}$ and $T_{\fa,\fb_1\fb_2}$ with real order $\fC$, pseudo-order $\fC'$, and error order $D$ in the sense of \Cref{defn pseudoT} for $\selfs$ $\cal E_k$, $k \in \qqq{k_0}$, and a $\pself$ $\wt\Sele_{k_0+1}$ with
\be\label{eq:psiSele}
\sizeself(\wt \Sele_{k_0+1}) : =\blam W^d\size(\wt\Sele_{k_0+1}) \lesssim W^{-c_{k_0+1}},
\ee
where $c_{k_0+1}$ is a constant satisfying $c_{k_0+1}\ge c_{k_0}+c_\xi$ and $\fC'$ is a constant satisfying that $\fC'\ge c_{k_0+1}+\fc$ with \be\label{eq:fc}
\fc:=\min\{c_\xi,(d-\log_W \blam)/2\}.
\ee

\end{proposition}

Next, we show that the leading terms in the $\pself$ $\wt\Sele_{k_0+1}$ form a $\self$ that satisfies the sum zero property \eqref{3rd_property0} and propeprties \eqref{eq:inf_form_Ej}--\eqref{4th_property_substrac}.

\begin{proposition}\label{cancellation property}
In the setting of \Cref{Teq}, if the pseudo-$\incomp$ is constructed with $\eta\ge \eta_0$ for some $\eta_0\le t_{Th}^{-1}$, then we can find $\Sele_{k_0+1}$ satisfying  \Cref{collection elements} and the properties \eqref{eq:inf_form_Ej}--\eqref{4th_property_substrac}, such that 
\be\label{eq:remain_self_error}
\sizeself(\wt \Sele_{k_0+1}-\Sele_{k_0+1}) :=\blam W^d\size(\wt \Sele_{k_0+1}-\Sele_{k_0+1}) \lesssim W^{-c_{k_0+1}-c_\xi}. 
\ee

\end{proposition}

The proof of \Cref{cancellation property} is similar to that of \cite[Lemma 5.8]{BandI} and will be deferred to \Cref{sec:sumzero}. 
By combining the above two propositions, we can conclude Theorem \ref{completeTexp} by induction.  

\begin{proof}[\bf Proof of Theorem \ref{completeTexp}] 
Suppose we have constructed a $T$-expansion up to order $\fC$. Combining Propositions \ref{Teq} and \ref{cancellation property}, 
we can obtain a pseudo-$T$-equation \eqref{mlevelT incomplete} of the form 
\begin{equation}
\label{mlevelT incomplete_1.0}
\begin{split}	\zT_{\fa,\fb_1 \fb_2}&=  m  \zthn_{\fa \fb_1}\overline G_{\fb_1\fb_2} 
	+ \sum_x (\zthn\Sele)_{\fa x}\zT_{x,\fb_1\fb_2}+ \sum_x (\zthn\cdot  \delta\Sele_{k_0+1})_{\fa x}\zT_{x,\fb_1\fb_2} \\
&+ \sum_x \zthn_{\fa x} \left[\PT_{x,\fb_1 \fb_2} +  \AT_{x,\fb_1 \fb_2}  + \WT_{x,\fb_1 \fb_2}  + \QT_{x,\fb_1 \fb_2}  +  (\Err_{D})_{x,\fb_1 \fb_2}\right]\, ,
	\end{split}
\end{equation}
where $\Sele:=\sum_{k=1}^{k_0+1}\Sele_k$ and $\delta\Sele_{k_0+1}:=\wt \Sele_{k_0+1}-\Sele_{k_0+1}$.  Due to \eqref{eq:remain_self_error}, we have that 
$$ \size\Big(\sum_x (\zthn\cdot \delta\Sele_{k_0+1})_{\fa x}\zT_{x,\fb_1\fb_2}\Big) \lesssim W^{2d}\cdot  \frac{\sizeself(\delta\Sele_{k_0+1})}{\blam W^d}\cdot \heta^2 \prec W^{-c_{k_0+1}-c_\xi}\heta.$$
Hence, we can include $\sum_x (\zthn\cdot \delta\Sele_{k_0+1})_{\fa x}T_{x,\fb_1\fb_2}$ into the higher-order graphs $\sum_x \zthn_{\fa x} \AT_{x,\fb_1 \fb_2}$ and get a $T$-equation \eqref{mlevelT incomplete} up to order $\fC'=c_{k_0+1} \ge \fC + c_\xi$. 
Then, we get from \eqref{mlevelT incomplete} that 
\begin{equation}
	\label{eq:solve-Teq-0}
	\begin{split}
		 \sum_x (I-\zthn\Sele )_{\fa x}& \mathring T_{x,\fb_1\fb_2}  = m  \zthn_{\fa \fb_1}\overline G_{\fb_1\fb_2}  \\
		 &+ \sum_x \zthn_{\fa x}\left[\PT_{x,\fb_1 \fb_2} +  \AT_{x,\fb_1 \fb_2}  + \WT_{x,\fb_1 \fb_2}  + \QT_{x,\fb_1 \fb_2}  +  (\Err_{D})_{x,\fb_1 \fb_2}\right]\,.
	\end{split}
\end{equation}
Solving \eqref{eq:solve-Teq-0} and recalling the definition \eqref{theta_renormal1}, we obtain that
$$\mathring T _{\fa,\fb_1\fb_2} = m  \zthn(\Sele)_{\fa \fb_1}\overline G_{\fb_1\fb_2} + \sum_x \zthn(\Sele)_{\fa x}\left[\PT_{x,\fb_1 \fb_2} +  \AT_{x,\fb_1\fb_2}  + \WT_{x,\fb_1\fb_2}  + \QT_{x,\fb_1\fb_2}  +  (\Err_{n,D})_{x,\fb_1\fb_2}\right], $$
which gives the $\fC'$-th order $T$-expansion \eqref{mlevelTgdef}. Similarly, solving the constructed $T$-equation for $T_{\fa,\fb_1\fb_2}$ yields the $\fC'$-th order $T$-expansion \eqref{mlevelTgdef_orig}. 

Repeating the above argument, we can construct the $T$-expansions order by order---each time the order of the $T$-expansion increases at least by $\fc$. After at most $\lceil\fC/\fc\rceil$ many inductions, we obtain a $\fC$-th order $T$-expansion and hence conclude Theorem \ref{completeTexp}.   
\end{proof}

The proof of Proposition \ref{Teq} is similar to that of Theorem 3.7 in \cite{BandII}. In the remainder of this section, we outline its proof by explaining the expansion strategy and stating several key lemmas, without providing all details of the proof. 

\subsection{Local expansions}\label{sec:local}

In this subsection, we state the local expansion rules for the Wegner orbital model. They have been proved in \cite{BandI} for random band matrices and also apply to the Wegner orbital model with $m(z)$ defined in \eqref{self_mWO} and $S\equiv S(\lambda)$ defined in \Cref{def:projlift}.

\begin{lemma}[Weight expansions, Lemma 3.5 of \cite{BandI}] \label{ssl} 
Suppose the Wegner orbital model holds, and $f$ is a differentiable function of $G$. Then,
\begin{align} 
  \Gc_{xx}  f(G)&=  m \sum_{  \al} S_{x\al}  \Gc_{xx}\Gc_{\al\al} f (G)  +m^3 \sum_{  \al,\beta}S^{+}_{x\al} S_{\al \beta}  \Gc_{\al\al}\Gc_{\beta\beta} f (G) \nonumber\\
 &-  m  \sum_{ \al} S_{x \al} G_{\al x}\partial_{ h_{ \al x}} f (G) -  m^3 \sum_{ \al,\beta} S^{+}_{x\al}S_{\al \beta} G_{\beta \al}\partial_{ h_{ \beta\al}} f(G) + \cal Q_w \, ,\label{Owx}\end{align}
 where 
 $\cal Q_w$ is a sum of $Q$-graphs defined as
 \begin{align} 
  \cal Q_w &=  Q_x \left[\Gc_{xx}  f(G)\right] + \sum_{\al} Q_\al \left[  S^{+}_{x\al}\Gc_{\al\al} f(G)\right] \nonumber\\
 & - m  Q_x\Big[ \sum_{\al} S_{x\al}  \Gc_{\al\al} G_{xx} f(G) \Big]  - m^3 \sum_\al Q_\al\Big[ \sum_{ \beta}S_{x\al}^{+}S_{\al\beta}\Gc_{\beta\beta}G_{\al\al} f(G) \Big] \nonumber\\
  &+  m  Q_x  \Big[  \sum_\al S_{x \al}  G_{\al x}\partial_{ h_{\al x}} f(G)\Big]+ m^3 \sum_\al Q_\al \Big[  \sum_{ \beta} S^{+}_{x\al}S_{\al \beta} G_{\beta \al}\partial_{ h_{\beta \al}} f(G)\Big]. \nonumber
\end{align} 
\end{lemma}

\begin{lemma}[Edge expansions, Lemma 3.10 of \cite{BandI}] \label{Oe14}
Suppose the Wegner orbital model holds, and $f$ is a differentiable function of $G$. Consider a graph
\be\label{multi setting}
\cal G := \prod_{i=1}^{k_1}G_{x y_i}  \cdot  \prod_{i=1}^{k_2}\overline G_{x y'_i} \cdot \prod_{i=1}^{k_3} G_{ w_i x} \cdot \prod_{i=1}^{k_4}\overline G_{ w'_i x} \cdot f(G).
\ee
If $k_1\ge 1$, then we have that 
\begin{align} 
  \cal G & =m\delta_{xy_1} {\cal G}/{G_{xy_1}} + m   \sum_\al S_{x\al }\Gc_{\al \al} \cal G \nonumber\\
  &+\sum_{i=1}^{k_2} |m|^2  \left( \sum_\al S_{x\al }G_{\al y_1} \overline G_{\al y'_i}\right)\frac{\cal G}{G_{x y_1} \overline G_{xy_i'}}+ \sum_{i=1}^{k_3} m^2 \left(\sum_\al S_{x\al }G_{\al y_1} G_{w_i \al} \right)\frac{\cal G}{G_{xy_1}G_{w_i x}}      \nonumber \\
& + \sum_{i=1}^{k_2}m  \Gc^-_{xx} \left( \sum_\al S_{x\al }G_{\al y_1} \overline G_{\al y'_i}\right)\frac{\cal G}{G_{x y_1} \overline G_{xy_i'}}  + \sum_{i=1}^{k_3} m  \Gc_{xx}  \left(\sum_\al S_{x\al }G_{\al y_1} G_{w_i \al} \right)\frac{\cal G}{G_{xy_1}G_{w_i x}}   \nonumber\\
 &+(k_1-1) m  \sum_\al S_{x\al } G_{x \al} G_{\al y_1}\frac{ \cal G}{G_{xy_1}}   + k_4 m   \sum_\al S_{x\al }\overline G_{\al x} G_{\al y_1}  \frac{\cal G}{G_{xy_1}} \nonumber\\
 &- m    \sum_\al S_{x\al } \frac{\cal G}{G_{x y_1}f(G)}G_{\al y_1}\partial_{ h_{\al x}}f (G)  +\cal Q_{e} \, ,\label{Oe1x}
\end{align}
where $\cal Q_{e} $ is a sum of $Q$-graphs, 
\begin{align*}
  \cal Q_{e} &= Q_x \left( \cal G\right) -  \sum_{i=1}^{k_2}m Q_x \left[ \overline G_{xx}  \left( \sum_\al S_{x\al }G_{\al y_1} \overline G_{\al y'_i}\right)\frac{\cal G}{G_{x y_1} \overline G_{xy_i'}} \right] \\
& - \sum_{i=1}^{k_3} m Q_x \left[ G_{xx} \left(\sum_\al S_{x\al }G_{\al y_1} G_{w_i \al} \right)\frac{\cal G}{G_{xy_1}G_{w_i x}}\right] -  m Q_x \left[  \sum_\al S_{x\al }\Gc_{\al \al}  \cal G\right]\\
& -(k_1-1) m Q_x \left[  \sum_\al S_{x\al }G_{x \al} G_{\al y_1}  \frac{ \cal G}{G_{xy_1}}\right]  - k_4 m Q_x \left[  \sum_\al S_{x\al }\overline G_{\al x} G_{\al y_1}  \frac{\cal G}{G_{xy_1}}\right] \\
& +m Q_x \left[  \sum_\al S_{x\al } \frac{\cal G}{G_{x y_1}f(G)}G_{\al y_1}\partial_{h_{\al x}}f(G)  \right].
 \end{align*} 
We refer to the above expansion as the \emph{edge expansion with respect to $G_{xy_1}$}. The edge expansion with respect to other $G_{xy_i}$ can be defined in the same way. The edge expansions with respect to $\overline G_{x y'_i}$, $ G_{ w_i x}$, and $ \overline G_{ w'_i x}$ can be defined similarly by taking complex conjugates or matrix transpositions of \eqref{Oe1x}.   
 \end{lemma}

\begin{lemma} [$GG$ expansion, Lemma 3.14 of \cite{BandI}]\label{T eq0}
Suppose the Wegner orbital model holds. 
Consider a graph $\cal G= G_{xy}   G_{y' x }  f (G)$, where $f$ is a differentiable function of $G$. We have that 
\begin{align}
 \cal G& =m\delta_{xy}G_{y' x}f(G)+  m^3  S^{+}_{xy} G_{y' y} f(G)  + m \sum_\al  S_{x\al} \Gc_{\al \al} \cal G + m^3 \sum_{\al,\beta}  S^{+}_{x\al}  S_{\al\beta} \Gc_{\beta \beta} G_{\al y}   G_{y'\al} f(G) \nonumber \\
    &  + m\Gc_{xx }    \sum_\al S_{x\al}G_{\al y}   G_{y'\al} f(G) + m^3 \sum_{\al,\beta}  S^+_{x\al} S_{\al\beta}  \Gc_{\al\al }  G_{\beta y}   G_{y'\beta} f(G)  \nonumber\\
    & - m \sum_{ \al}  S_{x\al}G_{\al y} G_{y' x} \partial_{ h_{\al x}}f(G) - m^3 \sum_{\al,\beta} S^{+}_{x\al} S_{\al\beta}G_{\beta y} G_{y' \al} \partial_{ h_{\beta \al}}f(G) + \cal Q_{GG} ,\label{Oe2x}
    \end{align}
where $\cal Q_{GG} $ is a sum of $Q$-graphs,
    \begin{align}
  \cal Q_{GG}&=  Q_x \left(\cal G\right)- m\delta_{xy}Q_x\left(G_{y' x}f(G)\right)+  m^2 \sum_\al Q_\al\Big[ S^+_{x\al}  G_{\al y}   G_{y'\al} f(G) \Big]  - m^3 Q_y\Big[S^{+}_{xy}  G_{y' y} f(G)\Big]  \nonumber\\
   &- m Q_x\Big[\sum_\al  S_{x\al} \Gc_{\al\al} \cal G\Big] - m^3 \sum_\al Q_\al\Big[\sum_{ \beta}  S^{+}_{x\al}  S_{\al\beta} \Gc_{\beta \beta} G_{\al y}   G_{y'\al} f(G)\Big] - mQ_x\Big[ G_{xx }  \sum_\al S_{x\al} G_{\al y}   G_{y'\al} f(G)\Big] \nonumber\\
&- m^3 \sum_\al Q_\al\Big[ \sum_{ \beta}  S^{+}_{x\al}  S_{\al\beta} G_{\al\al }  G_{\beta y}   G_{y'\beta} f(G)\Big]  + m Q_x\Big[\sum_{ \al} S_{x\al} G_{\al y} G_{y' x} \partial_{ h_{ \al x}}f(G)\Big] \nonumber\\
& + m^3\sum_\al Q_\al\Big[ \sum_{\beta} S^{+}_{x\al} S_{\al\beta}G_{\beta y} G_{y' \al} \partial_{h_{\beta \al}}f (G)\Big] .\nonumber 
 \end{align}
\end{lemma} 

Corresponding to the above three lemmas, we can define graph operations that represent the weight, edge, and $GG$ expansions. All these operations are called \emph{local expansions on the vertex $x$}, in the sense that they do not create new molecules---all new vertices created in these expansions connect to $x$ through paths of dotted or waved edges---in contrast to the global expansion that will be defined in Section \ref{sec:global_WO} below. 
By repeatedly applying the local expansions in Lemmas \ref{ssl}--\ref{T eq0}, we can expand an arbitrary normal graph into a sum of \emph{locally standard graphs}, defined as follows.

 \begin{definition} [Locally standard graphs] \label{deflvl1}
A graph is \emph{locally standard} if 
\begin{itemize}
\item[(i)] it is a normal graph without $P/Q$ labels; 

\item[(ii)] it has no self-loops (i.e., weights or light weights) on vertices;

\item[(iii)] any internal vertex is either standard neutral or connected with no solid edge (i.e., $G$ or $\Gc$ edges). 
\end{itemize}
Here, a vertex is said to be \emph{standard neutral} if it satisfies the following two properties:
\begin{itemize}
\item it has a neutral charge, where the charge of a vertex is defined by counting the incoming and outgoing blue solid (i.e., $G$ or $\Gc$) edges and red solid (i.e., $G^-$ or $\Gc^-$) edges: 
$$\#\{\text{incoming $+$ and outgoing $-$ solid edges}\}- \#\{\text{outgoing $+$ and incoming $-$ solid edges}\} ;$$

\item it is only connected with three edges: a $G$ edge, a $\overline G$ edge, and an $S$-waved edge. 
\end{itemize}
In other words, the edges connected with a standard neutral vertex $x$ take the form of a $T$-variable:
\begin{equation}\label{12in4T}
T_{x_0,y_1 y_2}=\sum_\al S_{x_0 x}G_{ x y_1}\overline G_{ x y_2} \ \ \text{or}\ \ T_{y_1y_2, x_0}=\sum_ x G_{ y_1 x }\overline G_{y_2  x}S_{ x x_0} .
\end{equation}
\end{definition}

Roughly speaking, \emph{locally standard graphs} are graphs where all $G$ edges are paired into $T$-variables on internal vertices. As discussed in Section 3.4 of \cite{BandI}, by applying local expansions repeatedly, we can expand any normal graph into a linear combination of locally standard, recollision, higher order, or $Q$-graphs.

\begin{lemma} \label{lvl1 lemma}
Let $\mathcal G$ be a normal graph. For any large constant $C>0$, we can expand $\cal G$ into a sum of $\OO(1)$ many graphs:
\begin{align}\label{expand lvl1}
\mathcal G =\mathcal G_{local} + \ATn  + \cal Q,
\end{align}
where $\mathcal G_{local} $ is a sum of locally standard graphs,  
$\ATn$ is a sum of graphs of scaling size $\OO(W^{-C})$, and $\QGn$ is a sum of $Q$-graphs of scaling size $\lesssim \size(\cal G)$. Some of the graphs on the RHS may be recollision graphs with respect to the external vertices, i.e., there is at least one dotted or waved edge between a pair of internal and external vertices (recall Definition \ref{Def_recoll}). 
In addition, if $\cal G$ is doubly connected, then all graphs on the RHS of \eqref{expand lvl1} are also doubly connected.

\end{lemma}
\begin{proof}
The equation \eqref{expand lvl1} is obtained by repeatedly applying the expansions in Lemmas \ref{ssl}--\ref{T eq0} to $\cal G$ until every resulting graph is locally standard, of sufficiently small scaling size, or a $Q$-graph. Since the proof is almost identical to that for \cite[Lemma 3.22]{BandI}, we omit the details here.
\end{proof}

\subsection{Global expansions}\label{sec:global_WO}

In this subsection, we define the key \emph{global expansions}, which are called ``global" because they may create new molecules in contrast to local expansions. Suppose we have constructed a $T$-expansion up to order $\fC$ by the induction hypothesis. Now, given a locally standard graph, say $\cal G$, produced from the local expansions. 
A global expansion involves replacing the $T$-variable containing a standard neutral vertex by a $\fC$-th order $T$-expansion. More precisely, picking a standard neutral vertex $\al$ in a locally standard graph, so that the edges connected to it take the form of one of the $T$-variables in \eqref{12in4T}. Without loss of generality, suppose this $T$-variable is $T_{x_0,y_1y_2}$ (the expansion of $T_{y_1y_2, x_0}$ can be obtained by switching the order of matrix indices). If we want to derive the $T$-expansion of $T_{\fa,\fb_1\fb_2}$, then can we apply the $\fC$-th order $T$-expansion in \eqref{mlevelTgdef_orig} to $T_{x_0,y_1y_2}$; otherwise, to derive the $T$-expansion of \smash{$\zT_{\fa,\fb_1\fb_2}$}, we can write $T_{x_0,y_1y_2}$ as
\be\label{eq:T-zT}
T_{x_0,y_1y_2}=\zT_{x_0,y_1y_2}+\frac{ G_{y_2 y_1} - \overline G_{y_1 y_2}}{2\ii N\eta}\ee
using \eqref{eq:T-cT}, and then apply the $\fC$-th order $T$-expansion in \eqref{mlevelTgdef} to $\zT_{x_0,y_1y_2}$.

In a global expansion, for each graph obtained by replacing $\zT_{x_0,y_1y_2}$ (or $T_{x_0,y_1y_2}$) with a graph on the RHS of \eqref{mlevelTgdef} (or \eqref{mlevelTgdef_orig}) that is not a $Q$-graph, we either stop expanding it (according to some stopping rules that will be defined later) or continue the local and global expansions of it. On the other hand, if we have replaced \smash{$\zT_{x_0,y_1y_2}$} (or $T_{x_0,y_1y_2}$) by a $Q$-graph, we will get a graph of the form  
\be\label{QG}\cal G_0=\sum_x \Gamma_0 Q_x (\wt\Gamma_0) ,\ee
where both the graphs $\Gamma_0$ and $\wt\Gamma_0$ do not contain $P/Q$ labels. Then, we need to perform the so-called \emph{$Q$-expansion} to expand \eqref{QG} into a sum of $Q$-graphs and graphs without $P/Q$ labels (so that we can conduct further local or global expansions on them).

\subsubsection{$Q$-expansions}\label{subsec_Qexp}

A complete process of $Q$-expansions is divided into the following three steps.


\medskip
\noindent{\bf Step 1:} Recall that $H^{(x)}$ is the $(N-1)\times(N-1)$ minor of $H$ obtained by removing the $x$-th row and column of $H$. Define the resolvent minor 
$G^{(x)}(z):=(H^{(x)}-z)^{-1}.$ 
By the Schur complement formula, we have the following resolvent identity:
\be\label{resol_exp0}G_{x_1x_2}=G_{x_1x_2}^{(x)}+\frac{G_{x_1 x}G_{xx_2}}{G_{xx}},\quad \forall x, x_1,x_2\in \Z_L^d.
\ee
Correspondingly, we introduce a new type of weights on $x$, $(G_{xx})^{-1}$ and $(\overline G_{xx})^{-1}$, and a new label $(x)$ to $G$-edges. More precisely, if a weight on $x$ has a label ``$(-1)$", then it represents a $(G_{xx})^{-1}$ or $(\overline G_{xx})^{-1}$ factor depending on its charge; if an edge/weight has a label $(x)$, then it represents a $G^{(x)}$ entry. 
Applying \eqref{resol_exp0} to expand the resolvent entries in $\Gamma$ one by one, we can write that
\be\label{decompose_gamma} \Gamma_0=\Gamma_0^{(x)}+\sum_{\omega} \Gamma_\omega,\ee 
where $\Gamma^{(x)}$ is a graph whose weights and solid edges all have the $(x)$ label, and hence is independent of the $x$-th row and column of $H$. The second term on the RHS of \eqref{decompose_gamma} is a sum of $\OO(1)$ many graphs. 
Using \eqref{decompose_gamma}, we can expand $\cal G$ as
\be\label{QG2}
\cal G_0=\sum_x \Gamma_0 Q_x (\wt \Gamma_0)=\sum_{\omega}\sum_x \Gamma_\omega Q_x(\wt \Gamma_0) + \sum_x Q_x ( \Gamma_0 \wt\Gamma_0 )- \sum_{\omega}\sum_x Q_x (\Gamma_\omega \wt\Gamma_0 ),
\ee
where the second and third terms are sums of $Q$-graphs.

Next, we expand graphs $\Gamma_\omega$ in \eqref{QG2} into graphs without $G^{(x)}$, $(G_{xx})^{-1}$, or $(\overline G_{xx})^{-1}$ entries. First, we apply the following expansion to the $G^{(x)}$ entries in $\Gamma_{\omega}$: 
\be\label{resol_reverse}G_{x_1x_2}^{(x)}= G_{x_1x_2} - \frac{G_{x_1 x}G_{xx_2}}{G_{xx}},\quad \forall x, x_1,x_2\in \Z_L^d. \ee
In this way, we can write $\Gamma_{\omega}$ as a sum of graphs \smash{$\sum_{\zeta} \Gamma_{\omega,\zeta}$}, where each graph $\Gamma_{\omega,\zeta}$ does not contain any $G^{(x)}$ entries. Second, we expand the $(G_{xx})^{-1}$ and $(\overline G_{xx})^{-1}$ weights in $\Gamma_{\omega,\zeta}$ using the Taylor expansion
\be\label{weight_taylor}\frac{1}{G_{xx}}=\frac{1}{m} + \sum_{k=1}^{D}\frac1m\left(-\frac{G_{xx}-m}{m}\right)^k + \cal W^{(x)}_{D},\quad \cal W^{(x)}_{D}:=\sum_{k>D}\left(-\frac{G_{xx}-m}{m}\right)^k . \ee
We will regard $\cal W^{(x)}_{D}$ and $\overline{\cal W}^{(x)}_{D}$ as a new type of weights on $x$ of scaling size $\heta^{D+1}$. Expanding the product of all Taylor expansions of the \smash{$(G_{xx})^{-1}$ and $(\overline G_{xx})^{-1}$} weights, we can expand every $\Gamma_{\omega,\zeta}$ into a sum of graphs that do not contain weights with label $(-1)$. To summarize, in this step, we obtain the following lemma.

\begin{lemma}[Lemma 4.11 of \cite{BandII}]\label{Q_lemma1}
Given a normal graph of the form \eqref{QG}, for any large constant $D>0$, we can expand $\cal G_0$ into a sum of $\OO(1)$ many graphs through the expansions in Step 1: 
\be\label{Q_eq1}\cal G_0 = \sum_{\omega} \sum_x \Gamma_\omega Q_x(\wt \Gamma_{\omega}) + \sum_{\zeta} \sum_ x Q_x \left( \cal G_\zeta \right) + \cal G_{err},\ee
where $\Gamma_{\omega}$, $\wt \Gamma_{\omega}$, and $\cal G_\zeta $ are normal graphs without $P/Q$, ${(x)}$, or $(-1)$ labels, and $\cal G_{err}$ is a sum of error graphs of scaling size $\OO(W^{-D})$. Furthermore, if $\Gamma_0$ does not contain any edges or weights attached to $x$, then each term $\sum_x \Gamma_\omega Q_x(\wt \Gamma_{\omega})$ has a strictly smaller scaling size than $\cal G_0$: 
\be\label{eq:smallerscalingQ} \size\bigg(\sum_x \Gamma_\omega Q_x(\wt \Gamma_{\omega})\bigg)\lesssim \heta^{1/2} \size(\cal G_0),
\ee
and $\Gamma_\omega$ contains at least one vertex connected to $x$ through a solid or dotted edge. 
\end{lemma}


\noindent{\bf Step 2:} In this step and Step 3, we will remove the $Q_x$ label in  
\be\label{eq:PGammaw} \Gamma_\omega Q_x(\wt \Gamma_{\omega})=\Gamma_\omega \wt \Gamma_{\omega} - \Gamma_\omega P_x(\wt \Gamma_{\omega}).\ee
It suffices to write $P_x (\wt \Gamma_{\omega})$ into a sum of graphs without the $P_x/Q_x$ label. In Step 2, we use the following lemma to remove the weights or solid edges attached to vertex $x$.

\begin{lemma}[Lemma 4.12 of \cite{BandII}]\label{lemm:removewe}
Let $f$ be a differentiable function of $G$. 
We have that
\begin{align}
	  P_x \left[(G_{xx}  -  m) f(G)\right] &=  P_x\Big[ m^2 \sum_{\al} S_{x\al} (G_{\al\al}-m)  f(G) + m \sum_{\al} S_{x\al} (G_{\al\al}-m) (G_{xx}-m)f(G)\Big] \nonumber\\
	& -  P_x  \Big[  m \sum_\al S_{x\al} G_{\al x}\partial_{  h_{\al x}} f(G)\Big]. \label{gHw}
\end{align}
Given a graph $\cal G$ taking the form \eqref{multi setting}, we have the following identity if $k_1\ge 1$:
\begin{align} 
	  P_x[\cal G]   &= m\delta_{xy_1}P_x[{\cal G}/{G_{xy_1}}] + \sum_{i=1}^{k_2}|m|^2 P_x \left[  \left( \sum_\al S_{x\al }G_{\al y_1} \overline G_{\al y'_i}\right)\frac{\cal G}{G_{x y_1} \overline G_{xy_i'}} \right] \nonumber\\
   &+ \sum_{i=1}^{k_3} m^2 P_x \left[ \left(\sum_\al S_{x\al }G_{\al y_1} G_{w_i \al} \right)\frac{\cal G}{G_{xy_1}G_{w_i x}}\right] \nonumber \\
	 & + \sum_{i=1}^{k_2}m P_x \left[ \left(\overline G_{xx} -\overline m\right) \left( \sum_\al S_{x\al }G_{\al y_1} \overline G_{\al y'_i}\right)\frac{\cal G}{G_{x y_1} \overline G_{xy_i'}} \right] \nonumber\\
	 &+ \sum_{i=1}^{k_3} m P_x \left[ (G_{xx}-m) \left(\sum_\al S_{x\al }G_{\al y_1} G_{w_i \al} \right)\frac{\cal G}{G_{xy_1}G_{w_i x}}\right] \nonumber \\
	&+  m P_x \left[  \sum_\al S_{x\al }\left(G_{\al \al}-m\right) \cal G\right] +(k_1-1) m P_x \left[  \sum_\al S_{x\al }G_{\al y_1} G_{x \al} \frac{ \cal G}{G_{xy_1}}\right] \nonumber\\
	&+ k_4 m P_x \left[  \sum_\al S_{x\al }G_{\al y_1} \overline G_{\al x} \frac{\cal G}{G_{xy_1}}\right]  - m P_x \left[  \sum_\al S_{x\al }\frac{\cal G}{G_{x y_1}f(G)}G_{\al y_1}\partial_{ h_{\al x}}f (G)\right].\label{gH0} 
\end{align} 
\end{lemma}



If a normal graph $P_x(\cal G)$ contains weights or solid edges attached to $x$, then we will apply the expansions in \Cref{lemm:removewe} to it. We repeat these operations until every graph does not contain any weight or solid edge attached to $x$, and hence obtain the following lemma.

\begin{lemma}[Lemma 4.13 of \cite{BandII}]\label{Q_lemma2}
Suppose $\cal G$ is a normal graph without $P/Q$, $(x)$, or $(-1)$ labels. 
Then, given any large constant $D>0$, we can expand $P_x(\cal G)$ into a sum of $\OO(1)$ many graphs: 
$$P_x(\cal G) = \sum_{\omega}  P_x(\cal G_{\omega}) +  \cal G_{err}\ ,$$
where $\cal G_{\omega}$ are normal graphs without weights or solid edges attached to $x$ (and without $P/Q$, $(x)$, or $(-1)$ labels), and $\cal G_{err}$ is a sum of error graphs of scaling size $\OO(W^{-D})$. 
\end{lemma}

With Lemma \ref{Q_lemma2}, in Step 2 we can expand $Q_x(\wt \Gamma_{\omega})$ as
\be\label{PxGamma} Q_x(\wt \Gamma_{\omega})=   \wt \Gamma_{\omega} + \sum_{\zeta}P_x(\wt \Gamma_{\omega,\zeta}) + \cal G_{err},\ee
where $\wt\Gamma_{\omega,\xi}$ are normal graphs without weights or solid edges attached to $x$, and $\cal G_{err}$ is a sum of error graphs. 

\medskip
\noindent{\bf Step 3:} In this step, we further expand the graphs $P_x(\wt \Gamma_{\omega,\zeta})$ in \eqref{PxGamma}. Suppose $\cal G$ is a normal graph without weights or solid edges attached to $x$. Using the resolvent identity \eqref{resol_exp0}, we can write $\cal G$ into a similar form as in \eqref{decompose_gamma}:
$$\cal G = \cal G^{(x)} + \sum_\omega \cal G_{\omega}.$$
Here, $\cal G^{(x)}$ is a graph whose weights and solid edges all have the $(x)$ label, while every graph $\cal G_{\omega}$ has a strictly smaller scaling size than $\cal G$ (by a factor of \smash{$\heta^{1/2}$}), at least two new solid edges connected with $x$, and some weights with label $(-1)$ on $x$. Then, we can expand $P_x(\cal G)$ as 
$$P_x(\cal G) = \cal G- \sum_\omega \cal G_{\omega} + \sum_\omega P_x(\cal G_{\omega}).$$
Next, as in Step 1, we apply \eqref{resol_reverse} and \eqref{weight_taylor} to remove all $G^{(x)}$, $(G_{xx})^{-1}$, and $(\overline G_{xx})^{-1}$ entries from $\cal G_{\omega}$. 
This gives the following result. 

\begin{lemma}[Lemma 4.14 of \cite{BandII}]\label{Q_lemma3}
Suppose $\cal G$ is a normal graph without $P/Q$, $(x)$, or $(-1)$ labels. Moreover, suppose $\cal G$ has no weights or solid edges attached to $x$. Then, for any large constant $D>0$, we can expand $P_x(\cal G)$ into a sum of $\OO(1)$ many graphs: 
$$P_x(\cal G) = \cal G+ \sum_\xi  \cal G_{\xi}  + \sum_\gamma  P_x({\cal G}'_{\gamma}) + \cal G_{err} \, ,$$
where \smash{$\cal G_{\xi}$ and ${\cal G}'_{\gamma}$} are normal graphs without $P/Q$, $(x)$, or $(-1)$ labels, and $\cal G_{err}$ is a sum of error graphs of scaling size $\OO(W^{-D})$. Moreover, $\cal G_\xi$ and ${\cal G}'_{\gamma}$ have strictly smaller scaling sizes than $\cal G$: 
$$\size\left(\cal G_\xi\right)+\size ({\cal G}'_{\gamma}) \lesssim \heta^{1/2} \size(\cal G),$$
and each of them has either new light weights or new solid edges connected to the vertex $x$. 
\end{lemma}

With this lemma, we can expand $P_x(\wt \Gamma_{\omega,\zeta}) $ in \eqref{PxGamma} as
$$P_x(\wt \Gamma_{\omega,\zeta}) =\wt \Gamma_{\omega,\zeta}+ \sum_\xi  \wt\Gamma_{\omega,\zeta,\xi} + \sum_\gamma  P_x (\wt \Gamma_{\omega,\zeta,\gamma}' ) +\cal G_{err},$$
where $\wt\Gamma_{\omega,\zeta,\xi}$ and $\wt \Gamma_{\omega,\zeta,\gamma}' $ satisfy the same properties as $ \cal G_{\xi}$ and ${\cal G}'_{\gamma} $ in Lemma \ref{Q_lemma3}, respectively. 
Next, we apply Step 2 again to the graphs \smash{$P_x(\wt \Gamma_{\omega,\zeta,\gamma}' )$}. Repeating Steps 2 and 3 for $\OO(1)$ many times, we can finally expand {$P_x (\wt \Gamma_{\omega})$} in \eqref{eq:PGammaw} into a sum of graphs without $P/Q$ label plus error graphs of sufficiently small scaling size. This completes the process of a single $Q$-expansion. 
\begin{remark}
We emphasize that $Q$-expansions, although complicated, are also \emph{local expansions on the vertex $x$}: all new vertices created in these expansions connect to $x$ through paths of dotted or waved edges. 
\end{remark}

\subsubsection{Pre-deterministic and globally standard property}

A locally regular graph may contain multiple $T$-variables that can be expanded. 
The choice of which variable to expand is a central aspect of our expansion strategy, which will be the primary focus of this subsection. 
Unlike local expansions, a global expansion can break the doubly connected property, i.e., a global expansion of a doubly connected graph may produce new graphs that are not doubly connected anymore. To maintain the doubly connected property, we are only allowed to expand a ``redundant" $T$-variable whose blue solid edge is ``redundant" in the following sense.   

\begin{definition}[Redundant edges]\label{def-redundant}
In a doubly connected graph, an edge is said to be redundant if after removing it, the resulting graph is still doubly connected. Otherwise, the edge is said to be pivotal. (Note red solid edges and edges inside molecules must be redundant, as they are not used in the black and blue nets.)
\end{definition}

To ensure the doubly connected property of all graphs in a $T$-expansion, it is essential that each locally regular graph from our expansions contains at least one redundant $T$-variable. 
Unfortunately, this property (i.e., doubly connected property with one redundant $T$-variable) is not preserved in global expansions. Following \cite{BandII,BandIII}, we define a slightly stronger \emph{sequentially pre-deterministic property} that is preserved in both local and global expansions. We first introduce the concept of isolated subgraphs.

\begin{definition}[Isolated subgraphs]\label{defn iso}
Let $\cal G$ be a doubly connected graph and $\cal G_{\cal M}$ be its molecular graph with all red solid edges removed. A subset of internal molecules in $\cal G$, say $\Pol$, is called \emph{isolated} if and only if $\Pol$ connects to its complement $\Pol^c$ exactly by two edges in $\cal G_{\cal M}$---a $\dashed$ edge in the black net and a blue solid, free, or $\dashed$ edge in the blue net. An isolated subgraph of $\cal G$ is a subgraph induced on an isolated subset of molecules.
 \end{definition}
 
An isolated subgraph of $\cal G$ is said to be \emph{proper} if it is induced on a proper subset of internal molecules. An isolated subgraph is said to be \emph{minimal} if it contains no proper isolated subgraphs. As a convention, if a graph $\cal G$ does not contain any proper isolated subgraph, then the \emph{minimal isolated subgraph} (MIS) refers to the subgraph induced on all internal molecules. On the other hand, given a doubly connected graph $\cal G$, an isolated subgraph is said to be \emph{maximal} if it is \emph{not} a proper isolated subgraph of any isolated subgraph of $\cal G$.  

\begin{definition}[Pre-deterministic property]\label{def PDG}
A doubly connected graph $\mathcal G$ is said to be {\bf pre-deterministic} if there exists an order of all internal blue solid edges, say $b_1 \preceq b_2 \preceq ... \preceq b_k$, such that 
\begin{itemize}
\item[(i)] $b_1$ is redundant;
\item[(ii)] for $i \in \qqq{k-1}$, if we replace each of $b_1,...,b_i$ by a $\dashed$ or free edge, then $b_{i+1}$ becomes redundant.
\end{itemize}
We will refer to the order $b_1 \preceq b_2 \preceq ... \preceq b_k$ as a pre-deterministic order. 
\end{definition}

With the above two definitions, we are ready to define the \emph{sequentially pre-deterministic} (SPD) property.

\begin{definition}[SPD property]\label{def seqPDG}
A doubly connected graph $\mathcal G$ 
is said to be {\bf sequentially pre-deterministic} if it satisfies the following properties.
\begin{itemize}

\item[(i)] All isolated subgraphs of $\mathcal G$ that have non-deterministic closure (recall Definition \ref{def_sub}) forms a sequence $(\Iso_{j})_{j=0}^k$ such that $\Iso_0 \supset \Iso_{1} \supset \cdots \supset \Iso_{k},$ where $\Iso_0$ is the maximal isolated subgraph and $\Iso_k$ is the MIS.

\item[(ii)] 
The MIS $\Iso_k$ is pre-deterministic. Let $\mathcal G_{\mathcal M}$ be the molecular graph without red solid edges.  For any $0 \leq j \leq k-1$, if we replace $\Iso_{j+1}$ and its two external edges in $\cal G_{\mathcal M}$ by a single $\dashed$ or free edge, then $\Iso_j$ becomes pre-deterministic. 
\end{itemize}

\end{definition}

By definition, a subgraph has a non-deterministic closure if it contains solid edges and weights inside it, or if it is connected with some external solid edges. Moreover, by \Cref{defn iso}, the two external edges in the above property (ii) are exactly the black and blue external edges in the black and blue nets, respectively. Note that an SPD graph $\cal G$ with a non-deterministic MIS has at least one redundant blue solid edge, that is, the first blue solid edge $b_1$ in a pre-deterministic order of the MIS. As proved in \cite{BandII}, the SPD property is preserved not only in local expansions but also in a global expansion of the $T$-variable containing $b_1$, provided we replace the $T$-variable also with an SPD graph---in \Cref{def genuni2} below, we will impose the SPD property on the graphs in the $T$-expansions \eqref{mlevelTgdef} and \eqref{mlevelTgdef_orig} that are used in the global expansion.

\begin{lemma}[Lemma A.17 of \cite{BandIII}]\label{lem_SPD_preserve}
Let $\cal G$ be an SPD normal graph. 
\begin{enumerate}
\item[(i)] Applying the local expansions in \Cref{sec:local} and the $Q$-expansions in \Cref{subsec_Qexp} to $\cal G$, all the resulting graphs are still SPD.
    
\item[(ii)] Replacing a $T$-variable containing the first blue solid edge in a pre-deterministic order of the non-deterministic MIS by an SPD graph, we still get an SPD graph.
\end{enumerate}

\end{lemma}

The above lemma shows that as long as every locally regular graph in our expansion process has a non-deterministic MIS, we can find a redundant $T$-variable to expand. However, in some steps, we may get graphs with a deterministic MIS that is connected with an external blue solid edge. This blue solid edge is pivotal, and the SPD property cannot guarantee the existence of a redundant blue solid edge in the graph. Thus, if the graph is locally regular, we cannot expand it anymore using either the local or global expansions. However, this issue actually does not occur for our graphs from the expansions, since they satisfy an even  stronger property than SPD:  

\begin{definition}[Globally standard graphs]\label{defn gs}
A doubly connected graph $\cal G$ is said to be \emph{globally standard} if it is SPD and its proper isolated subgraphs are all weakly isolated. Here, an isolated subgraph is said to be {\bf weakly isolated} if it has at least two external red solid edges; otherwise, it is said to be strongly isolated. 
\end{definition}

By this definition, if a globally standard graph $\cal G$ has a deterministic MIS whose closure is non-deterministic, then the MIS is connected with only one external blue solid edge (since it is isolated) and at least two external red solid edges. Thus, the graph cannot be locally regular and we can apply local expansions to it.


Now, we are ready to ``complete" the definitions of the $T$-expansion (\Cref{defn genuni}) and $T$-equation (\Cref{def incompgenuni}) by imposing the SPD or globally standard property on the graphs in the $T$-expansion/equation. We remark that the definition below is not restricted to the Wegner orbital model, since as explained below \Cref{def_atom}, the molecular graphs for the block Anderson and Anderson orbital models have the same global structure as those for the Wegner orbital model.  

\begin{definition} [$T$-expansion and $\incomp$: additional properties]\label{def genuni2}
The graphs in Definitions \ref{defn genuni}, \ref{def incompgenuni}, and \ref{defn pseudoT} satisfy the following additional properties:
\begin{enumerate}
		
\item $\PT_{x,\fb_1\fb_2}$ is a sum of globally standard $\{\fb_1,\fb_2\}$-recollision graphs. 
		
\item  $\AT_{x,\fb_1\fb_2}$ is a sum of SPD graphs. 
		
\item  $\WT_{x,\fb_1\fb_2}$ is a sum of SPD graphs, each of which has exactly one redundant free edge in its MIS. 
		
\item  $\QT_{x,\fb_1\fb_2} $ is a sum of SPD $Q$-graphs. Moreover, the vertex in the $Q$-label of every $Q$-graph belongs to the MIS, i.e., all solid edges and weights have the same $Q$-label $Q_x$ for a vertex $x$ inside the MIS. 
		
\item Each $\self$ or $\pself$ $\Sele_k$ is a sum of globally standard deterministic graphs without free edges. 
	
\end{enumerate}
\end{definition}

In a global expansion, if we replace the $T$-variable containing the first blue solid edge in a pre-deterministic order of the MIS with a $T$-expansion satisfying \Cref{def genuni2}, the resulting graphs also satisfy the corresponding properties in  \Cref{def genuni2} as shown in the following lemma.

\begin{lemma}[Lemma A.18 of \cite{BandIII}]\label{lem globalgood}
Let $\cal G$ be a globally standard graph without $P/Q$ labels, and let $\Iso_k$ be its MIS with non-deterministic closure. For local expansions, we have that:
\begin{itemize}
\item[(i)] Applying the local expansions to the graph $\cal G$ still gives a sum of globally standard graphs. Furthermore, if we apply the local expansions in Lemmas \ref{ssl}--\ref{T eq0} on a vertex $x$ in $\Iso_k$, then in every new $Q$-graph, the vertex in the $Q$-label still belongs to the MIS with non-deterministic closure. 
\end{itemize}
Let $T_{x,y_1y_2}$ be a $T$-variable that contains the first blue solid edge in a pre-deterministic order of $\Iso_k$. Then, corresponding to the terms on the RHS of \eqref{mlevelTgdef} and \eqref{eq:T-zT} (or \eqref{mlevelTgdef_orig}), we have: 
\begin{itemize}
\item[(ii)] If we replace $T_{x,y_1 y_2}$ by a $\dashed$ edge between $x$ and $y_1$ and a red solid edge between $y_1$ and $y_2$, the new graph is still globally standard and has one fewer blue solid edge.


\item[(iii)] If we replace $T_{x,y_1 y_2}$ by a free edge between $x$ and $y_1$ and a solid edge between $y_1$ and $y_2$, the new graph is still globally standard and has exactly one redundant free edge in its MIS. 

\item[(iv)] If replace $T_{x,y_1y_2}$ by a graph in $ \sum_\al \zthn(\Sele)_{x\al} \WT_{\al,y_1y_2} $, the new graph is SPD and has exactly one redundant free edge in its MIS. After removing the redundant free edge, the resulting graph has scaling size at most $\OO[\blam \size(\cal G)]$. 


\item[(v)] If we replace $T_{x,y_1y_2}$ by a graph in $\sum_\al \zthn_{x\al} (\Selek)\PT_{\al,y_1 y_2}$ or $\sum_\al \zthn_{x\al} (\Selek')\PT'_{\al,y_1 y_2}$, the new graph is still globally standard and has a strictly smaller scaling size than $\size(\cal G)$ by a factor \smash{$\OO(\heta^{1/2}\vee (\blam \heta))$}.


\item[(vi)] If we replace $T_{x,y_1y_2}$ by a graph in $ \sum_\al \zthn(\Sele)_{x\al} \AT_{\al,y_1y_2} $ or $ \sum_\al \zthn(\Sele')_{x\al} \AT'_{\al,y_1y_2} $, the new graph is SPD and has a strictly smaller scaling size than $\size(\cal G)$ by a factor $\OO(W^{-\fC-\fc})$.


\item[(vii)] If we replace $T_{x,y_1y_2}$ by a graph in $ \sum_\al \zthn(\Sele)_{x\al} \QT_{\al,y_1y_2} $ or $ \sum_\al \zthn(\Sele')_{x\al} \QT'_{\al,y_1y_2} $ and apply $Q$-expansions, we get a sum of $\OO(1)$ many new graphs: 
\be\label{mlevelTgdef3_Q}
\sum_\omega \cal G_\omega  + \cal Q   + \cal G_{err} ,
\ee
where every $\cal G_\omega$ is globally standard and has no $P/Q$ label or free edge; 
$\cal Q$ is a sum of $Q$-graphs, each of which has no free edge, is SPD, and has an MIS containing the vertex in the $Q$-label; $\cal G_{err}$ is a sum of doubly connected graphs whose scaling sizes are at most $\OO[W^{-D} \size(\cal G)]$. 

\item[(viii)] If we replace $T_{x,y_1y_2}$ by a graph in $ \sum_\al \zthn(\Sele)_{x\al} (\Err_{D})_{\al,y_1y_2}$ or $ \sum_\al \zthn(\Sele')_{x\al} (\Err_{D}')_{\al,y_1y_2}$, the new graph is doubly connected and has scaling size at most $\OO[W^{-D} \size(\cal G)]$.  

\end{itemize}
\end{lemma}

To summarize, cases (ii) and (v) give globally standard graphs, which can be expanded further; cases (iii), (iv), and (vi) produce SPD graphs, which can be included in the higher order or free graphs in the new $T$-expansion; case (vii) produces globally standard graphs along with additional $Q$-graphs and error graphs in the new $T$-expansion; case (viii) leads to error graphs in the new $T$-expansion. We also remark that some globally standard graphs can be $\{\fb_1,\fb_2\}$-recollision graphs, which need not be expanded further.

\subsection{Proof of \Cref{Teq}} \label{sec:expanstrat}

With the above preparations, we are ready to complete the proof of \Cref{Teq}. Similar to the proofs of \cite[Theorem 3.7]{BandII} and \cite[Proposition 4.13]{BandIII}, our proof is based on the following expansion \Cref{strat_global}. To state it, we first define the stopping rules. We will stop expanding a graph if it is a normal graph and satisfies at least one of the following properties: 
\begin{itemize}
\item[(S1)] it is a $\{\fb_1,\fb_2\}$-recollision graph; 

\item[(S2)] its scaling size is at most $\OO(W^{-C}\heta)$ for some constant $C>\fC+2\fc$; 

\item[(S3)] it contains a redundant free edge;

\item[(S4)] it is a $Q$-graph; 

\item[(S5)] it is \emph{non-expandable}, that is, its MIS with non-deterministic closure is locally standard and has no redundant blue solid edge. 
\end{itemize}
Note that in a non-expandable graph, there is no redundant $T$-variable in the MIS for us to expand, thus explaining the name ``non-expandable". If a locally regular graph $\cal G$ satisfies that
\be\label{det_eq}
\text{the subgraph induced on all internal vertices is deterministic,}
\ee
then $\cal G$ is non-expandable. On the other hand, we will see that all non-expandable graphs from our expansions must satisfy \eqref{det_eq} due to the globally standard property of our graphs. These non-expandable graphs will contribute to $\sum_x (\thn\Sele)_{\fa x}T_{x,\fb_1\fb_2}$ or $\sum_x (\zthn\Sele')_{\fa x}T_{x,\fb_1\fb_2}$. For the term $\sum_x (\zthn\Sele)_{\fa x}T_{x,\fb_1\fb_2}$, recall that $\Sele$ can be written as $\Sele=\Sele_{\LK}\otimes \mathbf E$, where $\Sele_{\LK}$ is translation invariant on \smash{$\wt\Z_n^d$} by \eqref{self_sym}. As a consequence, we have $P^\perp \Sele=\Sele P^\perp$, which implies that 
	$$ \sum_x (\zthn\Sele)_{\fa x}T_{x,\fb_1\fb_2}= \sum_x (\zthn P^\perp\Sele)_{\fa x}T_{x,\fb_1\fb_2}= \sum_x (\zthn \Sele P^\perp)_{\fa x}T_{x,\fb_1\fb_2}=\sum_x (\zthn \Sele )_{\fa x}\mathring T_{x,\fb_1\fb_2}.$$
This contributes to the second term on the RHS of \eqref{mlevelT incomplete}.

Without loss of generality, we will focus on deriving the $T$-expansion of $\zT_{\fa,\fb_1\fb_2}$. The only difference in deriving the $T$-expansion of $T_{\fa,\fb_1\fb_2}$ is that, in the global expansion of a $T_{x,y_1y_2}$-variable, we will use the $\fC$-th order $T$-expansion \eqref{mlevelTgdef_orig} instead of \eqref{mlevelTgdef} and \eqref{eq:T-zT}. 

\begin{strategy}[Global expansion strategy]\label{strat_global}

Given the above stopping rules (S1)--(S5), we expand $\zT_{\fa,\fb_1\fb_2}$ according to the following strategy.

 \vspace{5pt}

\noindent{\bf Step 0}: We start with the premilinary $T$-expansion \eqref{eq:3rd}, where we only need to expand $\sum_x \zthn_{\fa x} \mathcal A^{(>3)}_{x,\fb_1\fb_2}$ since all other terms already satisfy the stopping rules. We apply local expansions to it and obtain a sum of new graphs, each of which either satisfies the stopping rules or is locally standard. At this step, there is only one internal molecule in every graph. Hence, the graphs are trivially globally standard.  

\vspace{5pt}
\noindent{\bf Step 1}: Given a globally standard input graph, we perform local expansions on vertices in the MIS with non-deterministic closure. We send the resulting graphs that already satisfy the stopping rules (S1)--(S5) to the outputs. Every remaining graph is still globally standard by Lemma \ref{lem globalgood}, and its MIS is locally standard (i.e., the MIS contains no weight and every vertex in it is either standard neutral or connected with no solid edge).

\vspace{5pt}
\noindent{\bf Step 2}: Given a globally standard input graph $\cal G$ with a locally standard MIS, say $\Iso_k$, we find a $T_{x,y_1y_2}$ or $T_{y_1y_2,x}$ variable that contains the first blue solid edge in a pre-deterministic order of $\Iso_k$. If we cannot find such a $T$-variable, then we stop expanding $\cal G$.

\vspace{5pt}
\noindent{\bf Step 3}: We apply the global expansion (as stated in (ii)--(viii) of \Cref{lem globalgood}) to the $T_{x,y_1y_2}$ or $T_{y_1y_2,x}$ variable chosen in Step 2. (If we aim to derive the $T$-expansion of $T_{\fa,\fb_1\fb_2}$, we will expand $T_{x,y_1y_2}$ with \eqref{mlevelTgdef_orig}, while if we want to derive the $T$-expansion of \smash{$\zT_{\fa,\fb_1\fb_2}$}, we will expand $T_{x,y_1y_2}$ with \eqref{mlevelTgdef} and \eqref{eq:T-zT}.)  
We send the resulting graphs that satisfy the stopping rules (S1)--(S5) to the outputs. The remaining graphs are all globally standard by Lemma \ref{lem globalgood}, and we send them back to Step 1. 
\end{strategy} 

We iterate Steps 1--3 in the above strategy until all outputs satisfy the stopping rules (S1)--(S5). Suppose this iteration will finally stop.  
Then, the graphs satisfying stopping rules (S1)--(S4) can be included into 
\be\label{RAWQ}
\sum_x \zthn_{\fa x}\left[\PT_{x,\fb_1 \fb_2} + \AT_{x,\fb_1\fb_2}  + \WT_{x,\fb_1\fb_2}  + \QT_{x,\fb_1\fb_2}  +  (\Err_{D})_{x,\fb_1\fb_2}\right].
\ee
If a graph, denoted as $\cal G_{\fa,\fb_1\fb_2}$, is an output of Strategy \ref{strat_global} and fails to meet conditions (S1)--(S4), then it is either non-expandable or lacks a $T$-variable required by Step 2 of Strategy \ref{strat_global}. In either case, $\cal G_{\fa,\fb_1\fb_2}$ contains a locally standard MIS $\Iso_k$, which, by the pre-deterministic property of $\Iso_k$, does not contain any internal blue solid edge. If $\Iso_k$ is a proper isolated subgraph, then due to the weakly isolated property, $\Iso_k$ has at least two external red solid edges, at most one external blue solid edge, and no internal blue solid edge. Hence, $\Iso_k$ cannot be locally standard, which gives a contradiction. This means $\Iso_k$ must be the subgraph of $\cal G_{\fa,\fb_1\fb_2}$ induced on all internal vertices, and $\Iso_k$ is locally standard and does not contain any internal blue solid edge. Thus, the graph $\cal G_{\fa,\fb_1\fb_2}$ satisfies \eqref{det_eq} and contains a standard neutral vertex connected with $\fb_1$ and $\fb_2$, showing that it can be included into the term \smash{$\sum_x (\zthn\Sele)_{\fa x}T_{x,\fb_1\fb_2}=\sum_x (\zthn \Sele )_{\fa x}\mathring T_{x,\fb_1\fb_2}$}. In sum, the outputs from \Cref{strat_global} can be written into the form \eqref{mlevelT incomplete}.  

Now, as an induction hypothesis, suppose we have constructed a $T$-expansion \eqref{mlevelTgdef} up to order $\fC$ following \Cref{strat_global}. To conclude \Cref{Teq}, we need to establish the following two facts:
\begin{itemize}
    \item[(a)] The expansion process will finally stop after $\OO(1)$ many iterations of Steps 1--3 in \Cref{strat_global}.
    
    \item[(b)] The outputs from \Cref{strat_global} indeed form a $\pTeq$ of real order $\fC$ and pseudo-order $\fC'\ge \fC+\fc$.

    \item[(c)] $\wt\Sele_{k_0+1} = \Sele-\sum_{k=1}^{k_0}\Sele_k$ is a $\pself$ satisfying \eqref{eq:psiSele}.
\end{itemize}

To prove (a), we first notice that by \Cref{lvl1 lemma}, each round of Step 1 will stop after $\OO(1)$ many local expansions. Hence, we focus on the global expansion process. By the argument in the proof of \cite[Lemma 6.8]{BandII}, it suffices to show that after each round of global expansions (as given in (ii)--(viii) of \Cref{lem globalgood}) of a globally standard graph, say $\cal G$, every new graph from the expansions will get closer to the stopping rules:   
\begin{itemize}
\item[(1)] it satisfies the stopping rules already;
\item[(2)] it has strictly smaller scaling size than $\cal G$ by a factor $\OO(\heta^{1/2}\vee (\blam\heta))=\OO(W^{-\fc})$; 
\item[(3)] it becomes more deterministic, i.e., it has fewer blue solid edges than $\cal G$.    
\end{itemize}
Note that (iii), (iv), (vi), and (viii) satisfy (1), (v) satisfies (1) or (2), and (ii) satisfies (2) or (3). For the case (vii), the graphs $\cal Q$ and $\cal G_{err}$ in \eqref{mlevelTgdef3_Q} satisfies (1). It remains to discuss the graphs $\cal G_{\omega}$ in \eqref{mlevelTgdef3_Q}. If we have replaced $T_{x,y_1y_2}$ with a $Q$-graph in $\cal Q'_{x,y_1y_2}$, then the graphs $\cal G_{\omega}$ trivially satisfy (2) by \eqref{eq:smallQ}. Suppose we have replaced $T_{x,y_1y_2}$ with a $Q$-graph that is not in $\cal Q'_{x,y_1y_2}$, i.e., a graph of the form \smash{$\sum_{w} \zthn(\Sele)_{xw}\cal G_{w,y_1y_2}$} with $\cal G_{w,y_1y_2}$ taking the form \eqref{eq:largeQW} with $x=w$ and $\fb_{1,2}=y_{1,2}$ (recall that \smash{$\cal Q^{(2)}$} was defined in \eqref{eq:Q2}). 
If $\cal G_{w,y_1y_2}$ is one of the graphs in \eqref{eq:largeQW}, except for the expression 
\be\label{eq:leadingQs} - m  \delta_{wy_1}  Q_{y_1}(\overline G_{y_1y_2})+Q_w \left(G_{wy_1} \overline G_{wy_2} \right)- |m|^2\sum_{\al}  S_{w\al}  Q_w\left( G_{\al y_1}\overline G_{\al y_2}  \right),\ee
we have 
$$\size\Big(\sum_{w} \zthn(\Sele)_{xw}\cal G_{w,y_1y_2}\Big) \lesssim  \blam \heta^{1/2}\size(T_{x,y_1y_2}).$$
Then, \eqref{eq:smallerscalingQ} tells us that the graphs $\cal G_{\omega}$ satisfy property (2): 
\be\label{eq:Gomega1}\size(\cal G_\omega) \lesssim \blam \heta \size(\cal G) 
. 
\ee
Next, if $\cal G_{w,y_1y_2}=- m  \delta_{wy_1}  Q_{y_1}(\overline G_{y_1y_2})$ , then we can write $$\sum_{w} \zthn(\Sele)_{xw}\cal G_{w,y_1y_2}= - m  \zthn(\Sele)_{xy_1}Q_{y_1}(\overline G_{y_1y_2})  = - m  \zthn(\Sele)_{xy_1}Q_{y_1} ( \Gc^-_{y_1y_2}),$$
whose scaling sizes satisfy that
\begin{align*}
    \size\left[\zthn(\Sele)_{xy_1}Q_{y_1} ( \Gc^-_{y_1y_2})\right] &\lesssim \frac{\blam}{W^d}\heta^{1/2} \lesssim \heta^{1/2}\size(T_{x,y_1y_2}). 
\end{align*} 
Since $Q$-expansions do not increase scaling sizes, the graphs $\cal G_{\omega}$ also satisfy property (2). Finally, consider 
$$\cal G_{w,y_1y_2}=Q_w \left(G_{wy_1} \overline G_{wy_2} \right) \quad \text{or}\quad - |m|^2\sum_{\al}  S_{w\al}  Q_w\left( G_{\al y_1}\overline G_{\al y_2}  \right).$$
Note the subgraph obtained by removing $T_{x,y_1y_2}=\sum_w S_{xw}G_{wy_1}\overline G_{wy_2}$ from $\cal G$ has no edges or weights attached to the vertex $w$. 
Thus, in Step 1 of the $Q$-expansion, we must have applied either \eqref{resol_exp0} or \eqref{resol_reverse} at least once and kept the second term in this step. Each such expansion decreases the scaling size by \smash{$\heta^{1/2}$}. 
Thus, if this process occurs at least twice, then $\cal G_\omega$ will satisfy \eqref{eq:Gomega1}. Otherwise, the blue and red solid edges attached to vertex $x$ remain unmatched: there are either two additional blue solid edges or two additional red solid edges connected to $x$. Consequently, the subsequent Steps 2 and 3 of the $Q$-expansion, or the local expansions on $x$, will further reduce the scaling size by another factor of \smash{$\heta^{1/2}$}. 

In sum, we have shown that after each round of global expansion, the new graphs satisfy the above properties (1)--(3). Thus, following the argument in the proof of \cite[Lemma 6.8]{BandII}, we can show that the expansion process will finally stop after at most $C(\fC/\fc)^2$ many iterations of Steps 1–3 in \Cref{strat_global} with $C$ being an absolute constant. This concludes fact (a). 

To show (b), we perform the same expansions as those used in the construction of the $\fC$-th order $T$-expansion; that is, we apply the same local expansions and choose the same $T$-variables to expand whenever possible. There are only two key differences: (i) in the stopping rule (S2), we choose a larger threshold $C$ compared to that in the construction of the $\fC$-th order $T$-expansion; (ii) in the global expansion step, we replace the $T$-variable with a higher order $T$-expansion than that used in the $\fC$-th order case. It is straightforward to see that the expansions carried out in \Cref{strat_global} yield an expression of the same form as \eqref{mlevelT incomplete}, where all graphs with scaling size up to order $W^{-\fC}\heta$, as well as all deterministic graphs $\cal G$ in $\Sele$ with scaling size up to $\blam W^d\size(\cal G_{xy})\prec W^{-\fC}$ coincide with those appearing in the $\fC$-th order $T$-expansion.  (For a formal justification of this claim, see the discussion in \cite[Section 6.4]{BandII}.) The property \eqref{eq:condhigher} for the $\pTeq$ follows from property (2) above, which ensures that the higher order graphs in the $\pTexp$ have scaling sizes at least $\OO(W^{-\fc})$ smaller than those in the $\fC$-th order $T$-expansion.

To conclude the proof, it remains to show that the deterministic graphs in $\wt\Sele_{k_0+1}=\Sele-\sum_{k=1}^{k_0}\Sele_k$ are $\pselfs$. By the above construction, we know that the graphs in \smash{$\wt\Sele_{k_0+1}$} are globally standard, indicating that every graph is doubly connected and each diffusive edge in it is redundant. 
Now, recalling \Cref{collection elements}, we still need to prove that any such deterministic graph $\cal G$ satisfies \eqref{self_sym} and \eqref{4th_property0} with 
$\sizeself(\cal G)=\blam W^d \size(\cal G_{xy})$, $\cal G$ can be written as $\cal G_{\LK}\otimes \bE$, and 
\be\label{eq:sizeselfG}
\sizeself(\cal G) \lesssim \blam^2 /W^d
\ee
First, we notice that for the Wegner orbital model, the matrices $S$, $S^\pm,$ and $\zthn$ can be written as 
\be\label{eq:SSLK} S=S_{\LK}\otimes \bE,\quad S^\pm=S^\pm_{\LK}\otimes \bE,\quad \zthn=\zthn_{\LK}\otimes \bE, \ee
with the matrices $S_{\LK}$, $S_{\LK}^\pm$, and $\zthn_{\LK}$ being symmetric and translation invariant on $\wt \Z_n^d$. Since $\cal G$ is a deterministic graph formed with these matrices, it can be written as $\cal G_{\LK}\otimes \bE$, where $\cal G_{\LK}$ is a deterministic graph formed with the matrices \smash{$S_{\LK}$, $S_{\LK}^\pm$, and $\zthn_{\LK}$}. Using the symmetry and translation invariance of these matrices, the argument in the proof of \cite[Lemma A.1]{BandI} tells us that $\cal G_{\LK}$ is also symmetric and translation invariant. This verifies the property \eqref{self_sym} for $\cal G$.  
Next, the property \eqref{4th_property0} is an immediate consequence of \Cref{dG-bd}. 
Finally, for \eqref{eq:sizeselfG}, we know that the graphs in $\Sele$ must come from the expansion of $\sum_x \zthn_{\fa x}\mathcal A^{(>3)}_{x,\fb_1\fb_2}$ in \eqref{eq:3rd}. Thus, we have the following bound, which yields \eqref{eq:sizeselfG}: 
$$\size\Big(\sum_x (\zthn\cal G)_{\fa x}T_{x,\fb_1\fb_2}\Big)= \blam  W^d \size(\cal G_{xy})\heta = \sizeself(\cal G) \heta \lesssim \size \Big(\sum_x \zthn_{\fa x}\mathcal A^{(>3)}_{x,\fb_1\fb_2}\Big)\lesssim \blam  \heta^2 .$$

\subsection{Proof of \Cref{lemma ptree}} \label{sec:ptree}


The proof of \Cref{lemma ptree} is based on the following \Cref{no dot} for doubly connected graphs. It can be regarded as a non-deterministic generalization of \Cref{dG-bd}. To state it, we introduce the following $\eta$-denepdent parameter:
$$\pheta:= \frac{\blam}{W^d}+ \frac{1}{N\eta} \frac{L^5}{W^5}.$$
Then, we introduce two variants of the scaling size defined in \Cref{def scaling}, called the \emph{pseudo-scaling sizes}: 
given a normal graph $\cal G$ where every solid edge has a $\circ$, we define its two types of $\pscale$s as 
	\begin{align}
		\psize_1(\cal G): =\size(\cal G) \left({\peta}/{\heta}\right)^{ \frac{1}{2}\#\{\text{solid edges}\}},\quad  \psize_2(\cal G): =\size(\cal G) \left({\pheta}/{\heta}\right)^{ \frac{1}{2}\#\{\text{solid edges}\}}, \label{eq_defpsize}
	\end{align}
    where we recall that $\peta$ was defined in \eqref{eq:peta}. 
	If $\cal G$ is a sum of graphs $\cal G=\sum_k \cal G_k$, then we define $\psize_s(\cal G):=\max_k \psize_s(\cal G_k)$ for $s\in\{1,2\}$. Note that if $\eta\ge \eta_*$, then both $\peta$ and $\pheta$ are dominated by $\blam$, so the two types of $\pscale$s match the scaling size up to constant factors. 


\begin{lemma}\label{no dot}
Suppose $d\ge 7$ and $\|G(z)-M(z)\|_{w}\prec 1$ for $\eta\ge N^\fd \eta_\circ$ (recall \eqref{etacirc}). Let $\cal G$ be a doubly connected normal graph without external vertices. Pick any two vertices of $\cal G$ and fix their values $x , y\in \Z_L^d$ (where $x$ and $y$ can be the same vertex). Then, the resulting graph $\cal G_{xy}$ satisfies that 
\be\label{bound 2net1}
\left|\cal G_{xy}\right| \prec {\psize_2\left(\cal G_{xy}\right)} \frac{W^{d-2}}{\langle x-y\rangle^{d-2}} \frac{\cal A_{xy}}{\pheta^{1/2}},
\ee
where  $\cal A_{xy}$ are non-negative random variables satisfying $\|\cal A(z)\|_{w}\prec 1$.  
Furthermore, if $\|G(z)-M(z)\|_{s}\prec 1$,  then we have 
\be\label{bound 2net1 strong}
\left|\cal G_{xy}\right| \prec {\psize_1\left(\cal G_{xy}\right)} \frac{W^{d-2}}{\langle x-y\rangle^{d-2}} \frac{\cal A_{xy}}{\peta^{1/2}},
\ee
where $\cal A_{xy}$ are non-negative random variables satisfying $\|\cal A(z)\|_{s}\prec 1$. 
The above bounds also hold for the graph $ {\cal G}^{{\rm abs}}$, which is obtained by replacing each component (including edges, weights, and the coefficient) in $\cal G$ with its absolute value and ignoring all the $P$ or $Q$ labels (if any). 
 \end{lemma}

To prove the main results in the current paper, it suffices to state \Cref{no dot} for $\eta\ge N^\fd\eta_*$. We have chosen to state it for a wider range of $\eta$ because we will use it in \Cref{rem:Improveeta} below to explain the extension of our proof to the smaller $\eta$ case with $ N^\fd\eta_\circ \le \eta\le \eta_*$, as promised in \Cref{rem:smalleta}. 
The proof of \Cref{no dot} relies on the following basic estimates. 

\begin{lemma}\label{w_s}
Suppose $d \geq 7$ and $\eta \geq W^{\fd}\eta_\circ$. Given any matrices $ \mathcal A^{(1)}(z)$ and $ \mathcal A^{(2)}(z)$ with non-negative entries and argument $z=E+\ii \eta$, we have the following estimates:
	\begin{align}
	\label{keyobs3}
&\sum_{x_i}\mathcal A^{(1)}_{x_i \al}(z)\cdot \prod_{j=1}^k B_{x_i y_j } \prec \blam \pheta^{1/2} \cdot \Gamma(y_1,\cdots, y_k)\|\mathcal A^{(1)}(z)\|_{w}\, ,\\
&\sum_{x_i}\mathcal A^{(1)}_{x_i \al}(z)\cdot \prod_{j=1}^k B_{x_i y_j } \prec \blam \peta^{1/2} \cdot \Gamma(y_1,\cdots, y_k)\|\mathcal A^{(1)}(z)\|_{s}\,  ,\label{keyobs3s}
\end{align}
where $k\ge 1$ and $\Gamma(y_1,\cdots, y_k)$ is a sum of $k$ different products of $(k-1)$ $B$ entries:  
\begin{equation}
	\label{defn_Gamma}\Gamma(y_1,\cdots, y_k):= \sum_{i=1}^k \prod_{j\ne i}B_{y_{i}y_j}.
\end{equation}
In addition, if $ \|\mathcal A^{(i)}(z)\|_w \le 1$ (resp.~$ \|\mathcal A^{(i)}(z)\|_s \le 1$) for $i\in \{1,2\}$ and define $\cal A(z)$ as 
\begin{equation}\label{keyobs2}
{\mathcal A}_{\alpha \beta}: = \frac{\blam^{-1}}{\Gamma(y_1,\cdots, y_k)}\sum_{x_i}\mathcal A^{(1)}_{x_i \alpha }\mathcal A^{(2)} _{x_i \beta}\cdot \prod_{j=1}^k B_{x_i y_j },
\end{equation}
then we have
$\pheta^{-1/2}\|\cal A(z)\|_w\prec 1$ (resp.~$\peta^{-1/2}\|\cal A(z)\|_s\prec 1$).
\end{lemma}

\begin{proof}
Define the subsets $\mathcal I_l = \{ x \in \Z^d_L: \gE{x - y_l} \leq \min_{j\ne l}\gE{x- y_j}   \}$. 
Then, we have $B_{x y_l}=\max_{j}B_{xy_j}$ for $x \in \mathcal I_l$, and for $j\ne l$, $\gE{y_l-y_j} \le \gE{x- y_j}+\gE{x-y_l} \le 2\gE{x- y_j}$, implying that $B_{x_i y_j }\lesssim B_{y_l y_j }$. 
Thus, we can bound that 
	\begin{equation}
	\label{eq:1140223}
		\sum_{x_i \in I_l}\mathcal A^{(1)}_{x_i \al}\cdot \prod_{j=1}^k B_{x_i y_j } \lesssim \sum_{x_i \in I_l}\mathcal A^{(1)}_{x_i \al} B_{x_i y_l } \cdot \prod_{j \not = l} B_{y_l y_j } \le  \sum_{x_i \in \Z_L^d}\mathcal A^{(1)}_{x_i \al} B_{x_i y_l } \cdot \prod_{j \not = l} B_{y_l y_j }\,.
	\end{equation}
Without loss of generality, assume that $\mathcal \|\cal A^{(1)}\|_w=1$. With $K_n := 2^n W$ and the definition of the weak norm, we have that
	\begin{align*}
		 &\sum_{x_i \in \Z_L^d}\mathcal |\mathcal A^{(1)}_{x_i \al}| B_{x_i y_l }   \leq \sum_{ 1 \leq n \leq \log_2 \frac L W}\sum_{K_{n-1} \leq \gE{x_i - y_l} \leq K_n}\mathcal |\mathcal A^{(1)}_{x_i \al}| B_{x_i y_l } \\
		&\leq  \sum_{ 1 \leq n \leq \log_2 \frac L W}\max_{K_{n-1} \leq \gE{x_i - y_l} \leq K_n} B_{x_i y_l}\cdot  \sum_{K_{n-1} \leq \gE{x_i - y_l} \leq K_n} |\mathcal A^{(1)}_{x_i \al}|  \\
		& \prec \max_{ 1 \leq n \leq \log_2 \frac L W}\frac{\blam}{W^2K_n^{d-2}}\cdot K_n^d \sqrt{\conc(K_n,\lambda, \eta)}\\
		& \leq \blam \max_{ 1 \leq n \leq \log_2 \frac L W}   \left[ \left(\frac{\blam }{W^7 K_n^{d-7}}+ \frac{K_n^5}{W^5 N\eta} + \sqrt{\frac{\blam K_n^{10-d}}{W^{10}N\eta}\frac{L^2}{W^2}}\right)\left(\frac{\blam }{W^7 K_n^{d-7}} + \frac{K_n^3}{W^3N\eta}\frac{L^2}{W^2}\right)\right]^{\frac{1}{4}}\\
  &\lesssim  \blam \pheta^{1/2}.
	\end{align*}
Combined with \eqref{eq:1140223}, it yields \eqref{keyobs3}. 
The estimate \eqref{keyobs3s} follows from a similar argument---we only need to replace the third step of the above derivation with the following estimate:
$$\sum_{K_{n-1} \leq \gE{x_i - y_l} \leq K_n} |\mathcal A^{(1)}_{x_i \al}| \prec K_n^d \sqrt{\frac{\blam}{W^2K_n^{d-2}}+\frac{1}{N\eta}}.$$
The estimates $\pheta^{-1/2}\|\cal A(z)\|_w\prec 1$ and $\peta^{-1/2}\|\cal A(z)\|_s\prec 1$ in the second statement follow by controlling the right-hand sides of \eqref{def_weakA} and \eqref{def_strongA} using \eqref{keyobs3} and \eqref{keyobs3s}, respectively. Since the derivation of them 
is straightforward, we omit the details.  
\end{proof}

\begin{proof}[\bf Proof of \Cref{no dot}] 
The proof of \Cref{no dot} is almost the same as that of \cite[Lemma 6.10]{BandI}. Hence, we only outline the proof strategy without giving all details. 
Roughly speaking, the proof proceeds as follows: we first choose a ``center" for each molecule, which can be any vertex inside the molecule. Next, we bound the internal structure of each molecule $\cal M$ with center $\al$ (i.e., the subgraph induced on $\cal M$ with $\al$ regarded as an external vertex; recall \Cref{def_sub}) using its scaling size. Now, the problem is reduced to bounding the global structure, denoted as a graph $\cal G^{\rm{global}}_{xy}$, that consists of $x$, $y$, the centers of molecules, double-line edges, free edges, and solid edges between them.
These double-line edges represent $B_{\al\beta}$ factors between vertices $\al$ and $\beta$, and the solid edges represent $\cal A_{\al\beta}$ factors whose weak norm (for the proof of estimate \eqref{bound 2net1}) or strong norm (for the proof of estimate \eqref{bound 2net1 strong}) is bounded by $1$. 
Furthermore, $\cal G^{\rm{global}}_{xy}$ is doubly connected with a black net consisting of double-line edges and a blue net consisting of double-line, free, and solid edges. 

To bound \smash{$\cal G^{\rm{global}}_{xy}$}, we choose a blue spanning tree consisting of the internal vertices and the external vertex $y$, which serves as the root of the tree. By the doubly connected property of $\cal G$, the external vertex $x$ must be connected to the tree by at least one blue edge. Then, we sum over an internal vertex located at a leaf of the spanning tree. 
By \Cref{w_s}, we can bound this summation by a linear combination of new graphs, which \emph{have one fewer internal vertex and are still doubly connected}. In applying \Cref{w_s}, we will use \eqref{keyobs3} or \eqref{keyobs3s} if the chosen vertex is not connected with $x$; otherwise, we will use \eqref{keyobs2}. In particular, \eqref{keyobs2} guarantees that in each new graph, $x$ still connects to the blue spanning tree through a blue edge. 
We iterate the above argument: we select a leaf vertex of the blue spanning tree in each new graph and bound the sum over this vertex using \Cref{no dot}. Continuing this process, we can finally bound $\cal G^{\rm{global}}_{xy}$ by a doubly connected graph consisting of vertices $x$ and $y$ only. This graph provides at least a factor $B_{xy}\cal A_{xy}$ times a coefficient denoted by $c(\cal G)$. Keeping track of the above argument carefully, we see that    
$$ c(W)\heta\pheta^{1/2} \lesssim \psize_2(\cal G_{xy}),\quad \text{or}\quad  c(W)\heta\peta^{1/2} \lesssim \psize_1(\cal G_{xy}),$$
if we count the scaling size of $B_{xy}$ as $\heta$ and the scaling size of $\cal A_{xy}$ as $\pheta^{1/2}$ or $\peta^{1/2}$. This concludes the proof of \eqref{bound 2net1} and \eqref{bound 2net1 strong}.
\end{proof}

\begin{remark}\label{rem:Improveeta}
\Cref{w_s} above shows that when $N^\fd\eta_\circ\le \eta\le \eta_*$ (recall \eqref{eq:defeta*} and \eqref{etacirc}), each solid edge only contributes a factor of \smash{$\peta^{1/2}$ or $\pheta^{1/2}$} to the size of a doubly connected graph. Consequently, as shown by \Cref{no dot}, we should use pseudo-scaling sizes to bound the non-deterministic doubly connected graphs. When $\eta\le \eta_*$, the pseudo-scaling sizes of the graphs in $T$-expansions can be much larger than their true scaling sizes.  In particular, the pseudo-scaling sizes of the higher order graphs do not satisfy the bound \eqref{eq:smallA} anymore. As a consequence, \Cref{locallaw-fix} may not apply, since the condition \eqref{Lcondition1} is not strong enough to control the higher order graphs (to see why, one can refer to the proof of \Cref{lemma ptree} below). 

However, this technical issue can be easily addressed. In fact, if the pseudo-scaling sizes of the higher order graphs satisfy the bound 
$$\psize_2\Big(\sum_{x} \zthn(\Selek)_{\fa x}\AT_{x,\fb_1\fb_2}\Big) \lesssim W^{-\fC'}\heta $$
for some constant $\fC'>0$ smaller than $\fC+\fc$ in  \eqref{eq:smallA}, we can still prove \Cref{locallaw-fix} as long as the condition \eqref{Lcondition1} is replaced by $L^2/W^2\cdot  W^{-\fC'} \le W^{-\errL}.$
In this paper, we have constructed $T$-expansions up to arbitrarily high order, where the scaling sizes of the higher order graphs can be made smaller than $\OO(W^{-D})$ for any large constant $D>0$. The pseudo-scaling sizes of these graphs, although worse than the scaling sizes, can be made as small as possible if we use $T$-expansions of sufficiently high order.  This allows us to prove the local law, \Cref{thm_locallaw}, for all $\eta\ge N^\fd\eta_\circ$. With the local law, we can improve other main results, including delocalization, QUE, and quantum diffusion as mentioned in \Cref{rem:smalleta}. 
\end{remark}




\Cref{lemma ptree} is a simple consequence of the following lemma.  
\begin{lemma}\label{lem highp1}
Suppose the assumptions of \Cref{lemma ptree} hold. Assume that
\begin{equation}
	\label{initial_p}
T_{xy} \prec B_{xy}+\widetilde\Phi^2,\quad \forall \ x,y \in \Z_L^d,
\end{equation}
for a deterministic parameter $\widetilde\Phi $ satisfying $0\le \widetilde\Phi \le W^{-\e}$ for some constant $\e>0$. 
Then, there exists a constant $c>0$ such that for any fixed $p\in \N$, 
\begin{equation}
	\label{locallawptree}
\E T_{xy} (z) ^p \prec  \big(B_{xy} +W^{-c} \widetilde\Phi^2+(N\eta)^{-1}\big)^p ,
\end{equation}
as long as the constant $\errw$ in \eqref{eq:cond-ewb} is chosen sufficiently small depending on $d$, $\fc$, $\fC$, and $\errL$. 
\end{lemma}

\begin{proof}[\bf Proof of \Cref{lemma ptree}]
We start with $T_{xy}\prec B_{xy} +\wt\Phi_0^2$, where $\wt\Phi_0:= W^{\errw}\heta^{1/2}$ due to \eqref{eq:cond-ewb}. Then, combining \eqref{locallawptree} with Markov's inequality, we obtain that
$$T_{xy} (z) \prec  B_{xy} +W^{-c} \wt\Phi_0^2 +(N\eta)^{-1} \ .$$
Hence, \eqref{initial_p} holds with a smaller parameter $\wt\Phi=\wt\Phi_1:=W^{-c/2}\wt\Phi_0 + (N\eta)^{-1/2}$. For any fixed $D>0$, repeating this argument for $\left\lceil D/c\right\rceil$ many times gives that 
$$T_{xy} (z) \prec  B_{xy} +(N\eta)^{-1}+W^{-D} .$$
This concludes \eqref{pth T} as long as we choose $D$ sufficiently large.
\end{proof}

The proof of \Cref{lem highp1} is very similar to those for \cite[Lemma 8.1]{BandII} and \cite[Lemma 4.7]{BandIII}, so we only describe an outline of it without presenting all details. 

\begin{proof}[\bf Proof of Lemma \ref{lem highp1}] 
Using the $\fC$-th order $T$-expansion \eqref{mlevelTgdef} and \eqref{eq:T-zT} with $\fb_1=\fb_2=\fb$, we get  
\begin{equation}
\label{eq:ETp-Texp}
\begin{split}
        \E T_{\fa \fb}^p =\E T_{\fa \fb}^{p-1} \Big\{ & m \zthn(\Sele)_{\fa \fb} \overline{G}_{\fb\fb} + \frac{\im G_{\fb\fb}}{N\eta} \\
        &+\sum_x \zthn(\Sele)_{\fa x}\left[\PT_{x,\fb \fb} +  \AT_{x,\fb \fb}  + \WT_{x,\fb \fb}  + \QT_{x,\fb \fb}  +  (\Err_{D})_{x,\fb \fb}\right] \Big\}.
\end{split}
\end{equation}
By \eqref{BRB} and the condition \eqref{eq:cond-ewb}, the first two terms on the RHS of \eqref{eq:ETp-Texp} can be bounded by 
\be\label{estimates_ptree1}
m \zthn(\Sele)_{\fa \fb} \overline{G}_{\fb\fb} + \frac{\im G_{\fb\fb}}{N\eta} \prec B_{\fa\fb}+\frac{1}{N\eta}.
\ee
As stated in \Cref{defn genuni}, the graphs in $\PT,$ $\AT$, $\WT'$ (recall that $\WT=\WT'/(N\eta)$), and $\Err_{D}$ satisfy the doubly connected property and the scaling size conditions \eqref{eq:smallR}, \eqref{eq:smallA}, \eqref{eq:smallW}, and \eqref{eq:smallRerr}. Then, with Lemmas \ref{no dot} and \ref{w_s} as inputs, we can show that 
\begin{equation}\label{estimates_ptree}
\begin{split}
& \sum_x \zthn(\Sele)_{\fa x}\PT_{x,\fb \fb}  \prec W^{-\fc +2\errw} B_{\fa \fb}\,,\quad  \sum_x \zthn(\Sele)_{\fa x}\WT_{x,\fb \fb} \prec \frac{W^{-\fc+2\errw}}{N\eta}\frac{L^2}{W^2} (B_{\fa \fb} + \wt \Phi^2)\,,  \\
&\sum_x \zthn(\Sele)_{\fa x}\AT_{x,\fb \fb}  \prec W^{-(\fC +\fc)(1- C\errw)}\cdot \frac{L^2}{W^2} (B_{\fa \fb} + \wt \Phi^2)\,, \\ 
&\sum_x \zthn(\Sele)_{\fa x}\Err_{x,\fb \fb}  \prec W^{-D(1- C\errw)}\cdot \frac{L^2}{W^2} (B_{\fa \fb} + \wt \Phi^2)\,, 
\end{split}
\end{equation}
where $C>0$ is a constant that does not depend on $\errL$ and $\errw$. Here, the $W^{\errw}$ factors are due to the bound \eqref{eq:cond-ewb} on $G$ edges. By keeping track of the scaling sizes of the graphs from our \Cref{strat_global}, it is evident that each solid edge provides at least a scaling size of order $(\blam \heta)^{1/2}\lesssim W^{-\fc/2}$. Thus, a graph of scaling size $\Omega(W^{-A})$ ($A=\fc,\fC+\fc,D$) has at most $2A/\fd$ solid edges, which lead to a factor $\OO(W^{A C\errw})$ for $C= 2/\fc$.

Given any graph in $\QT_{x,\fb\fb}$, say $\cal G_{x\fb}$ with label $Q_y$, we first apply the $Q$-expansions to $T^{p-1}_{\fa\fb}\cal G_{x\fb}$ and then apply local expansions on $y$. This yields the following expansion for any large constant $D>0$: 
\begin{equation}\label{eq:Tp}
\E T^{p-1}_{\fa\fb}\sum_x \zthn(\Sele)_{\fa x}\cal G_{x\fb} = \sum_{k=1}^{p-1}\sum_{\omega_k}  \E \left[T^{p-1-k}_{\fa\fb} (\cal G_{\omega_k})_{\fa\fb}\right] + \OO(W^{-D}) \, ,
\end{equation}
where the RHS is a sum of $\OO(1)$ many new graphs. For each $k \in \qqq{p-1}$ and $\omega_k$, $(\cal G_{\omega_k})_{\fa\fb}$ represents a doubly connected graph without $P/Q$ labels, where $k$ $T_{\fa\fb}$ variables have been ``broken" during the $Q$-expansion. Specifically, each of the $k$ $T_{\fa\fb}$ variables has at least one (blue or red) solid edge that has been pulled by the $Q$ or local expansions at $y$. (Here, the ``pulling" of a solid edge refers to its replacement by two new solid edges due to partial derivatives $\partial_{h_{\al\beta}}$ or the operations described in \eqref{resol_exp0} and \eqref{resol_reverse}.) 
Furthermore, the scaling size of $(\cal G_{\omega_k})_{\fa\fb}$ satisfies that 
\be\label{eq:explainQsize}\size((\cal G_{\omega_k})_{\fa\fb}) \lesssim \blam\heta^{k+1} \cdot  \heta^{[k+\mathbf 1(k=1)]/2}. \ee
To see why this holds, notice that the scaling size of $ T^{k}_{\fa\fb}\sum_x \zthn(\Sele)_{\fa x}\cal G_{x\fb}$ is at most $\blam \heta^{k+1}$ (the worst case occurs when $\cal G$ takes one of the forms in \eqref{eq:leadingQs} with $w=x$ and $y_1=y_2=\fb$). Each pulling of a solid edge decreases the scaling size at least by \smash{$\heta^{1/2}$}, so $k$ broken $T_{\fa\fb}$ variables lead to a factor \smash{$\OO(\heta^{k/2})$}. 
Finally, if only one solid edge is pulled to vertex $y$, then the arguments below \eqref{eq:leadingQs} tell that the $Q$ and local expansions on $y$ provide at least one additional \smash{$\heta^{1/2}$} factor. This explains \eqref{eq:explainQsize}. Now, using the doubly connected property and the scaling size condition \eqref{eq:explainQsize} for $(\cal G_{\omega_k})_{\fa\fb}$, we can show that  
\begin{equation}\label{estimates_ptreeQ}
\begin{split}
&(\cal G_{\omega_k})_{\fa\fb} \prec  \left[W^{-\fc/2+\errw}(B_{\fa\fb} + \wt \Phi^2)\right]^{k+1}  \, .
\end{split}
\end{equation}
In the derivation of the estimates in \eqref{estimates_ptree} and \eqref{estimates_ptreeQ}, the inequalities in Lemma \ref{w_s} will take the place of inequalities (8.10)--(8.12) from \cite{BandII}. With these inequalities and the (doubly connected and scaling size) properties of $\PT,$ $\AT$, $\WT'$, $\Err_{D}$, and $\cal G_{\omega_k}$, we can prove \eqref{estimates_ptree} and \eqref{estimates_ptreeQ} using exactly the same arguments as in \cite[Section 8.2]{BandII}. So we omit the details of the proof.

Combining \eqref{estimates_ptree1}, \eqref{estimates_ptree}, \eqref{eq:Tp}, and \eqref{estimates_ptreeQ}, using the condition \eqref{Lcondition1}, we obtain that 
\begin{align*}
    \E T_{\fa \fb}^p &\prec \left[B_{\fa\fb}+ \frac{1}{N\eta}+\left(W^{-\fc/2+\errw}+W^{-\errL+(\fC+\fc) C\errw}\right)\left(B_{\fa\fb}+\wt\Phi^2\right)\right]\E T_{\fa\fb}^{p-1}\\
    &+ \sum_{k=1}^{p-1}\left[ W^{-\fc/2+\errw} \left(B_{\fa\fb}+\wt\Phi^2\right) \right]^{k+1}\E T_{\fa\fb}^{p-1-k},
\end{align*}
as long as we take $\errw<\min\{\fc/2,\errL/((\fC+\fc) C)\}$. Then, the desired result \eqref{locallawptree} follows from an application of H\"older's and Young's inequalities.  
\end{proof}

\section{$V$-expansion and coupling renormalization}\label{sec:Vexpansion}

In this section, we establish a generalization of the $T$-expansion, called \emph{$V$-expansion} (where $V$ stands for ``vertex"); see \Cref{thm:Vexp} below. It describes the expansion of a graph of the form $T_{\fa,\fb_1\fb_2}\Gamma$, where $\Gamma$ is an arbitrary locally regular graph. In general terms, a $V$-expansion can be regarded a general form of the global expansion, where we replace $T_{\fa,\fb_1\fb_2}$ with a $T$-expansion and apply subsequent $Q$-expansions if necessary. 
What differentiates our $V$-expansions from the global expansion defined in \Cref{sec:global_WO} is the key coupling (or vertex) renormalization mechanism---as discussed in \Cref{sec:idea}. This cancellation property is stated as \eqref{eq:vertex_sum_zero} below, whose proof is the main focus of this section. The $V$-expansion and coupling renormalization mechanism will be the key tools for our proof of \Cref{gvalue_continuity} in \Cref{sec:continuity} below.

\subsection{Complete $T$-expansion}

In defining the $V$-expansions, we will replace the $T$-variable $T_{\fa,\fb_1\fb_2}$ with a \emph{complete} $T$-expansion instead of the \emph{incomplete} $T$-expansion as stated in \Cref{defn genuni}. Roughly speaking, a complete $T$-expansion is an expansion of the $T$-variable that includes only deterministic, $Q$, and error graphs. It is obtained by further expanding the graphs in $\PT$, $\AT$, and $\WT$ into deterministic, $Q$, and error graphs. 
One advantage of the complete $T$-expansion is that it takes a much simpler form than the $T$-expansion so that we do not need to handle non-deterministic recollision, higher order, and free graphs in the global expansion of a $T$-variable. However, in contrast to $T$-expansions, the complete $T$-expansions are not ``universal". This is because some graphs in $\AT$ and $\WT$ are only SPD and may be non-expandable (recall (S5) in \Cref{sec:expanstrat}). Hence, expanding the $T$-variables in them can break the doubly connected property and therefore lose a factor $L^2/W^2$ (due to the summation over a diffusive edge). As a consequence, the scaling size of our graphs can increase during the expansions if $L$ is too large. To address this issue, we really need to assume that $L$ satisfies the upper bound \eqref{Lcondition1}, under which the $L^2/W^2$ factor can be compensated by the decrease in the scaling size.        

Since the doubly connected property is a convenient tool to estimate graphs, to maintain it, we introduce a new type of ``ghost" edges to our graphs. 

	
\begin{definition}[Ghost edges]
We use a dashed edge between vertices $x$ and $y$ to represent a $W^2/(\blam L^2)$ factor and call it a \emph{ghost edge}. We do not count ghost edges when calculating the scaling size of a graph, i.e., the scaling size of a ghost edge is $1$. Moreover, the doubly connected property in Definition \ref{def 2net} is extended to graphs with ghost edges by including these edges in the blue net.   
\end{definition}

For each ghost edge added to a graph, we multiply the coefficient of the graph by a $\blam L^2/W^2$ factor. In other words, given any graph $\mathcal G$, suppose we introduce $k_\gh$ many ghost edges to it. Then, the new graph, denoted by \smash{$\wt{\cal G}$}, has a coefficient \smash{$ \cof(\wt{\cal G})= \cof (\cal G)\cdot ( {L^2}/{W^2})^{k_\gh}.$} 

Given a doubly connected graph $\cal G$ where every solid edge has a $\circ$, we define its \emph{weak scaling size} as 
\begin{align}
\wsize(\mathcal G): =&~|\cof(\cal G)|  \heta^{\#\{\text{free edges}\}+\#\{\text{diffusive edges}\}+\frac{1}{2}\#\{\text{solid edges}\}}  W^{-d(\#\{ \text{waved edges}\} -  \#\{\text{internal atoms}\})}.\label{eq:def-size}
	\end{align}
Compared with \eqref{eq_defsize}, the only difference is that every free edge is associated with a factor $\heta$ instead of $({N\eta})^{-1}\cdot {L^2}/{W^2}$.  
The following lemma shows that as long as $L$ is not too large, we can construct a $\nonuni$ whose graphs are properly bounded.
	
\begin{lemma}[Complete $T$-expansion]\label{def nonuni-T}
Under the assumptions of Theorem \ref{thm_locallaw}, suppose the local laws 
\begin{equation}\label{locallaw0}
T_{xy}\prec B_{xy}+(N\eta)^{-1},\quad 		\|G (z) -M(z)\|_{\max}^2 \prec \heta,
\end{equation}
hold for a fixed $z= E+\ii\eta$ with $|E|\le  2-\kappa$ and $\eta \in [\eta_0,\blam^{-1}]$ for some $W^\fd \eta_*\le \eta_0\le \blam^{-1}$. Fix any constants $D>\fC>0$ such that \eqref{Lcondition1} holds. Suppose that there is a $\fC$-th order $T$-expansion (with $\eta\ge \eta_0$ and $D$-th order error) satisfying Definition \ref{defn genuni}. 
Then, for the Wegner orbital model, $T_{\fa,\fb_1 \fb_2}$ can be expanded into a sum of $\OO(1)$ many normal graphs (which may contain ghost and free edges):  
\begin{align}
T_{\fa,\fb_1 \fb_2} & =   m  \wt \thn_{\fa \fb_1}\overline G_{\fb_1\fb_2}  + \frac{G_{\fb_2 \fb_1} - \overline G_{\fb_1 \fb_2}}{2\ii N\eta} +\sum_{\mu} \sum_x \wt \thn_{\fa x}\mathcal D^{(\mu)}_{x \fb_1}f_{\mu;x,\fb_1\fb_2}(G)+   \sum_\nu \sum_{x} \wt \thn_{\fa x}\mathcal D^{(\nu)}_{x \fb_2} \wt f_{\nu;x,\fb_1\fb_2}(G) \nonumber\\
& +  \sum_{\gamma} \sum_{x} \wt \thn_{\fa x}\mathcal D^{(\gamma)}_{x , \fb_1 \fb_2}g_{\gamma;x,\fb_1\fb_2}(G) + \sum_{x}\wt \thn_{\fa x} \mathcal Q_{x,\fb_1\fb_2}   + \Err_{\fa,\fb_1 \fb_2}. \label{mlevelTgdef weak}
\end{align} 
The graphs on the RHS of \eqref{mlevelTgdef weak} satisfy the following properties. 
\begin{enumerate}
\item $\Err_{\fa,\fb_1 \fb_2}$ is an error term satisfying $\Err_{\fa,\fb_1 \fb_2}\prec W^{-D}.$ 
			
\item $\wt\thn$ is a deterministic matrix satisfying $|\wt \thn_{xy}| \prec B_{xy}$. $\mathcal D^{(\mu)}_{x \fb_1}$, $\mathcal D^{(\nu)}_{x \fb_2}$, and $\mathcal D^{(\gamma)}_{x , \fb_1 \fb_2}$ are deterministic doubly connected graphs when treating $x, \fb_1, \fb_2$ as internal vertices, and they have weak scaling sizes of order $\OO(W^{-d-c})$ for a constant $c>0$ when treating $x, \fb_1, \fb_2$ as external vertices. 

\item If we denote the molecules of $\fb_1$ and $\fb_2$ by $\cal M_1$ and $\cal M_2$, then \smash{$f_{\mu;x,\fb_1\fb_2} $, $\wt f_{\nu;x,\fb_1\fb_2} $, and $g_{\gamma;x,\fb_1\fb_2}$} are all graphs consisting solely of solid edges within $\cal M_1$ and $\cal M_2$, as well as solid edges between $\cal M_1$ and $\cal M_2$, with coefficients of order $\OO(1)$. Furthermore, $f_{\mu;x,\fb_1\fb_2}$ contains one solid edge with $-$ charge between $\fb_1$ and $\fb_2$, while \smash{$\wt f_{\nu;x,\fb_1\fb_2}$} contains one solid edge with $+$ charge between $\fb_1$ and $\fb_2$.

\item $\mathcal Q_{x,\fb_1\fb_2}$ is a sum of $Q$-graphs. 
For the Wegner orbital model, we denote the expression obtained by removing the graphs in \eqref{eq:largeQW} from $\cal Q_{x,\fb_1\fb_2}$ as $\cal Q'_{x,\fb_1\fb_2}$. Then, $\QT'_{x,\fb_1\fb_2}$ is a sum of $Q$-graphs satisfying the following properties:
\begin{enumerate}

\item[(a)] They are SPD graphs (with ghost edges included in the blue net). 

\item[(b)] The vertex in the $Q$-label of each of them belongs to the MIS with non-deterministic closure. 

\item[(c)] Its weak scaling size is of order $\OO(W^{-c}\heta)$ for a constant $c>0$. 

\item[(d)] There is an edge, blue solid, waved, $\dashed$, or dotted, connected to $\fb_1$; there is an edge, red solid, waved, $\dashed$, or dotted, connected to $\fb_2$. 
\end{enumerate}
\end{enumerate}
\end{lemma}
 
\begin{proof}
For the Wegner orbital model, under the condition \eqref{Lcondition1}, the construction of the complete $T$-expansion using the $\fC$-th order $T$-expansion has been performed in the proof of \cite[Lemma 3.2]{BandIII}. 
We omit the details. 
\end{proof}
	
The construction of the complete $T$-expansion gives that $\wt \thn=\zthn(\Sele)$, where $\Sele$ is a sum of $\selfs$ $\Sele_k$, $k=1,\ldots, k_0$, and a $\pself$ $\Sele_{k'}$ with $\sizeself(\Sele_{k'})\prec W^{-\fC}$ (with $\Sele_{k'}$ potentially containing free or ghost edges). Under the condition \eqref{Lcondition1}, this $\pself$ plays the same role as a $\self$ with parameter $\sizeself(\Sele_{k'})\cdot L^2/W^2$. Then, applying \eqref{BRB} yields the estimate \smash{$|\wt \thn_{xy}| \prec B_{xy}$}. 
Furthermore, by combining this estimate with \Cref{dG-bd} for $ \cal D^{(\mu)}$ and $\cal D^{(\nu)}$, we can show that 
$$ (\wt\thn \cal D^{(\mu)})_{x\fb_1}\prec W^d\size(\cal D^{(\mu)}_{x\fb_1})\cdot B_{x\fb_1},\quad (\wt\thn \cal D^{(\nu)})_{x\fb_2}\prec W^d\size(\cal D^{(\nu)}_{x\fb_2})\cdot B_{x\fb_2}.$$
Thus, we will regard $\wt \thn$, $\wt\thn \cal D^{(\mu)}$, and $\wt\thn \cal D^{(\nu)}$ as new types of diffusive edges.

\begin{definition}[New diffusive edges]
In applying the complete $T$-expansion, we will treat $\wt \thn_{xy}$, $(\wt\thn \cal D^{(\mu)})_{xy}$, and \smash{$(\wt\thn \cal D^{(\nu)})_{xy}$} as new types of diffusive edges between $x$ and $y$. Their (weak) scaling sizes are counted as $\heta$, \smash{$\blam \size(\cal D^{(\mu)}_{xy})$}, and {$\blam \size(\cal D^{(\nu)}_{xy})$}, respectively. Moreover, the doubly connected property in Definition \ref{def 2net} extends to graphs with these new diffusive edges (which can belong to either the black net or the blue net). 
\end{definition}

\subsection{$V$-expansions} 

We now present the key tool for our proof, the $V$-expansion, in \Cref{thm:Vexp}. This is a global expansion of a $T$-variable, which involves substituting with a complete $T$-expansion followed by subsequent $Q$-expansions and local expansions. The \emph{vertex renormalization mechanism} (or equivalently, \emph{molecule sum zero property}) for the $V$-expansion is stated as part of the theorem in \eqref{eq:vertex_sum_zero}.

\begin{theorem}[$V$-expansion for the Wegner orbital model]\label{thm:Vexp}
Under the assumptions of \Cref{def nonuni-T}, consider the Wegner orbital model. Suppose ${\cal G}$ is a normal graph of the form 
   \be\label{eq:regularG}
   {\cal G}=T_{\fa,\fb_1\fb_2}\cdot \Gamma,
   \ee
where $\fa,\fb_1,\fb_2$ are external vertices and $\Gamma$ is an arbitrary normal subgraph without $P/Q$ labels. Let ${\cal E}_G(\Gamma)$ and ${\cal E}_{\overline G}(\Gamma)$ denote the collections of $G$ and $\bar G$ edges in $\Gamma$, respectively, and denote $\ell(\Gamma):=|{\cal E}_G(\Gamma)|\wedge |{\cal E}_{\overline G}(\Gamma)|$. 
Plugging the complete $T$-expansion \eqref{mlevelTgdef weak} into $T_{\fa,\fb_1\fb_2}$, applying the $Q$ and local expansions to the resulting graphs if necessary, 
we can expand \eqref{eq:regularG} as follows for any large constant $D>0$:
\begin{equation}
   \begin{split}\label{eq_Vexp}
       & {\cal G}=m  \wt \thn_{\fa \fb_1}\overline G_{\fb_1\fb_2}  \Gamma + \frac{G_{\fb_2 \fb_1} - \overline G_{\fb_1 \fb_2}}{2\ii N\eta} \Gamma + \sum_x \wt \thn_{\fa x}\mathcal D (x, \fb_1,\fb_2) \Gamma \\
    &+\sum_{p=1}^{\ell(\Gamma)}\sum_{x} \wt \thn_{\fa x}
   \Gamma^{\star}_{p}(x,\fb_1,\fb_2) + \sum_x \wt \thn_{\fa x}\mathcal A_{x,\fb_1\fb_2} + \sum_x \wt \thn_{\fa x}\mathcal R_{x,\fb_1\fb_2} +  \mathcal Q_{\fa,\fb_1\fb_2}  + \Err_D(\fa,\fb_1,\fb_2) ,
   \end{split}
   \end{equation}
   where $\cal D(x,\fb_1,\fb_2)$ is defined as
   $$\cal D(x,\fb_1,\fb_2):=\sum_{\mu} \mathcal D^{(\mu)}_{x \fb_1}f_{\mu;x,\fb_1\fb_2}(G)+   \sum_\nu \mathcal D^{(\nu)}_{x \fb_2} \wt f_{\nu;x,\fb_1\fb_2}(G)  +  \sum_{\gamma} \mathcal D^{(\gamma)}_{x , \fb_1 \fb_2}g_{\gamma;x,\fb_1\fb_2}(G),$$
   $\mathcal Q_{\fa,\fb_1\fb_2}$ is a sum of $\OO(1)$ many $Q$-graphs, and $\Err_D(\fa,\fb_1,\fb_2)$ is a sum of error graphs with $\Err_D(\fa,\fb_1,\fb_2)\prec W^{-D}$. $\mathcal A_{x,\fb_1\fb_2}$, $\mathcal R_{x,\fb_1\fb_2}$, and $\Gamma^{\star}_{p}(x,\fb_1,\fb_2)$ are sums of $\OO(1)$ many \emph{locally standard graphs} without $P/Q$ labels, which satisfy the following properties. 
   For any graph $\cal G'$, we denote by $k_{\mathsf p}(\cal G')$ the number of $G$ or \smash{$\overline G$} edges in ${\cal E}_G(\Gamma)\cup {\cal E}_{\overline G}(\Gamma)$ that have been broken (or pulled) during the expansions (i.e., solid edges that have disappeared in the final expansion \eqref{eq_Vexp}).  
\begin{itemize}
\item[(i)] $\mathcal R_{x,\fb_1\fb_2}$ is a sum of recollision graphs $\cal G'(x,\fb_1,\fb_2)$ with respect to the vertices in $\cal G$. That is, during the expansions, one of the following two cases occurs: (1) no new internal molecules are created, or (2) some existing molecules in $\cal G$ are merged with newly created internal molecules. Moreover, the weak scaling size of $\cal G'(x,\fb_1,\fb_2)$ satisfies 
\be\label{eq:sizesR} 
\wsize\Big(\sum_x\wt\thn_{\fa x}\cal G'(x,\fb_1,\fb_2)\Big) \lesssim \wsize(\cal G)\cdot \heta^{\frac{1}{2}k_{\mathsf p}(\cal G'(x,\fb_1,\fb_2))}.\ee

\item[(ii)] $\mathcal A_{x,\fb_1\fb_2}$ is a sum of non-recollision graphs $\cal G'(x,\fb_1,\fb_2)$, whose weak scaling size satisfies
\be\label{eq:sizesA} 
\wsize\Big(\sum_x\wt\thn_{\fa x}\cal G'(x,\fb_1,\fb_2)\Big) \lesssim \blam\wsize(\cal G)\cdot \heta^{\frac{1}{2}[1+k_{\mathsf p}(\cal G'(x,\fb_1,\fb_2))]}.
\ee

\item[(iii)] $\Gamma^{\star}_{p}(x,\fb_1,\fb_2)$ is a sum of locally regular \emph{star graphs} $\cal G_{p,\gamma}^\star(x,\fb_1,\fb_2)$ (with $\gamma$ denoting the label for our graphs), where $k_{\mathsf p}(\cal G_{p,\gamma}^\star)=2p$ and its weak scaling size satisfies 
\be\label{eq:starsize} \wsize\Big(\sum_{x} \wt \thn_{\fa x} \cal G_{p,\gamma}^\star(x,\fb_1,\fb_2)\Big) \lesssim \blam\wsize(\cal G)\cdot \heta^{p}.\ee
Moreover, these graphs $\cal G_{p,\gamma}^\star(x,\fb_1,\fb_2)$ have the following structure. 
\begin{itemize}
\item[(iii.1)] Only one new internal molecule $\cal M_x\ni x$ containing $x$ appears in the graphs $\cal G_{p,\gamma}^\star$. 
\item[(iii.2)] Some solid edges in ${\cal E}_G(\Gamma)\cup {\cal E}_{\overline G}(\Gamma)$ have been pulled to $\cal M_x$, and we denote them by
\be\label{eq:eiei} e_i=(a_i, b_i)\in {\cal E}_G (\Gamma), \quad \text{and}\quad \overline e_i=(\overline a_i, \overline b_i)\in {\cal E}_{\overline G} (\Gamma),\quad i=1,\ldots, p,\ee
where every $e_i$ (resp.~$\overline e_i$) represents a $G_{a_ib_i}$ (resp.~$\overline G_{\overline a_i \overline b_i}$) edge. Denote $b_0=\fb_1$, $\bar b_0=\fb_2$, and \smash{$E_p:=\{e_1,\ldots, e_p\}$}, \smash{$\overline E_p:=\{\overline e_1,\ldots, \overline e_p\}$}. Let 
\be\label{eq:pairingP}
\qquad\qquad \cal P(E_p,\overline E_p)\equiv \cal P_{b_0\bar b_0}(E_p,\overline E_p)=\{(a_{k},\overline a_{\pi(k)}): k=1,\ldots, p\}\cup \{(b_{l},\overline b_{\sigma(l)}): l=0,\ldots, p\}
\ee
denote a pairing of vertices such that each $a_k$ (resp.~$b_l$) vertex is paired with a $\overline a_{\pi(k)}$ (resp.~$\overline b_{\sigma(l)}$) vertex, under the restriction that $b_0$ is not paired with \smash{$\bar b_0$}. Here, $\pi:\{1,\ldots,p\}\to \{1,\ldots,p\}$ and $\sigma:\{0,\ldots,p\}\to \{0,\ldots,p\}$ represent bijections such that $\sigma(0)\ne 0$. 

\item[(iii.3)] Given $E_p$, $\overline E_p$, and $\cal P(E_p,\overline E_p)$, there exists a collections of internal vertices $\vec x(\ii):=(x_k(\ii):k=1,\ldots, p)$ and $\vec x(\ff):=(x_k(\ff):k=0,1,\ldots, p)$ belonging to $\cal M_x\setminus \{x\}$ such that 
\begin{align}
\cal G_{p,\gamma}^\star(x,\fb_1,\fb_2)&=\sum_{\omega} \sum_{\substack{x_k(\ii):k=1,\ldots, p; \\ x_k(\ff):k=0,\ldots, p}}  \mathfrak D_\omega^{\lambda}(x, \vec x(\ii),\vec x(\ff))\prod_{k=1}^p \left(G_{a_k x_k(\ii)}\overline G_{\overline a_{\pi(k)} x_k(\ii)}\right)  \nonumber\\
    & \times  \prod_{k=0}^p \left(G_{x_k(\ff)b_k}\overline G_{x_k(\ff)\bar b_{\sigma(k)} }\right) \cdot \frac{\Gamma}{\prod_{k=1}^p \left(G_{a_k b_k} \overline G_{\overline a_{k} \bar b_{k} }\right) }.\label{eq:Gpgamma}
\end{align} 
Here, the RHS involves $\OO(1)$ many deterministic graphs $\mathfrak D_\omega^{\lambda}(x, \vec x(\ii),\vec x(\ff))$ (where $\omega$ denotes the labels of these graphs). Each graph consists of external vertices $x, \vec x(\ii),\vec x(\ff)$, an order $\OO(1)$ coefficient, and waved (i.e., $S$ and $S^\pm$) edges connecting internal vertices. 
Moreover, $\mathfrak D_\omega^{\lambda}(x, \vec x(\ii),\vec x(\ff))$ satisfies that 
\be\label{eq:wavednumber} \# \left\{\text{waved edges in } \mathfrak D_\omega^{\lambda}\right\} = \# \left\{\text{vertices in } \cal M_x\setminus\{x\}\right\},\ee
and we have explicitly indicated the dependence of $ \fD_\omega^{\lambda}$ on $\lambda$.
\end{itemize}
Let $\gamma=(E_p,\overline E_p,\cal P(E_p,\overline E_p),\omega)$ label all graphs of the above forms. 
 Given $E_p$, $\overline E_p$, and $\cal P(E_p,\overline E_p)$, when $\lambda=0$, each $\fD_\omega^{\lambda=0}$ can be expressed as 
\be\label{eq:vertex_sum_zero_reduce}
\fD_\omega^{\lambda=0}(x,\vec x(\ii),\vec x(\ff))=\Delta_\omega \cdot  \prod_{k=1}^p  S_{xx_k(\ii)}(\lambda=0)\cdot \prod_{k=0}^p S_{xx_k(\ff)}(\lambda=0)
\ee
for a deterministic coefficient $\Delta_\omega\equiv \Delta_\omega(E_p, \overline E_p,\cal P(E_p,\overline E_p))=\OO(1)$. These coefficients are polynomials in $m, \bar m, (1-m^2)^{-1}, (1-\bar m^2)^{-1}$ (with $\OO(1)$ coefficients) and satisfy the {\bf molecule sum zero property} (or the {\bf coupling or $(2p+2)$-vertex renormalization}):
\be\label{eq:vertex_sum_zero}
\sum_\omega \Delta_\omega(E_p, \overline E_p,\cal P(E_p,\overline E_p)) =\OO(\eta).
\ee
The remainder part satisfies that 
\be\label{eq:vertex_sum_zero_diff}
\size\left[\fD_\omega^{\lambda}(x,\vec x(\ii),\vec x(\ff))- \fD_\omega^{\lambda=0}(x,\vec x(\ii),\vec x(\ff))\right] \lesssim \blam^{-1} W^{-(2p+1)d} .
\ee
\end{itemize}
\end{theorem}
To help the reader understand the structure of star graphs, we draw an example with $p=2$: 
\be\label{eq:star1}
\parbox[c]{14cm}{\includegraphics[width=\linewidth]{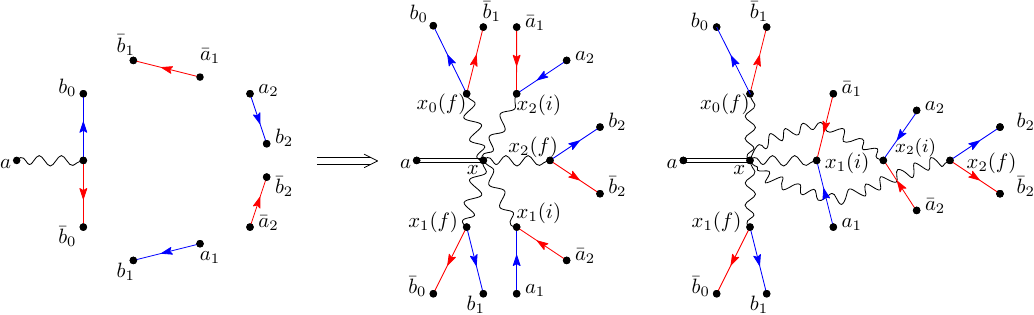}}
\ee
The first graph corresponds to the original graph $\cal G$, where we only draw the $T$-variable $T_{\fa,\fb_1\fb_2}\equiv T_{a,b_0\bar b_0}$ and the 4 solid edges pulled in the star graphs. It may lead to the two star graphs on the RHS that correspond to the following two pairings: 
\begin{align}
&\cal P(E_2,\bar E_2)=\{ \{b_0,\bar b_1\}, \{a_2,\bar a_1\}, \{b_2,\bar b_2\}, \{a_1,\bar a_2\},\{b_1,\bar b_0\}\} ,\quad \text{or} \ \label{eq:6-loop}\\
&\cal P(E_2,\bar E_2)=\{ \{b_0,\bar b_1\},\{b_1,\bar b_0\}, \{a_1,\bar a_1\}, \{a_2,\bar a_2\}, \{b_2,\bar b_2\}\}.\label{eq:4+2-loop}
\end{align} 
In these two star graphs, it is easy to observe their star structures where $x$ is the center and the waved edges may represent $S$ or $S^\pm$ edges. We also note that, unlike these two examples, there may be additional vertices along the path between $x$ and a vertex in $\vec x(\ii)$ or $\vec x(\ff)$.  

\begin{proof}[\bf Proof of \Cref{thm:Vexp}]
For simplicity of presentation, in the following proof, we regard all vertices in the original graph $\cal G$ as external vertices, while the internal vertices are the newly produced ones during the expansions. The first three graphs on the RHS of \eqref{eq_Vexp} are obtained by replacing $T_{\fa,\fb_1\fb_2}$ with the first five terms on the RHS of \eqref{mlevelTgdef weak}. $\cal Q_{\fa,\fb_1\fb_2}$ contains all the $Q$-graphs obtained from the expansion process, and $\Err_D(\fa,\fb_1,\fb_2) $ contains all the error graphs of order $\OO_\prec (W^{-D})$. We include all recollision graphs into $\mathcal R_{x,\fb_1\fb_2}$, where the bound \eqref{eq:sizesR} is due to the fact that every pulling of an edge in ${\cal E}_G(\Gamma)\cup {\cal E}_{\overline G}(\Gamma)$ decreases the weak scaling size at least by \smash{$\heta^{1/2}$}. We include all the other graphs satisfying \eqref{eq:sizesA} into $\cal A_{x,\fb_1\fb_2}$. For the rest of the proof, we focus on explaining how the star graphs $\Gamma^{\star}_p$ are obtained and proving the molecule sum zero property \eqref{eq:vertex_sum_zero}, which is the core of the proof of \Cref{thm:Vexp}. 

We claim that all the star graphs come from the expansions of the following term: 
\be\label{eq:prior_star}
\cal G_0=\sum_x \wt \thn_{\fa x} Q_x\left(G_{x \fb_1}   \overline G_{x \fb_2}\right)\Gamma.\ee
First, suppose we have replaced $T_{\fa,\fb_1\fb_2}$ by a $Q$-graph, denoted by $\sum_{x}\wt \thn_{\fa x} \mathcal Q^{(\omega)}_{x,\fb_1\fb_2}$, that is {\bf not} in 
\be\label{eq:largeQs}\sum_x \wt\thn_{\fa x}Q_x\left(G_{x \fb_1}   \overline G_{x \fb_2}\right) - |m|^2\sum_x \wt\thn_{\fa x}Q_x\left(T_{x, \fb_1\fb_2}\right) -   m\wt \thn_{\fa\fb_1} Q_{\fb_1} ( \overline G_{\fb_1 \fb_2}),\ee
where these graphs correspond to the first three terms on the RHS of \eqref{eq:Q2}. Then, we have 
$$ \wsize\Big( \sum_{x}\wt \thn_{\fa x} \mathcal Q^{(\omega)}_{x,\fb_1\fb_2}\Big)\le \blam \wsize (T_{\fa,\fb_1\fb_2})\cdot \heta^{1/2}.$$
Thus, applying the $Q$ and local expansions to 
$\sum_{x}\wt \Theta_{\fa x} \mathcal Q^{(\omega)}_{x,\fb_1\fb_2} \Gamma$ will produce graphs satisfying \eqref{eq:sizesA} and can hence be included into the last four terms on the RHS of \eqref{eq_Vexp}. 
The third graph in \eqref{eq:largeQs} does not contain any new internal vertex, so the expansion of \smash{$ - m\wt \thn_{\fa\fb_1}  Q_{\fb_1} ( \overline G_{\fb_1 \fb_2}) \Gamma$} will only produce recollision, $Q$, and error graphs belonging to the last three terms on the RHS of \eqref{eq_Vexp}.
Next, for the second term in \eqref{eq:largeQs}, using \eqref{resol_exp0} and \smash{$Q_x(G^{(x)}_{y\fb_1}\overline G^{(x)}_{y\fb_2})=0$}, we can write  
\begin{align}
Q_x\left(T_{x,\fb_1\fb_2}\right)&=\sum_y S_{xy} Q_x\left(\frac{G_{yx}G_{x\fb_1}}{G_{xx}}\overline G_{y\fb_2}\right) + \sum_y S_{xy} Q_x\left(G_{y \fb_1} \frac{\bar G_{yx}\bar G_{x\fb_2}}{\bar G_{xx}}\right)\nonumber\\
    &-\sum_y S_{xy} Q_x\left(\frac{G_{yx}G_{x\fb_1}}{G_{xx}}\frac{\bar G_{yx}\bar G_{x\fb_2}}{\bar G_{xx}}\right).\label{eq:QT}
\end{align} 
All three terms on the RHS have weak scaling size at most $ \wsize(T_{x,\fb_1\fb_2})\heta^{1/2}$, so $\sum_x \wt\thn_{\fa x}Q_x\left(T_{x,\fb_1\fb_2}\right)\Gamma$ will also only produce graphs belonging to the last four terms on the RHS of \eqref{eq_Vexp}.

Now, we focus on the expansion of \eqref{eq:prior_star}. Following \eqref{eq:eiei}, we denote the edges in $\cal E_G$ and $\cal E_{\overline G}$ as 
$$ \cal E_G=\{e_k=(a_k, b_k): k=1,\ldots, |\cal E_G|\},\quad \cal E_{\overline G}=\{\overline e_k=(\overline a_k, \overline b_k): k=1,\ldots, |\cal E_{\overline G}|\},$$
where every $e_k$ (resp.~$\overline e_k$) represents a $G_{a_k b_k}$ (resp.~$\overline G_{\overline a_k \overline b_k}$) edge.

\medskip

\noindent I. {\bf $Q$-expansions.} In Step 1 of the $Q$-expansion, using \eqref{QG2}, we can write $ \cal G_0$ as a sum of $Q$-graphs plus 
$\sum_x \wt\thn_{\fa x} Q_x\left(G_{x \fb_1}   \overline G_{x \fb_2}\right)  \Gamma_1$, where $\Gamma_1$ is defined as
\be\label{eq:Qstep1}
\Gamma_1:=  \sum_{\substack{1\le k\le |\cal E_G|,0\le l\le |\cal E_{\overline G}| \\ k+l\ge 1}} \sum_{\substack{\vec {\mathbf i}_k=(i_1,\ldots, i_k)\\ \subset \{1,\ldots,|\cal E_G|\} } }\sum_{\substack{\vec {\mathbf j}_l=(j_1,\ldots, j_l)\\     \subset \{1,\ldots,|\cal E_{\overline G}|\}}} (-1)^{k+l-1} \prod_{r=1}^k \frac{G_{a_{i_r}x}G_{xb_{i_r}}}{G_{xx}} \cdot \prod_{s=1}^l\frac{\overline G_{\overline a_{j_s}x}\overline G_{x\overline b_{j_s}}}{\overline G_{xx} }\cdot \Gamma_{- \vec {\mathbf i}_k , \vec {\mathbf j}_l} \ ,
\ee
where $\Gamma_{- \vec {\mathbf i}_k , \vec {\mathbf j}_l}$ is the graph obtained by removing the edges labeled by $\vec {\mathbf i}_k $ and $ \vec {\mathbf j}_l$ from $\Gamma$:
$$\Gamma_{- \vec {\mathbf i}_k , \vec {\mathbf j}_l}:=\frac{\Gamma}{\prod_{r=1}^k G_{a_{i_r}b_{i_r}}\cdot \prod_{s=1}^l G_{\overline a_{j_s}\overline b_{j_s}} }.$$
In the above expressions, if we replace some $1/G_{xx}$ by $1/G_{xx}-1/m$, then the resulting graph will satisfy \eqref{eq:sizesA} and its expansions will produce graphs in the last four terms on the RHS of \eqref{eq_Vexp}. Hence, to get the star graphs, it remains to expand the following expression:  
\be\label{eq:Qstep1_2}
\Gamma_2:=  \sum_{\substack{1\le k\le |\cal E_G|,0\le l\le |\cal E_{\overline G}| \\ k+l\ge 1}} \sum_{\substack{\vec {\mathbf i}_k=(i_1,\ldots, i_k)\\ \subset \{1,\ldots,|\cal E_G|\} } }\sum_{\substack{\vec {\mathbf j}_l=(j_1,\ldots, j_l)\\     \subset \{1,\ldots,|\cal E_{\overline G}|\}}} \frac{(-1)^{k+l-1}}{m^{k}\overline m^{l}} \prod_{r=1}^k  G_{a_{i_r}x}G_{xb_{i_r}}  \cdot \prod_{s=1}^l\overline G_{\overline a_{j_s}x}\overline G_{x\overline b_{j_s}} \cdot \Gamma_{- \vec {\mathbf i}_k , \vec {\mathbf j}_l}  \ .
\ee


Given $k,\ l,\ \vec {\mathbf i}_k,$ and $\vec {\mathbf j}_l$, we need to further expand 
$$\cal G_{\vec {\mathbf i}_k,\vec {\mathbf j}_l}=\frac{(-1)^{k+l-1}}{m^{k}\overline m^{l}} \sum_x \wt\thn_{\fa x} Q_x\left(G_{x \fb_1}   \overline G_{x \fb_2}\right)  \prod_{r=1}^k  G_{a_{i_r}x}G_{xb_{i_r}}  \cdot \prod_{s=1}^l \overline G_{\overline a_{j_s}x}\overline G_{x\overline b_{j_s}} \cdot \Gamma_{- \vec {\mathbf i}_k , \vec {\mathbf j}_l} .$$
In Step 2 of the $Q$-expansion, using the expansion \eqref{gH0}, we can write that 
\begin{align*}
    Q_x\left(G_{x \fb_1}   \overline G_{x \fb_2}\right) =&~ G_{x \fb_1}   \overline G_{x \fb_2} - m\delta_{x\fb_1}P_x( \overline G_{x \fb_2})- |m|^2 T_{x, \fb_1  \fb_2} +  |m|^2Q_x\left(T_{x, \fb_1  \fb_2}\right)\\
    &~- m  \sum_y S_{xy}P_x\left[(\overline G_{xx}-\overline m)G_{y \fb_1}   \overline G_{y \fb_2}\right]  - m\sum_y S_{xy} P_x\left[(G_{yy}-m) G_{x \fb_1}   \overline G_{x \fb_2}\right].
\end{align*}
The second term on the RHS is a $\{\fb_1\}$-recollision graph, the last two terms contain light weights, and the fourth term can be written as \eqref{eq:QT}. Hence, expansions involving them will produce recollision graphs or graphs satisfying \eqref{eq:sizesA}, which can be included in the last four terms on the RHS of \eqref{eq_Vexp}. To get star graphs, we only need to consider the expansions of  
\be\label{eq:G'ij}
\cal G'_{\vec {\mathbf i}_k,\vec {\mathbf j}_l}=\frac{(-1)^{k+l-1}}{m^{k}\overline m^{l}} \sum_x \wt\thn_{\fa x}  \left(G_{x \fb_1}   \overline G_{x \fb_2}-|m|^2 T_{x, \fb_1 \fb_2}\right) \cal F_{x,\vec {\mathbf i}_k , \vec {\mathbf j}_l},  
\ee
where we have abbreviated that 
$$ \cal F_{x,\vec {\mathbf i}_k , \vec {\mathbf j}_l}:= \prod_{r=1}^k  G_{a_{i_r}x}G_{xb_{i_r}}  \cdot \prod_{s=1}^l \overline G_{\overline a_{j_s}x}\overline G_{x\overline b_{j_s}} \cdot \Gamma_{- \vec {\mathbf i}_k , \vec {\mathbf j}_l}\ .$$

\noindent II. {\bf Edge expansions.} Next, we apply the edge expansions in \Cref{Oe14} (with $y_1=\fb_1$) to the graph
$G_{x \fb_1}   \overline G_{x \fb_2} \cdot \cal F_{x,\vec {\mathbf i}_k , \vec {\mathbf j}_l}$. 
In the resulting graphs, there is one term (i.e., the pairing between $G_{x \fb_1}$ and $\overline G_{x \fb_2}$) that cancels the term \smash{$ -|m|^2 T_{x, \fb_1 \fb_2}  \cdot \cal F_{x,\vec {\mathbf i}_k , \vec {\mathbf j}_l}.$}
Some graphs contain light weights, so expansions involving them will produce graphs in the last four terms on the RHS of \eqref{eq_Vexp}. Furthermore, terms involving pairings between $G_{x\fb_1}$ and $G_{xb_{i_r}}$ or $\overline G_{\overline a_{j_s}x}$ contains one more internal solid edge $G_{x\al}$ or $\overline G_{\al x}$. Hence, their expansions will also lead to graphs in the last four terms on the RHS of \eqref{eq_Vexp}. Then, we only need to expand the following term to get the star graphs:
\begin{align*}
\cal G''_{\vec {\mathbf i}_k,\vec {\mathbf j}_l}=&~\frac{(-1)^{k+l-1}}{m^{k}\overline m^{l}}  \sum_{x,\al} \wt\thn_{\fa x}  S_{x\al} \left[m^2\sum_{r=1}^k\left(G_{\al \fb_1} G_{a_{i_r}\al}\right)  \frac{\cal F_{x,\vec {\mathbf i}_k , \vec {\mathbf j}_l}}{G_{a_{i_r}x}} +  |m|^2 \sum_{s=1}^l \left(G_{\al \fb_1} \overline G_{\al\overline b_{j_s}}\right)  \frac{\cal F_{x,\vec {\mathbf i}_k , \vec {\mathbf j}_l}}{\overline G_{x\overline b_{j_s}}} \right]\\
&~ - \frac{(-1)^{k+l-1}}{m^{k}\overline m^{l}}  \sum_{x,\al} \wt\thn_{\fa x}  S_{x\al}   \overline G_{x \fb_2}   \prod_{r=1}^k  G_{a_{i_r}x}G_{xb_{i_r}}  \cdot \prod_{s=1}^l \overline G_{\overline a_{j_s}x}\overline G_{x\overline b_{j_s}} \cdot \left(G_{\al \fb_1}\partial_{h_{\al x}}\Gamma_{- \vec {\mathbf i}_k , \vec {\mathbf j}_l}\right) .
\end{align*}

Now, we apply the edge expansions in \Cref{Oe14} repeatedly. More precisely, suppose we have a graph (such as \smash{$\cal G''_{\vec {\mathbf i}_k,\vec {\mathbf j}_l}$}) that still contains $G/\overline G$ edges attached to $x$. Assume it takes the form given in \eqref{multi setting}, with $f(G)$ expressed as  
\be\label{eq:extf}f(G)=f_{ext}(G)\wt f(G),\ee
where $f_{ext}(G)$ is the subgraph consisting of all the external solid edges, and $\wt f(G)$ is the remaining graph consisting of the solid edges between external and internal vertices in $\cal M_x\setminus \{x\}$. 
Then, we apply the edge expansion in \eqref{Oe1x} and only keep terms of the following forms:
\be\label{eq:onlygraphs}
\begin{split} 
|m|^2\sum_\al S_{x\al }G_{\al y_1} & \overline G_{\al y'_i} \frac{\cal G}{G_{x y_1} \overline G_{xy_i'}},\quad  m^2 \sum_\al S_{x\al }G_{\al y_1} G_{w_i \al}\frac{\cal G}{G_{xy_1}G_{w_i x}}  ,\\
&- m    \sum_\al S_{x\al } \frac{G_{\al y_1}\partial_{ h_{\al x}}f_{ext}(G)}{G_{x y_1}f_{ext}(G)} \cal G.
\end{split}\ee
The first two cases correspond to neutral pairings at $\al$ (recall ``neutral charge" defined in \Cref{deflvl1}), while the last case involves the pulling of some external edges.  
All the other graphs are either recollision graphs or contain light weights or internal solid edges within the molecule $\cal M_x$, so expansions involving them will only contribute to the last four terms on the RHS of \eqref{eq_Vexp}. 
Above, we have considered the edge expansion with respect to $G_{xy_1}$. One can also perform edge expansions with respect to other edges in 
$$\prod_{i=2}^{k_1}G_{x y_i}  \cdot  \prod_{i=1}^{k_2}\overline G_{x y'_i} \cdot \prod_{i=1}^{k_3} G_{ w_i x} \cdot \prod_{i=1}^{k_4}\overline G_{ w'_i x}.$$

We repeatedly apply the edge expansions and keep only the graphs as in \eqref{eq:onlygraphs} that will finally contribute to the star graphs. We continue this process until all the resulting graphs have no solid edges attached to $x$. We denote these graphs by 
\be\label{eq:G'''}
\cal G'''_{\vec {\mathbf i}_k,\vec {\mathbf j}_l}=\sum_{x} \wt\thn_{\fa x} \sum_{\omega} \cal G_{\omega}(x,\vec {\mathbf i}_k,\vec {\mathbf j}_l,\Gamma), \ee
with $\omega$ representing their labels. It is not hard to see that the edge we choose at each step of the above edge expansion process does not affect the final expression in \eqref{eq:G'''}, as it should. In fact, from \eqref{eq:onlygraphs}, we see that the RHS of \eqref{eq:G'''} should involve all pairings of the solid edges connected to $x$, including those pulled to $x$ by the partial derivatives $\partial_{h_{\al x}}$.  


\medskip

\noindent III. {\bf $GG$ expansions}. After Step II, no edge connects to vertex $x$ in the graphs $\cal G_\omega$ in \eqref{eq:G'''}. Every internal vertex, say $\al$, in $\cal M_x\setminus\{x\}$ connects to $x$ through an $S$ edge and to external vertices through two solid edges. If these two solid edges have opposite charges, then $\al$ is locally regular and is one of the $x_k(\ii)$ or $x_k(\ff)$ vertices in \eqref{eq:Gpgamma}. However, there are also some internal vertices, say $x'$, connected with two edges of the same charge: $ G_{x' y}G_{y' x'}$ or $\overline G_{x'y} \overline G_{y'x'}$, where $y$ and $y'$ are external vertices. Then, we apply the $GG$ expansion in \Cref{T eq0} to it (with $x$ in \eqref{Oe2x} replaced by $x'$). Consider such a graph of the form $G_{x'y}   G_{y' x' }  f (G)$, where $f$ takes the form \eqref{eq:extf}. In this expansion, we only keep the terms 
\begin{align*}
- m \sum_{ \al}  S_{x'\al}G_{\al y} G_{y' x'} \wt f(G)\partial_{ h_{\al x}}f_{ext}(G),\quad  - m^3 \sum_{\al,\beta} S^{+}_{x'\al} S_{\al\beta}G_{\beta y} G_{y' \al} \wt f(G)\partial_{ h_{\beta \al}}f_{ext}(G).
\end{align*}
For each other graph, one of the following two cases holds:
\begin{itemize}
    \item it has a waved edge between $\cal M_x$ and an external vertex, and hence becomes a recollision graph;

    \item it contains light weights or internal solid edges within the molecule $\cal M_x$. 
\end{itemize}
Thus, expansions involving them will only produce graphs in the last four terms on the RHS of \eqref{eq_Vexp}. For a graph of the form $\overline G_{x'y}   \overline G_{y' x' }  f (G)$, we follow a similar expansion process by taking complex conjugates. 

We continue these $GG$ expansions until there are no same-colored pairings (i.e., $GG$ or $\bar G\bar G$ pairings). Again, we observe that the order of the $GG$ expansions (i.e., which pair of $GG$ or $\overline G\overline G$ edges we choose to expand at each step) does not affect the final result.

\medskip

After Step III, if some graphs are still not locally regular, we return to Step II and apply the edge expansion to them. By repeatedly applying Steps II and III, we obtain the star graphs in \eqref{eq_Vexp} (after renaming the vertices of the edges in $E_p\cup \bar E_p$), with $2p$ representing the external edges that have been pulled (once and only once) during the expansion process. 
As a result of Steps I--III defined above, it is easy to see that these star graphs satisfy \eqref{eq:starsize} and properties (iii.1)--(iii.3). It remains to prove the properties \eqref{eq:vertex_sum_zero_reduce}--\eqref{eq:vertex_sum_zero_diff}. 
First, the property \eqref{eq:vertex_sum_zero_reduce} follows easily from the definition of $S$. More precisely, when $\lambda=0$, we have $m(z,\lambda=0)=m_{sc}(z)$, and 
\be\label{eq:SSpm} S_{xy}(\lambda=0)=W^{-d}\mathbf 1_{[x]=[y]},\quad S^+_{xy}(\lambda=0)=\overline S^-_{xy}(\lambda=0) = [1-m_{sc}^2(z)]^{-1}W^{-d}\mathbf 1_{[x]=[y]}. \ee
Then, $\fD_\omega^{\lambda=0}(x,\vec x(\ii),\vec x(\ff))$ consists of a coefficient $f_\omega(z)$, which is a monomial of $m_{sc}(z)$, $\overline m_{sc}(z)$, $(1-m_{sc}^2(z))^{-1}$, and $(1-\overline m_{sc}^2(z))^{-1}$, and a graph (denoted by $\cal G_{\omega}$) that consists of waved edges representing $S(\lambda=0)$ and $S^\pm(\lambda=0)$ between internal vertices in $\cal M_x$. 
With \eqref{eq:SSpm}, we can write $\cal G_{\omega}$ as
\begin{align*}
    \cal G_{\omega} =&~  \frac{W^{-\# (\text{waved edges})\cdot d}W^{ (|\cal M_x| -2p -2) d}}{(1-m_{sc}^2(z))^{\# \{S^+ \text{ edges}\}} (1-\overline m_{sc}^2(z))^{\# \{S^- \text{ edges}\}}}  \prod_{k=1}^p\mathbf 1(x_k(\ii)\in [x])\cdot \prod_{k=0}^p\mathbf 1(x_k(\ff)\in [x])\\
    =&~ \frac{1}{(1-m_{sc}^2(z))^{\# \{S^+ \text{ edges}\}} (1-\overline m_{sc}^2(z))^{\# \{S^- \text{ edges}\}}}  \prod_{k=1}^p  S_{xx_k(\ii)}(\lambda=0)\cdot \prod_{k=0}^p S_{xx_k(\ff)}(\lambda=0) .
\end{align*} 
Above, $W^{ (|\cal M_x| -2p -2) d}$ accounts for the factors arising from the summation over the internal vertices in $\cal M_x\setminus \{x, x_1(\ii),\ldots, x_p(\ii), x_0(\ff),x_1(\ff),\ldots, x_p(\ff)\}$, and we have used \eqref{eq:wavednumber} in the second step. This concludes \eqref{eq:vertex_sum_zero_reduce} by setting  
$$ \Delta_\omega=f_\omega(z) \left(1-m_{sc}^2(z)\right)^{-\# \{S^+ \text{ edges}\}} \left(1-\overline m_{sc}^2(z)\right)^{-\# \{S^- \text{ edges}\}}.$$
The estimate \eqref{eq:vertex_sum_zero_diff} is a simple consequence of the fact that $m(z)-m_{sc}(z)=\OO(\lambda^2)$, 
$ S_{xy}(\lambda)-S_{xy}(\lambda=0) = \lambda^2\cdot W^{-d}\mathbf 1_{[x]\sim [y]},$ and \eqref{S+xy-0}. 

Finally, it remains to establish the molecule sum zero property \eqref{eq:vertex_sum_zero}, which is the core part of the proof. From the construction of the star graphs, we observe that $\Delta_\omega$ does not depend on the order of expansions or the choices of the subsets $E_p$ and $\overline E_p$---it depends solely on the pairing $\cal P(E_p,\overline E_p)$. To give a rigorous statement of this fact, let $E'_p=\{e'_1,\ldots, e'_p\}$ and \smash{$\overline E'_p=\{\overline e'_1,\ldots, \overline e'_p\}$} be another two subsets of $p$ external $G$ and $\overline G$ edges in $\Gamma$ (we allow for \smash{$E_p=E'_p$ and $\overline E_p=\overline E_p'$}), where
\be\label{eq:e'i} e_i'=(a_i', b_i')\in {\cal E}_G (\Gamma), \quad \text{and}\quad \overline e_i'=(\overline a_i', \overline b_i')\in {\cal E}_{\overline G} (\Gamma),\quad i=1,\ldots, p.\ee
Let $\phi:E_p\to E_p'$ and $\overline \phi:\overline E_p \to \overline E_p'$ be arbitrary bijections, i.e., there exist permutations $\mu, \overline \mu\in S_p$ such that 
$$ \phi(e_k)=(e'_{a(\mu(k))},e'_{b(\mu(k))}),\quad \overline \phi(\overline e_k)=(\overline e'_{\overline a(\overline \mu(k))},\overline e'_{\overline b(\overline \mu(k))}).$$
Then, let 
$\cal P'(E_p',\overline E_p')$ be a pairing induced by $\Phi=(\phi,\overline \phi)$ and \eqref{eq:pairingP}, i.e., 
$$\cal P'(E_p',\overline E_p')=\left\{\left(a'_{\mu(k)},\overline a'_{\pi\circ \overline \mu(k)}\right), \left(b'_{\mu(l)},\overline b'_{\sigma\circ \overline \mu(l)}\right): k\in \qqq{1,p}, \ l\in \qqq{0, p}\right\},$$
where as a convention, we let $\bar \mu(0)=0$. Under these notations, we have that 
\be\label{eq:permu_inv}
\sum_{\omega} \Delta_\omega ( \cal P(E_p,\overline E_p) )=\sum_{\omega} \Delta_\omega ( \cal P'(E_p',\overline E_p') ).
\ee

To take into account the invariance under bijections, we define the following concept of loop graphs. Given $E_p=\{e_1,\ldots, e_p\}$ and $\overline E_p=\{\overline e_1,\ldots, \overline e_p\}$ with the edges in \eqref{eq:eiei}, we cut the $2p+2$ edges in $\{(x,\fb_1), (x,\fb_2)\}\cup E_p\cup \overline E_p$ into $4p+2$ half-edges with $4p+2$ dangling ends 
\begin{equation}\label{pts}
\mathsf X_{p}:= \{  x_0(\ff), \ \bar x_0(\ff)\}\cup \{ x_k(\ii), x_k(\ff):k=1,\ldots, p\}\cup \{ \overline x_k(\ii), \overline x_k(\ff):k=1,\ldots, p \}. 
\end{equation}
More precisely, we have the following half-edges: $(x_0(\ff),b_0)$, $(\bar x_0(\ff),\bar b_0)$, $(a_k,x_k(\ii))$, $(x_k(\ff),b_k)$, $(\overline a_k,\overline x_k(\ii))$, and $(\overline x_k(\ff),\overline b_k)$ for $k=1,\ldots, p$. Then, we can encode an arbitrary pairing $\cal P(E_p,\overline E_p)$ into a pairing between these dangling ends, where a pairing between two ends, say $\al,\beta\in \mathsf X_p$, means setting $\al=\beta$. For example, a pairing $x_j(\ff)=\overline x_k(\ff)$ (resp.~$x_j(\ii)=\overline x_k(\ii)$) corresponds to $G_{x_j(\ff)b_j}\overline G_{x_j(\ff)\overline b_k}$ (resp.~$G_{a_jx_j(\ii)}\overline G_{\overline a_k x_j(\ii)}$). 
We denote by $ \mathbf \Pi_p$ the set of all possible pairings of $4p+2$ dangling ends. Specifically, every pairing $\Pi\in \mathbf \Pi_p$ takes the form $\Pi=\{\sigma_i : i=1,\ldots, 2p+1\}$ with $\sigma_i=\{\al_i,\beta_i\}$, $\al_i,\beta_i\in \mathsf X_p$, such that the following two properties hold: 
\begin{enumerate}
    \item the two elements of $\sigma_i$  have the same $\ii$ or $\ff$ label; 

    \item every $\sigma_i$ contains one $x$-type element (i.e., $x_0(\ff)$, $x_k(\ii)$, $x_k(\ff)$, $k=1,\ldots, p$) and one $\bar x$-type element (i.e., $\overline x_0(\ff)$, $\overline x_k(\ii)$, $\overline x_k(\ff)$, $k=1,\ldots, p$).
\end{enumerate}

Now, suppose a $\Pi\in \mathbf\Pi_p$ represents a $\cal P(E_p,\overline E_p)$ pairing. Then, we define 
\be\label{def:Pi}
\Delta(\Pi):=\sum_\omega \Delta_\omega(\cal P(E_p,\overline E_p)).
\ee
The symmetry \eqref{eq:permu_inv} implies that $\Delta(\Pi)$ depends solely on the \emph{loop structure} of $\Pi$, defined as follows. 


\begin{definition}[Loop graphs]\label{def_Loop} 
A graph $\Gamma$ with $2p$ vertices is called a \emph{loop graph} if it consists of $2p$ solid edges such that the following properties hold:  
 \begin{enumerate}
     \item $\Gamma$ is a union of disjoint polygons (or called loops). 
     \item Every vertex is connected with exactly one $G$ edge and one $\bar G$ edge.
     \item  All vertices have neutral charges (recall \Cref{deflvl1}).
\end{enumerate}
In other words, when there is only one loop in $\Gamma$, it takes the form
\begin{equation}\label{Gloop}
    \prod_{i=1}^{2p} G^{s_i}_{x_i x_{i+1}},\quad x_{2p+1}=x_1,\quad s_i=\begin{cases}
    \emptyset,\quad i\in 2\mathbb Z+1,
    \\
    \dagger ,\quad i\in 2\mathbb Z,
\end{cases}
\end{equation}
where $x_i$, $i=1,\ldots, 2p$, are the vertices and $G^{s_i}$ represents $G$ if $s_i=\emptyset$ and $G^\dagger$ if $s_i=\dagger$. The loop graphs with more than one loops can be written as products of graphs of the form \eqref{Gloop}.
\end{definition}

\begin{definition}[Loop structure of $\Pi$]  \label{def_lgPi} 
Given the $4p+2$ vertices in \eqref{pts}, an additional vertex $x$, and a pairing $\Pi=\{\sigma_k:k=1,\ldots, 2p+1\}\in \mathbf \Pi_p$, we define the graph $\mathsf L_G(\Pi)$ generated by $\Pi$ as
$$\mathsf L_G(\Pi):= G_{xx_0(\ff)} \overline G_{x\bar x_0(\ff)} \prod_{k=1}^p \left[G_{x_k(\ii) x_k(\ff)} \overline G_{\bar x_k(\ii) \bar x_k(\ff)}\right] \cdot \prod_{k=1}^{2p+1}\mathbf 1((\sigma_k)_1=(\sigma_k)_2), $$
where $(\sigma_k)_1$ and $(\sigma_k)_2$ denote the two elements of $\sigma_k$. In other words, $\mathsf L_G(\Pi)$ is the \emph{loop graph} obtained by identifying the end points of the $2p+2$ $G$ edges according to the pairing $\Pi$.
Then, we define the loop structure of $\Pi$ as a triple $\mathrm{Struc}(\Pi)=(k(\Pi),\ell_0(\Pi),\mathbf L(\Pi))$, where
\begin{enumerate}
     \item $k(\Pi)$ represents the number of loops in $\mathsf L_G(\Pi)$;
     \item $\ell_0(\Pi)+2$ represents the length of the loop containing $x$ (note $\ell_0(\Pi)$ only counts the edges that are not $(x,x_0(\ff))$ or $(x,\overline x_0(\ff))$);
     \item $\mathbf L(\Pi)=\{\ell_1, \ldots, \ell_{k(\Pi)-1}\}$ is the \emph{multiset} of lengths of the other loops that do not contain $x$.
 \end{enumerate}

\end{definition}

To give an example of loop graphs generated by pairings, we consider the paring \eqref{eq:6-loop}, which corresponds to 
\be\label{eq:6-loopPi}
\Pi_1=\{\{x_0(\ff),\bar x_1(\ff)\},\{x_2(\ii),\bar x_1(\ii)\},\{x_2(\ff),\bar x_2(\ff)\},\{x_1(\ii),\bar x_2(\ii)\},\{x_1(\ff),\bar x_0(\ff)\}\}, 
\ee
and the paring \eqref{eq:4+2-loop}, which corresponds to 
\be\label{eq:4+2-loopPi}
\Pi_2=\{\{x_0(\ff),\bar x_1(\ff)\},\{x_1(\ii),\bar x_1(\ii)\},\{x_1(\ff),\bar x_0(\ff)\},\{x_2(\ii),\bar x_2(\ii)\},\{x_2(\ff),\bar x_2(\ff)\}\}.\ee
We now draw the two loop graphs generated by these two pairings:
\be\label{eq:star2}
\parbox[c]{13cm}{\includegraphics[width=\linewidth]{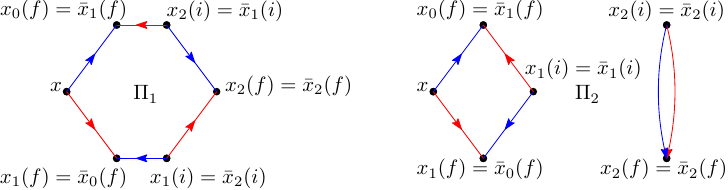}}
\ee
The pairing \eqref{eq:6-loopPi} generates a 6-edge loop with $\mathrm{Struc}(\Pi_1)=(1,4,\emptyset)$, while the pairing \eqref{eq:4+2-loopPi} generates a $4$-edge loop plus a 2-edge loop with $\mathrm{Struc}(\Pi_2)=(2,2,\{2\})$.

 
\begin{lemma}\label{lem:aLG}
$\Delta (\Pi)$ depends only on the loop structure of $\Pi$, i.e., $\Delta (\Pi)=\Delta (\Pi')$ whenever  
 $\mathrm{Struc}(\Pi)=\mathrm{Struc}(\Pi')$.
\end{lemma}
\begin{proof}
Given $E_p$, $\overline E_p$, $\cal P(E_p,\overline E_p)$, and the set \eqref{pts}, we choose an arbitrary pairing $\Pi \in \mathbf \Pi_p$ that represents $\cal P(E_p,\overline E_p)$. Pick two other subsets of $G$ and $\overline G$ edges denoted by $E'_p=\{e'_i:i =1,\ldots, p\}$ and \smash{$\overline E'_p=\{\overline e'_i:i =1,\ldots, p\}$}, where the edges take the form \eqref{eq:e'i}. Then, we again take the set of vertices \eqref{pts} that correspond to the following half-edges: $(x_0(\ff),b_0)$, $ (\bar x_0(\ff),\bar b_0)$, $ (a_k',x_k(\ii))$, $ (x_k(\ff),b_k')$, $ (\overline a_k',\overline x_k(\ii))$, and \smash{$(\overline x_k(\ff),\overline b_k')$} for $k=1,\ldots, p.$  Let $\Pi'$ be another pairing in $\mathbf \Pi_p$ that has the same loop structure as $\Pi$. Suppose $\Pi'$ represents a pairing \smash{$\cal P'(E_p',\overline E_p')$}. Now, in view of \eqref{eq:permu_inv}, it suffices to show that there exist bijections $\phi:E_p\to E_p'$ and $\overline \phi:\overline E_p \to \overline E_p'$ that induce {$\cal P'(E_p',\overline E_p')$} from $\cal P(E_p,\overline E_p)$.   

Notice that the loops in $\mathsf L_G(\Pi)$ and $\mathsf L_G(\Pi')$ naturally give such bijections. To describe one such bijection, we first establish a 1-1 correspondence between the loops in $\mathsf L_G(\Pi)$ and those in $\mathsf L_G(\Pi')$ such that
\begin{itemize}
    \item the loop containing $x$ in $\mathsf L_G(\Pi)$, denoted by $\text{Loop}_0$, maps to the loop containing $x$ in $\mathsf L_G(\Pi')$, denoted by $\text{Loop}'_0$;
    \item a loop $\text{Loop}_i\in \mathsf L_G(\Pi)$, $i\in \qqq{1,k(\Pi)-1}$, maps to a loop $\text{Loop}'_i\in \mathsf L_G(\Pi')$ that has the same length as  $\text{Loop}_i$.
\end{itemize}
Then, in $\text{Loop}_i$ and $\text{Loop}'_i$, $i=0,\ldots, \ell_{k(\Pi)-1}$, we enumerate their edges along the polygons according to an order fixed by a starting point and an initial edge.  For $i=0$, we can choose the starting point as $x$ in both loops and the initial edge as the blue solid edge $(x,x_0(f))$; for $i\ge 1$, we select a starting point $x_{k_i}(\ii)$ (resp.~$x_{l_i}(\ii)$) and an initial blue solid edge $(x_{k_i}(\ii),x_{k_i}(\ff))$ (resp.~$(x_{l_i}(\ii),x_{l_i}(\ff))$) on $\text{Loop}_i$ (resp.~$\text{Loop}_i'$) for some $k_i$ (resp.~$l_i$). 
Suppose these edges (enumerated according to the above order) correspond to the following edges in \smash{$E_p\cup \overline E_p$ and $E_p'\cup\overline E_p'$} (recall the notation $\ell_i$, $i\in \qqq{0,  \ell_{k(\Pi)-1}}$, introduced in \Cref{def_lgPi}): 
$$\begin{cases}
    e_{\pi_i(1)}, \overline e_{\overline \pi_i(1)}, \ldots , e_{\pi_i(\ell_i/2)}, \overline e_{\overline \pi_i(\ell_i/2)}, \ & \text{ on }\text{Loop}_i\\
    e'_{\pi_i'(1)}, \overline e'_{\overline \pi_i'(1)}, \ldots, e'_{\pi_i'(\ell_i/2)}, \overline e'_{\overline \pi_i'(\ell_i/2)},\ & \text{ on }\text{Loop}'_i
\end{cases}. $$
 Here, we did not include the edges $(x,x_0(\ff))$ and $(x,\overline x_0(\ff))$ for the $i=0$ case. Additionally, $\pi_i,\ \overline \pi_i, \ \pi_i',\ \overline \pi_i'$ represent some injections from $\qqq{1,\ell_i/2}$ to $\qqq{1,p}$. 
Now, we can define the bijections $\phi$ and $\overline \phi$ such that 
$$ \phi(e_{\pi_i(j)})=e'_{\pi_i'(j)},\quad \overline \phi(\overline e_{\overline \pi_i(j)})=\overline e'_{\overline \pi_i'(j)}, \quad \text{for}\quad i\in \qqq{0,k(\Pi)-1}, \ \ j\in \qqq{1,\ell_i/2}. $$
This concludes the proof.  
\end{proof}

Under the above notations, the sum zero property \eqref{eq:vertex_sum_zero} can be equivalently stated as follows.   
\begin{lemma} [Molecule sum zero property]\label{lem:vertex_zero}
In the setting of \Cref{thm:Vexp}, for any fixed $p\in\mathbb N$ and all pairings $\Pi\in \mathbf \Pi_p$, we have that 
\be\label{eq:vertex_renorm}
\Delta (\Pi)\equiv \Delta(\mathrm{Struc}(\Pi))=\OO(\eta). 
\ee
\end{lemma}

With this lemma, we conclude \eqref{eq:vertex_sum_zero}, which completes the proof of \Cref{thm:Vexp}. 
\end{proof}

The rest of this section is devoted to the proof of \Cref{lem:vertex_zero}. As mentioned in the introduction, this property is closely related to a deep coupling renormalization mechanism in Feynman diagrams. Before giving the formal proof, we first look at two examples with $p\in\{1,2\}$ to illustrate the concept of $V$-expansions and the key molecule sum zero property. 
In the first example, we consider the $p=1$ case with $4$ solid edges. 
\begin{example}\label{example_4sum0}
Consider the graph $ \cal G= T_{\fa,b_0\bar b_0}G_{a_1 b_1} \overline G_{\bar a_1\bar b_1}.$ We need to expand \eqref{eq:prior_star} with $\fb_1=b_0$, $\fb_2=\bar b_0$, and \smash{$\Gamma=G_{a_1 b_1} \overline G_{\bar a_1\bar b_1}$}. First, applying the $Q$-expansions, we get the following graphs as in  \eqref{eq:G'ij}:
 \begin{align}
	&\left( G_{x b_0} \overline {G}_{x \bar b_0} -|m|^2 T_{x,b_0\bar b_0}\right)\left(G_{a_1b_1} \frac{\overline G_{\bar a_1 x}\overline G_{x\bar b_1}}{\bar m}+\overline G_{\bar a_1\bar b_1}\frac{G_{a_1 x} G_{x b_1}}{m} - \frac{G_{a_1 x} G_{x b_1} \overline G_{\bar a_1 x}\overline G_{x\bar b_1}}{|m|^2} \right), \label{eq:Delta1}
\end{align}
where, for simplicity of presentation, we did not include the $\wt\thn_{\mathfrak{a} x}$ edge.
Next, applying the edge expansion to \eqref{eq:Delta1}, we get the following graphs that will lead to the star graphs:
\begin{align*}
&  \sum_{y_1}S_{xy_1} \left[ m  (G_{y_1b_0}\overline G_{y_1\bar b_1}) \overline G_{x\bar b_0} \overline G_{\bar a_1 x} G_{a_1 b_1} + \frac{m}{\bar m} (G_{y_1b_0}G_{a_1 y_1})  \bar {G}_{x \bar b_0} G_{xb_1}\overline G_{\bar a_1 x}\overline G_{x\bar b_1} \right] \\
+ &\sum_{y_1}S_{xy_1} \left[  m  (G_{yb_0}G_{a_1 y})\overline G_{x\bar b_0} G_{x b_1}  \overline G_{\bar a_1\bar b_1}+ (G_{y_1b_0}\bar G_{y_1 \bar b_1}) \bar G_{x\bar b_0} G_{a_1 x} G_{x b_1}\bar G_{\bar a_1 x} \right]\\
-& \sum_{y_1} S_{xy_1} \left[ (G_{y_1 \bar b_0} \bar G_{y_1\bar b_1})  \overline {G}_{x \bar b_0} G_{a_1 x} G_{xb_1} \overline G_{\bar a_1 x}+ \frac{m}{\bar m} (G_{y_1 b_0}G_{a_1 y_1}) \overline {G}_{x \bar b_0} G_{x b_1} \overline G_{\bar a_1 x}\overline G_{x\bar b_1} \right].
\end{align*}
Note that the two terms with coefficient $m/\bar m$ will not contribute to star graphs since they will give rise to graphs that involve a $GG$ pairing $G_{y_1b_0}G_{a_1 y_1}$. The fourth and fifth terms cancel each other. For the rest of the two terms, applying another edge expansion, we obtain that 
\begin{align*}
& \sum_{y_1,y_2} S_{xy_1}S_{xy_2}(G_{y_1b_0}\overline G_{y_1\bar b_1})   \left[ \bar m|m|^2 (\overline G_{y_2\bar b_0} \overline G_{\bar a_1 y_2}) G_{a_1 b_1} + |m|^2(\overline G_{y_2\bar b_0}G_{y_2b_1}) \overline G_{\bar a_1 x} G_{a_1 x} \right]\\
+ & \sum_{y_1,y_2} S_{xy_1}S_{xy_2} (G_{y_1b_0}G_{a_1 y_1}) \left[  m |m|^2 (\overline G_{y_2\bar b_0} G_{y_2 b_1})  \overline G_{\bar a_1\bar b_1}+ |m|^2 (\overline G_{y_2 \bar b_0} \overline G_{\bar a_1 y_2})G_{x b_1}  \overline G_{x\bar b_1} \right].
\end{align*}
Note that the last term contains $GG$ and $\bar G\bar G$ pairings, so it will not contribute to star graphs. Finally, by applying a $GG$-expansion to the first and third terms and an edge expansion to the second term, we derive the following graphs (upon renaming some matrix indices): 
\begin{align*}
&~\bar m^2 |m|^2  \sum_{y_1,y_2,y_3} S_{xy_1}S_{xy_2}S_{y_2y_3}(G_{y_1b_0}\overline G_{y_1\bar b_1})( G_{a_1 y_2}\overline G_{\bar a_1 y_2})(  G_{y_3 b_1}\overline G_{y_3\bar b_0})  \\
+& ~\bar m^4 |m|^2  \sum_{y_1,y_2,y_3,\al } S_{xy_1}S_{x\al }S^{-}_{\al y_2}S_{y_2 y_3}(G_{y_1b_0}\overline G_{y_1\bar b_1})  (G_{a_1 y_2}\overline G_{\bar a_1 y_2}) ( G_{y_3 b_1}\overline G_{y_3\bar b_0}) \\
+ &~ m^2 |m|^2 \sum_{y_1,y_2,y_3} S_{xy_2}S_{xy_3}S_{y_1y_2} (G_{y_1b_0}\overline G_{y_1\bar b_1})(G_{a_1 y_2}\overline G_{\bar a_1y_2}) (G_{y_3 b_1} \overline G_{y_3\bar b_0} ) \\
+ &~ m^4 |m|^2 \sum_{y_1,y_2,y_3,\al} S_{xy_3}S_{x\al}S^+_{\al y_2}S_{y_1y_2} (G_{y_1b_0}\overline G_{y_1\bar b_1})(G_{a_1 y_2}\overline G_{\bar a_1y_2}) (G_{y_3 b_1}\overline G_{y_3\bar b_0}) \\
+&~|m|^4 \sum_{y_1,y_2,y_3} S_{xy_1}S_{xy_2}S_{xy_3}(G_{y_1b_0}\overline G_{y_1\bar b_1}) (G_{a_1 y_2}\overline G_{\bar a_1 y_2} ) (G_{y_3b_1}\overline G_{y_3\bar b_0}) .
\end{align*}
Note that the first four graphs are not locally regular, since they contain a non-standard neutral vertex $y_2$. Then, applying another edge expansion, we get the following star graphs:
\be\label{eq:stars4loop}\sum_{y_1,y_2,y_3}\mathfrak D(x,y_1,y_2,y_3) (G_{y_1b_0}\overline G_{y_1\bar b_1})( G_{a_1 y_2}\overline G_{\bar a_1 y_2})(  G_{y_3 b_1}\overline G_{y_3\bar b_0}).\ee
Here, $\mathfrak D(x,y_1,y_2,y_3)$ is given by 
\begin{align*}
&    \mathfrak D(x,y_1,y_2,y_3)=\bar m^2 |m|^4 \sum_\al S_{xy_1}S_{x\al}S_{\al y_2}S_{\al y_3} +\bar m^4 |m|^4  \sum_{\al,\beta } S_{xy_1}S_{x\beta }S^{-}_{\beta \al}S_{\al y_2}S_{\al y_3} \\
    &\quad + m^2 |m|^4 \sum_{\al} S_{xy_3} S_{x\al}S_{\al y_1} S_{\al y_2}+m^4 |m|^4 \sum_{\al,\beta} S_{xy_3}S_{x\beta}S^+_{\beta \al}S_{\al y_1}S_{\al y_2} +|m|^4  S_{xy_1}S_{xy_2}S_{xy_3}\\
    &=\bar m^2 |m|^4  \sum_{\al} S_{xy_1}S^{-}_{x \al}S_{\al y_3}S_{\al y_2}  +m^2 |m|^4 \sum_{\al} S_{xy_3}S^+_{x \al}S_{\al y_1}S_{\al y_2} +|m|^4  S_{xy_1}S_{xy_2}S_{xy_3},
\end{align*}
where we used $ S(1 + m^2 S^+)=S^+$ in the second step. The expression \eqref{eq:stars4loop} corresponds to the pairing 
$\cal P(E_1,\bar E_1)=\{\{b_0,\bar b_1\}, \{a_1,\bar a_1\}, \{b_1,\bar b_0\}\}$, or equivalently, $ \Pi_0=\left\{\{x_0(\ff),\bar x_1(\ff)\},\{x_1(\ii),\bar x_1(\ii)\},\{x_1(\ff),\bar x_0(\ff)\}\right\}.$ The loop graph $\mathsf L_G(\Pi_0)$ generated by $\Pi_0$ is drawn as follows, which is a subgraph of $\mathsf L_G(\Pi_2)$ in \eqref{eq:star2}:
\be\label{eq:star3}
\parbox[c]{4cm}{\includegraphics[width=\linewidth]{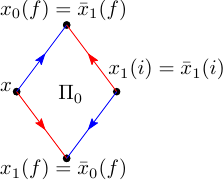}}
\ee
Note that $ \mathfrak D(x,y_1,y_2,y_3)$ gives the value of the coefficient $\Delta(\Pi_0)$: 
$$ \Delta(\Pi_0)=|m|^4 \left( 1 + \iota + \bar \iota \right),\quad \text{where}\ \ \iota:=\frac{m^2}{1-m^2}.$$
Using $|m|=1+\OO(\eta)$ when $\lambda=0$, we can derive that 
\be\label{eq:Pi4}
\iota + \bar \iota= -1+\OO(\eta) \ \Rightarrow \ \Delta(\Pi_0)=\OO(\eta) .\ee
This gives the molecule sum zero property \eqref{eq:vertex_renorm} for $\Pi_0$. 
\end{example}

In the second example, we consider the $p=2$ case with $6$ solid edges:
 
\begin{example}\label{example_6sum0}
Consider the graph
 $$\cal G=T_{\fa,b_0\bar b_0}\prod_{i=1}^2 \left(G_{a_i b_i}\overline G_{\bar a_i \bar b_i}\right).$$
We consider the two pairings in \eqref{eq:6-loop} and \eqref{eq:4+2-loop}, which correspond to $\Pi_1$ in \eqref{eq:6-loopPi} and $\Pi_2$ in \eqref{eq:4+2-loopPi}, respectively. Recall that they generate the two loop graphs in \eqref{eq:star2}.

To derive the coefficients $\Delta(\Pi_1)$ and $\Delta(\Pi_2)$, we need to expand \eqref{eq:prior_star} with $\Gamma=\prod_{i=1}^2 (G_{a_i b_i}\overline G_{\bar a_i \bar b_i})$. Applying the $Q$-expansions, we get the following graphs as in  \eqref{eq:G'ij}:
 \begin{align*}
	&\left( G_{x b_0} \overline {G}_{x \bar b_0} - |m|^2T_{x,b_0\bar b_0}\right)(\Gamma_1+\Gamma_2), 
\end{align*}
where, for simplicity of presentation, we did not include the $\wt\thn_{\mathfrak{a} x}$ edge, and for $i=2-j\in \{1,2\}$, $\Gamma_{i}$ is defined as follows:  
 \begin{align}
 	\Gamma_i=&~  m^{-1}G_{a_i x}G_{x b_i} \overline G_{\bar a_i \bar b_i}G_{a_j b_j}\overline G_{\bar a_j \bar b_j}+\bar m^{-1}G_{a_i b_i} \overline G_{\bar a_i x}\overline G_{x\bar b_i}G_{a_j b_j}\overline G_{\bar a_j \bar b_j}  \nonumber
 	\\
 	&~-|m|^{-2}G_{a_i x}G_{x b_i} \overline G_{\bar a_i x}\overline G_{x\bar b_i}G_{a_j b_j}\overline G_{\bar a_j \bar b_j}-|m|^{-2}G_{a_i x}G_{x b_i} \overline G_{\bar a_i \bar b_i}G_{a_j b_j}\overline G_{\bar a_j x}\overline G_{x \bar b_j}  \nonumber\\
 	&~ -\frac{1}{2} m^{-2}G_{a_i x}G_{x b_i} \overline G_{\bar a_i \bar b_i}G_{a_j x} G_{x b_j}\overline G_{\bar a_j \bar b_j}-\frac{1}{2}\bar m^{-2}G_{a_i b_i} \overline G_{\bar a_i x}\overline G_{x\bar b_i}G_{a_j b_j}\overline G_{\bar a_j x}\overline G_{x\bar b_j}   \nonumber\\
 	&~+ |m|^{-2}m^{-1} G_{a_i x}G_{x b_i} \overline G_{\bar a_i x}\overline G_{x\bar b_i}G_{a_j x}G_{x b_j}\overline G_{\bar a_j \bar b_j}+ |m|^{-2}\bar m^{-1}G_{a_i x}G_{x b_i} \overline G_{\bar a_i x}\overline G_{x\bar b_i}G_{a_j b_j}\overline G_{\bar a_j x}\overline G_{x\bar b_j} \nonumber\\
 	& ~-\frac{1}{2}|m|^{-4}G_{a_i x}G_{x b_i} \overline G_{\bar a_i x}\overline G_{x\bar b_i}G_{a_j x} G_{x b_j}\overline G_{\bar a_j x}\overline G_{x\bar b_j}.\label{eq:Gammai}
 \end{align}
Next, with a lengthy but straightforward calculation (see \Cref{subsec:6v}), 
we derive that  
\begin{align}
\Delta(\Pi_1) &= |m|^6\left(1+2 \iota+ 2 \bar \iota+\left|\iota\right|^2 +2\iota^2+2\bar \iota^2 + \iota^3+ \bar \iota^3\right), \quad 
 \Delta(\Pi_2)=0.\label{DeltaPi1}
\end{align}
Using the estimate \eqref{eq:Pi4}, we obtain that 
\begin{align*}
\Delta(\Pi_1) &=|m|^6 \left(\iota+\bar \iota+1\right)\left(\iota^2+\bar \iota^2 -|\iota|^2 +\iota +\bar \iota  +1\right)=\OO(\eta) .
\end{align*}
This gives the molecule sum zero property \eqref{eq:vertex_renorm} for $\Pi_1$ and $\Pi_2$.
\end{example}

We believe the two examples, Examples \ref{example_4sum0} and \ref{example_6sum0}, provide fairly strong evidence for the molecule sum zero property \eqref{eq:vertex_renorm}. For the rest of this section, we give a rigorous proof of this fact.

\subsection{Proof of \Cref{lem:vertex_zero}} 

In this subsection, we reduce the problem to a special case with $k(\Pi)=1$, i.e., there is only one loop in $\Pi$. 

\begin{lemma}[Reduction to the 1-loop case]\label{lem_F1T2}
For any fixed $p\in \N$, suppose that for each $q\in \qqq{1, p-1}$ and all $\Pi\in \mathbf \Pi_q$ with $k(\Pi)=1$, we have $\Delta(\Pi)=\OO(\eta)$. 
Then, for each $\Pi\in \mathbf \Pi_p$ with $k(\Pi)\ge 2$, we have that 
\be\label{eq:ind_zumzero}
\Delta(\Pi)=\OO(\eta). 
\ee
\end{lemma}
\begin{proof}
Given a $\Pi\in \mathbf \Pi_p$ with $k(\Pi)\ge 2$, we have $\ell_0(\Pi)=2q$ for some $q\in \qqq{1, p-1}$. Let $\Pi' \in  \mathbf \Pi_{q}$ be the pairing induced by the loop in $\Pi$ that contains vertex $x$. Then, \eqref{eq:ind_zumzero} is a simple consequence of the induction hypothesis assumed in this lemma and the following fact:
\begin{equation}
    \label{comPP}
\Delta(\Pi)= C_{\Pi\setminus \Pi'}\cdot \Delta(\Pi') ,
\end{equation}
where $C_{\Pi\setminus \Pi'}$ is a deterministic coefficient of order $\OO(1)$.

Due to \Cref{lem:aLG}, to show \eqref{comPP}, we only need to consider $\Delta(\Pi)$ constructed from the $V$-expansion of \eqref{eq:regularG} for a special graph $\Gamma$ with exactly $p$ $G$ edges and $p$ $\bar G$ edges. More precisely, we take
$$E_p=\left\{e_i=(a_i,b_i):i=1,\ldots, p\right\},\quad \bar E_p=\left\{\bar e_i=(\bar a_i,\bar b_i):i=1,\ldots, p\right\},\quad \text{and}\quad \Pi=\prod_{i=1}^p (G_{a_i b_i}\overline G_{\overline a_i \overline b_i}). $$
In the following proof, we always consider the half-edges defined by $(x_0(\ff),b_0)$, $(\bar x_0(\ff),\bar b_0)$, $(a_k,x_k(\ii))$, $(x_k(\ff),b_k)$, $(\overline a_k,\overline x_k(\ii))$, and $(\overline x_k(\ff),\overline b_k)$ for $k\in \qqq{p}$. In other words, $x_0(\ff)$ and $\bar x_0(\ff)$ are the internal vertices connected to $b_0=\fb_1$ and $\bar b_0=\fb_2$, which may change from step to step during the expansion process. Similarly, the vertices $x_k(\ii)$, $x_k(\ff)$, $\bar x_k(\ii)$, and $\bar x_k(\ff)$ are the internal vertices connected with $a_k$, $b_k$, $\overline a_k$, and $\bar b_k$, respectively. Without loss of generality, we assume that 
\be\label{eq:defPi'} \Pi'=\left\{\{x_0(\ff),\bar x_1(\ff)\},\{\bar x_1(\ii),x_1(\ii)\},\{x_1(\ff),\bar x_2(\ff)\},\ldots, \{\bar x_{q}(\ii),x_{q}(\ii)\},\{x_q(\ff),\bar x_0(\ff)\}\right\},\ee
which corresponds to the external edges in $E_q=\{e_i:i=1,\ldots, q\}$, $\bar E_q=\{\bar e_i:i=1,\ldots, q\}$, and the pairing 
$$\cal P(E_q,\bar E_q)=\{\{b_0,\bar b_1\}, \{\bar a_1, a_1\}, \{b_1,\bar b_2\},\ldots, \{\bar a_q,a_q\},\{b_q,\bar b_0\} \}.$$

We now examine the expansion process more closely and focus only on the terms that lead to star graphs corresponding to the pairing $\Pi$. One observation is that at each step, the relevant graphs are entirely determined by the equivalence classes (or pairings) of the vertices. In Step I, after the $Q$-expansions, we get a sum of graphs of the form \eqref{eq:G'ij}, where we have the following equivalence class of internal vertices  
$$\{x, x_{i_1}(\ii), x_{i_1}(\ff), \ldots, x_{i_k}(\ii), x_{i_k}(\ff), \overline x_{j_1}(\ii), \overline x_{j_1}(\ff), \ldots, \overline x_{j_l}(\ii), \overline x_{j_l}(\ff)\},$$
meaning that all half-edges are attached to $x$. In Step II, we get a sum of graphs of the form \eqref{eq:G'''}, where no solid edge is attached to $x$ and all half-edges are paired. These pairs are of the form 
\be\label{pairing}\left\{x_{k_i}^{c_i}(s_i), x_{k'_i}^{c'_i}(s_i')\right\},\quad k_i,k'_i\in \qqq{0,p},\ c_i,c'_i\in \{\emptyset,-\}, \ s_i,s_i'\in \{\ii, \ff\}, \ee 
where we adopt the conventions that $x^\emptyset \equiv x$ and $x^- \equiv \bar x$, and we have $s_i=s'_i$ if $c_i\ne c'_i$, and $s_i\ne s'_i$ if $c_i= c'_i$.
Finally, in Step III, we apply the $GG$ expansions to all $GG$ or $\bar G \bar G$ pairings. Each $GG$ expansion will pull an external edge and form pairings of the form \eqref{pairing}. We perform the $GG$ expansion repeatedly to every same-colored pairing until we have no $GG$ or $\bar G \bar G$ pairing in every resulting graph.  After Step III, if some graphs still contain non-standard neutral vertices, then we return to Step II and apply edge expansions to them. Repeating Steps II and III,  we finally get the desired star graphs that are determined by the pairings of the form \eqref{pairing}.   

Now, we look at the star graphs that are relevant to the pairing $\Pi$ and denote them as in \eqref{eq:Gpgamma}:
\begin{align}\label{eq:Gpgamma2}
    \cal G_{p,\gamma}^\star(x,\fb_1,\fb_2)=&\sum_{\omega} \sum_{x_k(\ii), x_k(\ff):k=1,\ldots, p}  \mathfrak D_\omega^{\lambda}(x,\vec x(\ii),\vec x(\ff))\cdot \cal G_{\Pi'} \cdot \cal G_{(\Pi')^c},
\end{align}
where $\cal G_{\Pi'}$ represents the product of solid edges related to vertices in $\Pi'$,  
\begin{align*}
\cal G_{\Pi'}= (G_{x_0(\ff) b_0 }\bar G_{ x_0(\ff) \bar b_{1}})(G_{x_q(\ff) b_q }\bar G_{ x_q(\ff) \bar b_{0}}) \cdot \prod_{k=1}^{q-1} (G_{x_k(\ff) b_k }\bar G_{ x_k(\ff) \bar b_{k+1}} )  \cdot \prod_{k=1}^q (G_{a_k x_k(\ii)}G_{\bar a_k x_k(\ii)}),
\end{align*}  
and $\cal G_{(\Pi')^c}$ represents the product of solid edges related to vertices in $\Pi\setminus \Pi'$. Let $\cal A$ be the set of vertices in \eqref{eq:defPi'}: 
$$\cal A:=\{x_0(\ff),\bar x_0(\ff)\}\cup \{x_{k}(\ii),x_{k}(\ff),\bar x_{k}(\ii),\bar x_{k}(\ff):k=1,\ldots, q\}.$$
We claim that during Steps II and III, in all relevant graphs that finally lead to star graphs with pairing $\Pi$, the vertices in $\cal A$ are only paired with vertices in $\cal A$, and not with any vertices in 
$$\cal A^c:=\{x_{k}(\ii),x_{k}(\ff),\bar x_{k}(\ii),\bar x_{k}(\ff):k=q+1,\ldots, p\}.$$ 

To see why this claim holds, suppose that at a certain step, a vertex $x_{i}^{c_i}(s_i)\in \cal A$ is paired with a vertex $x_{j}^{c_j}(s_j)\in \cal A^c$. If they have different charges ($c_i\ne c_j$), then this pairing $\{x_{i}^{c_i}(s_i),x_{j}^{c_j}(s_j)\}$ will not be affected in later expansions, and we cannot achieve the pairing $\Pi$ in the final star graphs. If they have the same charge ($c_i=c_j$), we need to apply a $GG$ expansion to this pairing at some stage. During the $GG$ expansion, it will pull an external edge, say $(a_k^{c_k},b_k^{c_k})$ where $c_k\in \{\emptyset,-\}$, and create two new pairings  
$$\{x_{i}^{c_i}(s_i), x_k^{c_k}(\ii)\},\ \{x_{j}^{c_j}(s_j),x_k^{c_k}(\ff)\} \quad \text{or}\quad \{x_{i}^{c_i}(s_i), x_k^{c_k}(\ff)\},\ \{x_{j}^{c_j}(s_j),x_k^{c_k}(\ii)\}.$$
Regardless of whether $a_k^{c_k},b_k^{c_k}\in \cal A$ or not, we get at least one pairing between a vertex in $\cal A$ and a vertex in $\cal A^c$. Consequently, we cannot achieve the pairing $\Pi$ in the final graphs. Hence, in Steps II and III of the expansion process, vertices in $\cal A$ (resp.~$\cal A^c$) are always paired with vertices in $\cal A$ (resp.~$\cal A^c$). This means that the expansions involving the edges in 
$$\cal B=E_q\cup \bar E_q \cup \{(x_0(\ff),b_0), (\bar x_0(\ff),\bar b_0)\}\cup\{ (a_k,x_k(\ii)), (x_k(\ff),b_k), (\overline a_k,\overline x_k(\ii)), (\overline x_k(\ff),\overline b_k): k=1,\ldots, q\},$$ 
are performed independently of the edges in 
$$\cal B^c=[(E_p\cup \bar E_p)\setminus(E_q\cup \bar E_q)]\cup \{(a_k,x_k(\ii)), (x_k(\ff),b_k), (\overline a_k,\overline x_k(\ii)),(\overline x_k(\ff),\overline b_k): k=q+1,\ldots, p\}.$$ 
As a consequence, for every star graph in \eqref{eq:Gpgamma2}, $\mathfrak D_\omega^{\lambda}(x,\vec x(\ii),\vec x(\ff))$ can be decomposed into two components involving vertices in $\cal A$ and $\cal A^c$, respectively:
$$\mathfrak D_\omega^{\lambda}(x,\vec x(\ii),\vec x(\ff))=\mathfrak D_\omega^{\lambda}(x,\cal A)\cdot \mathfrak D_\omega^{\lambda}(\cal A^c).$$
Given a fixed expression of $\mathfrak D_\omega^{\lambda}(\cal A^c)=\mathfrak D$, summing over all graphs containing such a graph, we get 
$$ \sum_\omega \mathfrak D_\omega^{\lambda}(x,\cal A)\cdot \mathfrak D_\omega^{\lambda}(\cal A^c) \mathbf 1(\mathfrak D_\omega^{\lambda}(\cal A^c)=\mathfrak D)= f^\lambda_{\mathfrak D}(x,\cal A)\cdot \mathfrak D,$$
where $f^\lambda_{\mathfrak D}(x,\cal A)$ represents a sum of graphs involving $x$ and vertices in $\cal A$. Due to our earlier claim, it does not depend on the choice of $\mathfrak D$, i.e., $f^\lambda_{\mathfrak D}(x,\cal A)=f^\lambda(x,\cal A)$. In particular, by taking $\mathfrak D$ to be an empty graph (i.e., no edge in $(E_p\cup \bar E_p)\setminus(E_q\cup \bar E_q)$ is pulled), we find that $f^\lambda(x,\cal A)$ must be equal to $\sum_\omega \Delta_\omega (E_q,\bar E_q, \cal P(E_q,\bar E_q))$. Summing over all graphs $\mathfrak D$ that form the pairing $\Pi\setminus \Pi'$ and setting $\lambda=0$, we conclude \eqref{comPP}. 
\end{proof}

Now, to show \Cref{lem:vertex_zero}, we only need to establish the following result. 
\begin{lemma}[Molecule sum zero property for 1-loop case]\label{lem_For1} 
For any fixed $p\in \N$, suppose that for each $q\in \qqq{1, p-1}$, all $\Pi\in \mathbf \Pi_q$, and $z=E+\ii \eta$ with $|E|\le 2-\kappa$ and $0<\eta\le 1$, we have $\Delta(\Pi)\equiv \Delta(\Pi,z)=\OO(\eta)$. Then, for each $\Pi\in \mathbf \Pi_p$ with $k(\Pi)=1$, we have that 
\be\label{eq:main_zumzero}
\Delta(\Pi,z)=\OO(\eta)
\ee
for all $z=E+\ii \eta$ with $|E|\le 2-\kappa$ and $0<\eta\le 1$.
\end{lemma}

\begin{proof} [\bf Proof of \Cref{lem:vertex_zero}]
Combining Lemmas \ref{lem_F1T2} and \ref{lem_For1}, we conclude \Cref{lem:vertex_zero} by induction.
\end{proof}

\subsection{$V$-expansions of GUE}
The rest of this section is devoted to proving \Cref{lem_For1} above. From the calculations in \Cref{subsec:6v}, 
one can see that the operations leading to star graphs and the value of $\Delta(\Pi)$ are so intricate that providing a direct proof of the molecule sum zero property in \eqref{eq:main_zumzero} is nearly impossible. Instead of evaluating $\Delta(\Pi)$ directly, we will prove the molecule sum zero property via the $V$-expansions for GUE. An ingredient for this proof is the following local law established in the literature for Wigner matrices.  

\begin{theorem}[Local law of GUE]\label{thm:GUE_local}
Let $\sH$ be an $\sN\times \sN$ ($\sN:=W^d$) GUE, whose entries are independent Gaussian random variables subject to the Hermitian condition $\sH=\sH^\dag$. 
The diagonal entries of $\sH$ are distributed as ${\cal N }_{\R}(0, \sN^{-1})$, while the off-diagonal entries are distributed as ${\cal N }_{\C}(0, \sN^{-1})$. Define its Green's function as $ \sG(z):=(\sH-z)^{-1}.$
Let $\kappa,\erre \in (0,1)$ be arbitrary small constants. For any constants $\tau,D>0$, the following local laws hold: 
\begin{equation}	\label{locallaw_GUE}
\P\bigg(\sup_{|E|\le 2- \kappa}\sup_{\sN^{-1+\fd}\le \eta\le 1} \max_{x,y \in \Z_\sN}|\sG_{xy} (z) -m_{sc}(z)\delta_{xy}| \le   \frac{\sN^\tau  }{\sqrt{\sN\eta}}\bigg) \ge 1- \sN^{-D}, 
\end{equation}
\begin{equation}	\label{locallaw_GUE_aver}
\P\bigg(\sup_{|E|\le 2- \kappa}\sup_{\sN^{-1+\fd}\le \eta\le 1} \max_{x,y \in \Z_\sN}\left|\frac{1}{\sN}\tr \sG (z) -m_{sc}(z) \right| \le   \frac{\sN^\tau}{\sN\eta}\bigg) \ge 1- \sN^{-D}, 
\end{equation}
as long as $\sN$ is sufficiently large.
\end{theorem}
\begin{proof}
    This follows directly from Theorem 2.1 of \cite{ErdYauYin2012Rig}.
\end{proof}

Note that GUE is a special case of the Wegner orbital model with $W=L$ and $\lambda=0$. In particular, the local expansions in \Cref{sec:local} also apply to $G=\sG$ with $m=m_{sc}$, $S_{xy}\equiv \sN^{-1}$ for $x,y\in \Z_{\sN}$, and 
\be\label{T_GUE} \sT_{x,yy'}:=\sum_{\al}S_{x\al}\sG_{\al y}\bar \sG_{\al y'}. \ee
Hence, our basic strategy is to first show that the $V$-expansions for GUE generate the same coefficient $\Delta(\Pi)$ as the Wegner orbital model. We will then prove the molecule sum zero property in the context of GUE using the local law presented in \Cref{thm:GUE_local}. 
However, in constructing the $V$-expansion for GUE, we will not use the complete $T$-expansion \eqref{mlevelTgdef weak} or the $T$-expansion \eqref{mlevelTgdef} to expand $T_{\fa,\fb_1\fb_2}$, because \smash{$\wt\thn$ and $\zthn$} vanish for GUE with $S_{xy}\equiv \sN^{-1}$. Instead, we will use the $T$-expansion in \Cref{2nd3rd T} without the zero mode removed.  
Similar to \Cref{thm:Vexp}, we can also derive the $V$-expansion for GUE with the $T$-expansion \eqref{seconduniversal}, the local expansions in Lemmas \ref{ssl}--\ref{T eq0}, and the $Q$-expansions defined in \Cref{subsec_Qexp}.

\begin{remark}
    Using the Ward's identity, we can rewrite $\sT_{x,yy'}$ in \eqref{T_GUE} as 
    $\sT_{x,yy'}= (\sG_{y'y}-\bar \sG_{yy'})/(2\ii \sN \eta) .$
    However, in the derivation of the $V$-expansion, we will not use this identity.
\end{remark}

With a slight abuse of notation, for the rest of the proof, we will adopt the following notations for GUE:
$$G\equiv \sG,\ \ T\equiv \sT, \ \ m\equiv m_{sc}, \ \ S=(S_{xy}),\ \ S^{+}=(S^{-})^\dag=\frac{S}{1-m_{sc}^2 S},\ \ \thn=\frac{S}{1-|m_{sc}|^2 S}.$$
All the definitions in \Cref{sec:graph_sizes} (except \Cref{def scaling}) apply to GUE. In particular, we define molecular graphs in the same manner as in \Cref{def_poly}, even though the molecules are no longer ``local" for GUE. 
Given a normal graph $\cal G$ consisting of solid edges (for $G/\bar G$), waved edges (for $S$ and $S^\pm$), diffusive edge (for $\thn$), and a coefficient $\cof(\cal G)$, we define the scaling size of $\cal G$ as
\begin{align}
		\size(\cal G): =&~ 
  ( \sN\eta)^{ -\frac{1}{2}\#\{\text{solid edges}\}}  (\sN\eta)^{-\#\{\text{diffusive edges}\}}   \sN^{\#\{\text{internal atoms}\}-\#\{ \text{waved edges}\}}. \label{eq_defsize_GUE}
	\end{align}
In the following proof, our graphs do not contain free edges. By \Cref{thm:GUE_local}, we have that for $z=E+\ii \eta$ with $|E|\le 2- \kappa$ and $\sN^{-1+\fd}\le \eta\le 1$, 
\be\label{eq:calG}|\cal G|\prec \size(\cal G).\ee
Now, we are ready to state the $V$-expansion for GUE. We will adopt the notations defined in \Cref{thm:Vexp}.

\begin{proposition}[$V$-expansions for GUE]\label{thm:Vexp_GUE}
In the setting of \Cref{thm:GUE_local}, fix any $z=E+\ii \eta$ with $|E|\le 2- \kappa$ and $\sN^{-1+\fd}\le \eta\le 1$. Suppose ${\cal G}$ is a normal graph of the form \eqref{eq:regularG}. 
Plugging the $T$-expansion \eqref{seconduniversal} into $T_{\fa,\fb_1\fb_2}$, applying the $Q$ and local expansions to the resulting graphs if necessary, we can expand \eqref{eq:regularG} as follows for any large constant $D>0$:
\begin{equation}
   \begin{split}\label{eq_Vexp_GUE}
        {\cal G}=&~m \thn_{\fa \fb_1}\overline G_{\fb_1\fb_2}  \Gamma  +\sum_{p=1}^{\ell(\Gamma)}\sum_{x}  \thn_{\fa x}
   \Gamma^{\star}_{p}(x,\fb_1,\fb_2) + \sum_x  \thn_{\fa x}\mathcal A_{x,\fb_1\fb_2} + \sum_x  \thn_{\fa x}\mathcal R_{x,\fb_1\fb_2} \\
   &~+  \mathcal Q_{\fa,\fb_1\fb_2}  + \Err_D(\fa,\fb_1,\fb_2) ,
   \end{split}
   \end{equation}
   where $\mathcal Q_{\fa,\fb_1\fb_2}$ is a sum of $\OO(1)$ many $Q$-graphs, and $\Err_D(\fa,\fb_1,\fb_2)$ is an error with $\Err_D(\fa,\fb_1,\fb_2) \prec \sN^{-D}$. The terms $\mathcal A_{x,\fb_1\fb_2}$, $\mathcal R_{x,\fb_1\fb_2}$, and $\Gamma^{\star}_{p}(x,\fb_1,\fb_2)$ are sums of $\OO(1)$ many normal graphs without $P/Q$ labels, which satisfy the following properties.  
\begin{itemize}

\item[(i)] $\mathcal R_{x,\fb_1\fb_2}$ is a sum of recollision graphs $\cal G'(x,\fb_1,\fb_2)$ with respect to the vertices in $\cal G$, whose scaling size satisfies
\be\label{eq:sizesR_GUE} \size\Big(\sum_x \thn_{\fa x}\cal G'(x,\fb_1,\fb_2)\Big) \lesssim  \size(\cal G)\cdot (\sN\eta)^{-\frac{1}{2}k_{\mathsf p}(\cal G'(x,\fb_1,\fb_2))}.\ee

\item[(ii)] $\mathcal A_{x,\fb_1\fb_2}$ is a sum of non-recollision graphs $\cal G'(x,\fb_1,\fb_2)$, whose scaling size satisfies 
\be\label{eq:sizesA_GUE}  \size\Big(\sum_x \thn_{\fa x}\cal G'(x,\fb_1,\fb_2)\Big)\lesssim \eta^{-1}\size(\cal G)\cdot (\sN\eta)^{-\frac{1}{2}[1+k_{\mathsf p}(\cal G'(x,\fb_1,\fb_2))]}.\ee

\item[(iii)] $\Gamma^{\star}_{p}(x,\fb_1,\fb_2)$ is a sum of locally regular star graphs $\cal G_{p,\gamma}^\star(x,\fb_1,\fb_2)$, where $k_{\mathsf p}(\cal G_{p,\gamma}^\star)=2p$ and its scaling size satisfies 
\be\label{eq:starsize_GUE} \size\Big(\sum_{x}  \thn_{\fa x} \cal G_{p,\gamma}^\star(x,\fb_1,\fb_2)\Big) \lesssim \eta^{-1}\size(\cal G)\cdot (\sN\eta)^{-p}.\ee
Moreover, these graphs $\cal G_{p,\gamma}^\star(x,\fb_1,\fb_2)$ takes the form in \eqref{eq:Gpgamma} with $\mathfrak D_\omega^{\lambda}(x,\vec x(\ii),\vec x(\ff))$ replaced by deterministic graphs 
$$\mathfrak D_\omega(x,\vec x(\ii),\vec x(\ff))\equiv \mathfrak D_\omega^{\lambda=0,W=L}(x,\vec x(\ii),\vec x(\ff)). $$
These graphs still satisfy \eqref{eq:wavednumber} and a similar equality to that in \eqref{eq:vertex_sum_zero_reduce}: 
\be\label{eq:vertex_sum_zero_reduce_GUE}
\fD_\omega(x,\vec x(\ii),\vec x(\ff))=\Delta_\omega \cdot \prod_{k=1}^p S_{xx_k(\ii)}\cdot \prod_{k=0}^p S_{xx_k(\ff)}=\Delta_\omega \sN^{-(2p+1)}, 
\ee
where the coefficient $\Delta_\omega$ is exactly the same as that in \eqref{eq:vertex_sum_zero_reduce}.
\end{itemize}
\end{proposition}
\begin{proof}
This proposition is essentially a special case of \Cref{thm:Vexp} with $\lambda=0$ and $W=L$, where the only difference is that we have applied \eqref{seconduniversal} to $T_{\fa,\fb_1\fb_2}$ instead of \eqref{mlevelTgdef weak}. However, this difference does not affect the proof. In the proof of this proposition, we consider the star graphs coming from the expansions of 
\be\nonumber
\cal G_0=\sum_x \thn_{\fa x} Q_x\left(G_{x \fb_1}   \overline G_{x \fb_2}\right)\Gamma,\ee
which takes the same form as \eqref{eq:prior_star} except that $\wt\thn$ is replaced by $\thn$ here. In particular, this change does not affect the expansion process and the coefficients $\Delta_\omega$ associated with the star graphs. Hence, we omit the details of the proof. 
\end{proof}

With \Cref{thm:Vexp_GUE}, to show the molecule sum zero property \eqref{eq:main_zumzero}, it suffices to prove the molecule sum zero property using a certain $V$-expansion of GUE. Given a pairing $\Pi\in \mathbf \Pi_p$, a natural candidate is the $V$-expansion for the loop graph $\sL_G(\Pi)$ generated by $\Pi$. Without loss of generality, assume that $\Pi$ is 
\be\label{eq:1-loop_Pi}\Pi=\left\{\{x_0(\ff),\bar x_1(\ff)\},\{\bar x_1(\ii),x_1(\ii)\},\{x_1(\ff),\bar x_2(\ff)\},\ldots, \{\bar x_{p}(\ii),x_{p}(\ii)\},\{x_p(\ff),\bar x_0(\ff)\}\right\}.\ee
It generates the loop graph 
\be\label{eq:LoopG}\cal G(x_0,x_1,\ldots, x_{2p+1}):=\mathsf L_G(\Pi)= \prod_{i=0}^{2p+1} G^{c_i}_{x_i x_{i+1}}, \quad c_i=\begin{cases}
    \emptyset , \ &\text{$i$ is even}\\
    \dag,\ &\text{$i$ is odd}
\end{cases},\ee
where we have renamed the indices of $\mathsf L_G(\Pi)$ as 
$$x_0=x_{2p+2}=x,   x_1=x_0(\ff)=\bar x_1(\ff),   x_2=\bar x_1(\ii)=x_1(\ii),\ldots, x_{2p}=\bar x_p(i)=x_p(i),  x_{2p+1}=x_p(\ff)=\bar x_0(\ff).$$
First, the following lemma gives the size of $\E\cal G(x_0,x_1,\ldots, x_{2p+1})$ when we sum over the indices $x_0,x_1,\ldots, x_{2p+1}$.

\begin{lemma}\label{lem:sameindices}
In the setting of \Cref{thm:Vexp_GUE}, let $\Sigma_1,\ldots, \Sigma_q$ be disjoint subsets that form a partition of the set of indices $\{x_0, \ldots, x_{2p+1}\}$ with $1\le q\le 2p+2$. 
We have that 
\be\label{eq:boundsame}
\sum_{y_1,\ldots, y_q\in \Z_{\sN}}^\star  \cal G(x_0,x_1,\ldots, x_{2p+1}) \mathbf 1(\Sigma_i=\{y_i\}: i=1,\ldots, q) \prec \sN\eta^{-q+1},
\ee
where $\Sigma_i=\{y_i\}$ means fixing the value of the indices in subset $\Sigma_i$ to be $y_i$ and $\sum^\star$ means summation subject to the condition that $y_1,\ldots, y_q$ all take different values. 
\end{lemma}

Summing over all possible partitions of $\{x_0, \ldots, x_{2p+1}\}$, this lemma implies that
\be\label{eq:correctsizeG} \sum_{x_0,\ldots, x_{2p+1}}\E\cal G(x_0,\ldots, x_{2p+1}) \prec \sN\eta^{-(2p+1)}.\ee
On the other hand, using the $V$-expansion in \Cref{thm:Vexp_GUE}, we find that the LHS of \eqref{eq:correctsizeG} has a leading term proportional to $\Delta(\Pi)\cdot \sN\eta^{-(2p+2)}$. Thus, we must have $\Delta(\Pi)=\oo(1)$, which implies the molecule sum zero property, as we will demonstrate below.

\begin{lemma}\label{lem_Lp1_St} 
In the setting of \Cref{thm:Vexp_GUE}, suppose the assumptions of \Cref{lem_For1} hold (i.e., the induction hypothesis $\Delta(\Pi',z)=\OO(\eta)$ for all $\Pi'\in \mathbf \Pi_q$, $q\in \qqq{1, p-1}$, and $z=E+\ii \eta$ with $|E|\le 2-\kappa$ and $0<\eta\le 1$). Then, for any $z=E+\ii\eta$ with $|E|\le 2-\kappa$ and $\eta=\sN^{-\serr}$ for a constant $\serr\in (0,1/2)$, we have
\begin{equation} \label{scsgun}
\sum_{x_0,x_1,\ldots, x_{2p+1}}^{\star} \mathbb E \cal G(x_0,x_1,\ldots, x_{2p+1})  =
 \Delta(\Pi)\cdot  \sN\left(\frac{\im m}{\eta}\right)^{2p+2} 
  + \OO_\prec\left(\sN\eta^{-(2p+1)} \right)  .
\end{equation}
\end{lemma}

Before proving Lemmas \ref{lem:sameindices} and \ref{lem_Lp1_St}, we first use them to complete the proof of \Cref{lem_For1}. 

\begin{proof}[\bf Proof of \Cref{lem_For1}]
Since $\Delta(\Pi)$ are fixed polynomials of $m,\, \bar m,\, (1-m^2)^{-1},\, (1-\bar m^2)^{-1}$ (with $\OO(1)$ coefficients), we have that $\Delta(\Pi)\equiv \Delta(\Pi,z)$ satisfies
\be\label{eq:Delta_eta=0}
\Delta(\Pi,z)-\Delta(\Pi,E) =\OO(\eta).
\ee
Thus, to prove \eqref{eq:main_zumzero}, it suffices to show that $\Delta(\Pi,E)=0$ for any fixed $E \in [-2+\kappa,2-\kappa]$. 

For this purpose, we take $z_0=E+\ii \eta_0$ with $\eta_0=\sN^{-\serr}$ for a constant $\serr\in (0,1/2)$. By \eqref{eq:boundsame}, the LHS of \eqref{scsgun} is bounded by \smash{$\sN\eta_0^{-(2p+1)}$}. Thus, we get from \eqref{scsgun} that 
\be\label{sum_scsgun} 
 \Delta(\Pi,z_0)\cdot  \sN\left({\im m}/{\eta_0}\right)^{2p+2} 
  + \OO_\prec\left(\sN\eta_0^{-(2p+1)} \right)\prec \sN\eta_0^{-(2p+1)},\ee
which implies that
$ \Delta(\Pi,z_0)\prec \eta_0 =\sN^{-\serr}.$ 
Together with \eqref{eq:Delta_eta=0}, this gives that 
$\Delta(\Pi,E)=\OO_\prec(\sN^{-\serr})$. Since $\Delta(\Pi,E)$ is a constant that depends only on $E$ and $\mathrm{Struc}(\Pi)$, and $\sN$ can be arbitrarily large, we must have $\Delta(\Pi,E)=0$. Together with \eqref{eq:Delta_eta=0}, this concludes the proof. 
\end{proof}

\subsection{Proof of \Cref{lem:sameindices} and \Cref{lem_Lp1_St}} 

For the proof of Lemma \ref{lem:sameindices}, we will use the following multi-resolvent local law for Wigner matrices established in \cite{CES_Multi-resol_EJP}.

\begin{lemma}\label{lem:multi}
Under the assumptions of \Cref{thm:GUE_local}, for a fixed $k\in \N$, we denote $\cal L:=\prod_{i=0}^{k} G^{c_i}$, where $c_i\in \{ \emptyset ,\dag\}$ for $i \in \qqq{0,k}$. Then, for any $z=E+\ii \eta$ with $|E|\le 2-\kappa$ and $\sN^{-1+\fd}\le \eta\le 1$, there exists a deterministic function $\cal M(z^{c_0},\ldots, z^{c_{k}}):\C^{k+1}\to \C$ such that  
\be\label{eq:multi-local}
\left|\cal L_{xy} - \cal M(z^{c_0},\ldots, z^{c_{k}})\delta_{xy}\right|\prec (\sN\eta)^{-1/2}\eta^{-k} ,\quad \forall x,y \in \Z_{\sN}.
\ee
Here, we denote $z^{c_i}=z$ if $c_i=\emptyset$ and $z^{c_i}=\bar z$ if $c_i=\dag$, and the function $\cal M(z^{c_0},\ldots, z^{c_{k}})$ satisfies that 
\be\label{eq:boundM}
\left|\cal M(z^{c_0},\ldots, z^{c_{k}})\right| \lesssim \eta^{-k}.
\ee
As a consequence, if we have a polygon $\prod_{i=0}^{k} G^{c_i}_{y_i y_{i+1}}$ of $(k+1)$ sides with $y_{k+1}=y_0$, then
\be\label{eq:loop}
\sum_{y_0,\ldots, y_{k}\in \Z_\sN} \prod_{i=0}^{k} G^{c_i}_{y_i y_{i+1}} =\tr\left( \prod_{i=0}^{k} G^{c_i} \right) \prec \sN \eta^{-k}.\ee
\end{lemma}
\begin{proof}
The estimate \eqref{eq:multi-local} is an immediate consequence of \cite[Theorem 2.5]{CES_Multi-resol_EJP}, where the explicit form of the function $\cal M$ is also presented. The estimate \eqref{eq:boundM} for $\cal M$ is given in \cite[Lemma 2.4]{CES_Multi-resol_EJP}. 
The estimate \eqref{eq:loop} is an immediate consequence of the multi-resolvent local law in \eqref{eq:multi-local}; it can also be proved directly using Ward's identity and the local law \eqref{locallaw_GUE}.
\end{proof}

With \Cref{lem:multi}, we can prove the following lemma, which implies \Cref{lem:sameindices} as a direct corollary. 
\begin{lemma}\label{lem:boundG}
Under the assumptions of \Cref{thm:GUE_local}, fix any $q\in \mathbb N$ and $z=E+\ii \eta$ with $|E|\le 2-\kappa$ and $\sN^{-1+\fd}\le \eta\le 1$. Consider a connected graph $\cal G(y_1,\ldots, y_q)$ consisting of $q$ vertices $y_1,\ldots, y_q$ and solid edges between them. Suppose that the charge of every vertex $y_i$, $i\in \qqq{1,q}$, is neutral (recall \Cref{deflvl1}) and every solid edge between two different vertices has no $\circ$ label. Then, we have that 
\be\label{eq:Gq}
\sum_{y_1,\ldots, y_q\in \Z_{\sN}} \cal G(y_1,\ldots, y_q) \prec \sN\eta^{-(q-1)}.  
\ee
In general, the estimate \eqref{eq:Gq} still holds if some solid edges have a $\circ$ label and if we add some dotted or $\times$-dotted edges (recall \Cref{def_graph1}) to the graph.  
\end{lemma}
\begin{proof}
First, note that due to the neutral charge condition, the degrees of the vertices in our graph are all even numbers. Next, we find all paths of solid edges taking the following form:
$$ \prod_{i=0}^k G^{c_i}_{x_i x_{i+1}},\quad \text{with}\quad x_i\in \{y_1,\ldots, y_q\}, \ c_i\in \{\emptyset,\dag\},  $$
such that each vertex $x_i$, $i\in \qqq{1,k}$, has degree 2, and the endpoints satisfy one of the following two scenarios: (1) $x_0=x_{k+1}$, or (2) $\deg(x_0)\ge 4$ and $\deg(x_{k+1})\ge 4$. Then, we treat 
$$\cal L^{\vec c}_{x_0 x_{k+1}}:=\eta^k\sum_{x_1,\ldots, x_{k}}\prod_{i=0}^k G^{c_i}_{x_i x_{i+1}}= \eta^k \left(\prod_{i=0}^k G^{c_i}\right)_{x_0 x_{k+1}},\quad \vec c=(c_0,c_1,\ldots, c_{k}),$$ 
as a ``long" solid edge of length $(k+1)$ between vertices $x_0$ and $x_{k+1}$. (The original solid edge can also be regarded as a long solid edge with $k=0$.) By \Cref{lem:multi}, we have that 
\be\label{eq:longlocal}
\mathring{\cal L}^{\vec c}_{xy}:= \cal L^{\vec c}_{xy} - \eta^k \cal M(\mathbf z^{\vec c})\delta_{xy} \prec (\sN\eta)^{-1/2},\quad \mathbf z^{\vec c}:=(z^{c_0},\ldots, z^{c_{k}}),
\ee
where $\eta^k \cal M(\mathbf z^{\vec c})$ is an order $\OO(1)$ scalar. Then, we can reduce $\sum_{y_1,\ldots, y_q}\cal G(y_1,\ldots, y_q)$ to a graph consisting of long solid edges between vertices, each of which has a degree at least 4:
$$\sum_{y_1,\ldots, y_q}\cal G(y_1,\ldots, y_q) = \eta^{-(q-\ell)}\sum_{y_1,\ldots, y_{\ell}}\cal G'(y_1,\ldots, y_\ell).$$
Here, the factor $\eta^{-(q-\ell)}$ comes from the vertices that have disappeared in this reduction, and without loss of generality, we assume that these vertices are labeled as $y_{\ell+1},\ldots, y_q$. We will refer to the above process as a ``long solid edge reduction".  

Next, in the graph $\cal G'(y_1,\ldots, y_\ell)$, we replace each weight (that is, a self-loop of a long solid edge) by an order $\OO(1)$ scalar plus a light weight; that is, we decompose each $\cal L^{\vec c}_{y_ky_k}$ of length $(k+1)$ as \be\label{eq:expand_w}
\cal L^{\vec c}_{y_ky_k}=\eta^k \cal M(\mathbf z^{\vec c}) + \zcL^{\vec c}_{y_ky_k} .\ee 
As in \Cref{def_graph1}, we use a self-loop with a $\circ$ to represent a light weight. Taking the product of all these decompositions, we can expand $\sum_{y_1,\ldots, y_{\ell}}\cal G'(y_1,\ldots, y_\ell)$ as a linear combination of $\OO(1)$ many new graphs:
$$\sum_{y_1,\ldots, y_{\ell}}\cal G'(y_1,\ldots, y_\ell) = \sum_\omega c_\omega \sum_{y_1,\ldots, y_{\ell}} \cal G_\omega(y_1,\ldots, y_\ell),$$
where $c_\omega$ are order $\OO(1)$ coefficients, and every $\cal G_\omega$ only contains light weights and solid edges between different vertices. For every new graph $\sum_{y_1,\ldots, y_{\ell}} \cal G_\omega(y_1,\ldots, y_\ell)$, we perform a long solid edge reduction again and expand each weight as in \eqref{eq:expand_w}. Repeating this process, we obtain that 
\be\label{eq:graph_red1}\sum_{y_1,\ldots, y_q}\cal G(y_1,\ldots, y_q) =\sum_{k=1}^q \eta^{-(q-k)} \sum_\omega c_{\omega,k} \sum_{y_1,\ldots, y_{k}}\cal G_{\omega,k}(y_1,\ldots, y_k),\ee
where $c_{\omega,k}$ are $\OO(1)$ coefficients, and we have renamed the remaining $k$ vertices as $y_1,\ldots, y_k$. 
When $k=1$, $\cal G_{\omega,1}(y_1)$ represents a self-loop on $y_1$, which is of order $\OO_\prec(1)$. For $k\ge 2$, each $\cal G_{\omega,k}(y_1,\ldots, y_k)$ satisfies that: (a) every weight in $\cal G_{\omega,k}$ is a light weight, and (b) every vertex in $\cal G_{\omega,k}$ has degree $\ge 4$. In \Cref{fig:graph_red}, we depict an example of the above graph reduction process. 

\begin{figure}[htb]
    \centering
\includegraphics[width=0.7\linewidth]{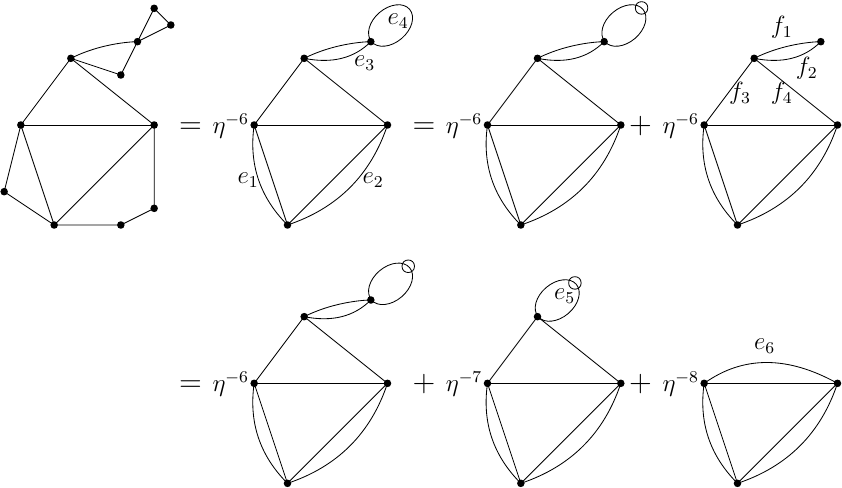}
\caption{We use undirected black solid edges to represent the blue and red solid edges in the graph. 
In the second graph, the edges $e_1$, $e_2,$ $e_3,$ $e_4$ are long solid edges of lengths $2,\ 3, \ 2$, $3$, respectively. In the second equality, we 
decompose the weight (i.e., self-loop) $e_4$ into a light weight (i.e., self-loop with a $\circ$) plus a scalar. In the latter case, we can further reduce $f_1$ and $f_2$ to a long solid edge $e_5$. The self-loop $e_5$ is then decomposed into a light weight and a scalar, which gives the second and third graphs in the third equality. In the last graph, we have also replaced $f_3$ and $f_4$ with another long solid edge $e_6$ of length 2. 
}
    \label{fig:graph_red}
\end{figure}

Now, we expand each solid edge in $\cal G_{\omega,k}$ with $k\ge 2$ as a scalar matrix plus a solid edge with $\circ$. If a solid edge is replaced by a scalar matrix, we merge their endpoints and include the scalar into the coefficient. In this way, we can expand every $\cal G_{\omega,k}$ into a linear combination of new graphs consisting of weights and solid edges with $\circ$. Again, we decompose every weight as shown in \eqref{eq:expand_w}, yielding a sum of new graphs 
\be\label{eq:graph_red2} \cal G_{\omega,k}(y_1,\ldots, y_k)=\sum_{l=1}^k \sum_{\gamma} c_{\gamma,l} \cal G_{\omega,k;\gamma,l}(y_1,\ldots, y_l),\ee
where $c_{\gamma,l}$ are $\OO(1)$ coefficients. The term $\cal G_{\omega,k;\gamma,1}(y_1)$ again represents a self-loop on $y_1$ of order $\OO_\prec(1)$, and $\cal G_{\omega,k;\gamma,l}(y_1,\ldots, y_l)$ for $l\ge 2$ consists solely of light weights and solid edges with $\circ$. For each $\cal G_{\omega,k;\gamma,l}$ with $l\ge 2$, we perform a long solid edge reduction as follows: if $y$ is a vertex with degree 2, attached with two edges $\zcL^{\vec c}_{xy}$ and $\zcL^{\vec c'}_{yw}$ with $\vec c=(c_0,c_1,\ldots, c_{r})$ and $\vec c'=(c_0',c_1',\ldots, c_{s}')$, then we perform the reduction
\begin{align*}
\sum_y \zcL^{\vec c}_{xy} \zcL^{\vec c'}_{yw}  &=\eta^{-1}\cdot  \zcL^{(\vec c,\vec c')}_{xw} -  [\eta^{s} \cal M(\mathbf z^{\vec c'})]  \zcL^{\vec c}_{xw} - [\eta^r \cal M(\mathbf z^{\vec c})] \zcL^{\vec c'}_{xw}\\
&+\eta^{-1}\cdot \eta^{r+s+1} \left[\cal M(\mathbf z^{(\vec c,\vec c')})- \cal M(\mathbf z^{\vec c})\cal M(\mathbf z^{\vec c'})\right] \delta_{xw},   
\end{align*}
i.e., we replace the two edges $(x,y),(y,w)$ with the sum of three long solid edges with $\circ$ and a scalar matrix. In the new graphs, we find another vertex with degree 2 and perform the above long solid edge reduction. Continuing this process, we can ultimately expand \eqref{eq:graph_red2} into a linear combination of graphs, each of which either has only one vertex or
consists of vertices with degrees $\ge 4$, light weights, and solid edges with $\circ$.  

In sum, we have expanded $\sum_{y_1,\ldots, y_q}\cal G(y_1,\ldots, y_q)$ as a sum of $\OO(1)$ many new graphs:
\be\label{eq:graph_red3}\sum_{y_1,\ldots, y_q}\cal G(y_1,\ldots, y_q) =\sum_{k=2}^q \eta^{-(q-k)} \sum_\mu \wt c_{\mu,k} \sum_{y_1,\ldots, y_{k}}\wt{\cal G}_{\mu,k}(y_1,\ldots, y_k)+\OO_\prec(\sN\eta^{-(q-1)}),\ee
where $\wt c_{\mu,k}$ are coefficients of order $\OO(1)$. Each $\wt{\cal G}_{\mu,k}(y_1,\ldots, y_k)$ with $k\ge 2$ consists of vertices with degrees $\ge 4$, light weights, and solid edges with $\circ$, and we have renamed the remaining $k$ vertices as $y_1,\ldots, y_k$. Now, using the local law \eqref{eq:longlocal} for long solid edges, we can bound $\wt{\cal G}_{\mu,k}(y_1,\ldots, y_k)$ as 
$$ \wt{\cal G}_{\mu,k}(y_1,\ldots, y_k) \prec \left[(\sN\eta)^{-1/2}\right]^{\frac{1}{2}\sum_{i=1}^k \deg(y_i)}\le (\sN\eta)^{-k},$$
where $\frac{1}{2}\sum_{i=1}^k \deg(y_i)$ represents the number of solid edges in $\wt{\cal G}_{\mu,k}$ and we have used that $\deg(y_i)\ge 4$ in the second step. Plugging it into \eqref{eq:graph_red3} concludes \eqref{eq:Gq}. 

In general, suppose the graph $\cal G(y_1,\ldots, y_q)$ contains some solid edges with a $\circ$ label, along with some dotted or $\times$-dotted edges. We then decompose every \smash{$\Gc$ or $\Gc^-$} edge into a solid edge plus a dotted edge (with an $\OO(1)$ coefficient), i.e., \smash{$\Gc_{y_iy_j}=G_{y_iy_j}-m\delta_{y_iy_j}$ or $\Gc^-_{y_iy_j}=\bar G_{y_iy_j}-\bar m\delta_{y_iy_j}$}.  We further decompose each $\times$-dotted edge, such as $\mathbf 1_{y_i\ne y_j}$, as $1-\mathbf 1_{y_i=y_j}$. Taking the product of all these decompositions, we can express the graph as a linear combination of new graphs consisting of dotted edges and solid edges without a $\circ$ label. In each new graph, we merge the vertices within each atom (recall \Cref{def_atom}) and get a graph consisting of solid edges only. We can check that all the resulting graphs satisfy the original setting for the graph $\cal G$ with at most $q$ vertices. This concludes the proof.
\end{proof}

\begin{proof}[\bf Proof of \Cref{lem:sameindices}]
We first identify the indices within each subset $\Sigma_i$ and get a new graph with $q$ external vertices:
$$ \cal G(y_1,\ldots, y_q):=\cal G(x_0,x_1,\ldots, x_{2p+1}) \mathbf 1(\Sigma_i=\{y_i\}: i=1,\ldots, q) .$$
Then, we add a $\times$-dotted edge between every pair of vertices. The resulting graph satisfies the setting in \Cref{lem:boundG}, so we conclude \eqref{eq:boundsame} immediately with \eqref{eq:Gq}.
\end{proof}

Next, we prove \Cref{lem_Lp1_St} using the $V$-expansion in \eqref{eq_Vexp_GUE}. Roughly speaking, we will show that for the loop graph $\cal G$ generated by $\Pi$ as in \eqref{eq:LoopG}, the star graph with a pairing corresponding to $\Pi$ is the dominant term after taking the expectation and the summation over all vertices. This leading term consists of a centering molecule $\cal M$ that contains $x$, a $\thn$ edge between $x_0$ and $\cal M$, and $(2p+1)$ pairs of $G$ and $\bar G$ edges between $x_i$ and $\cal M$. Take the $p=2$ case and the loop graph in \eqref{eq:star2} generated by $\Pi_1$ as an example. The star graphs corresponding to the pairings \eqref{eq:6-loopPi} and \eqref{eq:4+2-loopPi} take the following form:
\be\label{eq:star4}
\parbox[c]{10cm}{\includegraphics[width=\linewidth]{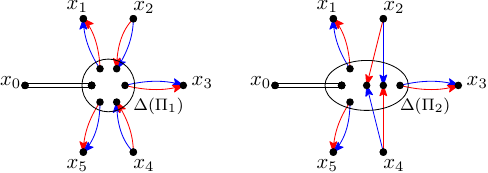}}
\ee
Here, the black circles represent the centering molecule $\cal M$ that contains $x$, and we have omitted the waved edges between $x$ and other vertices inside $\cal M$. In the star graph corresponding to $\Pi_1$ (resp.~$\Pi_2$), the subgraph induced on $\cal M$ takes the form $\Delta(\Pi_1)\sN^{-5}$ (resp.~$\Delta(\Pi_2)\sN^{-5}$). Then, we sum the two graphs in \eqref{eq:star4} over the external vertices $x_i$, $i\in \qqq{0,5}$. By Ward's identity and the local law \eqref{locallaw_GUE_aver}, the first graph in \eqref{eq:star4} yields 
$$\Delta(\Pi_1)\cdot  \frac{\sN}{1-|m|^2}\left( \frac{\im m+\OO_\prec((\sN\eta)^{-1}) }{\eta}\right)^5= \Delta(\Pi_1)\cdot  \sN\left( \frac{\im m}{\eta}\right)^6 \left[1 + \OO_\prec\left( (\sN\eta)^{-1} + \eta\right)\right], $$  
where we used that
\be\label{eq;Theta} \thn_{x_0 x} = \frac{1}{\sN}\frac{1}{1-|m|^2},\quad  \frac{|m|^2}{1-|m|^2}=\frac{\im m}{ \eta},\quad \text{and}\quad |m|=1-\OO(\eta). \ee
On the other hand, by \eqref{eq:loop}, the second graph in \eqref{eq:star4} satisfies a much better bound $\OO_\prec(\eta^{-7})$. 
For the other terms in \eqref{eq_Vexp_GUE}, the first graph on the RHS contains a $\thn$ edge, whose summation over $\fa=x_0$ provides an $\eta^{-1}$ factor, and a loop graph with $(2p+1)$ vertices, whose summation provides an $\sN\eta^{-2p}$ factor. The fourth term on the RHS of \eqref{eq_Vexp_GUE} can be bounded using a similar reasoning: every recollision graph essentially consists of a $\thn$ edge and a connected graph with $(2p+1)$ ``free" vertices, which can be bounded by $\sN\eta^{-2p}$ using \Cref{lem:boundG}. Finally, the higher order graphs (i.e., the third term on the RHS of \eqref{eq_Vexp_GUE}) are small because they contain more solid edges or light weights than the star graphs.  

We now give a rigorous justification for the above argument. 
\begin{proof}[\bf Proof of \Cref{lem_Lp1_St}]
For $x_1,x_2,\ldots, x_{2p+1}$ that all take different values, we consider the graph 
\be\label{eq:Gx1x2p}
\cal G_1(x_1 \ldots, x_{2p+1}):=\frac{1}{\sN}\sum_{x_0}  \cal G(x_0,x_1,\ldots, x_{2p+1}) =  T_{\fa,x_1x_{2p+1}} \prod_{i=1}^{2p}G^{c_i}_{x_i x_{i+1}} ,
\ee
where $\fa$ is an auxiliary external vertex introduced to form the $T$-variable. We claim that 
\be\label{scsgun2}
\sum_{x_1,x_2,\ldots, x_{2p+1}}\E\cal G_1(x_1 \ldots, x_{2p+1}) =  
 \Delta(\Pi)\cdot  \left({\im m}/{\eta}\right)^{2p+2} 
  + \OO_\prec\left(\eta^{-(2p+1)} \right)  .
\ee
Note that we can write   
\begin{align*}
  \sum_{x_0,x_1,\ldots, x_{2p+1}}^{\star} \mathbb E \cal G(x_0,x_1,\ldots, x_{2p+1})  &= \sN \sum_{x_1,x_2,\ldots, x_{2p+1}}^\star \E \cal G_1(x_1,\ldots, x_{2p+1}) \\
  &- \sum_{x_1,x_2,\ldots, x_{2p+1}}^\star \sum_{x_0\in \{x_1,\ldots, x_{2p+1}\}} \mathbb E \cal G(x_0,x_1,\ldots, x_{2p+1})   .
\end{align*}
By \Cref{lem:sameindices}, the second term on the RHS is bound by $\OO_\prec(\sN \eta^{-2p})$. Thus, \eqref{scsgun2} implies \eqref{scsgun}. 

To prove \eqref{scsgun2}, we apply the $V$-expansion in \eqref{eq_Vexp_GUE} to \eqref{eq:Gx1x2p} with $\fb_1=x_1$, $\fb_2=x_{2p+1}$, $\Gamma=\prod_{i=1}^{2p}G^{c_i}_{x_i x_{i+1}}$, and $D\ge 2p+1$, and obtain that 
\begin{align}
        {\cal G_1}(x_1,\ldots, x_{2p+1})=&~m \thn_{\fa x_1} \overline G_{x_1x_{2p+1}}  \Gamma  +\sum_{q=1}^{p}\sum_{x}  \thn_{\fa x}
   \Gamma^{\star}_{q}(x,x_1,x_{2p+1}) + \sum_x  \thn_{\fa x}\mathcal A_{x,x_1x_{2p+1}} + \sum_x  \thn_{\fa x}\mathcal R_{x,x_1x_{2p+1}} \nonumber\\
   &~+  \mathcal Q_{\fa,x_1x_{2p+1}}  + \Err_D(\fa,x_1,x_{2p+1}) .\label{eq_Vexp_GUE1}
   \end{align}
Taking the expectation and summing over the vertices $x_1,\ldots, x_{2p+1}$, the $\mathcal Q_{\fa,x_1x_{2p+1}}$ term vanishes and the $\Err_D$ term provides an error of order $\OO_\prec(\sN^{2p+1}\cdot \sN^{-D})=\OO_\prec(1)$. The first term on the RHS is bounded by 
\begin{align}
\sum_{x_1,\ldots, x_{2p+1}}^{\star} \mathbb E m \thn_{\fa x_1} \overline G_{x_1x_{2p+1}}  \prod_{i=1}^{2p}G^{c_i}_{x_i x_{i+1}} &= \frac{ \im m}{\bar m\sN\eta}\sum_{x_1,\ldots, x_{2p+1}}^{\star} \mathbb E \overline G_{x_1x_{2p+1}} \prod_{i=1}^{2p}G^{c_i}_{x_i x_{i+1}} \nonumber\\
&\prec \frac{1}{\sN\eta}\cdot \sN \eta^{-2p} =\eta^{-(2p+1)},\label{eq;Vleading}
\end{align}
where we used \eqref{eq;Theta} in the first step and \Cref{lem:sameindices} in the second step.  

\medskip
\noindent {\bf Star graphs.} Next, we consider the star graphs $\Gamma^{\star}_{q}(x,x_1,x_{2p+1})$, $1\le q\le p$. We pick a subset of edges $E_q=\{e_i=(a_i,b_i):i\in \qqq{q}\}\subset \{(x_{2i},x_{2i+1}):i\in \qqq{p}\}$ and $\bar E_q=\{\bar e_i=(\bar a_i,\bar b_i):i\in \qqq{q}\}\subset \{(x_{2i},x_{2i-1}):i\in \qqq{p}\}$ (recall \eqref{eq:LoopG}). Denote $b_0=x_1$ and $\bar b_0=x_{2p+1}$, and fix any pairing $\cal P(E_q,\bar E_q)$. Let $\cal G^\star (x,x_1,\ldots, x_{2p+1};\cal P(E_q,\bar E_q)) $ be the sum of all star graphs with the pairing $\cal P(E_q,\bar E_q)$. By (iii) of \Cref{thm:Vexp_GUE}, it writes 
\begin{align}
    \cal G^\star (x,x_1,\ldots, x_{2p+1};\cal P(E_q,\bar E_q)) =\Delta(\cal P(E_q,\bar E_q))  \sum_{\substack{y_k(\ii):k=1,\ldots, q \\ y_k(\ff):k=0,\ldots, q}}  \prod_{k=1}^q S_{xy_k(\ii)}\cdot \prod_{k=0}^q S_{xy_k(\ff)} \nonumber\\
     \times \prod_{k=1}^q \left(G_{a_k y_k(\ii)}\overline G_{\overline a_{\pi(k)} y_k(\ii)}\right) \cdot \prod_{k=0}^q \left(G_{y_k(\ff)b_k}\overline G_{\overline y_k(\ff)\bar b_{\sigma(k)} }\right) \cdot \frac{\Gamma}{\prod_{k=1}^q \left(G_{a_k b_k} \overline G_{\overline a_{k} \bar b_{k} }\right) },\label{eq:Gpgammastar}
\end{align} 
where $\pi:\{1,\ldots,q\}\to \{1,\ldots,q\}$ and $\sigma:\{0,\ldots,q\}\to \{0,\ldots,q\}$ represent bijections that correspond to the pairing $\cal P(E_q,\bar E_q)$.

Let $n_a$ denote the number of internal vertices (including $x$) and let $n_W$ denote the number of waved edges. For the graph \eqref{eq:Gpgammastar}, we have $n_a=n_W+1=2q+2$. Due to the locally standard property, all endpoints of solid edges have degree 2 and there are no light weights. Hence, the solid edges in $\cal G^\star$ must form (disjoint) loops and each loop has length $\ge 2$. We assume that there are $k$ loops, each having length $\ell_i\ge 2$, $i=1,\ldots, k$. By \Cref{lem:boundG}, the summation over the vertices of a loop of length $\ell_i$ can be bounded by $ \OO_\prec(\sN\eta^{-\ell_i+1})$. 
Thus, we can bound that 
\begin{align}
\sum_{x_1,\ldots, x_{2p+1}}^\star \sum_x \E\cal G^\star (x,x_1,\ldots, x_{2p+1};\cal P(E_q,\bar E_q)) &\prec |\Delta(\cal P(E_q,\bar E_q))|\sN^{2p+2-\sum_{i=1}^k \ell_i}\prod_{i=1}^k (\sN\eta^{-\ell_i+1})\nonumber\\
&= |\Delta(\cal P(E_q,\bar E_q))|\sN^{2p+2} (\sN\eta)^{-\sum_{i=1}^k(\ell_i-1)} .\label{eq:sizeloopsstar}
\end{align}
Note that every external vertex $x_i$, $i\in \qqq{1,2p+1}$, belongs to a loop with at least one internal vertex. As a consequence, we have that  
\be\label{eq:length_i}\sum_i (\ell_i-1) \ge  2p+1. \ee
Furthermore, when $q<p$, we have known that $|\Delta(\cal P(E_q,\bar E_q))|=\OO(\eta)$ by the induction hypothesis. Plugging it into \eqref{eq:sizeloopsstar} and using \eqref{eq;Theta}, we obtain that 
\begin{align}
\sum_{x_1,\ldots, x_{2p+1}}^\star \sum_x \thn_{\fa x}\E\cal G^\star (x,x_1,\ldots, x_{2p+1};\cal P(E_q,\bar E_q)) &\prec (\sN\eta)^{-1}\cdot \eta \sN^{2p+2} (\sN\eta)^{-(2p+1)} =\eta^{-(2p+1)}. \label{eq;Vstar1}
\end{align}

It remains to consider star graphs with $q=p$. In this case, all edges in $\Gamma$ have been pulled in the $V$-expansion and we have $\sum_i \ell_i = 4p+2$. If there exists a loop, say loop $j$, with $\ell_j \ge 4$, then   
$$ \sum_{i=1}^k (\ell_i-1) \ge \sum_{i\ne j}\frac{1}{2}\ell_i + \left(\frac{1}{2}\ell_j+1\right)\ge \frac{1}{2}\sum_{i=1}^k \ell_i + 1 = 2p+2.$$
Inserting it into \eqref{eq:sizeloopsstar}, we obtain that 
\begin{align}
    \sum_{x_1,\ldots, x_{2p+1}}^\star \sum_x \thn_{\fa x} \E\cal G^\star (x,x_1,\ldots, x_{2p+1};\cal P(E_q,\bar E_q)) &\prec  (\sN\eta)^{-1}\cdot \sN^{2p+2} (\sN\eta)^{-(2p+2)} \nonumber\\
    &=\sN^{-1}\eta^{-(2p+3)}\le \eta^{-(2p+1)}, \label{eq;Vstar2}
\end{align}
since $\eta\ge \sN^{-1/2}$. We are left with the case where all loops have length 2. In this case, it is easy to see that $\cal P(E_p,\bar E_p)$ must be chosen such that every $G$ edge is paired with a $\bar G$ edge, i.e., we must have 
\begin{align}
    &~ \cal G^\star (x,x_1,\ldots, x_{2p+1};\cal P(E_p,\bar E_p)) \nonumber\\
=&~\Delta(\cal P(E_p,\bar E_p))  \sum_{\substack{y_k(\ii):k=1,\ldots, p \\ y_k(\ff):k=0,\ldots, p}}  \prod_{k=1}^p S_{xy_k(\ii)}|G_{x_{2k} y_k(\ii) }|^2\cdot \prod_{k=0}^{p} S_{xy_k(\ff)} |G_{y_k(\ff) x_{2k+1}}|^2  .\label{eq;Vstar2.5}
\end{align} 
This corresponds to the desired pairing $\Pi\in \mathbf \Pi_p$  in \eqref{eq:1-loop_Pi}. Using Ward's identity, we can rewrite \eqref{eq;Vstar2.5} as
\begin{align}
    \cal G^\star (x,x_1,\ldots, x_{2p+1};\cal P(E_p,\bar E_p)) & = \Delta(\Pi) \prod_{k=1}^p\frac{\im G_{x_{2k}x_{2k}}}{\sN\eta}\cdot \prod_{k=0}^p \frac{\im G_{x_{2k+1}x_{2k+1}}}{\sN\eta}.
    \label{eq;Vstar2.75}
\end{align} 
Summing \eqref{eq;Vstar2.75} over the indices $x_1, \ldots, x_{2p+1}$ and using \eqref{locallaw_GUE_aver} and \eqref{eq;Theta}, we obtain that 
\begin{align}
\sum_{x_1,\ldots, x_{2p+1}}^\star \sum_x \thn_{\fa x} \E \cal G^\star (x,x_1,\ldots, x_{2p+1};\cal P(E_p,\bar E_p)) &=\frac{1}{1-|m|^2} \Delta(\Pi) \left({\im m}/{\eta}\right)^{2p+1} 
  + \OO_\prec\left(\sN^{-1}\eta^{-(2p+3)} \right)\nonumber\\
  &= \Delta(\Pi) \left({\im m}/{\eta}\right)^{2p+2} 
  + \OO_\prec\left(\eta^{-(2p+1)} \right)  .\label{eq;Vstar3}
\end{align}

\medskip
\noindent {\bf Recollision graphs.} 
A recollision graph $\cal G_R(x,x_1,\ldots, x_{2p+1})$ appears in the following two scenarios: 
\begin{itemize}
  \item[(1)] We replace a $G^{c_i}_{\al x_i}$ or $G^{c_i}_{x_i \al}$ edge by a dotted edge $m^{c_i} \delta_{\al x_i}$, where $i\in \qqq{1,2p+1}$, $c_i\in \{\emptyset,-\}$, and $\al$ is an internal vertex. 
  \item[(2)] In the $GG$ expansion \eqref{Oe2x}, we replace $G_{y' \al}G_{\al x_i}$ (resp.~$\bar G_{y' \al}\bar G_{\al x_i}$) by $m^3 S^{+}_{\al x_i}G_{y' x_i}$ (resp.~$\overline m^3 S^{-}_{\al x_i}\bar G_{y' x_i}$), where $i\in \qqq{1,2p+1}$ and $\al$ is an internal vertex.   
\end{itemize}
We denote by $\cal G_0(x,x_1,\ldots, x_{2p+1})$ the graph obtained just before one of these operations. Again, let $n_a(\cal G_R)$ and $n_W(\cal G_R)$ denote the number of internal vertices (including $x$) and the number of waved edges, respectively, in $\cal G_R$. 
We further assume that in $\cal G_R$, there are $k(\cal G_R)$ connected components (in terms of paths of solid edges) of the endpoints of solid edges, each having size $\ell_i(\cal G_R)\ge 2$, $i=1,\ldots, k(\cal G_R)$. We can also define these quantities for $\cal G_0$.
Keeping track of the $Q$ and local expansions, we can check that all our graphs without $P/Q$ labels satisfy the setting in \Cref{lem:boundG}.
Thus, the summation over all vertices of a connected component of size $\ell_i(\cal G_R)$ can be bounded by $ \OO_\prec(\sN\eta^{-\ell_i(\cal G_R)+1})$. 
 Thus, we can bound that 
\begin{align}
\sum_{x_1,\ldots, x_{2p+1}}^\star \sum_x \E\cal G_R(x,x_1,\ldots, x_{2p+1}) &\prec \sN^{2p+1+n_a(\cal G_R)-n_W(\cal G_R)-\sum_{i=1}^{k(\cal G_R)} \ell_i(\cal G_R)}\prod_{i=1}^{k(\cal G_R)}(\sN\eta^{-\ell_i(\cal G_R)+1})\nonumber\\
&= \sN^{2p+1+n_a(\cal G_R)-n_W(\cal G_R)} (\sN\eta)^{-\sum_{i=1}^{k(\cal G_R)}(\ell_i(\cal G_R)-1)} .\label{eq:sizeloops_recol}
\end{align}

Note that in our graphs, all internal vertices connect to $x$ through waved edges, implying that $n_W(\cal G_0)\ge n_a(\cal G_0) - 1$. By the argument leading to \eqref{eq:length_i}, we also have that 
$$\sum_{i=1}^{k(\cal G_0)}(\ell_i(\cal G_0)-1)\ge 2p+1.$$ 
After the operation in scenario (1) and merging vertices connected by dotted edges, we have that $n_W(\cal G_R)=n_W(\cal G_0)$, $n_a(\cal G_R)=n_a(\cal G_0)-1$, and 
\be\label{eq:klG}  \sum_{i=1}^{k(\cal G_R)}(\ell_i(\cal G_R)-1) \ge \sum_{i=1}^{k(\cal G_0)}(\ell_i(\cal G_0)-1) -1 \ge 2p,\ee
because either we have $k(\cal G_R)=k(\cal G_0)$ and $\ell_i(\cal G_R)=\ell_i(\cal G_0)-1$ for the connected component containing $x_i$, or we lose a connected component consisting of $\al$ and $x_i$ only, while all other connected components remain unchanged after operation (1).  
Applying \eqref{eq:klG} to \eqref{eq:sizeloops_recol}, we obtain that 
\begin{align}
\sum_{x_1,\ldots, x_{2p+1}}^\star \sum_x \E\cal G_R(x,x_1,\ldots, x_{2p+1}) &\prec  \sN^{2p+1} (\sN\eta)^{-2p} = \sN \eta^{-2p}.\label{eq;Vrecol0}
\end{align}
On the other hand, in our expansion strategy, we only apply the $GG$-expansion when the vertex $\al$ has degree 2. Hence, after the operation in scenario (2), the number of connected components is unchanged, so we still have \eqref{eq:klG} and that $n_W(\cal G_R)=n_W(\cal G_0)+1$, $n_a(\cal G_R)=n_a(\cal G_0)$. Then, with \eqref{eq:sizeloops_recol}, we again get the estimate \eqref{eq;Vrecol0}. 
In sum, using \eqref{eq;Vrecol0}, we obtain that 
\begin{align}
\sum_{x_1,\ldots, x_{2p+1}}^{\star} \sum_x \thn_{\fa x}\mathbb E \mathcal R_{x,x_1x_{2p+1}} \prec (\sN\eta)^{-1}\cdot  \sN \eta^{-2p} =\eta^{-(2p+1)} .\label{eq;Vrecol}
\end{align}

\medskip
\noindent {\bf Higher order graphs.} 
Finally, we consider the graphs in $\mathcal A_{x,x_1x_{2p+1}}$, which are all non-recollision graphs.
Pick one of the graphs, denoted by $\cal G_A(x, x_1,\ldots, x_{2p+1})$. 
With a similar argument as in \eqref{eq:sizeloops_recol}, we get 
\begin{align}\label{eq:sizeloops2}
\sum_{x_1,\ldots, x_{2p+1}}^\star \sum_x \E\cal G_A(x, x_1,\ldots, x_{2p+1}) \prec  \sN^{2p+1+n_a(\cal G_A)-n_W(\cal G_A)} (\sN\eta)^{-\sum_{i=1}^{k(\cal G_A)}(\ell_i(\cal G_A)-1)}.
\end{align}
Again, by the argument leading to \eqref{eq:length_i}, we have that 
$\sum_{i=1}^{k(\cal G_A)}(\ell_i(\cal G_A)-1)\ge 2p+1.$
Thus, if $n_a(\cal G_A)\le n_W(\cal G_A)$, we get from \eqref{eq:sizeloops2} and \eqref{eq;Theta} that 
\begin{align}\label{eq:sizeloops2A}
\sum_{x_1,\ldots, x_{2p+1}}^\star \sum_x \thn_{\fa x}\E\cal G_A(x, x_1,\ldots, x_{2p+1}) \prec (\sN\eta)^{-1}\cdot \sN^{2p+1} (\sN\eta)^{-(2p+1)} = \sN^{-1}\eta^{-(2p+2)}\le \eta^{-(2p+1)}.
\end{align}
It remains to deal with the case $n_a (\cal G_A)=n_{W}(\cal G_A)+1$. 

For this purpose, we perform local expansions to the graph $\cal G_A$ and obtain that
\begin{align}\label{expand GA}
\mathcal G_A(x, x_1,\ldots, x_{2p+1}) =\sum_\omega \mathcal G_{\omega}(x, x_1,\ldots, x_{2p+1}) + \PT_A + \cal E_A  + \QT_A,
\end{align}
where $\mathcal G_{\omega}$ are non-recollision locally standard graphs with coefficients of order 1, $\PT_A\equiv \PT_A(x, x_1,\ldots, x_{2p+1})$ is a sum of recollision graphs with respect to $\{x_1,\ldots, x_{2p+1}\}$, $\cal E_A\equiv\cal E_A(x, x_1,\ldots, x_{2p+1})$ is a sum of error graphs of scaling size $\OO(\sN^{-(2p+1)})$, and $\QT_A\equiv \QT_A(x, x_1,\ldots, x_{2p+1})$ is a sum of $Q$-graphs. After taking expectation and summing over all external vertices, the $Q$-graphs vanish, the error graphs are negligible, and the recollision graphs can be bounded in the same way as \eqref{eq;Vrecol}. It remains to consider the graphs $\cal G_\omega$.  If we have $n_a (\cal G_\omega)\le n_{W}(\cal G_\omega)$, then the estimate \eqref{eq:sizeloops2A} also holds for $\cal G_\omega$. Thus, we need to address the case where $n_a (\cal G_\omega)=n_{W}(\cal G_\omega)+1$. In this scenario, applying \eqref{eq:sizeloops2} to $\cal G_{\omega}$ gives that  
\begin{align}\label{eq:sizeloops2omega}
\sum_{x_1,\ldots, x_{2p+1}}^\star \sum_x \E\cal G_\omega(x, x_1,\ldots, x_{2p+1}) \prec  \sN^{2p+2} (\sN\eta)^{-\sum_{i=1}^{k(\cal G_\omega)}(\ell_i(\cal G_\omega)-1)}.
\end{align}
Similar to star graphs, the connected components in $\cal G_\omega$ must form (disjoint) loops, each of which has length $\ge 2$, due to the locally standard property. Hence, $\ell_i(\cal G_\omega)$ also represents the number of solid edges in the $i$-th loop.

By the condition \eqref{eq:sizesA_GUE} for $\cal G_A$, we have that 
$$ \size\Big(\sum_x \thn_{\fa x}\cal G_A(x, x_1,\ldots, x_{2p+1})\Big)\lesssim \eta^{-1}\size(\cal G_1)\cdot (\sN\eta)^{-\frac{1}{2}[1+k_{\mathsf p}(\cal G_A]},$$
where recall that $\cal G_1$ is defined in \eqref{eq:Gx1x2p} and $k_{\mathsf p}(\cal G_A)$ denotes the number of external solid edges in $\Gamma$ that has been pulled during the $V$-expansion. During the local expansions of $\cal G_A$, every pulling of an external solid edge in $\Gamma$ also increases the scaling size by a $(\sN\eta)^{-1/2}$ factor. Thus, the graphs $\cal G_\omega$ also satisfy that 
$$ \size\Big(\sum_x \thn_{\fa x}\cal G_\omega(x, x_1,\ldots, x_{2p+1})\Big)\lesssim \eta^{-1}\size(\cal G_1)\cdot (\sN\eta)^{-\frac{1}{2}[1+k_{\mathsf p}(\cal G_\omega]}.$$
Then, by the definition \eqref{eq_defsize_GUE}, we have 
\begin{align}
 \size\Big(\sum_x \thn_{\fa x}\cal G_\omega(x, x_1,\ldots, x_{2p+1})\Big) = \frac{1}{\sN \eta} \cdot \sN^{n_a(\cal G_\omega)-n_W(\cal G_\omega)}(\sN \eta)^{-\frac{1}{2}\sum_{i=1}^{k(\cal G_\omega)}\ell_i(\cal G_\omega)} \nonumber\\
 \lesssim \eta^{-1}\size(\cal G_1)\cdot (\sN\eta)^{-\frac{1}{2}[1+k_{\mathsf p}(\cal G_\omega)]} = \eta^{-1} (\sN \eta)^{-\frac{1}{2}(2p+2)-\frac{1}{2}[1+k_{\mathsf p}(\cal G_\omega)]}.\label{eq:sizerelation}
\end{align} 
Here, in calculating the scaling size of $\cal G_1$, we used the fact that all external solid edges in $\Gamma$ are off-diagonal, owing to our constraint that $x_1,\ldots, x_{2p+1}$ all take different values. 
Since $n_a (\cal G_\omega)=n_{W}(\cal G_\omega)+1$, equation \eqref{eq:sizerelation} gives that 
\be\label{eq:totallength} \sum_{i=1}^{k(\cal G_\omega)}\ell_i \ge 2p+3+k_{\mathsf p}(\cal G_\omega).\ee
Since each loop has at least two edges, we trivially have that $\ell_i-1\ge \ell_i/2$ for each $i\in \qqq{1,k(\cal G_\omega)}$, i.e., every edge contributes at least $1/2$ to the sum. On the other hand, we denote the subset of $2p-k_{\mathsf p}(\cal G_\omega)$ external edges that have not been pulled as $\mathsf E$. Every edge $e\in \mathsf E$ must belong to a loop, say $\cal L_e$, of length $\ge 4$. Moreover, such a loop $\cal L_e$ must contain at least two edges not in $\mathsf E$. 
To see this, we take the paths in $\cal L_e$ that consist solely of the edges in $\mathsf E$. The endpoints of these paths are all distinct external vertices, denoted $x_{k_1},\ldots, x_{k_r}$, where $r\ge 2$. Since $\cal G_\omega$ is not a recollision graph, 
to connect these endpoints into a loop, we need at least $r$ edges in $\mathsf E^c$. The above observation indicates that each edge in $\mathsf E$ contributes at least 1 to $\sum_{i=1}^{k(\cal G_\omega)}(\ell_i(\cal G_\omega)-1)$. Thus, we have 
\be\nonumber 
\sum_{i=1}^{k(\cal G_\omega)}(\ell_i(\cal G_\omega)-1)\ge |\mathsf E|+ \frac{1}{2} \left(\sum_{i=1}^{k(\cal G_\omega)} \ell_i(\cal G_\omega) - |\mathsf E| \right) = \frac{1}{2}\sum_i \ell_i(\cal G_\omega) + \frac{1}{2} \left(2p-k_{\mathsf p}(\cal G_\omega)\right)\ge 2p+\frac{3}{2}, 
\ee
where we used \eqref{eq:totallength} in the last step.  Since the LHS is an integer, we indeed have that 
\be\label{eq:totallength-1}
\sum_{i=1}^{k(\cal G_\omega)}(\ell_i(\cal G_\omega)-1)\ge   2p+2. \ee
Plugging it into \eqref{eq:sizeloops2omega}, we obtain that 
\begin{align*}
\sum_{x_1,\ldots, x_{2p+1}}^\star \sum_x \E\cal G_\omega(x, x_1,\ldots, x_{2p+1}) \prec   \eta^{-(2p+2)} .
\end{align*}
Since $\eta\ge \sN^{-1/2}$, we have that 
\begin{align*} 
\sum_{x_1,\ldots, x_{2p+1}}^\star \sum_x \thn_{\fa x}\E\cal G_\omega(x, x_1,\ldots, x_{2p+1}) \prec (\sN\eta)^{-1}\cdot \eta^{-(2p+2)}= \sN^{-1}\eta^{-(2p+3)} \le \eta^{-(2p+1)}.
\end{align*}

In sum, we have obtained that for $\eta\ge \sN^{-1/2}$,
\begin{align}
    \sum_{x_1,\ldots, x_{2p+1}}^\star \E \sum_x \thn_{\fa x}\mathcal A_{x,x_1x_{2p+1}} \prec \eta^{-(2p+1)} .\label{eq;Vhigher}
\end{align}

Combining \eqref{eq;Vleading}, \eqref{eq;Vstar1}, \eqref{eq;Vstar2}, \eqref{eq;Vstar3}, \eqref{eq;Vrecol}, and \eqref{eq;Vhigher},  we obtain that for any constant $\serr\in(0,1/2)$, \eqref{scsgun2} holds, which concludes the proof of \Cref{lem_Lp1_St}. 
\end{proof}

\section{Continuity estimate}\label{sec:continuity}

In this section, we present the proofs of the continuity estimate, \Cref{gvalue_continuity}, 
for the Wegner orbital model. Our main tools are the $T$-expansion and $V$-expansion, where the molecule sum zero (or coupling renormalization) property \eqref{eq:vertex_sum_zero} plays a key role. 


Since $\mathrm{Tr}(\mathcal A_{\mathcal I}^{p})$ can be written into a linear combination of loop graphs with $p$ solid edges, \Cref{gvalue_continuity} follows as an immediate consequence of the next lemma. 
\begin{lemma}\label{gvalue_continuity2}
In the setting of \Cref{lem: ini bound}, suppose 
\be\label{eq:cont_ini_eta}
T_{xy}(z)\prec B_{xy} +(N\eta)^{-1},\quad \|G(z)-M(z)\|_{\max}\prec \heta^{1/2}, 
\ee
for some $z =E+ \ii \eta$ with $|E|\le  2-\kappa$ and $\eta\in [\eta_0,\blam^{-1}]$. Consider a loop graph 
\be\label{eq_pgons}
\cal G_{\bx}(z) = \prod_{i=1}^{p}G^{c_i}_{x_i x_{i+1}}(z),
\ee
where $\bx:=(x_1,\cdots, x_{p})$, $x_{p+1}\equiv x_1$, and $c_i \in \{\emptyset,\dag\}$. 
Let $\cal I\subset \Z_L^d$ be a subset with $|\cal I|\ge W^d$ and denote $K:=|\cal I|^{1/d}$. Then, for any fixed $p\in 2\N$, we have that (recall $\conc(K,\lambda,\eta)$ defined in \eqref{eq:def-conc})
	\begin{equation}\label{eq_bound_Gx}
	\begin{split}
	  &\frac{1}{|\cal I|^p}\sum_{x_i \in \cal I: i =1,\ldots, p}\E\mathcal G_{\mathbf x}(z) \prec \conc(K,\lambda,\eta)^{{p-1}} \, .
	\end{split}
	\end{equation} 
\end{lemma}

\begin{proof}[\bf Proof of \Cref{gvalue_continuity}]
Since $\mathrm{Tr}(\mathcal A_{\mathcal I}^{p})$ can be written into a linear combination of $\OO(1)$ many loop graphs of the form \eqref{eq_pgons} (with the spectral parameter $z=\wt z$), we conclude \eqref{eq_bound_Gx0} by \eqref{eq_bound_Gx}.
\end{proof}

\subsection{Proof of \Cref{gvalue_continuity2}}	
We define the \emph{generalized doubly connected property} of graphs with external molecules, which generalizes the doubly connected property in Definition \ref{def 2net}. 

\begin{definition}[Generalized doubly connected property]\label{def 2netex}
A graph $\cal G$ with external molecules is said to be generalized doubly connected if its molecular graph satisfies the following property. There exists a collection, say $\cal B_{black}$, of $\dashed$ 
edges, and another collection, say $\cal B_{blue}$, of blue solid, $\dashed$, 
free, or ghost edges such that: 
\begin{itemize}
\item[(a)]  $\cal B_{black}\cap \cal B_{blue}=\emptyset$; 
\item[(b)] every internal molecule connects to external molecules through two disjoint paths: a path of edges in $\cal B_{black}$ and a path of edges in $\cal B_{blue}$. 
\end{itemize}
Simply speaking, a graph is generalized doubly connected if merging all its external molecules into a single internal molecule gives a doubly connected graph in the sense of Definition \ref{def 2net}. 
\end{definition}

As shown in \cite{BandIII}, expanding a graph without internal molecules in a proper way always generates generalized doubly connected graphs. Moreover, in these graphs, at most half of the external molecules are used in the generalized doubly connected property in the sense of property (b) of the following lemma. 

\begin{lemma}\label{mG-fe}
In the setting of \Cref{gvalue_continuity2}, suppose \eqref{eq:cont_ini_eta} holds and we have a $\nonuni$ \eqref{mlevelTgdef weak}. For the graph $\mathcal G_{\mathbf x}(z)$ in \eqref{eq_pgons}, let $\Sigma_1,\ldots, \Sigma_q$ be disjoint subsets that form a partition of the set of vertices $\{x_1, \ldots, x_{p}\}$ with $1\le q\le p$. Then, we identify the vertices in $\Sigma_i$, $i\in \qqq{1,q}$, and denote the resulting graph by $ \mathcal G_{\mathbf x, \mathbf \Sigma}(z)$, where $\mathbf \Sigma=(\Sigma_1,\ldots, \Sigma_q)$ and we still denote the external vertices by $\bx=(x_1,\ldots, x_q)$ without loss of generality. Suppose $x_1,\ldots, x_q\in \Z_L^d$ all take different values. Then, we have that for any constant $D>0$,
\begin{equation}\label{EGx}
\E[\mathcal G_{\mathbf x,\mathbf \Sigma}] = \sum_\mu\mathcal G_{\mathbf x}^{(\mu)} + \OO(W^{-D})\,,
\end{equation}
where the RHS is a sum of $\OO(1)$ many \emph{deterministic} normal graphs $\mathcal G_{\mathbf x}^{(\mu)}$, satisfying the following properties: 
\begin{itemize}
\item[(a)] $\mathcal G_{\mathbf x}^{(\mu)}$ are generalized doubly connected graphs in the sense of \Cref{def 2netex}.

\item[(b)] Consider all external molecules that are neighbors of internal molecules on the molecular graph, say $\cal M_1,\ldots, \cal M_n$. We can find at least $r \ge \lceil n/2\rceil$ of them, denoted by $\{\cal M_{k_i}\}_{i=1}^r$, that are simultaneously connected with redundant $\dashed$ or free edges in the following sense: (i) every $\cal M_{k_i}$, $i \in \qqq{r}$, connects to an internal molecule through a $\dashed$ or free edge $e_{k_i}$, and (ii) after removing all these $r$ edges $e_{k_i}$, $i \in \qqq{r}$, the resulting graph is still generalized doubly connected. We will call these molecules $\cal M_{k_i}$ as special external molecules. 

\item[(c)] 
All vertices are connected in $\mathcal G_{\mathbf x}^{(\mu)}$ through non-ghost edges.   

\end{itemize}
\end{lemma}
\begin{proof}
Since the Wegner orbital model also belongs to the random band matrix ensemble, this lemma is an immediate consequence of Lemma 6.2 (a)--(c) in \cite{BandIII}. For the later proof of the key \Cref{lem:mG-fe-size} below, we outline the expansion strategy adopted in the proof of this lemma. 

First, corresponding to Definition \ref{def 2netex}, we introduce the following definition.

\begin{definition}\label{GSS_external}
We define the generalized SPD and globally standard properties for graphs with external molecules (recall Definitions \ref{def seqPDG} and \ref{defn gs}) as follows.
\begin{itemize}
\item [(i)] A graph $\cal G$ with external molecules is said to satisfy the {\bf generalized SPD property} if merging all external molecules of $\cal G$ into one single internal molecule yields an SPD graph. 

\item[(ii)] A graph $\cal G$ with external molecules is said to be {\bf generalized globally standard} if it is generalized SPD and every proper isolated subgraph with non-deterministic closure is weakly isolated.
\end{itemize}
\end{definition}

Next, we will expand $\cal G_{\bx,\mathbf\Sigma}$ by applying Strategy \ref{strat_global} with the following modifications: 
\begin{itemize}
\item we use the generalized globally standard property in Definition \ref{GSS_external};
\item for a global expansion, we use the $V$-expansion \eqref{eq_Vexp} instead of the $T$-expansion; 
\item we stop expanding a graph if it is deterministic, its weak scaling size is $\OO(W^{-D})$, or it is a $Q$-graph.
\end{itemize} 
More precisely, we will adopt the following strategy. 

\begin{strategy}
\label{strat_global_weak2}
Given a large constant $D>0$ and a generalized globally standard graph without $P/Q$ labels, we perform one step of expansion as follows. 

\vspace{5pt}
\noindent{\bf Case 1:} Suppose we have a graph where all solid edges are between external molecules. Corresponding to Step 0 of Strategy \ref{strat_global}, we perform local expansions to get a sum of locally standard graphs such that all solid edges are between internal vertices (along with some graphs that already satisfy the stopping rules). Then, we apply \eqref{eq_Vexp} to an arbitrary $T$-variable and get a sum of generalized globally standard graphs (plus some graphs that already satisfy the stopping rules). 

\vspace{5pt}
\noindent{\bf Case 2:} If the given graph is not locally standard, we then perform local expansions on it and get a sum of generalized globally standard graphs that are locally standard. 

\vspace{5pt}
\noindent{\bf Case 3:} Given a generalized globally standard graph, we find a $T$-variable that contains the first blue solid edge in a pre-deterministic order of the MIS as in Step 2 of Strategy \ref{strat_global}, and then apply the $V$-expansion in \Cref{thm:Vexp} to it to get a sum of generalized globally standard graphs (plus some graphs that already satisfy the stopping rules). 
\end{strategy}
\noindent Applying this one step of expansion repeatedly, we can get \eqref{EGx} as shown in the proof of \cite[Lemma 6.2]{BandIII}.
\end{proof}

Now, to complete the proof of \Cref{gvalue_continuity2}, we need the following key result regarding the weak scaling size of the graphs in \eqref{EGx}.   
\begin{lemma}\label{lem:mG-fe-size}
In the setting of \Cref{mG-fe}, the weak scaling sizes (recall \eqref{eq:def-size}) of the deterministic graphs on the RHS of \eqref{EGx} satisfy that 
\be\label{mG-fe-size}
\wsize(\mathcal G_{\mathbf x}^{(\mu)})\lesssim \heta^{q-1}.
\ee
\end{lemma}

The proof of Lemma \ref{mG-fe-size} is deferred to the next subsection. Here, we discuss briefly why its proof depends crucially on the molecule sum zero property established in \eqref{eq:vertex_sum_zero}. In fact, without using the $V$-expansions and the molecule sum zero property, the argument in \cite[Section 6]{BandIII} yields a weaker condition than \eqref{mG-fe-size}:
\be\label{mG-fe-size-weak} \wsize(\mathcal G_{\mathbf x}^{(\mu)})\le \blam^{q/2-1}\heta^{q-1}.\ee
This will lead to a weaker continuity estimate than \eqref{eq_bound_Gx0} with $\conc(K,\lambda,\eta)$ replaced by 
\begin{equation}
\label{eq:def-conc_old}
\wt \conc(K,\lambda,\eta)
= \blam^{1/2}\conc(K,\lambda,\eta)\,.
\end{equation}
To help readers understand the statement of \Cref{mG-fe}, the weaker bound \eqref{mG-fe-size-weak}, and how to improve from \eqref{mG-fe-size-weak} to \eqref{mG-fe-size}, we provide a heuristic discussion based on some illustrating examples. These examples describe a typical class of deterministic graphs that will naturally appear in the expansions of $\cal G_{\bx, \mathbf \Sigma}$. The leading contribution comes from the case $q=p$, i.e., all external vertices take different values. 
The following figure shows a molecular graph with $p=2s+2$, where vertices $x_i$ denote external molecules and vertices $a_i$ denote internal molecules. 
\be\label{eq:exampletree1}
\parbox[c]{8cm}{\includegraphics[width=\linewidth]{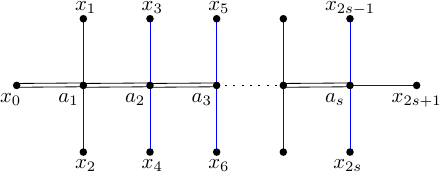}}
\ee
Here, the black double-line edges represent the diffusive edges in the black net, and the blue solid lines represent the diffusive or free edges in the blue net. 
Notice that among the two blue edges connected with an internal molecule $a_i$, $i \in \qqq{1, s-1}$, one of them is redundant for the generalized doubly connected property, while $a_s$ is connected with two redundant blue edges. Thus, we can choose $x_1, x_3 ,\ldots, x_{2s+1} $ as special external molecules among the $p=2s+2$ external molecules in the above graph so that properties (a) and (b) of \Cref{mG-fe} hold. 
To count the weak scaling size of the above graph, we assume that all molecules have no internal structures, i.e., each of them contains only one vertex. Then, in \eqref{eq:exampletree1}, there are $s$ internal vertices, $s$ diffusive edges, and $(2s+1)$ blue edges, each of which has a weak scaling size bounded by $\heta$. Thus, the size of \eqref{eq:exampletree1} is bounded by 
\smash{$\blam^{s} \heta^{2s+1}=\blam^{p/2-1}\heta^{p-1}$} as in \eqref{mG-fe-size-weak} with $q=p$.  

We now discuss how a graph as in \eqref{eq:exampletree1} can be generated in our expansions.  In our expansion strategy, new internal molecules are generated when we apply the $V$-expansion as in Cases 1 and 3 of \Cref{strat_global_weak2}. For example, suppose at some step, we need to expand a graph \eqref{eq:regularG} with $\fa,\fb_1,\fb_2$ being vertices in the molecules $x_0,x_1,x_2$. Then, the main issue is to deal with the graph $\cal G_0$ in \eqref{eq:prior_star}, drawn as the first graph in the following picture, where $a_1$ is the internal molecule containing vertex $x$ in \eqref{eq:prior_star}:  
\be\label{eq:exampletree_gen}
\parbox[c]{15cm}{\includegraphics[width=\linewidth]{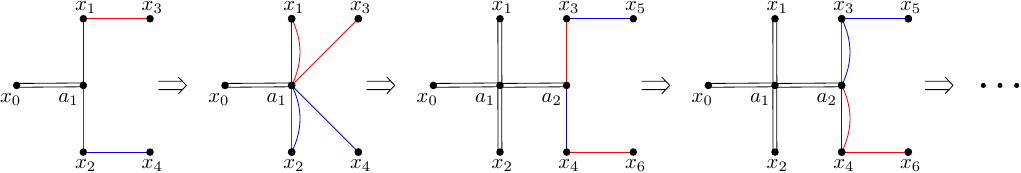}}
\ee
Here, we did not draw the full molecular graph or include the $Q$-label; instead, we have displayed only the edges relevant to the discussion below. 
Then, by applying $Q$ and local expansions, we obtain new graphs where the red solid edge $(x_1,x_3)$ and the blue solid edge $(x_2,x_4)$ are pulled to the molecule $a_1$, as shown in the second graph of \eqref{eq:exampletree_gen}. If the two solid edges between $a_1$ and $x_1$ are paired, they can become a diffusive edge after applying a further $V$-expansion. Similarly, the solid edges between $a_1$ and $x_2$ may also become a diffusive edge. Next, we apply the $V$-expansion to the $T$-variable formed by the solid edges $(a_1,x_3)$ and $(a_1,x_4)$. The main issue is to deal with the third graph of \eqref{eq:exampletree_gen} where the solid edges $(a_2,x_3)$ and $(a_2,x_4)$ have a $Q_y$ label with $y$ belonging to the molecule $a_2$. Applying $Q$ and local expansions again yields some new graphs of the form shown in the fourth graph of \eqref{eq:exampletree_gen}. Continuing this process will lead to the molecular graph depicted in  \eqref{eq:exampletree1}. 
Now, a key observation is that in this process, each time we generate a new molecule using the $V$-expansion in \Cref{thm:Vexp}, we acquire an additional small factor $\OO(\eta+\blam^{-1})=\OO(\blam^{-1})$ due to the molecule sum zero property \eqref{eq:vertex_sum_zero} and the estimate \eqref{eq:vertex_sum_zero_diff}. This means that each internal molecule is associated with a $\blam^{-1}$ factor, allowing us to cancel the $\blam^{q/2-1}$ factor in \eqref{mG-fe-size-weak} and get the improved estimate \eqref{mG-fe-size}.

In general, we believe that the \emph{leading contributions} to \eqref{EGx} should come from deterministic graphs whose molecular graphs are \emph{generalized doubly connected trees}, where each internal molecule has degree $\ge 4$, and all leaves are external molecules. Moreover, every internal molecule should be associated with a $\blam^{-1}$ factor. (However, we do not fully utilize these facts in this paper and \cite{BandIII} because proving them is quite challenging. We instead rely on two weaker properties, i.e., property (b) of \Cref{mG-fe} and the weak scaling size condition \eqref{mG-fe-size}.) 
For example, in the second graph of \eqref{eq:exampletree_gen}, if the blue edge $(a_1,x_1)$ and red edge $(a_1,x_3)$ are paired into a $T$-variable while the red edge $(a_1,x_2)$ and blue edge $(a_1,x_4)$ are paired into a $T$-variable, then expanding them may produce two branches emanating from molecule $a_1$ as shown in \Cref{fig:treeg3}. We can check that this graph is generalized doubly connected, with the dashed lines denoting the redundant edges in the blue net. Counting these dashed lines and the weak scaling size, we also see that the property (b) in \Cref{mG-fe} holds and its weak scaling size satisfies \eqref{mG-fe-size-weak}. With $V$-expansions, we can get an extra factor $\blam^{-1}$ at each internal molecule, which allows us to get \eqref{mG-fe-size}. 
 
  \begin{figure}[htbp]
\centering
\includegraphics[width=0.45\textwidth]{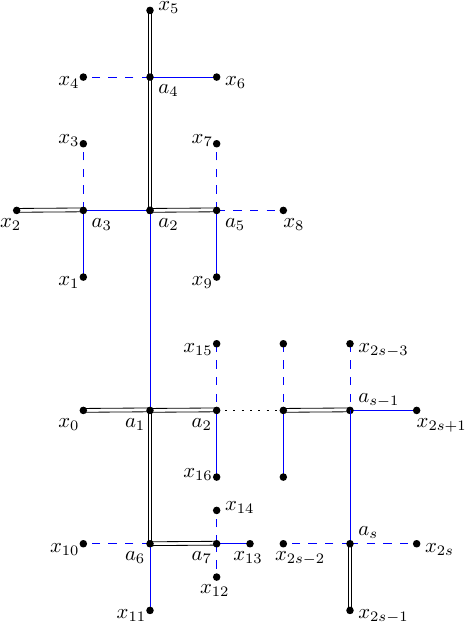} 
\caption{A tree-like molecular graph, where $x_i$'s denote external molecules and $a_i$'s denote internal molecules.  The black double-line edges represent the diffusive edges in the black net, and the blue lines represent the diffusive or free edges in the blue net. Here, the dashed lines represent the redundant edges in the generalized doubly connected property.}
\label{fig:treeg3}
\end{figure}



\begin{proof}[\bf Proof of \Cref{gvalue_continuity2}]
With Lemmas \ref{mG-fe} and \ref{lem:mG-fe-size}, the proof of \Cref{gvalue_continuity2} uses the same argument as shown in Section 6 of \cite{BandIII}. Hence, we only give a sketch of the proof without presenting all the details. 

By \Cref{mG-fe}, to show \eqref{eq_bound_Gx}, it suffices to prove that 
\be\label{eq_bound_Gx2}
\frac{1}{|\cal I|^p}\sum_{x_i \in \cal I: i =1,\ldots, p}\mathcal G_{\mathbf x}^{(\mu)} \prec \conc(K,\lambda,\eta)^{p-1}
\ee
for each deterministic graph $\mathcal G_{\mathbf x}^{(\mu)}$ satisfying \eqref{mG-fe-size} and properties (a)--(c) in \Cref{mG-fe}. For this purpose, we first bound \smash{$\mathcal G_{\mathbf x}^{(\mu)}$} by a new graph where the coefficient is bounded by its absolute value, the ghost and free edges remain unchanged, and the remaining graph is bounded (in the sense of absolute value) as follows: 
\begin{itemize}
\item[(1)] For each molecule in $\mathcal G_{\mathbf x}^{(\mu)}$, we choose a representative vertex in it, called \emph{center} of the molecule. For an external molecule containing an external vertex $x_i$, let $x_i$ be its center. If an external molecule contains multiple external vertices $x_{k_1}, \ldots, x_{k_r}$, then we choose one of them as the center. 


\item[(2)] We define a ``pseudo-waved edge" between vertices $x$ and $y$ as a waved edge representing  the factor $\mathcal W_{xy}:=W^{-d}\mathbf 1({|x-y|\le W(\log W)^2})$. Now, we bound the waved edges within molecules as follows. First, each waved edge between vertices $x$ and $y$ is bounded by a factor $\OO(W^{-d}\mathbf 1(|x-y|\le W(\log W)^{1+\tau}))$ for a constant $\tau\in (0,1)$. By the definition of $S$ and the bound \eqref{S+xy}, this leads to an error of order $\OO(W^{-D})$ for any large constant $D>0$. 
Second, suppose $x$ and $y$ belong to a molecule $\cal M$ with center $x_0$. We can further bound the factor $W^{-d}\mathbf 1(|x-y|\le W(\log W)^{1+\tau})$ by pseudo-waved edges between them and $x_0$, i.e., $W^d\cdot \mathcal W_{xx_0}\cdot \mathcal W_{yx_0}$. Finally, if there are multiple pseudo-wave edges between vertices $x$ and $x_0$, we only keep one of them, while the rest of them provide a power of $W$ that can be included into the coefficient. 

\item[(3)] By \eqref{thetaxy}, a diffusive edge between vertices $x$ and $y$ is bounded by a ``pseudo-diffusive edge" that represents the factor $B_{xy}$. Suppose $x$ and $y$ belong to molecules $\cal M_1$ and $\cal M_2$ with centers $x_0$ and $y_0$ ($\cal M_1$ and $\cal M_2$ can be the same molecule). Then, up to an error of order $\OO(W^{-D})$, we can further bound the $B_{xy}$ edge by 
$$ (\log W)^{2d}B_{x_0 y_0}\mathbf 1(|x-x_0|\le W (\log W)^2)\mathbf 1(|y-y_0|\le W(\log W)^2).$$
For a labeled diffusive edge $\zthn_{xy}(\Sele)$ or $(\zthn\Sele_1\zthn\Sele_2\zthn \cdots \zthn \Sele_k \zthn)_{xy}$ as in \eqref{BRB}, we can bound it in a similar manner. However, in the latter case, we multiply the coefficient of the graph by the extra factor $\prod_{i=1}^k \sizeself(\Sele_i)$.
\end{itemize}
After the above operations, we can bound $\mathcal G_{\mathbf x}^{(\mu)}$ by a new deterministic graph: 
\be\label{eq:Ggamma1}\mathcal G_{\mathbf x}^{(\mu)}\prec (\log W)^{C} {\mathcal G}_{\mathbf x}'+\OO(W^{-D}).\ee
Then, we sum over all internal vertices in ${\mathcal G}_{\mathbf x}'$ that are not centers of molecules and get that:
\be\label{eq:Ggamma2}{\mathcal G}_{\mathbf x}' \prec  (\log W)^{C}c''(\mathcal G_{\mathbf x}^{(\mu)})\cdot {\mathcal G}_{\mathbf x}'',\ee 
where $c''(\mathcal G_{\mathbf x}^{(\mu)})$ is a positive coefficient coming from the coefficient of $\mathcal G_{\mathbf x}^{(\mu)}$ and the above operations (1)--(3). ${\mathcal G}_{\mathbf x}''$ is a deterministic graph consisting of external vertices, internal vertices that are centers of internal molecules,  pseudo-diffusive, free, or ghost edges between them, as well as potential pseudo-waved edges between external vertices. This graph satisfies (a)--(c) of \Cref{mG-fe}. Moreover, if we define $\wsize({\mathcal G}_{\mathbf x}'')$ as (note that we do not count the factors from the ghost edges as in \eqref{eq:def-size})
\begin{align*}
    \wsize({\mathcal G}_{\mathbf x}'') =&\heta^{\#\{\text{free edges}\}+\#\{\text{pseudo-diffusive edges}\}} W^{-d (\#\{\text{pseduo-waved edges}\}-\#\{\text{internal vertices}\})}, 
\end{align*} 
then it satisfies that
$$c''(\mathcal G_{\mathbf x}^{(\mu)}) \wsize({\mathcal G}_{\mathbf x}'') \prec \wsize(\mathcal G_{\mathbf x}^{(\mu)})\lesssim \heta^{q-1}. $$

Next, we bound the summations over the internal vertices in ${\mathcal G}_{\mathbf x}''$. For this purpose, we introduce two new types of edges: 
\begin{itemize}
    \item a ``silent diffusive edge" between vertices $x$ and $y$ represents a factor		
  \be\label{eq:sil_diff}\wt B_{xy}:=\frac{\blam}{W^{4}\langle x-y\rangle^{d-4}};\ee

 \item a ``silent free edge" between vertices $x$ and $y$ represents a factor 
			$\frac{1}{N\eta} \frac{L^2}{W^2}.$
\end{itemize}
Like free and ghost edges, a silent free edge is just a scalar and we introduce it to help with analyzing the structures of graphs.   
With the generalized doubly connected property of ${\mathcal G}_{\mathbf x}''$, we can choose a collection of blue trees such that every internal vertex connects to an external vertex through a \emph{unique path} on a blue tree. These blue trees are disjoint, each of which contains an external vertex as the root. Then, we will sum over all internal vertices \emph{from leaves of the blue trees to the roots}. In bounding every summation, we use the following two estimates with $k\ge 1$ and $r\ge 0$: 
\be\label{keyobs4.3}
\begin{split}
& \sum_{x } \prod_{j=1}^k B_{ x y_j}\cdot \prod_{s=1}^r \wt B_{ x z_s} \cdot B_{x a} \\ 
&\prec \blam \sum_{l=1}^k \prod_{j: j\ne l}B_{y_j y_{l}} \cdot \Big( \wt B_{y_l a} \prod_{s=1}^r \wt B_{z_s a}+ \wt B_{y_l a} \prod_{s=1}^r \wt B_{z_s y_l} + \sum_{t=1}^r  \wt B_{y_l z_t}  \wt B_{z_t a} \prod_{s:s\ne t} \wt B_{z_s z_t} \Big) ,
\end{split}
\ee
\be\label{keyobs4.3_add}
\sum_{x} \prod_{j=1}^k B_{ x y_j}\cdot \prod_{s=1}^r \wt B_{ x z_s} \prec \blam \frac{L^2}{W^2}\sum_{l=1}^k \prod_{j: j\ne l}B_{y_j y_{l}} \cdot \Big(  \prod_{s=1}^r \wt B_{z_s y_l} + \sum_{t=1}^r  \wt B_{y_l z_t}  \prod_{s:s\ne t} \wt B_{z_s z_t} \Big) .
\ee
The LHS of \eqref{keyobs4.3} is a star graph consisting of $k$ pseudo-diffusive edges (where at least one of them is in the black tree), a pseudo-diffusive edge in the blue tree, and $r$ silent pseudo-diffusive edges connected with $x_i$, while every graph on the RHS is a connected graph consisting of $(k-1)$ pseudo-diffusive edges and $(r+1)$ silent pseudo-diffusive edges (together with an additional $\blam$ factor). The estimate \eqref{keyobs4.3_add} can be applied to the case where the blue edge is a ghost or free edge. (Of course, there may be multiple free edges connected with $x$, in which case we can still use \eqref{keyobs4.3} and \eqref{keyobs4.3_add} by multiplying some $(N\eta)^{-1}$ factors with them.) After applying \eqref{keyobs4.3_add}, we either lose a ghost edge or change one free edge to a silent free edge (together with an additional $\blam$ factor). 

Summing over all internal vertices in ${\mathcal G}_{\mathbf x}''$ from leaves of the blue trees to the roots and applying \eqref{keyobs4.3} and \eqref{keyobs4.3_add} to bound each summation, we obtain that 
\begin{equation}\label{eq:bound_EGx}
c''(\mathcal G_{\mathbf x}^{(\mu)}){\mathcal G}_{\mathbf x}'' \prec  \sum_\gamma c_\gamma \wt {\mathcal G}_{\mathbf x}^{(\gamma)} + \OO(W^{-D})\,,
\end{equation}
where the RHS is a sum of $\OO(1)$ many deterministic graphs $\wt {\mathcal G}_{\mathbf x}^{(\gamma)}$ consisting of external vertices and pseudo-diffusive, silent diffusive, free, silent free, and pseudo-waved edges between them, and $c_\gamma$ denotes the coefficient. Moreover, these graphs and coefficients satisfy the following properties: 
\begin{itemize}
\item[(i)] Suppose that $t$ of the $q$ external vertices are centers of external molecules. Without loss of generality, assume that $x_{1},\ldots, x_{t}$ are these centers. Among these $t$ vertices, there are $k_s\ge \lceil t/2\rceil$ special vertices, each connected to a \emph{unique pseudo-diffusive/free edge}. We call these $k_s$ pseudo-diffusive/free edges associated with the $k_s$ special vertices as special edges. 

\item[(ii)] We define $\wsize(\wt {\mathcal G}_{\mathbf x}^{(\gamma)})$ as 
\begin{align}\label{eq:cgamma0}
    \wsize(\wt {\mathcal G}_{\mathbf x}^{(\gamma)}) =&\heta^{Fr+SF+PD+SD} W^{-(q-t)d}, 
\end{align}
where $Fr$, $SF$, $PD$, and $SD$ denote the numbers of free, silent free, pseudo-diffusive, and silent diffusive edges in \smash{$\wt {\mathcal G}_{\mathbf x}^{(\gamma)}$}. Then, for every $\gamma$, we have that 
\be\label{eq:cgamma} c_\gamma \wsize (\wt {\mathcal G}_{\mathbf x}^{(\gamma)}) \lesssim c''(\mathcal G_{\mathbf x}^{(\mu)}) \wsize({\mathcal G}_{\mathbf x}'') \prec \wsize(\mathcal G_{\mathbf x}^{(\mu)})\lesssim \heta^{q-1} .\ee

			
\end{itemize}
The proof of \eqref{eq:bound_EGx} and these two conditions is similar to that in Section 6.1 of \cite{BandIII}, so we omit the details. 

Finally, we sum over the external vertices $x_1,\ldots, x_q$ in $\wt {\mathcal G}_{\mathbf x}^{(\gamma)}$. For the $q-t$ external vertices $x_{t+1},\ldots, x_q$ that are not centers, their averages provide a $K^{-(q-t)d}$ factor. 
For the remaining $t$ vertices, each of them is connected with at least one silent/non-silent pseudo-diffusive or free edge. Thus, its average over $\cal I$ provides at least a factor
\be\label{bad factor}
\frac{\blam}{W^4 K^{d-4}}+ \frac{1}{N\eta}\frac{L^2}{W^2}.
\ee
Furthermore, we will see that each average over an external vertex connected with a non-silent edge contributes a better factor
\be\label{good factor}
\frac{\blam}{W^2 K^{d-2}}+ \frac{1}{N\eta} + \frac{1}{K^{d/2}} \sqrt{\frac{\blam}{N\eta}\frac{L^2}{W^2}}.
\ee
By property (i) above, there are $k_s\ge \lceil t/2\rceil$ special external vertices, each associated with a unique special non-silent edge. Therefore, on average, each external vertex $x_i$, $i\in \qqq{1,t}$, provides at least a factor $\conc(K,\lambda,\eta)$.  

To rigorously justify the above argument, we estimate the averages over the special external vertices. We use the following two estimates to bound the average over such an external vertex $x$ connected with a special diffusive edge: for $k\ge 1$ and $r\ge 0$,  
\begin{align}
& \frac{1}{|\cal I|}\sum_{x\in \cal I} B_{xy_1}^2 \prod_{j=2}^k B_{ x y_j}\cdot \prod_{s=1}^r \wt B_{ x z_s} \prec  \frac{\blam}{W^{d}}\frac{\blam}{ K^d} \sum_{l=1}^k  \sum_{t=1}^r \prod_{j\in \qqq{k}\setminus \{l\}}B_{y_j y_{l}} \cdot \wt B_{ y_1 z_t} \cdot   \prod_{s:s\ne t} \wt B_{z_s z_t}  ,\label{keyobs4}\\
& \frac{1}{|\cal I|}\sum_{x\in \cal I} \prod_{j=1}^k B_{ x y_j}\cdot \prod_{s=1}^r \wt B_{ x z_s}  \prec \frac{\blam}{W^2 K^{d-2}} \sum_{l=1}^k  \sum_{t=1}^r \prod_{j\in \qqq{k}\setminus \{l\}}B_{y_j y_{l}} \cdot \wt B_{y_l z_t} \cdot   \prod_{s:s\ne t} \wt B_{z_s z_t} .\label{keyobs4_add}
\end{align}
Here, we only include the pseudo-diffusive and silent diffusive edges connected with $x$. If there are free or silent free edges connected with $x$, we can still use \eqref{keyobs4} and \eqref{keyobs4_add} by multiplying them with the corresponding factors. We apply the above two estimates in two different cases. In the first case, suppose the special diffusive edge associated with $x$ is paired with the special diffusive edge associated with $y_1$. Then, we use \eqref{keyobs4} to bound the average over $x$ by a factor $\blam^2 W^{-d}K^{-d}$ times a sum of new graphs, each of which is still connected and has \emph{two fewer} special vertices. If the first case does not happen, then we use \eqref{keyobs4_add} to bound the average over $x$ by a factor $\blam W^{-2}K^{-(d-2)}$ times a sum of new graphs, each of which is still connected and has \emph{one fewer} special vertex. 

Second, we sum over special external vertices connected with non-silent free edges. Given such an external vertex $x$, we use (1) the trivial identity $|\cal I|^{-1}\sum_{x\in \cal I}1=1$ if $x$ is only connected to silent/non-silent free edges, (2) the estimate \eqref{keyobs4_add} if $x$ is connected to at least one pseudo-diffusive edge, or (3) the following estimate if $x$ is connected to silent diffusive edges but no pseudo-diffusive edges:
\be\label{keyobs5_add}
\frac{1}{|\cal I|}\sum_{x\in \cal I} \prod_{s=1}^r \wt B_{ x z_s}  \prec \frac{\blam}{W^4 K^{d-4}}   \sum_{t=1}^r \prod_{s:s\ne t} \wt B_{z_s z_t} .
\ee
In this way, we can bound the average over $x$ by a factor $(N\eta)^{-1}$ times a sum of new graphs, each of which is still connected and has one fewer special external vertex. 

Third, after summing over all special external vertices, we then sum over the non-special external vertices one by one using $|\cal I|^{-1}\sum_{x\in \cal I}1=1$ or the estimate \eqref{keyobs5_add}. Each summation yields at least a factor \eqref{bad factor} times a sum of new connected graphs. Finally, the average over the last vertex is equal to $1$.  

Overall, after summing over all external vertices $x_1,\ldots, x_t$, we can obtain that 
\begin{align}
 \frac{1}{|\cal I|^q}\sum_{x_i \in \cal I: i \in \qqq{1,q}} \wt{\cal G}_{\bx}^{(\gamma)} 
&\prec \heta^{(Fr+SF+PD+SD)-(t-1)} \left(\frac{1}{K^d}\right)^{q-t}\left( \frac{\blam}{K^d}\right)^{k_1} \left( \frac{\blam}{W^2K^{d-2}} + \frac{1}{N\eta}\right)^{k_s-2k_1}  \nonumber\\
&\quad \times \left( \frac{\blam}{W^4K^{d-4}} + \frac{1}{N\eta}\frac{L^2}{W^2}\right)^{t-1-(k_s-k_1)}, \label{eq_reduce_to_aux2}
\end{align}
where $\lceil t/2\rceil \le k_s \le t-1$ is the number of special external vertices and $0\le k_1\le \lfloor k_s/2\rfloor$ is the number of times that \eqref{keyobs4} has been applied.  
Here, the first factor on the RHS comes from $(Fr+SF+PD+SD)-(t-1)$ silent/non-silent pseudo-diffusive and free edges that become self-loops during the summation (for example, the $\blam/W^{d}$ factor in \eqref{keyobs4} comes from a pseudo-diffusive edge that becomes a self-loop in the summation), all of which are bounded by $\heta$. The factor $K^{-d(q-t)}$ results from averaging over the $q-t$ external vertices that are not centers of molecules. The third and fourth factors come from averages over special external vertices, while the last factor on the RHS is due to averages over the remaining non-special external vertices. 
Together with \eqref{eq:cgamma0} and \eqref{eq:cgamma}, this leads to 
\begin{align*}
\frac{1}{|\cal I|^p}\sum_{x_i \in \cal I: i \in \qqq{p}}c_\gamma \wt{\mathcal G}_{\mathbf x}^{(\gamma)} &\prec \frac{1}{|\cal I|^{p-q}} \left( \frac{\blam}{K^d}\right)^{q-t+k_1} \left( \frac{\blam}{W^2K^{d-2}} + \frac{1}{N\eta}\right)^{k_s-2k_1} \left( \frac{\blam}{W^4K^{d-4}} + \frac{1}{N\eta}\frac{L^2}{W^2}\right)^{t-1-(k_s-k_1)}\\
&\prec \sum_{k \in \qqq{p}} \sum_{ \lceil p/2\rceil \le k\le p-1}  \left( \frac{\blam}{W^2K^{d-2}} + \frac{1}{N\eta}+ \frac{1}{K^{d/2}} \sqrt{\frac{\blam}{N\eta}\frac{L^2}{W^2}}\right)^{k}  \\ 
&\quad \times \left( \frac{\blam}{W^4K^{d-4}} + \frac{1}{N\eta}\frac{L^2}{W^2}\right)^{p-1-k}  \prec \conc(K,\lambda,\eta)^{p-1}.
\end{align*}	
Plugging it into \eqref{eq:bound_EGx} and further into \eqref{eq:Ggamma1} and \eqref{eq:Ggamma2}, we conclude \eqref{eq_bound_Gx} since $D$ is arbitrary.
\end{proof}

\subsection{Proof of \Cref{lem:mG-fe-size}}

In this subsection, we prove \eqref{mG-fe-size} by keeping track of the change of the weak scaling size of the graphs from the expansion process. We will assign a subset of solid edges, denoted by $E(\cal G)$, to each graph $\cal G$ from the expansion. Moreover, we will see that these subsets can be chosen in a way such that the following quantity always decreases during our expansion process:
\be\label{eq:psizeq}\wsize(\cal G)\cdot \heta^{\frac{1}{2}|E(\cal G)|}.\ee
For the initial graph $\cal G_{\bx,\mathbf \Sigma}$, we choose $E(\cal G_{\bx,\mathbf \Sigma})$ as a subset of edges between external vertices with two edges removed. 
Roughly speaking, the subset $E({\cal G})$ evolves during the expansion process as follows. If an edge in $E(\cal G)$ is ``pulled" in some sense, we remove it from $E(\cal G)$. Conversely, in the $V$-expansion \eqref{eq_Vexp}, if a graph in \smash{$\mathcal A_{x,\fb_1\fb_2}$} satisfies \eqref{eq:sizesA} but fails to satisfy a stronger bound
$$\wsize\Big(\sum_x\wt\thn_{\fa x}\cal G'(x,\fb_1,\fb_2)\Big) \lesssim \blam\wsize(\cal G)\cdot \heta^{1+\frac{1}{2}k_{\mathsf p}(\cal G'(x,\fb_1,\fb_2))},$$
then it must contain a new solid edge within the molecule containing $x$, and we add this edge to the subset $E(\cal G)$. 
For the initial graph, we select $E(\cal G_{\bx,\mathbf \Sigma})$ to be a subset consisting of $(q-2)$ of the $q$ solid edges, while for each deterministic graph, we trivially have \smash{$E(\cal G_{\bx}^{(\mu)})=\emptyset$}. This gives the inequality  
\be\label{eq:psizeqdet}\wsize(\cal G_{\bx}^{(\mu)})\lesssim \wsize(\cal G_{\bx,\mathbf \Sigma})\cdot \heta^{\frac{1}{2}|E(\cal G_{\bx,\mathbf \Sigma})|}\lesssim \heta^{q-1}.\ee
We remark that this argument is not completely rigorous---in the following proof, we will only consider the quantity \eqref{eq:psizeq} for locally standard graphs, where the solid edges form loops. In particular, we will not specify $E(\cal G_{\bx,\mathbf \Sigma})$ for the initial graph.

\begin{remark}
Given a loop graph $\cal G$ with $q$ solid edges, we choose $E(\cal G)$ to be a subset of $(q-2)$ solid edges for the following reasons. When we continuously apply the first four terms on the RHS of \eqref{mlevelTgdef weak} to each external vertex in the loop graph, every time we replace a pair of $G$ and $\overline G$ edges with a diffusive/free edge and a $G/\overline G$ edge, the size of the graph decreases by a factor \smash{$\heta^{1/2}$}. Consequently, after the first $(q-2)$ steps, the size of the graph decreases by $\heta^{(q-2)/2}=\heta^{|E(\cal G)|/2}$. For the last two edges, they can be paired and replaced by a single diffusive edge, in which case the size of the graph remains unchanged. The following illustration depicts this process for a loop graph with $q=6$ edges: 
\begin{center}
\includegraphics[width=\textwidth]{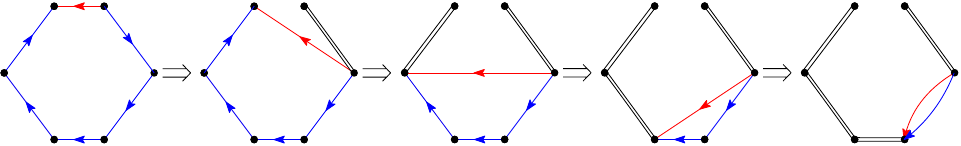}
\end{center}
Due to this observation, in the following proof, for a locally standard graph $\cal G$, we will choose $E(\cal G)$ such that 
\be\label{property_loop}
\text{each loop containing an edge in $E(\cal G)$ must include at least 2 solid edges that are not in $E(\cal G)$.} 
\ee
\end{remark}

We now look at the expansion process. We first consider the initial local expansion, where we expand the graph $\cal G_{\bx,\mathbf \Sigma}$ into a sum of $\OO(1)$ many locally regular graphs (where all solid edges are between internal vertices), $Q$-graphs, and graphs with negligible weak scaling size. 
In each locally regular graph $\cal G$, we denote by $n_a$ the number of internal vertices, $n_W$ the number of waved edges, and $t\in \qqq{1,q}$ the number of external molecules. Since two external vertices belong to the same molecule if and only if they are connected through waved edges, we can derive that 
\be\label{eq_nwna} n_W- n_a = q-t.\ee
Next, we can choose $(t-1)$ solid edges in $\cal G$ that form a spanning tree of the $t$ external molecules, and we denote the subset of these edges by $E_0(\cal G)$. We assume that these edges belong to $a$ loops of solid edges. Every such loop contains at least one solid edge that does not belong to $E_0(\cal G)$. Hence, we need to remove at most $a$ solid edges from $E_0(\cal G)$ to get $E(\cal G)$ so that property \eqref{property_loop} holds. Suppose we remove $a_0\le a$ edges from $E_0(\cal G)$. Then, using \eqref{eq_nwna} and the fact that $\cal G$ contains at least $(t-1+a)$ solid edges, we obtain  
$$\wsize(\cal G)\cdot \heta^{\frac{1}{2}|E(\cal G)|} \lesssim W^{-(q-t)d} \heta^{\frac{1}{2}(t-1+a)}\cdot \heta^{\frac{1}{2}(t-1-a_0)}\le \heta^{q-1}.$$

Now, given a locally regular graph $\cal G$ from the expansions, suppose we have chosen $E(\cal G)$ such that \smash{$\wsize(\cal G)\cdot \heta^{|E(\cal G)|/2} \lesssim \heta^{q-1}$}, and property \eqref{property_loop} holds. Then, we need to expand one of the $T$-variables in $\cal G$ according to \Cref{strat_global_weak2}. Suppose this $T$-variable is $T_{\fa,\fb_1\fb_2}$. Without loss of generality, assume that the $G$ and $\bar G$ edges (denoted by $e_1$ and $\bar e_1$) forming the $T$-variable do not belong to $E(\cal G)$. (Otherwise, we can find two solid edges $e_2$ and $e_3$ that do not belong to $E(\cal G)$ and lie in the loop containing $e_1$ and $\bar e_1$. In this case, we can remove $e_1$ and $\bar e_1$ from $E(\cal G)$ and substitute them with the edges $e_2$ and $e_3$.) We then apply the $V$-expansion in \Cref{thm:Vexp} to 
$$\cal G=\sum_{\fa,\fb_1,\fb_2}T_{\fa,\fb_1\fb_2}\Gamma(\fa,\fb_1,\fb_2)$$ 
by substituting the $\nonuni$ \eqref{mlevelTgdef weak} into $T_{\fa,\fb_1\fb_2}$. We divide the discussion into the following cases.

\medskip 
\noindent{\bf Case 1}: If we have replaced $T_{\fa,\fb_1\fb_2}$ by the first four terms on the RHS of \eqref{mlevelTgdef weak}, then in each new graph, say $\cal G_{new}$, we have a new diffusive or free edge between $\fa$ and the molecule $\cal M_{1}$ or $\cal M_2$, as well as a solid edge between between $\cal M_1$ and $\cal M_2$. Consequently, $\wsize(\cal G_{new})$ is smaller than $\wsize(\cal G)$ by at least \smash{$\heta^{1/2}$}. Moreover, the original loop (denoted by $\mathrm{L}_1$) of solid edges containing $e_1$ and $\bar e_1$ in $\cal G$ becomes a new loop (denoted by $\mathrm{L}_{new}$) in $\cal G_{new}$ with one fewer solid edge. We then remove one edge (if there is one) from $E(\cal G)$ that belongs to $\mathrm{L}_{new}$, while all the other edges in $E(G)$ remain in $E(\cal G_{new})$. This gives that 
\be\label{eq:GnewleG1}\wsize(\cal G_{new}) \heta^{\frac{1}{2}|E(\cal G_{new})|}\lesssim \wsize(\cal G) \heta^{\frac{1}{2}|E(\cal G)|}.\ee
If we have replaced $T_{\fa,\fb_1\fb_2}$ by the fifth term on the RHS of \eqref{mlevelTgdef weak}, then in each new graph $\cal G_{new}$, $\wsize(\cal G_{new})$ is smaller than $\wsize(\cal G)$ by at least \smash{$\blam\heta^2$}: $\blam$ comes from the summation of $x$ over \smash{$\wt\thn_{\fa x}$}, while the factor $\heta^2$ is due to the doubly connected property of \smash{$\cal D^{(\gamma)}_{x,\fb_1\fb_2}$}. Furthermore, the loop $\mathrm L_1$ becomes a new loop $\mathrm L_{new}$ with two fewer solid edges. Then, we remove two edges (if there are two) from $E(\cal G)$ that belong to $\mathrm L_{new}$, while all the other edges in $E(G)$ remain in $E(\cal G_{new})$. This gives that 
\be\label{eq:GnewleG2}\wsize(\cal G_{new})  \heta^{\frac{1}{2}|E(\cal G_{new})|}\lesssim \blam \heta \cdot \wsize(\cal G) \heta^{\frac{1}{2}|E(\cal G)|} \le \wsize(\cal G) \heta^{\frac{1}{2}|E(\cal G)|} .\ee

\medskip 
\noindent{\bf Case 2}: We consider the star graphs in \eqref{eq_Vexp}. We take the sum of star graphs \eqref{eq:Gpgamma} that correspond to a particular pairing $\cal P(E_p,\bar E_p)$:  
$$\cal G_{new}=\sum_{\fa,\fb_1,\fb_2} \sum_x \wt\thn_{\fa x}\cal G_{p,\gamma}^\star(x,\fb_1,\fb_2).$$
If an edge in $E(\cal G)$ is pulled in the star graph, then we remove it from $E(\cal G_{new})$; otherwise, we keep it. Using the \emph{sum zero property} \eqref{eq:vertex_sum_zero} and the estimate \eqref{eq:vertex_sum_zero_diff}, the coefficient of $\cal G_{new}$ is of order 
$$\cof(\cal G_{new})=\OO(\eta+\blam^{-1})=\OO(\blam^{-1}).$$ 
Thus, $\wsize(\cal G_{new})$ satisfies that   
$$ \wsize(\cal G_{new}) \lesssim \blam \cdot |\cof(\cal G_{new})| \heta^{\frac{1}{2}(|E(\cal G)|-|E(\cal G_{new})|)}\cdot \wsize(\cal G)= \heta^{\frac{1}{2}(|E(\cal G)|-|E(\cal G_{new})|)}\cdot \wsize(\cal G) ,$$
where $\blam$ comes from the summation of the diffusive edge $\wt\thn_{\fa x}$ over $x$, while the factor $\heta^{\frac{1}{2}(|E(\cal G)|-|E(\cal G_{new})|)}$ accounts for the decrease in the weak scaling size due to the $(|E(\cal G)|-|E(\cal G_{new})|)$ edges that have been pulled in the star graphs. Finally, consider any edge $e$ in $\cal G_{new}$ and the loop containing it. Removing $e$ from the loop creates a path between the two endpoints of $e$. This path either traverses two edges that are not in $E(\cal G)$ or passes through the center molecule $\cal M_x$ through two edges that have been pulled in the star graphs. Thus, the property \eqref{property_loop} still holds for $\cal G_{new}$.

\medskip 
\noindent{\bf Case 3}:  We consider the higher order graphs in $\sum_x \wt \thn_{\fa x}\mathcal A_{x,\fb_1\fb_2}$ from \eqref{eq_Vexp}, and they are not recollision graphs. Pick one of the graphs, denoted $\cal G_{new}$. Similar to Case 2, if an edge in $E(\cal G)$ is pulled in $\cal G_{new}$, then we remove it from $E(\cal G)$; otherwise, we keep it. This process results in a subset $E'(\cal G_{new})$. If $\wsize(\cal G_{new})$ satisfies  
\be\label{eq:smallerpsize} 
\wsize(\cal G_{new})\lesssim \blam \heta\cdot \wsize(\cal G)\heta^{\frac{1}{2}(|E(\cal G)|-|E'(\cal G_{new})|)} ,
\ee
then we set $E(\cal G_{new})=E'(\cal G_{new})$ and have that 
$$ \wsize(\cal G_{new}) \heta^{\frac{1}{2}|E(\cal G_{new})|}\le \blam \heta\cdot \wsize(\cal G)\heta^{\frac{1}{2}(|E(\cal G)|-|E(\cal G_{new})|)}\cdot \heta^{\frac{1}{2}|E(\cal G_{new})|} \le \wsize(\cal G)\heta^{\frac{1}{2}|E(\cal G)|}. $$
If \eqref{eq:smallerpsize} does not hold, then we at least have 
\be\label{eq:smallerpsize2}
\wsize(\cal G_{new})\lesssim \blam \heta^{1/2}\cdot \wsize(\cal G)\heta^{\frac{1}{2}(|E(\cal G)|-|E'(\cal G_{new})|)} 
\ee
by \eqref{eq:sizesA}. This scenario can only occur when we replace $T_{\fa,\fb_1\fb_2}$ by $\sum_x \wt\thn_{\fa x}\wt{\cal Q}_{x,\fb_1\fb_2}$, where $\wt{\cal Q}_{x,\fb_1\fb_2}$ denotes the sum of the $Q$-graphs in \eqref{eq:largeQW}, because all other $Q$-graphs have $\wsize$ smaller than $\wsize(T_{\fa,\fb_1\fb_2})$ by at least $\blam \heta$. Furthermore, if $T_{\fa,\fb_1\fb_2}$ is replaced by a $Q$-graph that is not included in \eqref{eq:largeQs}, then after Step 1 of the $Q$-expansion, each resulting graph that does not satisfy \eqref{eq:smallerpsize} either has a light weight or a vertex $x$ with non-neutral charge. In this case, we must perform at least one light weight expansion or a $GG$-expansion, which decreases the $\wsize$ of the graph by at least \smash{$\heta^{1/2}$}, ensuring that the resulting graphs still satisfy \eqref{eq:smallerpsize}. 

It remains to consider the $Q$-expansions involving the terms in \eqref{eq:largeQs}. Since we are considering the non-recollision graphs, only the first two terms in \eqref{eq:largeQs} are relevant. Following the $Q$-expansion process, we observe that a non-recollision graph $\cal G_{new}$ that does not satisfy \eqref{eq:smallerpsize} must have one (and only one) of the following structures:
\begin{center}
\includegraphics[width=10cm]{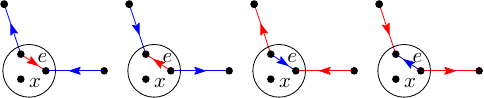}
\end{center}
In other words, there is a new solid edge $e$ within the molecule $\cal M_x$ connecting the two pulled edges of an edge in $E(\cal G)$. We then add this edge to $E'(\cal G_{new})$ and let $E(\cal G_{new})=E'(\cal G_{new})\cup\{e\}$. With \eqref{eq:smallerpsize2}, we get 
$$ \wsize(\cal G_{new}) \heta^{\frac{1}{2}|E(\cal G_{new})|}\le \blam \heta^{\frac12}\cdot \wsize(\cal G)\heta^{\frac{1}{2}(|E(\cal G)|-|E'(\cal G_{new})|)}\cdot \heta^{\frac{1}{2}|E(\cal G_{new})|} \le \wsize(\cal G)\heta^{\frac{1}{2}|E(\cal G)|}. $$
Finally, since recollisions between $\cal M_x$ and existing vertices in $\cal G$ do not occur, a similar argument as in Case 2 shows that \eqref{property_loop} holds for $\cal G_{new}$ under such a choice of $E(\cal G_{new})$.

\medskip 
\noindent{\bf Case 4}: Finally, we consider the recollision graphs in $\sum_x \wt \thn_{\fa x}\mathcal R_{x,\fb_1\fb_2}$ in \eqref{eq_Vexp}. Pick any graph $\cal G_{new}$ coming from the $Q$-expansion of a graph obtained by replacing $T_{\fa,\fb_1\fb_2}$ with a $Q$-graph in \eqref{mlevelTgdef weak}. This $Q$-graph has $\wsize$ at most $\blam \wsize(T_{\fa,\fb_1\fb_2})$. Thus, the graph before the $Q$-expansions has $\wsize$ at most $\blam\wsize(\cal G)$. We now check how the $\wsize$ of the graphs decreases with respect to the changes of $E(\cal G)$.

If an edge in $E(\cal G)$ is ``pulled" during the expansion process leading to $\cal G_{new}$, then we remove it from $E(\cal G)$; otherwise, we keep it. This process results in the subset $E(\cal G_{new})$. Here, a ``pulling" of a solid edge $e\in E(\cal G)$ refers to one of the following changes made to the edge:
\begin{itemize}
    \item[(1)] The edge becomes two edges due to the operations \eqref{resol_exp0} and \eqref{resol_reverse} in $Q$-expansions, or as a result of the partial derivatives $\partial_{\al\beta}$ in Gaussian integration by parts---these represent the ``normal pulling" of the edge. In this case, the $\wsize$ of the graph decreases by at least \smash{$\heta^{1/2}$}, compensating for the loss of the edge $e$ from $E(\cal G)$ to $E(\cal G_{new})$. 

    \item[(2)] The edge is not pulled as in case (1), but at least one of its endpoints is merged with another vertex due to the dotted edge associated with a new solid edge produced in the expansions. Denote such an endpoint by vertex $a$. We know that $a$ is standard neutral because the graph $\cal G$ is locally regular. In particular, in $\cal G$, $a$ is attached with a waved edge, the edge $e$, and another edge, say $e'$, which has an opposite charge compared to edge $e$. 
    This scenario occurs only when the edge $e'$ is pulled as in case (1) and we have replaced the new solid edge, say $e''$, attached to $a$ by a dotted edge with coefficient $m(z)$ or $\bar m(z)$. We draw the following pictures to illustrate this process:
    \begin{center}
\includegraphics[width=7cm]{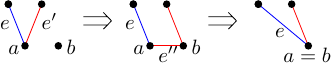}
\end{center}
Here, we did not indicate the directions of the solid edges, and we have taken $e$ to be a blue solid edge without loss of generality. 
From the first graph to the second, the edge $e'$ is pulled to a vertex $b$; from the second graph to the third, we have replaced the edge $e''$ by a dotted edge with a coefficient $\bar m(z)$ and merged the two vertices $a$ and $b$. After merging $a$ with $b$, we lose one internal vertex and one solid edge, resulting in a decrease in the $\wsize$ by \smash{$W^{-d}\heta^{-1/2}\asymp \blam^{-1}\heta^{1/2}$}. The \smash{$\heta^{1/2}$} factor compensates for the loss of the edge $e$ from $E(\cal G)$ to $E(\cal G_{new})$, and we get an additional $\blam^{-1}$ factor. 
    
\item[(3)] One solid edge (denoted by $e'$) attached to an endpoint (denoted by $a$) of the edge $e$ becomes an $S^{\pm}$ edge. This change is a result of the second term in the $GG$-expansion \eqref{Oe2x}. In this operation, we lose one solid edge while gaining an $S^{\pm}$ edge. Consequently, similar to case (2), the $\wsize$ of the graph decreases by \smash{$\blam^{-1}\heta^{1/2}$}, where the $\heta^{1/2}$ factor compensates for the loss of the edge $e$ from $E(\cal G)$ to $E(\cal G_{new})$, and we get an additional $\blam^{-1}$ factor.   

\end{itemize}

In sum, we see that if an edge in $E(G)$ has been ``pulled" according to the three cases described above, the $\wsize$ of the resulting graph decreases by at least \smash{$\heta^{1/2}$}. Furthermore, if case (2) or (3) occurs at least once, we gain an extra factor $\blam^{-1}$. Even if cases (2) and (3) do not occur for the edges in $E(\cal G)$, they must have occurred for some edges in $\cal G$ that do not belong to $E(\cal G)$, since $\cal G_{new}$ is a recollision graph. Thus, we always get an additional $\blam^{-1}$ factor, leading to the conclusion that
$$ \wsize(\cal G_{new}) \heta^{\frac{1}{2}|E(\cal G_{new})|}\le \blam \cdot \blam^{-1} \cdot \wsize(\cal G)\heta^{\frac{1}{2}(|E(\cal G)|-|E(\cal G_{new})|)}\cdot \heta^{\frac{1}{2}|E(\cal G_{new})|} = \wsize(\cal G)\heta^{\frac{1}{2}|E(\cal G)|}. $$
Finally, consider any edge $e$ in $\cal G_{new}$ and the loop containing it. In $\cal G_{new}$, one of the following cases occurs:
\begin{itemize}
    \item The loop remains untouched during the expansions and retains its original structure. 
    \item The loop contains at least two pulled edges, which do not belong to $E(\cal G_{new})$.
       

\end{itemize}
Thus, the property \eqref{property_loop} remains valid for $\cal G_{new}$. 

\medskip

Applying the expansion strategy, \Cref{strat_global_weak2}, repeatedly, we finally get a sum of $\OO(1)$ many deterministic graphs \smash{$\cal G_{\bx}^{(\mu)}$} through a sequence of $V$-expansions. From the above discussions, we see that the quantity \eqref{eq:psizeq} always decreases in our expansions for carefully chosen subsets of edges $E(\cal G)$. Hence, \eqref{eq:psizeqdet} holds, which concludes the proof of \eqref{mG-fe-size}. 




\appendix

\newpage

\section{Sum zero property for self-energies 
}\label{sec:sumzero}


In this section, we present the proof of \Cref{cancellation property} for the Wegner orbital model. We need to give a construction of $\Sele_{k_0+1}$ and prove that it satisfies the sum zero property \eqref{3rd_property0} and the properties \eqref{eq:inf_form_Ej}--\eqref{4th_property_substrac}. Following the ideas developed in \cite{BandI}, our proof is based on a comparison between the $\pself$ \smash{$\wt\Sele_{k_0+1}$} and its infinite space limit.  
 

\begin{lemma} \label{lem V-R wt}
In the setting of \Cref{cancellation property}, fix any $z=E+\ii \eta$ with $|E|\le 2-\kappa$ and $\eta_0\le \eta \le 1$. Any deterministic doubly connected graph $\cal G\equiv \cal G(m(z),S, S^{\pm}(z), \zthn(z),\{\Sele_k(z)\}_{k=1}^{k_0})$ in $\wt\Sele_{k_0+1}$ satisfies that 
\be\label{property V-R wt0}
 \left|{\cal G}_{xy}(z,W,L)-\cal G_{xy}^\infty(E,W,\infty)\right| \le \left(\eta+t_{Th}^{-1}\right)\frac{\sizeself(\cal G)}{ \langle x-y\rangle^{d-c}}, \quad \forall \ x,y\in \Z_L^d,
\ee
 \be\label{property V-R wt}
\Big|\sum_{y\in \Z_L^d} \cal G_{xy} (z,W,L) - \sum_{y\in \Z^d}  
\cal G^{\infty}_{xy}(E, W,\infty)\Big|   \le L^c \sizeself(\cal G) \left(\eta + t_{Th}^{-1}\right), \quad \forall \ x\in [0], 
\ee
for any small constant $c>0$. 
\end{lemma}

\begin{proof}
The proof of \Cref{lem V-R wt} is similar to those of \cite[Lemma 7.7]{BandI} and \cite[Lemmas 5.5 and 5.6]{BandIII}. It relies on the estimates \eqref{eq:S+-Sinf} and \eqref{Theta-wh}, \Cref{dG-bd}, the induction assumption \eqref{4th_property_substrac}, and the fact that $\cal G$ is doubly connected, with each diffusive edge being redundant. 
Hence, we omit the details of the proof.
\end{proof}

By studying the infinite space limit of $\wt\Sele_{k_0+1}$, the next lemma gives a construction of $\Sele_{k_0+1}^\infty(E,W,\infty)$, which in turn leads to a construction of $\Sele_{k_0+1}(z,W,L)$. 



\begin{lemma}\label{lem FT0}
In the setting of \Cref{cancellation property}, fix any $z=E+\ii \eta$ with $|E|\le 2-\kappa$ and $\eta_0\le \eta \le 1$. Then, any deterministic doubly connected graph $\wt{\cal G}(z,W,L)$ in $\wt\Sele_{k_0+1}$ can be expanded as 
\be\label{eq:G-expand}
\wt{\cal G}(z,W,L) = \cal G(z,W,L) + \delta \cal G(z,W,L),
\ee
such that the following properties hold.
\begin{itemize}
    \item[(a)] The infinite space limit $\cal G^\infty(E,W,\infty)$ can be written as 
   \be\label{eq:Ginf-expand}
\cal G^\infty(E,W,\infty) = \frac{\sizeself(\wt{\cal G})}{\blam} \left( \cal A_{\cal G}\otimes \mathbf E\right),
\ee
where $\cal A_{\cal G}:\ell^2(\wt\Z^d)\to \ell^2(\wt\Z^d)$ is a fixed operator that does not depend on $W$ and satisfies the properties \eqref{eq:inf_form_symm} and \eqref{eq:inf_form_decay}. Moreover, $\cal G-\cal G^\infty$ satisfies that
\be\label{G-G_property_substrac}
 \left|{\cal G}_{xy}(z,W,L)-{\cal G}^\infty_{xy}(E,W,\infty)\right| \le \left(\eta+t_{Th}^{-1}\right)\frac{\sizeself(\cal G)}{ \langle x-y\rangle^{d-c}}, \quad \forall \ x,y\in \Z_L^d,
\ee
for any small constant $c>0$. 

\item[(b)] We have \be\label{eq:remain_self_error2}
\sizeself(\delta\cal G) =\blam W^d\size(\delta\cal G) \lesssim \lambda^2 \sizeself(\cal G). 
\ee

\end{itemize}

\end{lemma}

\begin{proof}
The graph $\cal G$ consists of $m\equiv m(z,\lambda)$ and the entries of the following matrices: 
$$S , \quad S^\pm,\quad \zthn,\quad \text{and}\quad \{\Sele_{k}\}_{k=1}^{k_0}.$$
Recall that $S,$ $S^\pm$, and $\zthn$ can be written as in \eqref{eq:SSLK}, where $S_{\LK}$ is a tridiagonal matrix with  $(S_{\LK})_{[x][y]}=\delta_{[x][y]} + \lambda^2 \mathbf 1_{[x]\sim [y]}$, $S^+_{\LK} = (1-m^2 S_{\LK})^{-1},$ and $ \zthn_{\LK}$ can be expressed as 
$$ (\zthn_{\LK})_{[x][y]}=\frac{1}{1-|m|^2 S_{\LK}}- \frac{1}{n(1-|m|^2)}. $$
Then, as $L\to \infty$ and $\eta\to 0$, $m(z)$ converges to 
$$m(E)={\big(1+2d\lambda^2\big)^{-1/2}}m_{sc}\big({E}/\sqrt{1+2d\lambda^2}\big) = m_{sc}(E)+\OO(\lambda^2),$$
and the matrices $S$, $S^+$, and $\zthn$ on $\Z_L^d$ converges to the following operators on $\Z^d$:  
$$S_{\infty\to \infty}\otimes \mathbf E,\quad S^+_{\infinf}\otimes \mathbf E,\quad \thn_{\infinf}\otimes \mathbf E.$$
Here, $S_{\infty\to \infty}$, $\zthn_{\infinf}$, and $S^+_{\infinf}$ represent the following operators on $\wt\Z^d$: 
\be\label{eq:splitSSS}
\begin{split}
&    S_{\infinf} = (1+2d\lambda^2)I -  \lambda^2\wt\Delta,    
    \quad \thn_{\infinf}  = \frac{1+2d\lambda^2}{\lambda^2}{\wt\Delta}^{\inv} , \\
&     S^+_{\infinf} = \left(1-m(E)^2S_{\infinf}\right)^{-1} = \left[1-m_{sc}(E)^2\right]^{-1}\left( 1 + \delta S^+_\infinf\right),
\end{split}\ee
where we used $|m_{sc}(E)|=1$ for every $E\in [-2,2]$, $\wt\Delta$ denotes the Laplacian operator on $\wt\Z^d$:
$$\Delta_{[x][y]}=2d\delta_{[x][y]}-\lambda^2 \mathbf 1_{[x]\sim [y]}, $$
``$\inv$" stands for the pseudo-inverse, 
and the operator $\delta S^+_\infinf$ is defined as
$$\delta S^+_\infinf:=\left[m_{sc}(E_\lambda)^2 - m_{sc}(E)^2 -m_{sc}(E_\lambda)^2\frac{\lambda^2\wt\Delta}{1+2d\lambda^2} \right]\cdot S^+_\infinf,\quad E_\lambda:={E}/\sqrt{1+2d\lambda^2}.$$
Finally, recall that the infinite space limit of $\Sele_{k}$, $k\in \qqq{k_0}$, can be expressed as \eqref{eq:inf_form_Ej}. 

Due to the above observations, given $\wt{\cal G}^\infty$ written as $\wt{\cal G}^\infty=\cal B_{\wt{\cal G}}\otimes \mathbf E$ for some operator $\cal B_{\wt{\cal G}}$ on $\wt\Z^d$, we construct $\cal B_{\cal G}\otimes \mathbf E$ as follows, where $\cal B_{\cal G}$ is another operator defined on \smash{$\wt\Z^d$}: 
\begin{itemize}
    \item Replace each $S_\infinf$ in $\cal B_{\wt{\cal G}}$ by $I$.

    \item Replace each $\thn_\infinf $ in $\cal B_{\wt{\cal G}}$ by $\lambda^{-2}\wt\Delta^{\inv}=\blam \wt\Delta^{\inv}$.

    \item Replace each $S^+_\infinf$ and $S^-_\infinf$ in $\cal B_{\wt{\cal G}}$ by $(1-m_{sc}(E)^2)^{-1}I$ and $(1-\bar m_{sc}(E)^2)^{-1}I$. 

    \item Replace $m(E)$ in the coefficient by $m_{sc}(E)$. 
\end{itemize}
Correspondingly, we can define $\cal G$ from $\wt{\cal G}$ as follows: 
\begin{itemize}
    \item Replace each $S$ in $\wt{\cal G}$ by $I_n \otimes \mathbf E$ (with an error $\delta S=S-I_n \otimes \mathbf E$).

    \item Replace each $\zthn$ in $\wt{\cal G}$ by $(1+2d\lambda^2)^{-1}\zthn$ (with an error $\delta \zthn=2d\lambda^2 (1+2d\lambda^2)^{-1}\zthn$).

    \item Replace each $S^+$ and $S^-$ in $\wt{\cal G}$ by $(1-m_{sc}(z)^2)^{-1}I_n\otimes \mathbf E$ and $(1-\bar m_{sc}(z)^2)^{-1}I_n\otimes \mathbf E$ (with errors $\delta S^+=S^+-(1-m_{sc}(z)^2)^{-1}I_n\otimes \mathbf E$ and $\delta S^-=(\delta S^+)^-$). 

    \item Replace $m(z)$ in the coefficient by $m_{sc}(z)$ (with an error $m(z)-m_{sc}(z)=\OO(\lambda^2)$).   
\end{itemize}
Recall that $I_n$ and $\Delta_n$ denote the identity and Laplacian operators on $\wt\Z_n^d$, respectively. By definition, it is straightforward to check that $\cal B_{\cal G}\otimes \mathbf E$ is the infinite space limit of $\cal G$, and $\cal B_{\cal G}$ can be written as 
$$ \cal B_{\cal G} =\frac{\sizeself(\wt{\cal G})}{\blam}\cal A_{\cal G} $$
for some fixed operator $\cal A_{\cal G}$ that does not depend on $W$. 
The property \eqref{eq:inf_form_symm} for $\cal A_{\cal G}$ can be proved in the same manner as described below \eqref{eq:SSLK}, the property \eqref{eq:inf_form_decay} for $\cal A_{\cal G}$ is a consequence of \Cref{dG-bd}, and the property \eqref{G-G_property_substrac} follows from \Cref{lem V-R wt}. Finally, by the construction of $\cal G$, the scaling size of \smash{$\delta \cal G=\wt{\cal G}-\cal G$} receives an additional $\lambda^2$ factor compared to the leading term $\cal G$. Specifically, if we count the scaling sizes of $\delta S$, $\delta S^\pm$, and $\delta\zthn$ by $\lambda^2W^{-d}$, $\lambda^2W^{-d}$, and $\lambda^2\heta$, respectively, we get \eqref{eq:remain_self_error2}. This concludes the proof. 
\end{proof}

\subsection{Proof of \Cref{cancellation property}}


Given any graph $\wt{\cal G}$ in $\wt\Sele_{k_0+1}$, we have obtained its leading term $\cal G$ as in \eqref{eq:G-expand}. Summing over all these graphs $\cal G$, we get the desired $\Sele_{k_0+1}$. With \Cref{lem FT0}, we see that $\Sele_{k_0+1}$ satisfies the properties \eqref{eq:inf_form_Ej}, \eqref{eq:inf_form_symm}, \eqref{eq:inf_form_decay}, and \eqref{4th_property_substrac}. We still need to prove the key sum zero properties \eqref{3rd_property0} and \eqref{eq:inf_form_zero}. By \eqref{property V-R wt}, the sum zero property \eqref{3rd_property0} is an immediate consequence of \eqref{eq:inf_form_zero}. Hence, it remains to prove that 
\be\label{eq:sumzero_k0+1}
\sum_{[x]\in \wt\Z_n^d}\cal A_{k_0+1}([0],[x]) =0.
\ee

\begin{lemma}\label{lem maincancel}
If the setting of \Cref{cancellation property}, fix any $z = E+\ii \eta_0$ with $|E|\le 2-\kappa$ and $\eta_0=t_{Th}^{-1}$. Suppose $L$ is chosen such that 
\be\label{Lcondition} W^{c_{k_0+1} - \fc/2 + \errL} \le \frac{L^2}{W^2} \le  W^{c_{k_0+1} - \errL} \ee 
for a constant $0<\errL<\fc/4$. Recall that $\Sele_{k_0+1}'$ is obtained by replacing each $\zthn$ edge in $\Sele_{k_0+1}$ by a $\thn$ edge, as mentioned in \Cref{defn genuni}. We have that
\be\label{small SnE0}
\begin{split}
\Big| \sum_{[x]\in \wt\Z_n^d} (\Sele_{k_0+1}')^{\LK}_{[0][x]} \left(z,W, L\right) \Big|\le W^{-c_{k_0+1} -\e}/\blam 
\end{split}
\ee
for a constant $\e>0$ that depends only on $\fc$. 
\end{lemma}

In the proof of \Cref{lem:thn-zthn} in \Cref{subsect:pfthn-zthn} below (cf.~estimate \eqref{eq:G-G'}), we will show that 
$$ \sum_y\left( \Sele_{k_0+1}-\Sele_{k_0+1}'\right)_{xy}(z,W,L) \prec \sizeself(\Sele_{k_0+1})\eta_0. $$
This estimate, together with \eqref{small SnE0}, yields that:
\be\nonumber
\Big| \sum_{[x]\in \wt\Z_n^d} (\Sele_{k_0+1})^{\LK}_{[0][x]} \left(z,W, L\right) \Big|\prec W^{-c_{k_0+1} -\e}/\blam + \psi(\Sele_{k_0+1})\eta_0. 
\ee
Next, combining this estimate with \eqref{property V-R wt}, we obtain that  
\be\label{eq:smallpsi}  W^{-d}\sum_{x\in[0]}\sum_{y\in\Z^d} (\Sele_{k_0+1}^\infty)_{xy} \left(E,W, \infty\right)  \prec W^{-c_{k_0+1} -\e}/\blam+ \psi(\Sele_{k_0+1})\eta_0 \lesssim W^{-c_{k_0+1} -\e\wedge \errL}/\blam, \ee
under the condition \eqref{Lcondition}. On the other hand, recall that $\Sele_{k_0+1}^\infty$ takes the form \eqref{eq:inf_form_Ej}, which implies that 
$$W^{-d}\sum_{x\in[0]}\sum_{y\in\Z^d} (\Sele_{k_0+1}^\infty)_{xy} \left(E,W, \infty\right) = \blam^{-1}W^{-c_{k_0+1}}\cdot \sum_{[x]\in \wt\Z_n^d}\cal A_{k_0+1}([0],[x]).$$
Together with \eqref{eq:smallpsi}, it implies that 
$$ \sum_{[x]\in \wt\Z_n^d}\cal A_{k_0+1}([0],[x])\prec W^{-\e\wedge \errL}.$$
Note the LHS is a constant that does not depend on $W$, while the RHS can be arbitrarily small as $W\to \infty$. Thus, we must have \eqref{eq:sumzero_k0+1}, which completes the proof of \Cref{cancellation property}.

\begin{proof}[\bf Proof of \Cref{lem maincancel}]
First, in the setting of \Cref{cancellation property}, we already have a $\fC$-th order $T$-expansion. Furthermore, solving the pseudo-$T$-equation we have constructed as in the argument around \eqref{eq:solve-Teq-0}, we obtain a $(c_{k_0+1}-\fc)$-th order $T$-expansion. 
Then, by \Cref{locallaw-fix}, we know that the local laws \eqref{locallaw1} and \eqref{locallaw2} hold for $z=E+\ii \eta_0$ under \eqref{Lcondition}. 

Second, in the setting of \Cref{cancellation property}, we have the following $T$-equation:
 \begin{equation}
\label{mlevelT incomplete_1.1}
\begin{split}	T_{\fa,\fb_1 \fb_2}&=  m  \thn_{\fa \fb_1}\overline G_{\fb_1\fb_2} 
	+ \sum_x (\thn\Sele')_{\fa x}T_{x,\fb_1\fb_2}+ \sum_x (\thn\cdot  \delta\Sele'_{k_0+1})_{\fa x}T_{x,\fb_1\fb_2} \\
&+ \sum_x \thn_{\fa x} \left[\PT'_{x,\fb_1 \fb_2} +  \AT'_{x,\fb_1 \fb_2}  + \QT'_{x,\fb_1 \fb_2}  +  (\Err_{D}')_{x,\fb_1 \fb_2}\right]\, ,
	\end{split}
\end{equation}
which is the $T_{\fa,\fb_1 \fb_2}$ version of \eqref{mlevelT incomplete_1.0}. 
Now, by letting $\fb_1=\fb_2=\fb$ in equation \eqref{mlevelT incomplete_1.1} and taking the expectation of both sides, we obtain that 
\begin{equation}
\label{mlevelT incomplete_1.2}
\begin{split}	\E T_{\fa\fb}&=  m  \thn_{\fa \fb}\E \overline G_{\fb\fb} + \sum_{k=1}^{k_0}\sum_x (\thn\Sele_k')_{\fa x}\E T_{x\fb}+ \sum_x (\thn\Sele'_{k_0+1})_{\fa x}\E T_{x\fb} \\
&+ \sum_x \thn_{\fa x} \left[\E\PT'_{x,\fb \fb} +  \E\AT'_{x,\fb \fb} +  \E(\Err_{D}')_{x,\fb \fb}\right]\, ,
	\end{split}
\end{equation}
where, as discussed below \eqref{mlevelT incomplete_1.0}, we have included $\sum_x (\thn\cdot  \delta\Sele'_{k_0+1})_{\fa x}T_{x\fb}$ into the higher order graphs and still denote the resulting graphs by  $\sum_x \thn_{\fa x}  \AT'_{x,\fb \fb}$ with a slight abuse of notation.   

Next, we sum both sides of \eqref{mlevelT incomplete_1.2} over $\fa,\fb\in \Z_L^d$. Using Ward's identity \eqref{eq_Ward}, we get that 
\begin{align}\label{Rcancel10}
&\sum_{\fa,\fb} \mathbb E T_{\fa \fb} = \sum_{\fa, x}S_{\fa x}\cdot \mathbb E \sum_\fb |G_{x\fb}|^2 = (1+2d\lambda^2) \frac{ \sum_{x}\im  \left(\mathbb E G_{xx}\right)}{\eta_0}. 
\end{align}
For the first term on the RHS of \eqref{mlevelT incomplete_1.2}, using the definition of $\thn$, we obtain that 
\begin{align}\label{Rcancel20}
 m  \sum_{\fa,\fb}\thn_{\fa\fb} \E \overline G_{\fb\fb}  = (1+2d\lambda^2) \frac{m \sum_{\fb}\E \overline G_{\fb\fb}}{1-(1+2d\lambda^2)|m|^2}.
\end{align}
For the fourth term on the RHS of \eqref{mlevelT incomplete_1.2}, we can prove that   
\begin{align}\label{Rcancel30} 
\sum_{\fa,\fb} \sum_x \thn_{\fa x} \E\PT'_{x,\fb\fb} \prec L^d \frac{W^{-\fc}}{\eta_0}.
\end{align}
For the last two terms on the RHS of \eqref{mlevelT incomplete_1.2}, we can prove that
\begin{align}\label{Rcancel40} 
\sum_{\fa,\fb} \sum_x \thn_{\fa x}  \left[\E\AT'_{x,\fb\fb}+  \E(\Err_{D}')_{x,\fb \fb}\right]  \prec  L^d\frac{W^{-c_{k_0+1}-\fc}}{\blam\eta_0^2}.
\end{align}
The above two estimates \eqref{Rcancel30} and \eqref{Rcancel40} can be proved with the estimates in \Cref{no dot}, the local laws \eqref{locallaw1} and \eqref{locallaw2}, and the conditions \eqref{eq:smallR}, \eqref{eq:smallA} (with $\fC$ replaced by $c_{k_0+1}$), and \eqref{eq:smallRerr}. Since the proof is identical to that of Lemma 5.12 in \cite{BandI}, we omit the details. 
Finally, for the second and third terms on the RHS of \eqref{mlevelT incomplete_1.2}, we have that 
\begin{align}
\sum_{\fa,x} \left(\thn \Sele_{k}' \right)_{\fa x} \sum_{\fb}\E T_{x \fb} &=\frac{1+2d\lambda^2}{1-(1+2d\lambda^2)|m|^2}   \sum_{\al,x} \left( \Sele_{k}'\right)_{\al x}  \sum_{\fb}\E T_{x \fb} \nonumber\\
&= \frac{(1+2d\lambda^2)\sum_x \im (\E G_{xx})}{\eta_0[1-(1+2d\lambda^2)|m|^2]}\cdot   \sum_{[x]} \left( \Sele_{k}'\right)^{\LK}_{[0][x]},\label{Rcancel21}
\end{align}
where we applied Ward's identity to $\sum_x T_{x\fb}$ and used the fact that $\sum_{[x]} \left( \Sele_{k}'\right)^{\LK}_{[\al] [x]}$ does not depend on $\al$. Applying the sum zero property \eqref{3rd_property0} to $\Sele_k'$ for $k\in \qqq{k_0}$ and using $1-(1+2d\lambda^2)|m|^2\sim \eta_0= t_{Th}^{-1}$, we get 
\begin{align}\label{Rcancel22}
\sum_{\fa,x} \left(\thn \Sele_{k}' \right)_{\fa x} \sum_{\fb}\E T_{x \fb} \prec \frac{\sizeself(\Sele_k) \sum_x \im (\E G_{xx})}{\eta_0}   \lesssim W^{-\fc}\frac{ \sum_x \im (\E G_{xx})}{\eta_0} ,
\end{align}
where the second step follows from $\sizeself(\Sele_k)\lesssim \blam^2/W^d \le W^{-\fc}$. Now, plugging \eqref{Rcancel10}--\eqref{Rcancel22} into \eqref{mlevelT incomplete_1.2} and dividing both sides by $N=L^d$, we get
\begin{align}
(1+2d\lambda^2) \frac{ \im  \mathbb E \qq{G}}{\eta_0} &=  \frac{ (1+2d\lambda^2)m \mathbb E \overline {\qq{G}}}{1-(1+2d\lambda^2)|m|^2}  +  \frac{(1+2d\lambda^2)\im  \mathbb E \qq{G}}{[1-(1+2d\lambda^2)|m|^2]\eta_0}\sum_{[x]} \left( \Sele_{k_0+1}'\right)^{\LK}_{[0] [x]}\label{Rcancelsum0}\\
&+   \OO_\prec \left(  \frac{W^{-\fc}}{\eta_0} +\frac{W^{-c_{k_0+1}-\fc}}{\blam\eta_0^2} +W^{-\fc}\frac{ \im \E \qq{G} )}{\eta_0}\right) ,\nonumber
\end{align}
where $\qq{G}:=N^{-1}\tr G$. By the local law \eqref{locallaw2}, we have that
\be\label{insert Gxx}\E \qq{G} =m(z)+\OO(\heta^{1/2}) =m(z)+\OO(W^{-\fc}) .\ee 
Moreover, using the identity 
\be\nonumber \frac{|m_{sc}(z_\lambda)|^2}{1-|m_{sc}(z_\lambda)|^2}  = \frac{ \im m_{sc}(z_\lambda)}{\im z_\lambda} ,\quad z_\lambda:={z}/\sqrt{1+2d\lambda^2},\ee
we obtain that 
\be\label{insert m^2}
\frac{ \im m(z)}{\eta_0}= \frac{|m|^2}{1-(1+2d\lambda^2)|m(z)|^2}.\ee
Inserting \eqref{insert Gxx} and \eqref{insert m^2} into \eqref{Rcancelsum0}, we get that
$$ \sum_{[x]} \left( \Sele_{k_0+1}'(z)\right)^{\LK}_{[0] [x]} \prec  W^{-c_{k_0+1}-\fc}/\blam + W^{-\fc}\eta_0.$$
Together with the lower bound in condition \eqref{Lcondition}, we conclude \eqref{small SnE0} for $\e=\fc/2$. 
 \end{proof}

\section{Proof of local law: block Anderson and Anderson orbital models}\label{sec:ext}

In this section, we extend the proof of the local law for the Wegner orbital model---as developed in Sections \ref{sec:Texp}--\ref{sec:continuity} of the main text and \Cref{sec:sumzero}---to the block Anderson and Anderson orbital models.
In this setting, it is more convenient to work with the following $\cT$ variables:
\be\label{general_Tc}
\begin{split}
    & \cT_{x,yy'}:=\sum_{\al}S_{x\al}\Gc_{\al y}\Gc^-_{\al y'},\quad \cT_{yy',x}:=\sum_{\al}\Gc_{y\al}\Gc^-_{y' \al}S_{\al x},\\
    & \czT_{x,yy'}:=\sum_w P^\perp_{xw} \cT_{w,yy'}= \cT_{x,y y'} -  {N}^{-1}\sum_x \cT_{x,y y'},\quad \czT_{yy',x}:=\sum_{w}\cT_{yy',x}P^\perp_{wx}.
\end{split} 
\ee
To prepare for the proofs,  we begin by presenting the  $T$-expansions of these $T$-variables.



\subsection{Preliminary $T$-expansion}\label{sec_second_T}

In this subsection, we present the preliminary $T$-expansions for the block Anderson and Anderson orbital models. We only provide the $T$-expansions for $\cT_{\fa,\fb_1\fb_2}$ and \smash{$\czT_{\fa,\fb_1\fb_2}$}, while the corresponding $T$-expansions for $\cT_{\fb_1\fb_2,\fa}$ and \smash{$\czT_{\fb_1\fb_2,\fa}$} can be obtained by considering the transposition of indices. 
Using $G= (\lambda \Psi + V -z)^{-1}$ and the definition of $M$ in \eqref{def_G0}, we can write that 
\be\label{G-M}
M^{-1}\Gc=M^{-1}G-1=- ( V + m) G.
\ee 
Taking the partial expectation of $VG$ and applying Gaussian integration by parts to the $V$ entries, we can obtain the following counterpart of Lemmas \ref{2nd3rd T} and \ref{2ndExp}. 

\begin{lemma}\label{2ndExp0}
For the block Anderson and Anderson orbital models in \Cref{def: BM}, we have that 
\begin{equation}
	\label{eq:2ndT}
 \begin{split}
\cT_{\fa,\fb_1 \fb_2}
&	=   \sum_x (\thn M^0 S)_{\fa x }   \left(M _{x \fb_1 }\bar M_{ x \fb_2} +\Gc _{x \fb_1 }\bar M_{ x \fb_2}
	+M_{x \fb_1 }\Gc^{-} _{ x \fb_2}\right)\\
 &+\sum_x \thn_{\fa x} \left(\mathcal R^{(2)}_{x,\fb_1 \fb_2} +\mathcal A^{(>2)}_{x,\fb_1 \fb_2}  + \mathcal Q^{(2)}_{x,\fb_1 \fb_2}\right),
 \end{split}
\end{equation}
\begin{equation}
	\label{eq:2nd0}
 \begin{split}
		 \czT_{\fa,\fb_1 \fb_2}
&	=   \sum_x (\zthn M^0 S)_{\fa x }   \left(M _{x \fb_1 }\bar M_{ x \fb_2} +\Gc _{x \fb_1 }\bar M_{ x \fb_2}
	+M_{x \fb_1 }\Gc^{-} _{ x \fb_2}\right)\\
 &+\sum_x \zthn_{\fa x} \left(\mathcal R^{(2)}_{x,\fb_1 \fb_2} +\mathcal A^{(>2)}_{x,\fb_1 \fb_2}  + \mathcal Q^{(2)}_{x,\fb_1 \fb_2}\right),
 \end{split}
\end{equation}
where $\mathcal R^{(2)}_{x,\fb_1 \fb_2}$, $\mathcal A^{(>2)}_{x,\fb_1 \fb_2}$, and $\mathcal Q^{(2)}_{x,\fb_1 \fb_2}$ are defined as  
\begin{align*}
\mathcal R^{(2)}_{x,\fb_1 \fb_2}:= &\sum_{y,\beta}M_{xy} S_{y \beta}\left( \Gc^-_{xy} M_{\beta \fb_1} \bar M_{\beta \fb_2} +  \Gc^-_{xy} M_{\beta \fb_1}\Gc^-_{\beta \fb_2} + \Gc^-_{xy} \Gc_{\beta \fb_1} \bar M_{\beta \fb_2} +  \Gc_{\beta\beta} M_{y \fb_1} \Gc^- _{x\fb_2} \right)\,,\\
\mathcal A^{(>2)}_{x,\fb_1 \fb_2}:= &\sum_{y,\beta} M_{xy} S_{y \beta}\left( \Gc^-_{xy} \Gc _{\beta \fb_1}\Gc^-_{\beta \fb_2} +  \Gc_{\beta\beta} \Gc _{y \fb_1} \Gc^- _{x\fb_2} \right)\,,\\
\mathcal Q^{(2)}_{x,\fb_1 \fb_2}:=&\sum_{y}M_{xy}Q_y \left[(M^{-1}\Gc)_{y \fb_1} \Gc^-_{x\fb_2}\right]- \sum_{y,\beta} M^0_{xy}S_{y \beta}  Q_y \left( G _{\beta \fb_1}G^-_{\beta \fb_2}\right)  - \sum_{y,\beta} M_{xy}S_{y \beta}  Q_y \left(\Gc_{\beta\beta} M_{y \fb_1} \Gc^- _{x\fb_2}\right) \\
&- \sum_{y,\beta} M_{xy}S_{y \beta}  Q_y \left(\Gc^-_{xy} G _{\beta \fb_1}G^-_{\beta \fb_2} +\Gc_{\beta\beta} \Gc_{y \fb_1} \Gc^- _{x\fb_2}\right),
\end{align*}
and $M^{-1}=\lambda\Psi-z-m$ by \eqref{def_G0}.
\end{lemma}
\begin{proof}
Using \eqref{G-M}, we can write that
\begin{align*}  
	&~  \Gc _{x\fb_1}\Gc^{-}_{x\fb_2}- \sum_{\al}  M_{x\al}Q_\al\left[(M^{-1}\Gc)_{\al \fb_1} \Gc^{-}_{x\fb_2}  \right]  = - \sum_{\al,\beta}M_{x\al}P_\al\left(V_{\al\beta} G_{\beta \fb_1 }\Gc^{-}_{x\fb_2}  \right) - \sum_{\al,\beta}m M_{x\al}P_\al\left( G _{\al \fb_1 }\Gc^{-}_{x\fb_2} \right)  .
\end{align*}
Then, applying Gaussian integration by parts with respect to $V_{\al\beta}$, we obtain that 
\begin{align*}  
	&~  \Gc _{x\fb_1}\Gc^{-}_{x\fb_2}- \sum_{\al}  M_{x\al}Q_\al\left[(M^{-1}\Gc)_{\al \fb_1} \Gc^{-}_{x\fb_2}  \right]  \\
	=&~ \sum_{\al,\beta}M_{x\al}S_{\al\beta}P_\al\left(G_{\beta\beta}G_{\al \fb_1 }\Gc^{-}_{x\fb_2} \right) - \sum_{\al,\beta}m M_{x\al}P_\al\left( G_{\al \fb_1 }\Gc^{-}_{x\fb_2}\right) + \sum_{\al,\beta}M_{x\al}S_{\al\beta}P_\al\left(G_{\beta \fb_1 }G^-_{\beta \fb_2}G^-_{ x\al}\right) \\
	=&~\sum_{\al,\beta}|M_{x\al}|^2 S_{\al \beta} G _{\beta \fb_1}G^-_{\beta \fb_2}   +\sum_{\al,\beta}M_{x\al} S_{\al \beta} G _{\beta \fb_1}G^-_{\beta \fb_2} \Gc^-_{x\al} + \sum_{\al,\beta}  M_{x\al} S_{\al \beta}\Gc_{\beta\beta} G _{\al \fb_1} \Gc^- _{x\fb_2}
	\\ 
	&\quad- \sum_{\al,\beta}M_{x\al} S_{\al \beta} Q_\al\left(G _{\beta \fb_1}G^-_{\beta \fb_2} G^-_{x\al}\right) - \sum_{\al,\beta}  M_{x\al} S_{\al \beta}Q_\al\left(\Gc_{\beta\beta} G _{\al \fb_1} \Gc^{-}_{x\fb_2}\right)   .
\end{align*}
From this equation, we get that
\begin{align*}
	&\sum_\beta (1-M^0 S)_{x\beta}\Gc _{\beta \fb_1 }\Gc^- _{\beta \fb_2} =\sum_{\beta}(M^0 S)_{x \beta} \left( M_{\beta \fb_1}\bar M_{\beta \fb_2}+ \Gc_{\beta \fb_1}\bar M_{\beta \fb_2} + M_{\beta \fb_1}\Gc^-_{\beta \fb_2} \right) \\ 
	&\quad +\sum_{\al,\beta}M_{x\al} S_{\al \beta} \left(G _{\beta \fb_1}G^-_{\beta \fb_2} \Gc^-_{x\al}+ \Gc_{\beta\beta} G _{\al \fb_1} \Gc^- _{x\fb_2}\right)+ \sum_{\al}  M_{x\al}Q_\al\left[(M^{-1}\Gc)_{\al \fb_1} \Gc^{-}_{x\fb_2} \right] 
	\\ 
	&\quad  - \sum_{\al,\beta}  M_{x\al} S_{\al \beta}Q_\al\left(G _{\beta \fb_1}G^-_{\beta \fb_2} G^-_{x\al}+\Gc_{\beta\beta} G _{\al \fb_1} \Gc^{-}_{x\fb_2}\right) ,
\end{align*}
solving which gives
\begin{align}  
	 \Gc _{x \fb_1 } \Gc^{-}_{x \fb_2} 
	&=  \sum_{ \beta} [(1-M^0 S)^{-1}M^0 S]_{x \beta } \left(M _{\beta \fb_1 }\bar M_{ \beta \fb_2} + \Gc _{\beta \fb_1 } \bar M_{ \beta \fb_2}
	+M_{\beta \fb_1 }\Gc^{-} _{ \beta \fb_2}\right)\nonumber
	\\
 & +\sum_{y,\al,\beta}(1-M^0S)^{-1}_{x y} M_{y\al} S_{\al \beta}\left(  G _{\beta \fb_1 }G^-_{\beta \fb_2} \Gc^-_{y\al} +  \Gc_{\beta\beta} G _{\al \fb_1 } \Gc^- _{y\fb_2} \right)  + Q_{x,\fb_1 \fb_2},\label{eq:GGbar_exp}
\end{align}
where $Q_{x,\fb_1\fb_2}$ is a sum of $Q$-graphs defined as 
\begin{align*}
	Q_{x,\fb_1\fb_2}&=\sum_{y,\al}(1-M^0S)^{-1}_{xy}M_{y\al}Q_\al\bigg[(M^{-1}\Gc)_{\al \fb_1} \Gc^-_{y\fb_2}   - \sum_{\beta}  S_{\al\beta} \left(G _{\beta \fb_1}G^-_{\beta \fb_2} G^-_{y\al}+\Gc_{\beta\beta} G _{\al \fb_1} \Gc^- _{y\fb_2}\right)\bigg] .
\end{align*}
Multiplying \eqref{eq:GGbar_exp} with $S$ (resp.~$P^\perp S$) to the left, expanding $G$ as $G=\Gc+M$, and renaming the indices, we obtain the expansions \eqref{eq:2ndT} (resp.~\eqref{eq:2nd0}). 
\end{proof}

From \eqref{eq:2nd0}, we can further derive the preliminary $T$-expansion for the block Anderson/Anderson orbital model by performing further expansions with respect to \smash{$\Gc^-_{xy}$ and $\Gc_{\beta\beta}$} in \smash{$\mathcal A^{(>2)}_{x,\fb_1\fb_2}$}, which corresponds to \Cref{3rdExp} above for the Wegner orbital model.

\begin{lemma}[Preliminary $T$-expansion for $\BA$ and $\AO$]\label{2ndExp_AO}
For the block Anderson and Anderson orbital models, we have that
\begin{equation}
	\label{eq:2ndT_AO}
 \begin{split}
 \czT_{\fa,\fb_1 \fb_2}
	&=   \sum_x (\thn M^0 S)_{\fa x }   \left(M _{x \fb_1 }\bar M_{ x \fb_2} +\Gc _{x \fb_1 }\bar M_{ x \fb_2}
	+M_{x \fb_1 }\Gc^{-} _{ x \fb_2}\right) \\
 &+\sum_x \thn_{\fa x} \left(\mathcal R^{(3)}_{x,\fb_1 \fb_2} +\mathcal A^{(>3)}_{x,\fb_1 \fb_2}  + \mathcal Q^{(3)}_{x,\fb_1 \fb_2}\right),
 \end{split}
\end{equation}
\begin{equation}
	\label{eq:2nd_AO}
 \begin{split}
 \czT_{\fa,\fb_1 \fb_2}
	&=   \sum_x (\zthn M^0 S)_{\fa x }   \left(M _{x \fb_1 }\bar M_{ x \fb_2} +\Gc _{x \fb_1 }\bar M_{ x \fb_2}
	+M_{x \fb_1 }\Gc^{-} _{ x \fb_2}\right) \\
 &+\sum_x \zthn_{\fa x} \left(\mathcal R^{(3)}_{x,\fb_1 \fb_2} +\mathcal A^{(>3)}_{x,\fb_1 \fb_2}  + \mathcal Q^{(3)}_{x,\fb_1 \fb_2}\right),
 \end{split}
\end{equation}
where $\mathcal R^{(3)}_{x,\fb_1 \fb_2}$, $\mathcal A^{(>3)}_{x,\fb_1 \fb_2}$, and $\mathcal Q^{(3)}_{x,\fb_1 \fb_2}$ are defined as 
\begin{align*}
\mathcal R^{(3)}_{x,\fb_1 \fb_2}&:=\mathcal R^{(2)}_{x,\fb_1 \fb_2}+\sum_{y} M_{xy} \mathcal R^{(3)}_{x,y;\fb_1\fb_2} ,\quad
\mathcal A^{(>3)}_{x,\fb_1 \fb_2}:=\sum_{y} M_{xy} \mathcal A^{(>3)}_{x,y;\fb_1\fb_2} ,\\
 \mathcal Q^{(3)}_{x,\fb_1 \fb_2}&:=\mathcal Q^{(2)}_{x,\fb_1 \fb_2}  +\sum_{y}M_{xy} \mathcal Q^{(3)}_{x,y;\fb_1\fb_2} .
\end{align*}
Here, $\mathcal R^{(3)}_{x,y;\fb_1 \fb_2}$, $\mathcal A^{(>3)}_{x,y;\fb_1 \fb_2}$, and $\mathcal Q^{(3)}_{x,y;\fb_1 \fb_2}$ are defined as 
\begin{align*}
&\mathcal R^{(3)}_{x,y;\fb_1 \fb_2}:= \sum_{ \beta , \al_1,\beta_1}S_{ y \beta }\overline M_{x\al_1} S_{\al_1\beta_1} G^-_{\beta_1 y} \left(G_{ \beta  \al_1}M_{\beta_1 \fb_1} \Gc^-_{ \beta  \fb_2}  +G^-_{ \beta  \beta_1} \Gc _{ \beta  \fb_1} \overline M_{\al_1 \fb_2}  \right)\\
&\quad + \sum_{\beta,\al_1,\beta_1} S^+_{y \beta} M_{\beta\al_1}S_{\al_1\beta_1}G_{\beta_1 \beta}\left(  G_{y\beta_1}M_{\al_1\fb_1}\Gc^- _{x\fb_2}    +  G^- _{x \al_1} \Gc _{y \fb_1}\overline M_{\beta_1 \fb_2}  \right)\\
&\quad +\sum_{\beta,\al_1,\beta_1,\al_2,\beta_2}S_{ y\beta}\overline M_{x\al_1}\overline M_{\al_1 y}S^-_{\al_1 \beta_1}\overline M_{\beta_1\al_2}S_{\al_2\beta_2}G^-_{\beta_2 \beta_1}  \left( G_{\beta \al_2}M_{\beta_2 \fb_1}\Gc^-_{\beta \fb_2} + G^-_{\beta \beta_2} \Gc _{\beta \fb_1}\overline M_{\al_2 \fb_2} \right),\\
    &\mathcal A^{(>3)}_{x,y;\fb_1\fb_2}:= \sum_{\beta,\al_1,\beta_1} S_{ y\beta}\overline M_{x \al_1} S_{\al_1\beta_1} \left( \Gc^-_{\beta_1\beta_1} \Gc^-_{\al_1 y} \Gc _{\beta \fb_1}\Gc^-_{\beta \fb_2} +  G_{\beta_1 y}^- G_{\beta \al_1} \Gc_{\beta_1 \fb_1}\Gc^-_{\beta \fb_2}   +  G^-_{\beta_1y} G^-_{\beta \beta_1} \Gc _{\beta \fb_1}\Gc^-_{\al_1 \fb_2} \right)\\
    &\quad + \sum_{\beta,\al_1,\beta_1} S^+_{y \beta} M_{\beta\al_1}S_{\al_1\beta_1}\left(\Gc_{\al_1 \beta}\Gc_{\beta_1\beta_1}\Gc _{y \fb_1} \Gc^- _{x \fb_2} + G_{\beta_1 \beta} G_{y\beta_1}\Gc_{\al_1\fb_1} \Gc^- _{x \fb_2}   + G_{\beta_1 \beta}  G^- _{x  \al_1} \Gc _{y \fb_1}\Gc^-_{\beta_1 \fb_2}  \right)\\
    &\quad +\sum_{\beta,\al_1,\beta_1,\al_2,\beta_2}S_{ y\beta}\overline M_{x\al_1}\overline M_{\al_1 y}S^-_{\al_1 \beta_1}\overline M_{\beta_1\al_2}S_{\al_2\beta_2}   \\
 &\qquad \times \left(\Gc^-_{\beta_2\beta_2} 
 \Gc^-_{\al_2 \beta_1} \Gc _{\beta \fb_1}\Gc^-_{\beta \fb_2} + G^-_{\beta_2 \beta_1} G_{\beta \al_2}\Gc_{\beta_2 \fb_1}\Gc^-_{\beta \fb_2} + G^-_{\beta_2 \beta_1}  G^-_{\beta \beta_2} \Gc _{\beta \fb_1}\Gc^-_{\al_2 \fb_2} \right) ,\\
    &\mathcal Q^{(3)}_{x,y;\fb_1\fb_2}:=\sum_{\beta,\al_1} S_{y\beta}\overline M_{x\al_1}  Q_{\al_1}\left[(M^{-1}\Gc)^-_{\al_1 y}  \Gc _{\beta \fb_1}\Gc^-_{\beta \fb_2} \right] +\sum_{\beta,\al_1} S^+_{y \beta} M_{\beta\al_1} Q_{\al_1}\left[(M^{-1}\Gc)_{\al_1 \beta} \Gc _{y \fb_1} \Gc^- _{x\fb_2}\right]  \\
    &\quad +\sum_{\beta,\al_1,\beta_1,\al_2}S_{y\beta}\overline M_{x\al_1}\overline M_{\al_1y}S^-_{\al_1 \beta_1}\overline M_{\beta_1\al_2} Q_{\al_2}\left[(M^{-1}\Gc)_{\al_2 \beta_1} \Gc _{\beta \fb_1}\Gc^-_{\beta \fb_2}\right]  \\
    &\quad - \sum_{\beta,\al_1,\beta_1}S_{y\beta} \overline M_{x\al_1}\bar M_{\al_1 y} S_{\al_1\beta_1} Q_{\al_1}\left(\Gc^-_{\beta_1\beta_1}\Gc _{\beta \fb_1}\Gc^-_{\beta \fb_2}\right) -\sum_{\beta,\al_1,\beta_1}  S^+_{y \beta} M^+_{\beta\al_1}  S_{\al_1\beta_1} Q_{\al_1}\left(  \Gc_{\beta_1\beta_1} \Gc _{y \fb_1} \Gc^- _{x \fb_2}\right)\\
    &\quad  - \sum_{\beta,\al_1,\beta_1,\al_2,\beta_2}S_{y\beta}\overline M_{x\al_1}\overline M_{\al_1y}S^-_{\al_1 \beta_1} (\bar M_{\beta_1\al_2})^2 S_{\al_2\beta_2} Q_{\al_2} \left(\Gc^-_{\beta_2\beta_2} \Gc _{\beta \fb_1}\Gc^-_{\beta \fb_2} \right)\\
    &\quad  - \sum_{\beta,\al_1,\beta_1}S_{y\beta} \overline M_{x\al_1} S_{\al_1\beta_1} Q_{\al_1}\left(\Gc^-_{\beta_1\beta_1}\Gc^-_{\al_1 y} \Gc _{\beta \fb_1}\Gc^-_{\beta \fb_2}+ G^-_{\beta_1 y}  G_{\beta \al_1}G_{\beta_1 \fb_1}\Gc^-_{\beta \fb_2}   + G^-_{\beta_1 y} G^-_{\beta \beta_1} \Gc _{\beta \fb_1}G^-_{\al_1 \fb_2}  \right)\\
     &\quad -\sum_{\beta,\al_1,\beta_1}  S^+_{y \beta} M_{\beta\al_1}  S_{\al_1\beta_1} Q_{\al_1}\left(  \Gc_{\beta_1\beta_1} \Gc_{\al_1 \beta}\Gc _{y \fb_1} \Gc^- _{x \fb_2}+ G_{\beta_1\beta} G _{y \beta_1}G_{\al_1\fb_1} \Gc^- _{x \fb_2} + G_{\beta_1\beta}G^- _{x  \al_1} \Gc _{y \fb_1} G^-_{\beta_1\fb_2}\right)\\    
    &\quad - \sum_{\beta,\al_1,\beta_1,\al_2,\beta_2}S_{y\beta}\overline M_{x\al_1}\overline M_{\al_1y}S^-_{\al_1 \beta_1}\overline M_{\beta_1\al_2} S_{\al_2\beta_2} \\
    &\qquad \times Q_{\al_2}\left(\Gc^-_{\beta_2\beta_2}\Gc^-_{\al_2 \beta_1} \Gc _{\beta \fb_1}\Gc^-_{\beta \fb_2}  + G^-_{\beta_2 \beta_1}  G_{\beta \al_2}G_{\beta_2 \fb_1}\Gc^-_{\beta \fb_2}  + G^-_{\beta_2 \beta_1} G^-_{\beta \beta_2} \Gc _{\beta \fb_1}G^-_{\al_2 \fb_2} \right) .
\end{align*}
\end{lemma}

\begin{proof}
To get the expansions in \eqref{eq:2ndT_AO} and \eqref{eq:2nd_AO} from \eqref{eq:2ndT} and \eqref{eq:2nd0}, we first use the expansion in \Cref{lanlw} below to expand \smash{$\sum_{\beta} S_{y \beta}\Gc^-_{xy} \Gc _{\beta \fb_1}\Gc^-_{\beta \fb_2} $ in $\mathcal A^{(>2)}_{x,\fb_1 \fb_2}$} with respect to $\Gc^-_{xy}$. This gives some graphs in \smash{$\mathcal R^{(3)}_{x,\fb_1 \fb_2}$, $\mathcal A^{(>3)}_{x,\fb_1 \fb_2}$, and $\mathcal Q^{(3)}_{x,\fb_1 \fb_2}$}, along with the following graph: 
\be\label{eq:new_graph3} \sum_{\beta,\al_1,\beta_1}S_{y \beta}\bar M_{x\al_1} \bar M_{\al_1 y}S_{\al_1\beta_1} \Gc_{\beta_1\beta_1}^- \Gc _{\beta \fb_1}\Gc^-_{\beta \fb_2}.\ee
Next, we apply \Cref{lem_lweight} below to expand \eqref{eq:new_graph3} with respect to $\Gc_{\beta_1\beta_1}^-$ and expand $\sum_{\beta}S_{y \beta} \Gc_{\beta\beta} \Gc _{y \fb_1} \Gc^- _{x\fb_2}$ in \smash{$\mathcal A^{(>2)}_{x,\fb_1 \fb_2}$ with respect to $\Gc^-_{\beta\beta}$}. This leads to the expansions \eqref{eq:2ndT_AO} and \eqref{eq:2nd_AO}. \end{proof}

Although the expressions in \Cref{2ndExp_AO} are lengthy, in the proof, we will only use their scaling sizes and some ``simple graphical structures" of them. 
Take the expansion \eqref{eq:2nd_AO} as an example, we first observe that all new vertices generated in the expansion belong to the same molecule as $x$. The terms 
$$\sum_x (\zthn M^0 S)_{\fa x }   \left(M _{x \fb_1 }\bar M_{ x \fb_2} +\Gc _{x \fb_1 }\bar M_{ x \fb_2}
+M_{x \fb_1 }\Gc^{-} _{ x \fb_2}\right) +\sum_x \zthn_{\fa x} \mathcal R^{(3)}_{x,\fb_1 \fb_2} $$ 
consist of $\{\fb_1,\fb_2\}$-recollision graphs, where an internal vertex connects to $\fb_1$ or $\fb_2$ via a dotted edge (in the case of the block Anderson model) or a waved edge (in the case of the Anderson orbital model). 
Recall the scaling sizes defined in \Cref{def scaling}. It is easy to check that
$$\size \Big(\sum_x (\zthn M^0 S)_{\fa x } M _{x \fb_1 }\bar M_{ x \fb_2}\Big)\lesssim \heta= \size(\czT_{x,yy'}),$$
which means that this term is the leading term of the $\cal T$-expansion \eqref{eq:2nd_AO} in the sense of scaling size. All other recollision graphs have strictly smaller scaling sizes: 
$$ \size\bigg(\sum_x (\zthn M^0 S)_{\fa x }   \left(\Gc _{x \fb_1 }\bar M_{ x \fb_2}+M_{x \fb_1 }\Gc^{-} _{ x \fb_2}\right) +\sum_x \zthn_{\fa x} \mathcal R^{(3)}_{x,\fb_1 \fb_2}\bigg)\lesssim \heta^{1/2} \size(\czT_{x,yy'}).  $$
The graphs in $\mathcal A^{(>3)}_{x,yy'}$ are higher-order graphs in the sense that they have strictly smaller scaling sizes than the leading term: 
$$ \size\left( \mathcal A^{(>3)}_{x,yy'}\right)\le  \blam \heta^{2} \le W^{-\fd} \heta,$$
under the condition \eqref{eq:cond-lambda2}. Finally, the term $\mathcal Q^{(3)}_{x,\fb_1 \fb_2}$ consists of $Q$-graphs only. Moreover, notice that the following two terms in \smash{$\mathcal Q^{(3)}_{x,\fb_1 \fb_2}$} are the leading terms with largest scaling sizes: 
\begin{align}
&\mathcal Q^{a}_{x,\fb_1 \fb_2}:= \sum_{y}M_{xy}Q_y \left[(M^{-1}\Gc)_{y \fb_1} \Gc^-_{x\fb_2}\right]- \sum_{y,\beta}M^0_{xy}S_{y \beta}  Q_y \left( G _{\beta \fb_1}G^-_{\beta \fb_2}\right), \label{eq_largeQ} \end{align}
which have scaling size
\begin{align*}
&\size\bigg(\sum_{x}\zthn_{\fa x}\mathcal Q^{a}_{x,\fb_1 \fb_2} \bigg) \lesssim \blam \heta, \end{align*}
and the following eight terms are the ``sub-leading terms":
\begin{align}
&\mathcal Q^{b}_{x,\fb_1 \fb_2}:=-\sum_{y,\beta}M_{xy}S_{y \beta}  Q_y \left(\Gc^-_{xy} G _{\beta \fb_1}G^-_{\beta \fb_2} +\Gc_{\beta\beta} \Gc_{y \fb_1} \Gc^- _{x\fb_2}\right)\nonumber\\
&\quad +\sum_{y,\beta,\al_1} M_{xy}S_{y\beta}\overline M_{x\al_1}  Q_{\al_1}\left[(M^{-1}\Gc)^-_{\al_1 y}  \Gc _{\beta \fb_1}\Gc^-_{\beta \fb_2} \right]  +\sum_{y,\beta,\al_1}M_{xy} S^+_{y \beta} M_{\beta\al_1} Q_{\al_1}\left[(M^{-1}\Gc)_{\al_1 \beta} \Gc _{y \fb_1} \Gc^- _{x\fb_2}\right] \nonumber\\
&\quad +\sum_{y,\beta,\al_1,\beta_1,\al_2}M_{xy}S_{y\beta}\overline M_{x\al_1}\overline M_{\al_1y}S^-_{\al_1 \beta_1}\overline M_{\beta_1\al_2} Q_{\al_2}\left[(M^{-1}\Gc)_{\al_2 \beta_1} \Gc _{\beta \fb_1}\Gc^-_{\beta \fb_2}\right]\nonumber\\
&\quad - \sum_{y,\beta,\al_1,\beta_1}M_{xy}S_{y\beta} \overline M_{x\al_1}\bar M_{\al_1 y} S_{\al_1\beta_1} Q_{\al_1}\left(\Gc^-_{\beta_1\beta_1}\Gc _{\beta \fb_1}\Gc^-_{\beta \fb_2}\right) \nonumber\\
&\quad -\sum_{y,\beta,\al_1,\beta_1}M_{xy}  S^+_{y \beta} M^+_{\beta\al_1}  S_{\al_1\beta_1} Q_{\al_1}\left(  \Gc_{\beta_1\beta_1} \Gc _{y \fb_1} \Gc^- _{x \fb_2}\right)\nonumber\\
&\quad  - \sum_{y,\beta,\al_1,\beta_1,\al_2,\beta_2}M_{xy}S_{y\beta}\overline M_{x\al_1}\overline M_{\al_1y}S^-_{\al_1 \beta_1} (\bar M_{\beta_1\al_2})^2 S_{\al_2\beta_2} Q_{\al_2} \left(\Gc^-_{\beta_2\beta_2} \Gc _{\beta \fb_1}\Gc^-_{\beta \fb_2} \right) , \label{eq_largeQ2} 
\end{align}
which have scaling size 
\begin{align*}
&\size\bigg(\sum_{x}\zthn_{\fa x}\mathcal Q^{b}_{x,\fb_1 \fb_2} \bigg) \lesssim \blam \heta^{\frac3 2}.  
\end{align*}
All the other graphs in $\sum_x \zthn_{\fa x}\mathcal Q^{(3)}_{x,\fb_1 \fb_2}$ have scaling sizes at most $\OO(\heta^{3/2}+\blam \heta^{2})$.  

The above discussions show that after replacing $\czT_{\fa,\fb_1\fb_2}$ with the graphs on the RHS of \eqref{eq:2nd_AO}, the resulting graphs have scaling sizes of \smash{at most $\size(\czT_{\fa,\fb_1\fb_2})$}, except those obtained by replacing \smash{$\czT_{\fa,\fb_1\fb_2}$} with the $Q$-graphs in \eqref{eq_largeQ} and \eqref{eq_largeQ2}. However, after applying the $Q$-expansions (defined as in \Cref{subsec_Qexp}) to these graphs, the leading terms are still $Q$-graphs, but the non-$Q$-graphs will automatically gain at least one more solid edge, i.e., a factor of \smash{$\heta^{1/2}$}. Then, by substituting \smash{$\czT_{\fa,\fb_1\fb_2}$} with the $Q$-graphs in \eqref{eq_largeQ2} and conducting $Q$-expansions, the resulting non-$Q$-graphs will have scaling sizes of at most $\OO(\blam \heta^{2})=\OO( W^{-\fd} \heta)$. 
On the other hand, the $Q$-expansions involving the $Q$-graphs in \eqref{eq_largeQ} will lead to at least two additional solid edges in the resulting non-$Q$-graphs, as discussed below \eqref{eq:leadingQs} in the context of the Wegner orbital model. Hence, the new graphs from the $Q$ and local expansions will have scaling sizes of at most $\OO(\blam \heta^{2})=\OO( W^{-\fd} \heta)$.  


\subsection{$T$-expansion, $\pTexp$, and complete $T$-expansion}

Similar to \Cref{defn genuni,def incompgenuni,defn pseudoT} and \Cref{def nonuni-T}, we define the $T$-expansion, $T$-equation, $\pTexp$, and complete $T$-expansion for the block Anderson and Anderson orbital models as follows. 

\begin{definition} [$T$-expansion]\label{defn genuni_BA}
In the setting of \Cref{thm_locallaw}, let $z=E+\ii \eta$ with $|E|\le 2-\kappa$ and $\eta\ge \eta_0$ for some $\eta_0\in [\eta_*,1]$. For the block Anderson and Anderson orbital models, suppose we have a sequence of $\selfs$ $\cal E_k$, $k=1, \ldots, k_0$, satisfying \Cref{collection elements} with $\eta\ge \eta_0$ and the properties (i)--(iii) in \Cref{defn genuni}. Let \smash{$\Sele=\sum_{k=1}^{k_0} \Sele_k$}, and denote by $\Sele_k'$ the term obtained by replacing each \smash{$\zthn$} edge in $\Sele_k$ by a $\thn$ edge.
Then, a $T$-expansion of \smash{$\czT_{\fa,\fb_1\fb_2}$} up to order $\fC$ (with $\eta\ge \eta_0$ and $D$-th order error) is an expression of the form
\begin{equation}
	\label{mlevelTgdef_BA}
	\begin{split}
		\czT_{\fa,\fb_1 \fb_2}&=  \sum_{x} [\zthn(\Selek) M^0 S]_{\fa x}\left(M _{x \fb_1 }\bar M_{x \fb_2}
	+ \Gc _{x \fb_1 } \bar M_{ x \fb_2}	+M_{x \fb_1 }\Gc^{-} _{ x \fb_2}\right) \\
		&+ \sum_x \zthn(\Selek)_{\fa x} \left[\PT_{x,\fb_1 \fb_2} +  \AT_{x,\fb_1 \fb_2}  + \WT_{x,\fb_1 \fb_2}  + \QT_{x,\fb_1 \fb_2}  +  (\Err_{D})_{x,\fb_1 \fb_2}\right]\, ;
	\end{split}
\end{equation}
a $T$-expansion of $\cT_{\fa,\fb_1\fb_2}$ up to order $\fC$ (with $\eta\ge \eta_0\vee t_{Th}^{-1}$ and $D$-th order error) is an expression of the form
\begin{equation}
	\label{mlevelTgdef_BA_orig}
	\begin{split}
		\cT_{\fa,\fb_1 \fb_2}&=  \sum_{x} [\thn(\Selek') M^0 S]_{\fa x}\left(M _{x \fb_1 }\bar M_{x \fb_2}
	+ \Gc _{x \fb_1 } \bar M_{ x \fb_2}	+M_{x \fb_1 }\Gc^{-} _{ x \fb_2}\right) \\
		&+ \sum_x \thn(\Sele')_{\fa x} \left[\PT'_{x,\fb_1 \fb_2} +  \AT'_{x,\fb_1 \fb_2}    + \QT'_{x,\fb_1 \fb_2}  +  (\Err_{D}')_{x,\fb_1 \fb_2}\right]\, .
	\end{split}
\end{equation}
Here, $\PT',$ $  \AT'$, $ \QT'$, and $\Err_{D}'$ are obtained by replacing each $\zthn$ edge in $\PT,$ $  \AT$, $ \QT$, and $\Err_{D}$ by a $\thn$ edge. Moreover, the graphs in $\PT,$ $  \AT$, $\WT$, $ \QT$, and $\Err_{D}$ satisfy exactly the same properties as in \Cref{defn genuni}, except that the property (5) for $Q$-graphs becomes:
\begin{enumerate}
\item[(5')] For the block Anderson and Anderson orbital models, we denote the expression obtained by removing the 10 graphs in \eqref{eq_largeQ} and \eqref{eq_largeQ2} from \smash{$\cal Q_{x,\fb_1\fb_2}$ as $\cal Q'_{x,\fb_1\fb_2}$. Then, $\QT'_{\fa,\fb_1\fb_2}$} is a sum of $Q$-graphs without any free edge and satisfying \eqref{eq:smallQ}.		
\end{enumerate}
Finally, we require the graphs in $\Sele$, $\PT,$ $  \AT$, $\WT$, and $ \QT$ to satisfy the additional properties in \Cref{def genuni2}.
\end{definition}

\begin{definition}[$\incomp$]\label{def incompgenuni_BA}
In the setting of \Cref{defn genuni_BA}, a $T$-equation of $\czT_{\fa,\fb_1\fb_2}$ for the band Anderson and Anderson orbital models up to order $\fC$ (with $\eta\ge \eta_0$ and $D$-th order error) is an expression of the following form corresponding to \eqref{mlevelTgdef_BA}: 
\begin{equation}\label{mlevelT incomplete_BA}
\begin{split}
	\czT_{\fa,\fb_1\fb_2}&=  \sum_{ x} (\zthn M^0 S)_{\fa x}\left(M _{ x \fb_1}\bar M_{ x \fb_2}
	+ \Gc _{ x \fb_1} \bar M_{  x \fb_2}	+M_{ x \fb_1}\Gc^{-} _{  x \fb_2}\right) +\sum_x (\zthn \Sele)_{\fa x} \czT_{ x,\fb_1\fb_2} \\
		&+ \sum_ x \zthn_{\fa x} \left[\PT_{ x,\fb_1\fb_2} +  \AT_{ x,\fb_1\fb_2}  + \WT_{ x,\fb_1\fb_2}  + \QT_{ x,\fb_1\fb_2}  +  (\Err_{D})_{ x,\fb_1\fb_2}\right]\,,
		\end{split}
\end{equation}
where $\Sele,\, \PT,\, \AT,\, \WT,\, \QT,\, \Err_{\fC,D}$ are the same expressions as in Definition \ref{defn genuni_BA}. The $T$-equation of $\cT_{\fa,\fb_1\fb_2}$ corresponding to \eqref{mlevelTgdef_BA_orig} can be defined in a similar way. 
\end{definition}

\begin{definition} [$\PTexp$/equation]\label{defn pseudoT_BA} 
Fix constants $\fC'\ge \fC>0$ and a large constant $D>\mathfrak C'$. In the setting of \Cref{defn genuni_BA}, suppose we have a sequence of $\selfs$ $\cal E_k$, $k \in \qqq{1,k_0}$, and $\pselfs$ $\Sele_{k'}$, $k' \in \qqq{k_0+1,  k_1}$, satisfying the properties (i)--(iii) in \Cref{defn pseudoT}. Let \smash{$\Sele=\sum_{k=1}^{k_1} \Sele_k$}, and denote by $\Sele_k'$ the term obtained by replacing each \smash{$\zthn$} edge in $\Sele_k$ by a $\thn$ edge. Then, for the block Anderson and Anderson orbital models, a $\pTexp$ of \smash{$\czT_{\fa,\fb_1\fb_2}$} (resp.~$\cT_{\fa,\fb_1\fb_2}$) with real order $\fC>0$, pseudo-order $\fC'\ge \fC$, and error order $D>\fC'$ is still an expression of the form \eqref{mlevelTgdef_BA} (resp.~\eqref{mlevelTgdef_BA_orig}).
The graphs in these expansions still satisfy all the properties in \Cref{defn genuni_BA} except for the two changes described in \Cref{defn pseudoT}. 
Similarly, we can define the $\pTeq$s of real order $\fC$, pseudo-order $\fC'$, and error order $D$ for our RBSOs as in \Cref{def incompgenuni_BA}, where some $\selfs$ become $\pselfs$. 
\end{definition}

\begin{lemma}[Complete $T$-expansion]\label{def nonuni-T_BA}
Under the assumptions of Theorem \ref{thm_locallaw}, suppose the local laws in \eqref{locallaw0} hold for a fixed $z= E+\ii\eta$ with $|E|\le  2-\kappa$ and $\eta \in [\eta_0,\blam^{-1}]$ for some $W^\fd \eta_*\le \eta_0\le \blam^{-1}$. Fix any constants $D>\fC>0$ such that \eqref{Lcondition1} holds. 
Suppose that there is a $\fC$-th order $T$-expansion (with $\eta\ge \eta_0$ and $D$-th order error) satisfying Definition \ref{defn genuni_BA}. 
Then, for the block Anderson and Anderson orbital models, $T_{\fa,\fb_1 \fb_2}$ can be expanded into a sum of $\OO(1)$ many normal graphs:  
\begin{align}
T_{\fa,\fb_1 \fb_2} & = \sum_{x}S_{\fa x}\left(M_{x\fb_1}\Gc^-_{x\fb_2}+\Gc_{x\fb_1}\bar M_{x\fb_2}+M_{x\fb_1}\bar M_{x\fb_2}\right)  +  \frac{G_{\fb_2 \fb_1} - \overline G_{\fb_1 \fb_2}}{2\ii N\eta}  + \frac{1}{N}\sum_x M_{x\fb_1}\bar M_{x\fb_2} \nonumber\\
& -\frac1{N}\sum_x (M_{x\fb_1}\Gc^-_{x\fb_2}+\Gc_{x\fb_1}\bar M_{x\fb_2})+ \sum_{x}(\wt\thn M^0 S)_{\fa x}\left(M _{x \fb_1 }\bar M_{x \fb_2}+ \Gc _{x \fb_1 } \bar M_{ x \fb_2}	+M_{x \fb_1 }\Gc^{-} _{ x \fb_2}\right)  \nonumber\\ 
& +\sum_{\mu} \sum_x \wt \thn_{\fa x}\mathcal D^{(\mu)}_{x \fb_1}f_{\mu;x,\fb_1\fb_2} (G) +   \sum_\nu \sum_{x} \wt \thn_{\fa x}\mathcal D^{(\nu)}_{x \fb_2} \wt f_{\nu;x,\fb_1\fb_2}(G)\nonumber\\
& +  \sum_{\gamma} \sum_{x} \wt \thn_{\fa x}\mathcal D^{(\gamma)}_{x , \fb_1 \fb_2}g_{\gamma;x,\fb_1\fb_2}(G) + \sum_{x}\wt \thn_{\fa x} \mathcal Q_{x,\fb_1\fb_2}   + \Err_{\fa,\fb_1 \fb_2}.\label{mlevelTgdef weak_BA}
\end{align}
The graphs on the RHS of \eqref{mlevelTgdef weak_BA} also satisfy the properties (1)--(4) in \Cref{def nonuni-T}, except that the property (4) for $Q$-graphs becomes: 
\begin{enumerate}
\item[(4')] For the block Anderson and Anderson orbital models, we denote the expression obtained by removing the 10 graphs in \eqref{eq_largeQ} and \eqref{eq_largeQ2} from $\cal Q_{x,\fb_1\fb_2}$ as $\cal Q'_{x,\fb_1\fb_2}$. Then, $\QT'_{x,\fb_1\fb_2}$ is a sum of $Q$-graphs satisfying the properties (a)--(d) in \Cref{def nonuni-T}.
\end{enumerate}
\end{lemma}
 
\begin{proof}
Under the condition \eqref{Lcondition1}, the construction of the complete $T$-expansion for random band matrices using the $\fC$-th order $T$-expansion has been performed in the proof of \cite[Lemma 3.2]{BandIII}. For the block Anderson and Anderson orbital models, the construction follows exactly the same strategy (based on the expansions we will describe in \Cref{sec:graphs}). We omit the details. 
\end{proof}

Similar to \Cref{completeTexp}, we can construct a sequence of $T$-expansions up to arbitrarily high order for the block Anderson and Anderson orbital models.
\begin{theorem}[Construction of $T$-expansions]   \label{completeTexp_BA} 
In the setting of \Cref{thm_locallaw}, let $ \eta_0=W^{\fd}\eta_*$. 
For any fixed constants $D>\fC>0$, we can construct $T$-expansions for $\czT_{\fa,\fb_1 \fb_2}$ and $\cT_{\fa,\fb_1 \fb_2}$ that satisfy Definition \ref{defn genuni_BA} up to order $\fC$ (with $\eta\ge \eta_0$ and $D$-th order error).
\end{theorem}

We are now ready to outline the main modifications to the proof of the local law for the Wegner orbital model presented in Sections \ref{sec:pflocal}–\ref{sec:continuity} and \Cref{sec:sumzero}. These sections have been focused on establishing the following key components for the proof: \Cref{completeTexp}, \Cref{gvalue_continuity}, and \Cref{lemma ptree}.
For the block Anderson and Anderson orbital models, the proof of \Cref{completeTexp_BA} follows closely the argument of \Cref{completeTexp}. We will detail the necessary adjustments in \Cref{sec:graphs} (regarding the expansion strategy) and \Cref{sec:sumzero_BA} (concerning the proof of the sum zero property for $\selfs$). 
The proof of the local law in \Cref{sec:pflocal} carries over almost verbatim to the block Anderson and Anderson orbital models, with the exception that we need to establish the corresponding versions of \Cref{gvalue_continuity} and \Cref{lemma ptree}. The required modifications to their proofs will be discussed in \Cref{sec:hightree} and \Cref{subsec:appdcont}, respectively.



\subsection{Construction of the $T$-expansion}\label{sec:graphs}

Similar to the proof of \Cref{completeTexp}, the proof of \Cref{completeTexp_BA} is also divided into two parts. In the first part, we establish the following analogue of \Cref{Teq}, i.e., \Cref{Teq_BA}. The second part is dedicated to proving \Cref{cancellation property} in the context of the block Anderson and Anderson orbital models; this will be deferred to \Cref{sec:sumzero_BA} below. 
By combining these two parts, we can establish \Cref{completeTexp_BA} via induction, as discussed in \Cref{sec:constr}.

\begin{proposition}[Construction of the pseudo-$\incomp$] \label{Teq_BA}
Fix any large constants $D>\fC>0$ and $W^\fd \eta_*\le \eta_0 \le 1$. Suppose we have constructed a $T$-expansion up to order $\fC$ (with $\eta\ge \eta_0$ and $D$-th order error) that satisfies Definition \ref{defn genuni_BA} for a sequence of $\selfs$ $\Sele_k$, $k \in\qqq{k_0}$, satisfying \Cref{collection elements} and the properties (i)--(iii) in \Cref{Teq}. Then, for the block Anderson and Anderson orbital models, we can construct pseudo-$T$-equations of \smash{$\czT_{\fa,\fb_1\fb_2}$} and $\cT_{\fa,\fb_1\fb_2}$ with real order $\fC$, pseudo-order $\fC'$, and error order $D$ in the sense of \Cref{defn pseudoT_BA} for $\selfs$ $\cal E_k$, $k \in \qqq{k_0}$, and a $\pself$ \smash{$\wt\Sele_{k_0+1}$} satisfying \eqref{eq:psiSele}.
\end{proposition}

In the proof of \Cref{completeTexp_BA}, we adopt a similar expansion strategy as in \Cref{sec:constr}, with the main modifications occurring in the local expansions. Below, we present the local expansions for the block Anderson and Anderson orbital models.    


\begin{lemma}[Basic expansion]\label{lanlw}
For the block Anderson and Anderson orbital models in \Cref{def: BM}, given a graph $\Gamma$, we have that
\begin{align} \label{eq:BE}
 \Gc_{xy} \bigGamma =  &
   \sum_{\al,\beta}  M_{x\al} S_{\al \beta}
 \Gc_{\beta\beta} G _{\al y} \bigGamma
  - \sum_{\al,\beta}M_{x\al} S_{\al \beta} G _{\beta y} 
  \partial_{h_{ \beta\al}} \bigGamma   + Q,
 \end{align}
where $Q$ is a sum of $Q$-graphs defined as
\begin{align*}
Q=-\sum_{\al,\beta}M_{x\al}S_{\al\beta}  Q_{\al}\left[\Gc_{\beta \beta}G_{\al y}\Gamma\right] +\sum_{\al,\beta}M_{x\al}S_{\al\beta} Q_{\al}\left[G_{\beta y}\partial_{h_{\beta\al}}\Gamma\right]+ \sum_{\al}M_{x\al}Q_\al\left[(M^{-1}\Gc)_{\al y} \bigGamma\right].
\end{align*}
\end{lemma}

\begin{proof}
With \eqref{G-M}, we can write that 
\begin{align*}
(M^{-1}\Gc)_{xy} \bigGamma& = -  [( V+ m) G]_{xy}\Gamma   = - \sum_{\al }P_{x}\left[ V_{x\al}G_{\al y}\Gamma\right]  -m P_x\left(G_{xy}\Gamma\right) + Q_x\left[(M^{-1}\Gc)_{xy} \bigGamma\right] \\
&=  \sum_{\al} S_{x\al} P_{x}\left[ G_{\al\al}G_{x y}\Gamma\right] -\sum_{\al}S_{x\al} P_{x}\left[G_{\al y}\partial_{h_{\al x}}\Gamma \right] - m  P_{x}\left( G_{xy}\Gamma \right) +Q_x\left[(M^{-1}\Gc)_{xy} \bigGamma\right]\\
&= \sum_{\al}S_{x\al} P_{x}\left[\Gc_{\al\al}G_{xy}\Gamma\right] -\sum_{\al}S_{x\al} P_{x}\left[G_{\al y}\partial_{h_{\al x}}\Gamma\right]+Q_x\left[(M^{-1}\Gc)_{xy} \bigGamma\right].
\end{align*}
Multiplying both sides with $M$, we obtain that 
\begin{align*}
 \Gc_{xy} \bigGamma& = \sum_{\al,\beta}M_{x\al} S_{\al\beta} \Gc_{\beta\beta}G_{\al y}\Gamma  -\sum_{\al,\beta}M_{x\al} S_{\al\beta}G_{\beta y}\partial_{h_{\beta \al}}\Gamma - \sum_{\al,\beta}M_{x\al}S_{\al\beta}Q_{\al}\left[ \Gc_{\beta\beta}G_{\al y}\Gamma\right] \\
 &  + \sum_{\al,\beta}M_{x\al}S_{\al\beta} Q_{\al}\left[G_{\beta y}\partial_{h_{\beta \al}}\Gamma\right]+\sum_{\al}M_{x\al}Q_\al\left[(M^{-1}\Gc)_{\al y} \bigGamma\right],
\end{align*}
which is the expansion \eqref{eq:BE}.
\end{proof}

Take the block Anderson model as an example, we discuss the purpose of this basic expansion; a similar discussion applies to the Anderson orbital model. The expansion \eqref{eq:BE} introduces two new vertices $\al$, which is in the same atom as $x$, and $\beta$, which is in the same molecule as $x$. 
Moreover, we can classify the non-$Q$ graphs on the RHS of \eqref{eq:BE} into three cases.  
 \begin{itemize}
 	\item  Graphs with much smaller scaling sizes.
 	\item  Graphs with the same order of scaling size but strictly fewer $\Gc$ edges (i.e., the graph is more deterministic). There are two cases: 
 	\begin{itemize}
 		\item If the derivative $\partial_{h_{\beta\al}}$ acts on a $\Gc_{y'x'}$ edge in $\bigGamma$, then the relevant graph is 
 		$$ 
 		\sum_{\al,\beta}M_{x\al} S_{\al \beta} M _{\beta y} 
 		M_{y'\beta}M_{\al x'}\frac{\bigGamma}{\Gc _{y'x'}}
 		. 
 		$$
 		Note that for this graph to be non-negligible, $x$ and $x'$ must belong to the same atom, $y$ and $y'$ must belong to the same atom, and $x$ and $y$ must belong to the same molecule. 
 		
 		\item  If the derivative $\partial_{h_{\beta\al}}$ acts on a $\Gc_{x'y'}^-$ edge in $\bigGamma$, then the relevant graph is 
 		$$  \sum_{\al,\beta}M_{x\al} S_{\al \beta} 
 		M _{\beta y} 
 		\overline M_{ \beta y'}\overline M_{ x'\al} \frac{ \bigGamma }{\Gc_{ x'y'}^-}.
 		$$ 
 		Again, for this graph to be non-negligible, $x$ and $x'$ must belong to the same atom, $y$ and $y'$ must belong to the same atom, and $x$ and $y$ must belong to the same molecule. 
 	\end{itemize} 
 	\item  Graphs with the same order of scaling size and the same number of $\Gc$ edges. There are three cases:
 	\begin{itemize}
 		\item There is a new self-loop: 
 		$$  \sum_{\al,\beta}  M_{x\al} S_{\al \beta}
 		\Gc_{\beta\beta} M _{\al y} \cdot \bigGamma.$$
 		
 		\item If the derivative $\partial_{h_{\beta\al}}$ acts on a $\Gc_{y' x'}$ edge in $\bigGamma$, then we get a pair of edges of the same color: 
 		$$\sum_{\al,\beta}M_{x\al} S_{\al \beta} \Gc_{\beta y} 
 		\Gc_{y'\beta}M_{\al x'}\frac{\bigGamma}{\Gc_{y'x'}}.$$
 		
 		\item  If the derivative $\partial_{h_{\beta\al}}$ acts on a $\Gc_{x'y'}^-$ edge in $\bigGamma$, then we get a pair of edges of different colors: 
 		$$\sum_{\al,\beta}M_{x\al} S_{\al \beta}
 		\Gc_{\beta y} 
 		\Gc^{-}_{ \beta y'}
 		M^-_{ x'\al} \frac{\bigGamma}{\Gc^{-} _{ x'y'}}.
 		$$
 	\end{itemize} 
 	Given a vertex or atom, denote its degree of $\Gc$ edges by \smash{$deg_{\Gc}(\cdot)$}.  In all these three cases, the atom containing $\beta$ is a new atom and has degree 2.
 	On the other hand, the degree of \smash{$\Gc$} on the atom containing $x$ and $x'$ is reduced by 2. 
\end{itemize} 
From the above discussion, we see that performing the basic expansion repeatedly yields a sum of graphs satisfying one of the following conditions:
 \begin{itemize}
 	\item it has a sufficiently small scaling size;
 	
 	\item it is a $Q$-graph;
 	
 	\item it is a recollision graph with respect to some subset of vertices;
 	
 	\item it is deterministic;
 	
 	\item every atom $\bx$ and vertex $x$ satisfy that
 	\be\label{degA}
 	deg_{\Gc}(\bx)\in \{0,2 \},\quad deg_{\Gc}(x)\in \{0,2 \}.
 	\ee
 \end{itemize}

 
Next, we describe the self-loop expansion (corresponding to the weight expansion in \Cref{ssl} for the Wegner orbital model).

\begin{lemma} [Self-loop expansion]\label{lem_lweight}
For the block Anderson and Anderson orbital models in \Cref{def: BM}, given a graph $\Gamma$, we have that
\begin{align} \label{eq:LW}
\Gc _{xx} \Gamma  
= &\sum_{y}(1+M^+S^+)_{xy}
 \Big(   \sum_{\al,\beta}  M_{y\al} S_{\al \beta} \Gc_{ \al y} 
 \Gc_{\beta\beta}  \bigGamma
  - \sum_{\al,\beta}M_{y\al} S_{\al \beta} G _{\beta y} 
   \partial_{h_{ \beta\al}} \bigGamma\Big) + Q_1,
 \end{align}
where $Q_1$ is a sum of $Q$-graphs defined as 
\begin{align*} 
Q_1&=\sum_{y,\al}(1+M^+S^+)_{xy} M_{y\al}Q_\al\left[(M^{-1}\Gc)_{\al y}  \Gamma \right] - \sum_{y,\al ,\beta}(1+M^+S^+)_{xy} M_{y\al} S_{\al \beta} Q_\al\left[
\Gc_{\beta\beta} G_{\al y} \Gamma\right] 
\\
&+ \sum_{y,\al ,\beta}(1+M^+S^+)_{xy} M_{y\al}S_{\al \beta}  Q_\al\left[G _{\beta y} 
\partial_{h_{ \beta\al}} \Gamma\right].
\end{align*}
\end{lemma}
\begin{proof}
Using \eqref{G-M}, we can write that
\begin{align*}  
 \Gc_{xx}\Gamma=&\sum_{\al}M_{x\al}(M^{-1}\Gc) _{\al x}  \Gamma     =  - \sum_{\al,\beta}M_{x\al}P_\al \left[  V_{\al \beta}G_{\beta x}\Gamma\right] - m \sum_\al M_{x\al}P_\al (G_{\al x}\Gamma)  + \sum_{\al}M_{x\al}Q_\al\left[(M^{-1}\Gc )_{\al x}  \Gamma \right]   \\
=&
\sum_{\al,\beta} M_{x\al} S_{\al \beta} P_\al \left[
\Gc_{\beta\beta} G _{\al x} \Gamma \right]
- \sum_{\al,\beta}M_{x\al} S_{\al \beta} P_\al \left[G _{\beta x} 
\partial_{h_{ \beta\al}} \Gamma\right] + \sum_{\al}M_{x\al}Q_\al\left[(M^{-1}\Gc )_{\al x}  \Gamma \right] \\ 
= & \sum_{\al, \beta}  M_{x\al}  M_{ \al x}S_{\al \beta} P_\al \left[
\Gc_{\beta\beta}  \Gamma \right]
+ \sum_{\al ,\beta}   M_{x\al}S_{\al \beta} P_\al\left[ 
\Gc_{\beta\beta} \Gc _{\al x} \Gamma\right] 
- \sum_{\al,\beta} M_{x\al} S_{\al \beta} P_\al \left[ G _{\beta x} 
\partial_{h_{ \beta\al}} \Gamma\right]\\
&+ \sum_{\al}M_{x\al}Q_\al\left[(M^{-1}\Gc )_{\al x}  \Gamma \right].
 \end{align*}
This gives
\begin{align*} 
\sum_\al & (1-M^+ S)_{x\al}\Gc _{\al\al}  \Gamma = \sum_{\al,\beta} M_{x\al} S_{\al \beta}
\Gc_{\beta\beta} \Gc _{\al x} \Gamma - \sum_{\al,\beta}  M_{x\al} S_{\al \beta} G _{\beta x} \partial_{h_{ \beta\al}} \Gamma + \sum_{\al}M_{x\al}Q_\al\left[(M^{-1}\Gc )_{\al x}  \Gamma \right]  \\
& - \sum_{\al ,\beta}  M_{x\al} M_{ \al x} S_{\al \beta} Q_\al \left[
	\Gc_{\beta\beta}  \Gamma \right]
	- \sum_{\al ,\beta}  M_{x\al} S_{\al \beta} Q_\al\left[
	\Gc_{\beta\beta} \Gc _{\al x} \Gamma\right]  + \sum_{\al,\beta} M_{x\al} S_{\al \beta}Q_\al\left[ G _{\beta x} 
	\partial_{h_{ \beta\al}} \Gamma\right].
\end{align*}
Solving this equation, we get \eqref{eq:LW}.  \end{proof}

Again, taking the block Anderson model as an example, we continue the discussion following \eqref{degA}. 
If \smash{$\Gc _{xx}\bigGamma$} satisfies \eqref{degA}, then all non-$Q$ graphs on the RHS of \eqref{eq:LW} have strictly smaller scaling sizes than \smash{$\Gc _{xx}\bigGamma$}.  To see this, note that the first term on the RHS of \eqref{eq:LW} has a much smaller scaling size. For the second term on the RHS, suppose the derivative $\partial_{h_{\beta\al}}$ acts on a \smash{$\Gc_{y'x'}$} edge in $\bigGamma$. Then, we have that 
\be\label{eq:afterloop}
\sum_{ \al,\beta} M_{x\al} S_{\al \beta} G_{\beta x} 
G_{y'\beta}G_{\al x'}\frac{\bigGamma}{\Gc _{y'x'}} + \sum_{y,\al,\beta} (M^+S^+)_{xy}M_{y\al} S_{\al \beta} G_{\beta y} 
G_{y'\beta}G_{\al x'}\frac{\bigGamma}{\Gc _{y'x'}}
. 
\ee
We expand the $G$ entries in these terms as $G=\Gc+M$ and observe the following four cases: 
\begin{itemize}
	\item If $G_{\beta x} G_{y'\beta}G_{\al x'}$ is replaced by a product of three $\Gc$ edges in the first term, then the resulting graph has a much smaller scaling size than \smash{$\Gc_{xx} \Gc_{y'x'}$}. A similar argument applies to the second term. 
	
	\item If $G_{\beta x} G_{y'\beta}G_{\al x'}$ is replaced by a product of three $M$ edges in the first term, then we have
	\begin{align*}
		\sum_{ \al,\beta} M_{x\al} S_{\al \beta} M_{\beta x} 
		M_{y'\beta}M_{\al x'}\frac{\bigGamma}{\Gc _{y'x'}}.
	\end{align*}
For this graph to be non-negligible, $x'$ must be in the same atom as $x$, which contradicts \eqref{degA}. On the other hand, if we replace $G_{\beta y} G_{y'\beta}G_{\al x'} $ by three $M$ edges, we obtain
$$\sum_{y,\al,\beta} (M^+S^+)_{xy}M_{y\al} S_{\al \beta} M_{\beta y} 
M_{y'\beta}M_{\al x'}\frac{\bigGamma}{\Gc _{y'x'}}.$$
Then, $x'$, $y$, $y'$, $\al$, $\beta$ are all in the same atom, so this graph has a much smaller scaling size than \smash{$\Gc_{xx} \Gc_{y'x'}$} due to the two waved edges $(M^+S^+)_{xy}S_{\al \beta}$.

\item Suppose $G_{\beta x} G_{y'\beta}G_{\al x'}$ is replaced by one $\Gc$ edge and two $M$ edges in the first term of \eqref{eq:afterloop}. For example, in the graph
$$\sum_{ \al,\beta} M_{x\al} S_{\al \beta} \Gc_{\beta x} 
M_{y'\beta}M_{\al x'}\frac{\bigGamma}{\Gc _{y'x'}},$$
$x$ and $x'$ are in the same atom, which contradicts \eqref{degA}. All other cases and the second term of \eqref{eq:afterloop} can be handled similarly.

\item Suppose $G_{\beta x} G_{y'\beta}G_{\al x'}$ is replaced by a product of two $\Gc$ edges and one $M$ edge in the first term of \eqref{eq:afterloop}. It is evident that the terms 
$$\sum_{ \al,\beta} M_{x\al} S_{\al \beta} M_{\beta x} 
\Gc_{y'\beta}\Gc_{\al x'}\frac{\bigGamma}{\Gc _{y'x'}},\quad \sum_{ \al,\beta} M_{x\al} S_{\al \beta} \Gc_{\beta x} 
M_{y'\beta}\Gc_{\al x'}\frac{\bigGamma}{\Gc _{y'x'}} $$
have much smaller scaling sizes than $\Gc_{xx} \Gc_{y'x'}$ because of the extra $S_{\al \beta}$ edge. Similar arguments apply to the second term. Next, for the graph
$$ 
\sum_{ \al,\beta} M_{x\al} S_{\al \beta} \Gc_{\beta x} 
\Gc_{y'\beta}M_{\al x'}\frac{\bigGamma}{\Gc _{y'x'}}
$$
to be non-negligible, $x$ and $x'$ must be in the same atom, which contradicts \eqref{degA}. Finally, since $y$ is in the same atom as $x'$ in the graph 
$$\sum_{y,\al,\beta} (M^+S^+)_{xy}M_{y\al} S_{\al \beta} \Gc_{\beta y} 
\Gc_{y'\beta}M_{\al x'}\frac{\bigGamma}{\Gc _{y'x'}},$$
it has a much smaller scaling size than $\Gc_{xx} \Gc_{y'x'}$ because of the extra $(M^+S^+)_{xy}$ edge.
\end{itemize}
If the derivative $\partial_{h_{\beta\al}}$ has acted on a $\Gc_{x'y'}^-$ edge in $\bigGamma$, the arguments are similar.

For a pair of edges of the same color, we have the following $GG$ expansion (which corresponds to the $GG$-expansion in \Cref{T eq0} for the Wegner orbital model).

\begin{lemma}[$GG$ expansion]
For the block Anderson and Anderson orbital models in \Cref{def: BM}, given a graph $\Gamma$, we have that 
\begin{align} 
&\Gc _{y'x}\Gc _{xy}  \bigGamma 
= \sum_\beta S^+ _{x\beta } M _{\beta y}M_{y'\beta}\bigGamma
+ \sum_{ \beta}S^+_{x\beta}\left(\Gc _{\beta y}M_{y'\beta}+M_{\beta y}\Gc _{y'\beta}\right)
   \bigGamma \label{GGGamma}
 \\ & +\sum_{z,\al,\beta}\left(1+M^+S^+\right)_{xz}M_{z\al} S_{\al \beta}\left(   
 \Gc_{\beta\beta} G _{\al y} \Gc _{y'z}\bigGamma
 + \Gc _{\al z}G _{\beta y}G_{y'\beta}
   \bigGamma -  G _{\beta y} \Gc _{y'z}
     \partial_{h_{ \beta\al}}   \bigGamma\right) + Q_2,\nonumber
 \end{align}
where $Q_2$ is a sum of $Q$-graphs defined as 
 \begin{align*}
	Q_2&=\sum_{z,\al}\left(1+M^+S^+\right)_{xz}M_{z\al}Q_\al\left[(M^{-1}\Gc)_{\al y} \Gc_{y'z} \bigGamma\right]  - \sum_{z,\al,\beta}\left(1+M^+S^+\right)_{xz} M_{z\al} S_{\al \beta}Q_\al\left(\Gc_{\beta\beta} G _{\al y} \Gc _{y'z}\bigGamma\right)\\
	& - \sum_{z,\al,\beta}\left(1+M^+S^+\right)_{xz}M_{z\al} S_{\al \beta} Q_\al\left(G _{\beta y}G_{y'\beta} G _{\al z}\bigGamma\right) + \sum_{z,\al,\beta}\left(1+M^+S^+\right)_{xz}M_{z\al} S_{\al \beta} Q_\al\left(G _{\beta y} \Gc _{y'z}  \partial_{h_{ \beta\al}}  \bigGamma\right).
\end{align*}
\end{lemma}

\begin{proof}
	Using \eqref{G-M}, we can write that
\begin{align*}  
& \Gc _{y'x}\Gc _{xy} \bigGamma = \sum_{\al}  M_{x\al}(M^{-1}\Gc)_{\al y} \Gc_{y'x} \bigGamma \\
& = - \sum_{\al,\beta}M_{x\al}P_\al\left(h_{\al\beta} G_{\beta y }\Gc_{y'x} \bigGamma\right) - \sum_{\al,\beta}m M_{x\al}P_\al\left( G _{\al y }\Gc_{y'x} \bigGamma\right) + \sum_{\al}  M_{x\al}Q_\al\left[(M^{-1}\Gc)_{\al y} \Gc_{y'x} \bigGamma\right] \\
& = \sum_{\al,\beta}M_{x\al}S_{\al\beta}P_\al\left(G_{\beta\beta}G_{\al y }\Gc_{y'x} \bigGamma\right) - \sum_{\al,\beta}m M_{x\al}P_\al\left( G_{\al y }\Gc_{y'x} \bigGamma\right) + \sum_{\al,\beta}M_{x\al}P_\al\left(S_{\al\beta}G_{\beta y }G_{y'\beta}G_{\al x}\bigGamma\right) \\
&\quad - \sum_{\al,\beta}M_{x\al}P_\al\left(S_{\al\beta}G_{\beta y }\Gc_{y'x} \partial_{h_{\beta\al}}\bigGamma\right) + \sum_{\al}  M_{x\al}Q_\al\left[(M^{-1}\Gc)_{\al y} \Gc_{y'x} \bigGamma\right] \\
&=\sum_{\al,\beta}M_{x\al} M_{\al x}S_{\al \beta} G_{y'\beta} G _{\beta y} \bigGamma +\sum_{\al,\beta}M_{x\al} S_{\al \beta} G_{y'\beta} G _{\beta y}\Gc_{\al x}\bigGamma + \sum_{\al,\beta}  M_{x\al} S_{\al \beta}\Gc_{\beta\beta} G _{\al y} \Gc _{y'x}\bigGamma
\\ 
&\quad - \sum_{\al,\beta}M_{x\al} S_{\al \beta} G _{\beta y} \Gc _{y'x}  \partial_{h_{ \beta\al}}  \bigGamma + \sum_{\al}  M_{x\al}Q_\al\left[(M^{-1}\Gc)_{\al y} \Gc_{y'x} \bigGamma\right]  - \sum_{\al,\beta}  M_{x\al} S_{\al \beta}Q_\al\left(\Gc_{\beta\beta} G _{\al y} \Gc _{y'x}\bigGamma\right)\\
&\quad - \sum_{\al,\beta}M_{x\al} S_{\al \beta} Q_\al\left(G _{\beta y}G_{y'\beta} G _{\al x}\bigGamma\right)+ \sum_{\al,\beta}M_{x\al} S_{\al \beta} Q_\al\left(G _{\beta y} \Gc _{y'x}  \partial_{h_{ \beta\al}}  \bigGamma\right).
\end{align*}
From this equation, we get 
\begin{align*}
	&\sum_\beta (1-M^+S)_{x\beta}\Gc _{y'\beta}\Gc _{\beta y } \bigGamma =\sum_{\beta}(M^+S)_{x \beta} \left(\Gc_{y'\beta} M_{\beta y}+ M_{y'\beta} \Gc_{\beta y}+ M_{y'\beta} M_{\beta y}\right) \bigGamma \\ 
	&\quad +\sum_{\al,\beta}M_{x\al} S_{\al \beta} G_{y'\beta} G _{\beta y}\Gc_{\al x}\bigGamma + \sum_{\al,\beta}  M_{x\al} S_{\al \beta}\Gc_{\beta\beta} G _{\al y} \Gc _{y'x}\bigGamma - \sum_{\al,\beta}M_{x\al} S_{\al \beta} G _{\beta y} \Gc _{y'x}  \partial_{h_{ \beta\al}}  \bigGamma
	\\ 
	&\quad  + \sum_{\al}  M_{x\al}Q_\al\left[(M^{-1}\Gc)_{\al y} \Gc_{y'x} \bigGamma\right]  - \sum_{\al,\beta}  M_{x\al} S_{\al \beta}Q_\al\left(\Gc_{\beta\beta} G _{\al y} \Gc _{y'x}\bigGamma\right)\\
&\quad - \sum_{\al,\beta}M_{x\al} S_{\al \beta} Q_\al\left(G _{\beta y}G_{y'\beta} G _{\al x}\bigGamma\right)+ \sum_{\al,\beta}M_{x\al} S_{\al \beta} Q_\al\left(G _{\beta y} \Gc _{y'x}  \partial_{h_{ \beta\al}}  \bigGamma\right),
\end{align*}
solving which gives \eqref{GGGamma}. \end{proof}

Again, taking the block Anderson model as an example, with a discussion similar to that below \eqref{eq:afterloop}, we can show that if \smash{$\Gc _{y'x}\Gc _{xy}  \bigGamma$} is a graph satisfying \eqref{degA}, then, except for the first term, all other non-$Q$ graphs on the RHS of \eqref{GGGamma} have much smaller scaling sizes than \smash{$\size(\Gc _{y'x}\Gc _{xy}  \bigGamma)$}. 

With the above discussions, we can see that performing the local expansions on an arbitrary normal graph repeatedly yields a sum of $\OO(1)$ many locally standard graphs. This allows us to establish a result similar to that in \Cref{lvl1 lemma} for the block Anderson model. For the Anderson orbital model, the discussion is very similar (and actually a bit simpler). 

After defining the local expansions, the global expansion strategy is the same as that for the Wegner orbital model. First, as mentioned below \Cref{def_atom}, under our notation for atoms and molecules, the graphs for the block Anderson and Anderson orbital models have the same atomic and molecular structures as those for the Wegner orbital model (or random band matrices). Thus, the same \Cref{strat_global} applies to the block Anderson and Anderson orbital models. Specifically, every input graph is globally standard, and we find a \smash{$\cT_{x,y_1y_2}$ or $\cT_{y_1y_2,x}$} (resp.~\smash{$\czT_{x,y_1y_2}$ or $\czT_{y_1y_2,x}$}) variable that contains the first blue solid edge in a pre-deterministic order of the MIS. We then expand it with \eqref{mlevelTgdef_BA_orig} (resp.~\eqref{mlevelTgdef_BA}), and apply the $Q$-expansions if necessary. The $Q$-expansions can be defined in the same way as that in \Cref{subsec_Qexp}, while the only difference is that we will replace the expansions in \Cref{lemm:removewe} with the local expansions defined in \Cref{lanlw} and \Cref{lem_lweight}.    

Now, with a similar argument as in \Cref{sec:constr}, we can establish \Cref{Teq_BA} for the block Anderson and Anderson orbital models. Since the modifications to the argument are minor, we will omit the full details here for the sake of simplicity in presentation.

\subsection{High-moment estimates of $T$-variables}\label{sec:hightree}

We now present the proof of \Cref{lemma ptree} for the block Anderson and Anderson orbital models, which is similar to that for the Wegner orbital model in \Cref{sec:ptree}. 
It is also based on the $p$-th moment estimate of $T_{xy}$ in \eqref{locallawptree}.  
Using the definitions in \eqref{general_Tc} and $\sum_x S_{xy}=1$, we can write that
\begin{align}
T_{\fa\fb}&=\czT_{\fa\fb}+\sum_{x}S_{\fa x}\left(M_{x\fb}\Gc^-_{x\fb}+\Gc_{x\fb}\bar M_{x\fb}+M^0_{x\fb}\right)+{N}^{-1}\sum_x |\Gc_{x\fb}|^2\nonumber\\
&=\czT_{\fa\fb}+\sum_{x}S_{\fa x}\left(M_{x\fb}\Gc^-_{x\fb}+\Gc_{x\fb}\bar M_{x\fb}+M^0_{x\fb}\right) + \frac{\im G_{\fb\fb}}{N\eta} \nonumber\\
&\quad + \frac{\im m}{N(\eta+\im m)} -\frac1{N}\sum_x (M_{x\fb}\Gc^-_{x\fb}+\Gc_{x\fb}\bar M_{x\fb}) ,\label{ETab}
\end{align}
where we used \eqref{L2M} and Ward's identity \eqref{eq_Ward} in the second step.
Thus, we can write the LHS of \eqref{locallawptree} as
\begin{align*}
\E T_{\fa\fb}^p &= \E T_{\fa\fb}^{p-1}\Big( \czT_{\fa\fb}+\sum_{x}S_{\fa x}\left(M_{x\fb}\Gc^-_{x\fb}+\Gc_{x\fb}\bar M_{x\fb}+M^0_{x\fb}\right) \Big) \\
&+\E T_{\fa\fb}^{p-1}\left( \frac{\im G_{\fb\fb}}{N\eta} + \frac{\im m}{N(\eta+\im m)}-\frac1{N}\sum_x (M_{x\fb}\Gc^-_{x\fb}+\Gc_{x\fb}\bar M_{x\fb})\right).
\end{align*}
Then, we expand $\czT_{\fa\fb}$ with the $T$-expansion \eqref{mlevelTgdef_BA} and write:
\begin{align}
\E T_{\fa\fb}^p &= \E T_{\fa\fb}^{p-1}\sum_{x}S_{\fa x}\left(M_{x\fb}\Gc^-_{x\fb}+\Gc_{x\fb}\bar M_{x\fb}+M^0_{x\fb}\right)  \nonumber \\
& +\E T_{\fa\fb}^{p-1}\left( \frac{\im G_{\fb\fb}}{N\eta} + \frac{\im m}{N(\eta+\im m)} -\frac1{N}\sum_x (M_{x\fb}\Gc^-_{x\fb}+\Gc_{x\fb}\bar M_{x\fb})\right)\nonumber\\
&+ \E T_{\fa\fb}^{p-1} \sum_{x} [\zthn(\Selek) M^0 S]_{\fa x}\left(M _{x \fb_1 }\bar M_{x \fb_2}+ \Gc _{x \fb_1 } \bar M_{ x \fb_2}	+M_{x \fb_1 }\Gc^{-} _{ x \fb_2}\right)  \nonumber\\
&+\E T_{\fa\fb}^{p-1} \sum_x \zthn(\Selek)_{\fa x} \left[\PT_{x,\fb_1 \fb_2} +  \AT_{x,\fb_1 \fb_2}  + \WT_{x,\fb_1 \fb_2}  + \QT_{x,\fb_1 \fb_2}  +  (\Err_{D})_{x,\fb_1 \fb_2}\right]\, . \label{eq:Tree_BA}
\end{align} 
Using \eqref{BRB}, the condition \eqref{eq:cond-ewb}, and the estimates of $M$ in \Cref{lem:propM}, we obtain that
\begin{align*}
&\sum_{x}S_{\fa x}\left(M_{x\fb}\Gc^-_{x\fb}+\Gc_{x\fb}\bar M_{x\fb}+M^0_{x\fb}\right)+\frac{\im G_{\fb\fb}}{N\eta} +\frac{\im m}{N(\eta+\im m)}-\frac1{N}\sum_x (M_{x\fb}\Gc^-_{x\fb}+\Gc_{x\fb}\bar M_{x\fb})\\
&+ \sum_{x} [\zthn(\Selek) M^0 S]_{\fa x}\left(M _{x \fb_1 }\bar M_{x \fb_2}+ \Gc _{x \fb_1 } \bar M_{ x \fb_2}	+M_{x \fb_1 }\Gc^{-} _{ x \fb_2}\right)\prec B_{\fa\fb}+\frac{1}{N\eta}.
\end{align*}
Finally, for the terms in line \eqref{eq:Tree_BA}, they can be bounded in the same way as \eqref{estimates_ptree}--\eqref{estimates_ptreeQ} and we omit the details of the proof. This leads to the estimate \eqref{locallawptree} and completes the proof of \Cref{lemma ptree} for the block Anderson and Anderson orbital models.


\subsection{Continuity estimate}\label{subsec:appdcont}

The proof of \Cref{gvalue_continuity} for the block Anderson and Anderson orbital models is nearly identical to that in Sections \ref{sec:Vexpansion} and \ref{sec:continuity}, with one major modification regarding the $V$-expansion, \Cref{thm:Vexp}. For the block Anderson and Anderson orbital models, by using the complete $T$-expansion \eqref{mlevelTgdef weak_BA} and performing $Q$-expansions, we can obtain a similar $V$-expansion as in \eqref{eq_Vexp} (with some obvious modifications to the first two terms on the RHS). The same proof used for \Cref{thm:Vexp} shows that all properties (i)--(iii) in \Cref{thm:Vexp} still hold. In particular, the star graphs comes from replacing $T_{\fa,\fb_1\fb_2}$ with \smash{$\sum_x  \wt\thn_{\fa x} \mathcal Q^{a}_{x,\fb_1 \fb_2}$}, where $ \mathcal Q^{a}$ was defined in \eqref{eq_largeQ}, and applying the $Q$-expansions. The key molecule sum zero property \eqref{eq:vertex_sum_zero} can be proved in exactly the same way via loop graphs and the $V$-expansion of GUE. The only change in the proof lies in the proof of property \eqref{eq:vertex_sum_zero_diff}.

We first consider the Anderson orbital model. In this case, the variance matrix $S$ does not depend on $\lambda$. Moreover, we have the estimate \eqref{S+xy-0} for $S^\pm(\lambda)-S^\pm(\lambda=0)$ and $m(z)-m_{sc}(z)=\OO(\lambda^2)$ by \eqref{eq:expandm}. Thus, replacing $m(z)$ in the coefficient and the waved edges in $\fD_\omega^{\lambda}(x,\vec x(\ii),\vec x(\ff))$ with their $\lambda=0$ versions introduces an extra $\lambda^2=\blam^{-1}$ factor. The main issue arises from possible $\lambda\Psi$ and $M$ edges in $\fD_\omega^{\lambda}$. Note that replacing each $M(\lambda)$ edge with an $M(0)$ edge may only yield a $\lambda$ factor due to the off-diagonal entries of $M(\lambda)$, specifically: $M_{[x][y]}(\lambda)-M_{[x][y]}(\lambda=0)=M_{[x][y]}(\lambda)$ with $[x]\sim [y]$. To address this issue, it suffices to prove that 
\be\label{eq:even_degree}
\text{every graph in $\fD_\omega^{\lambda}$ must contain an even number of off-diagonal $\lambda\Psi$ and $M$ edges}. 
\ee
 For this purpose, we define the total degree $\deg_{\#}$ of a vertex $x$ in a graph (which may contain $G$ edges) as:
\begin{align}
\deg_{\#}(x)=\#\{\text{solid, $\Psi$, and $M$ edges}\} +2 \#\{\text{$S$ and $S^\pm$ edges}\}.
\end{align}
From the local expansions (including the $Q$-expansions) defined in \Cref{sec:graphs}, it is easy to observe that the total degrees of the internal vertices exhibit parity symmetry, meaning that $\deg_{\#}(x)$ for each internal vertex $x$ in our graphs must remain even during the expansions. This fact immediately leads to \eqref{eq:even_degree}, since $\fD_\omega^{\lambda}$ does not contain any solid edge.

For the block Anderson model, the proof of \eqref{eq:vertex_sum_zero_diff} is similar, with the following minor change: given $z=E+\ii\eta$, we choose $z_\lambda=E_\lambda+\ii\eta$ with $E_\lambda=-2\re m(E)$. In other words, $E_\lambda$ is chosen such that $\re m_{sc}(E_\lambda) = \re m(E)$, which implies $m_{sc}(z_\lambda) = m(z)+\OO(\eta)$. 
We already know that the molecule sum zero property \eqref{eq:vertex_sum_zero} holds at $z_\lambda$. Then, instead of \eqref{eq:vertex_sum_zero_diff}, it suffices to show that 
\be\label{eq:vertex_sum_zero_diff2}
\size\left[\fD_\omega^{\lambda}(x,\vec x(\ii),\vec x(\ff);z)- \fD_\omega^{\lambda=0}(x,\vec x(\ii),\vec x(\ff);z_\lambda)\right] \lesssim \blam^{-1} W^{-(2p+1)d} ,
\ee
where we have indicated the dependence of these two terms on the spectral parameter. In this case, we have $m(z)-m_{sc}(z)=\OO(\eta)=\OO(\blam^{-1})$ and the estimate \eqref{S+xy-0} for $S^\pm(\lambda)-S^\pm(\lambda=0)$. Moreover, the same argument as above shows that \eqref{eq:even_degree} holds, leading to an extra $\lambda^2$ factor. To improve this factor to $\blam^{-1}$, we notice that the $\lambda\Psi$ or $M$ edges are always associated with local block averages. By the definition of $\Psi$ and the estimate \eqref{Mbound}, we have
$$ \frac{1}{W^d}\sum_{x\in [0]}|\lambda \Psi_{xy}|^2=\OO\left( {\lambda^2}/{W}\right),\quad  \frac{1}{W^d}\sum_{x\in [0]}| M_{xy}|^2=\OO\left({\lambda^2}/{W}\right),\quad \forall y \in \Z_L^d.$$
With these bounds and the Cauchy-Schwarz inequality, we find that every off-diagonal $\lambda\Psi$ or $M$ edge essentially contributes an addition $\OO(\blam^{-1/2})$ factor; we omit the details since the argument is straightforward. This concludes the proof of \eqref{eq:vertex_sum_zero_diff2}.

\subsection{Sum zero property for self-energies} \label{sec:sumzero_BA}

Finally, as the last piece of the proof of \Cref{completeTexp_BA}, we need to establish \Cref{cancellation property} for the block Anderson and Anderson orbital models, particularly the sum zero property for the $\self$ $\Sele_{k_0+1}$. The only change in the argument from \Cref{sec:sumzero} lies in the proof of \Cref{lem FT0}. 
For this proof, we need to derive the forms of $S_\infinf$, $S_\infinf^+$, and $\thn_{\infinf}$ as in \eqref{eq:splitSSS} and extract their leading terms expressed as $W$-dependent coefficients multiplied by $W$-independent operators. 
We write $S$, $S^\pm$ and $\thn$ as tensor products 
$$ S=I_{n} \otimes \bE ,\quad \zthn=\left(\wt P^\perp\frac{1}{1-\KM}\wt P^\perp\right)\otimes \bE , \quad \Sp=\frac{1}{1-\KMp}\otimes \bE .$$
The infinite space limit of $S$ is $S_\infinf=I\otimes \mathbf E$, where $I$ is the identity operator, so we do not change $S$ edges from $\cal G$ to $\wt{\cal G}$. For $\KMp$, using \eqref{eq:expandm} and \Cref{lem:propM}, we can write  
$$\KMp = m_{sc}^2(z)+\delta \KMp(z),$$
where $\delta\KMp(z)$ is an error matrix of order $\|\delta\KMp(z)\|=\OO(\lambda^2)$. Thus, similar to the proof of \Cref{lem FT0}, we can replace each $S^+$ or $S^-$ in \smash{$\wt{\cal G}$} with $(1-m_{sc}(z)^2)^{-1}I_n\otimes \mathbf E$ or $(1-\bar m_{sc}(z)^2)^{-1}I_n\otimes \mathbf E$ when defining $\cal G$. 
It remains to deal with $(1-\KM)^{-1}$.

For $k\in \N$, define a sequence of $ \Z_n^d \times  \Z_n^d$ matrices $\J_k$ and integers $d_k$ as 
\be\label{eq:Jks}(\J_k)_{xy} = \mathbf 1_{|x-y|=k} ,\quad  d_k:=|\{x\in  \Z_n^d: \|x\|_1=k\}|.\ee
In particular, when $k=1$, we have $ d_1=2d$ and $\J_1=2dI_n-\Delta_n$. Using \eqref{eq:expandm} and the expansion \eqref{eq:expandM}, we can obtain an expansion of \smash{$\KM$} in terms of $\blam^{-1}$ as
\begin{align}
\KM =& 
m_0(z,\lambda, W)+ \blam^{-1}\left[|m_{sc}(z)|^2+ \e_1(z,\lambda, W)\right]J_1 + \sum_{k=2}^{(n-1)/2}\e_k(z)\blam^{-k}J_k \, , \label{eq:M0LK}
\end{align}
where $\e_1(z,\lambda, W)$ is a function of order $\OO(\lambda^2)$, and $\e_k(z)$ are non-negative functions of order $\OO(1)$.     
Recall that by \eqref{L2M}, we have 
$$ \sum_{[y]}(\KM)_{[x][y]}= 1-\frac{\eta}{\im m + \eta}.$$ 
Then, with \eqref{eq:M0LK}, we can write: 
\begin{align*}
1-\KM &= \frac{\eta}{\im m + \eta} +\sum_{[y]}(\KM)_{[x][y]} - \KM\\
&=\frac{\eta}{\im m + \eta} + \blam^{-1}\left[|m_{sc}(z)|^2+\e_1(z,\lambda,W)\right]\Delta_n + \sum_{k=2}^{(n-1)/2}\e_k(z)\blam^{-k}(d_k I_n-J_k).
\end{align*}
Thus, in constructing $\cal G$, we replace each $\zthn$ in $\wt{\cal G}$ by the leading term 
$$\mathring\Theta(z,W,L):=\blam \wt P^\perp\left[\frac{\blam\eta}{\im m + \eta} +  |m_{sc}(z)|^2\Delta_n\right]^{-1}\wt P^\perp \otimes \mathbf E .$$ 
As $L\to \infty$ and $\eta\to 0$, we know that $\Theta(z,W,L)$ converges to the operator $\blam \wt \Delta^{\inv}$.


Finally, we can check that the $\zthn-\mathring\Theta$ indeed introduces an error. We can express this as:
\begin{align*}
 \zthn-\mathring\Theta = -\blam^{-1} \zthn \left[\wt P^\perp \bigg(\e_1(z,\lambda,W)\Delta_n + \sum_{k=2}^{(n-1)/2}\e_k(z)\blam^{-k+1}(d_k I_n-J_k) \bigg)\wt P^\perp\otimes \mathbf E \right]   \mathring\Theta \, .
\end{align*}
With the same argument as in Section \ref{sec:Fourier} below, we can show, via Fourier series, that   
$$\zthn_{xy}-\mathring\Theta_{xy} \prec \lambda^2 B_{xy},\quad \forall x,y\in \Z_L^d.$$
Hence, $\zthn-\mathring\Theta$ behaves like a diffusive edge,  while introducing an addition $\lambda^2$ factor to the graph. 



\section{Proof of quantum diffusion}\label{sec:pf_Qdiff}

The proof of \Cref{thm_diffu} is based on the $T$-expansion constructed in \Cref{completeTexp} (for the Wegner orbital model) or \Cref{completeTexp_BA} (for the block Anderson and Anderson orbital models), the local laws in \Cref{thm_locallaw}, and the estimate \eqref{bound 2net1 strong} in \Cref{no dot}. 

For the Wegner orbital model, we have a $T$-expansion \eqref{mlevelTgdef} up to arbitrarily high order $\fC$. Setting $\fb_1=\fb_2=\fb$ in \eqref{mlevelTgdef}, taking expectation, and using the identity \eqref{eq:T-cT}, we get that
\begin{align}
		\E T_{\fa \fb}&=  m  \zthn_{\fa \fb}(\Sele)\E \overline G_{\fb\fb} + \frac{ \im \E G_{\fb \fb}}{N\eta} + \sum_x \zthn_{\fa x} (\Selek)\E\left[\PT_{x,\fb \fb} +  \AT_{x,\fb \fb}  + \WT_{x,\fb \fb}  +  (\Err_{D})_{x,\fb \fb}\right]\nonumber\\
  &= \zthn_{\fa \fb}(\Sele)  \left[|m|^2 +m(\E\bar G_{\fb\fb}-\bar m)\right] + \sum_x \zthn_{\fa x} (\Selek)\E \PT_{x,\fb \fb}  + \frac{ \im \E G_{\fb\fb} + \sum_x \zthn_{\fa x} (\Selek)\E  \WT'_{x,\fb \fb}}{N\eta}  \nonumber\\
  &\quad  + \sum_x \zthn_{\fa x} (\Selek)\E  \left[ \AT_{x,\fb \fb} +(\Err_{D})_{x,\fb \fb} \right]  \label{mlevelTgdef_QD} \, ,
\end{align}
where we recall that $\WT'_{x,\fb \fb}$ is obtained by removing the free edge from $\WT_{x,\fb \fb}$. By definition, we have \smash{$\zthn (\Sele)S(0)=\zthn (\Sele)$}, where recall that $S(0)$ was defined in \eqref{eq:SBA}. 
Then, we define $\cal G$ as \be\label{eq:calGxb}\cal G_{x\fb} 
:= \sum_y S(0)_{xy} \E \PT_{y,\fb \fb} + m(\E\bar G_{\fb\fb}-\bar m) S(0)_{x\fb} .\ee 
For the bound \eqref{eq:boundGaly}, we first have that 
\be\label{eq:EG-m}m(\E\bar G_{\fb\fb}-\bar m) S(0)_{x\fb} \prec \heta^{1/2}  S(0)_{x\fb} \prec \heta^{1/2} \exp\left( -|x-\fb|/W\right)\ee
by the local law \eqref{locallawmax} for $\bar G_{\fb\fb}-\bar m$. Next, we notice that the graphs in $\PT_{x,\fb \fb}$ are doubly connected since they are $\fb$-recollision graphs so that $\fb$ belongs to the same molecule as some internal vertex. Then, using \eqref{bound 2net1 strong} (recall that $\peta$ is equivalent to $\heta$ when $\eta\ge \eta_*$), we obtain that 
\be\label{eq:ERx} \sum_y S(0)_{xy}\E \PT_{y,\fb \fb} \prec \sum_y S(0)_{xy} \size(\PT_{y,\fb \fb}) \frac{W^{d-2}}{\langle y-\fb\rangle^{d-2}} \frac{\cal A_{y\fb}}{\heta^{1/2}},\ee
where $\cal A_{xy}$ are non-negative random variables satisfying $\|\cal A(z)\|_{s}\prec 1$, and 
$$ \size(\PT_{y,\fb \fb})\le \blam^{-1}\size\Big(\sum_{x} \zthn_{\fa x}(\Selek)\PT_{x,\fb_1\fb_2}\Big) \lesssim \blam^{-1}W^{-\fd}\heta $$
by the condition \eqref{eq:smallR}. Plugging it into \eqref{eq:ERx} and using the definition \eqref{def_strongA} for $\|\cal A(z)\|_s$, we get that 
\begin{align}
\sum_y S(0)_{xy}&\E \PT_{y,\fb \fb}  \prec \frac{W^{-\fd}\heta^{1/2}}{\blam}   \frac{W^{d-2}}{\langle x-\fb\rangle^{d-2}} \E\sum_y S(0)_{xy}\cal A_{y\fb}\prec \frac{W^{-\fd}\heta^{1/2}}{\blam}   \frac{W^{d-2}}{\langle x-\fb\rangle^{d-2}}  \left(B_{x\fb}+\frac{1}{N\eta}\right)^{1/2}  \label{eq:ES0R} \\
    &\le  \frac{W^{-\fd}\heta^{1/2}}{\blam} \cdot  \frac{W^{d}}{\langle x-\fb\rangle^{d}}  \left(B_{x\fb}\frac{\langle x-\fb\rangle^{4}}{W^4}+\frac{1}{N\eta}\frac{L^4}{W^4}\right)^{1/2} \le  \frac{W^{-\fd}\heta}{\blam W^{-d}} \cdot  \frac{1}{\langle x-\fb\rangle^{d}}  =\frac{W^{-\fd}}{\langle x-\fb\rangle^{d}}\nonumber
\end{align}
for $\eta\ge W^\fd\eta_*$. Together with \eqref{eq:EG-m}, this gives the estimate \eqref{eq:boundGaly} for $\cal G$ defined in \eqref{eq:calGxb}. 

Next, we prove that 
\be\label{eq:ThetaW}
\sum_x \zthn_{\fa x} (\Selek) \WT'_{x,\fb \fb} \prec \size\Big(\sum_{x} \zthn_{\fa x}(\Selek)\WT'_{x,\fb_1\fb_2}\Big) \lesssim  W^{-\fd} \, . 
\ee
Consider any graph $\cal G'_{x\fb}$ in $\WT'_{x,\fb \fb}$. If it is a $\fb$-recollision graph, then we can bound it in the same way as the graphs in $\PT_{y,\fb \fb}$. Otherwise, it can be expressed as
\be \label{Ahoforms_add}
\cal G'_{x\fb} = \sum_{y, y'}(\cal G_0)_{xyy'} G_{y\fb}\overline G_{y'\fb}, \quad  \quad  \cal G'_{x\fb}= \sum_{y} (\cal G_0)_{xy}\Theta_{y\fb},\ee
or in other forms obtained by setting some indices among $x,y,y'$ to be equal to each other.
By property (7) of Definition \ref{defn genuni}, the graph $\cal G_0$ is doubly connected (with $x,y,y'$ regarded as internal vertices). Without loss of generality, it suffices to consider the first form in \eqref{Ahoforms_add} with $y=y'$, i.e., we will show that for $\cal G'_{x\fb}= \sum_{y} (\cal G_0)_{xy}|G_{y\fb}|^2$, 
\be\label{eq:ThetaW2}
\sum_{x,y} \zthn(\Sele)_{\fa x}  \E \left[(\cal G_0^{{\mathrm{abs}}})_{xy}|G_{y\fb}|^2\right] \prec \size\Big(\sum_{x} \zthn_{\fa x}(\Selek)\cal G'_{x\fb}\Big)\, .
\ee
The second expression in \eqref{Ahoforms_add} is easier to estimate, while the first expression can be bounded by combining \eqref{eq:ThetaW2} with the Cauchy-Schwarz inequality. To prove \eqref{eq:ThetaW2}, we use \eqref{BRB} for \smash{$\zthn(\Sele)$} and \eqref{bound 2net1 strong} for $(\cal G_0^{{\mathrm{abs}}})_{xy}$ to get that
\begin{align*}
&\sum_{x,y} \zthn(\Sele)_{\fa x}  \E \left[(\cal G_0^{{\mathrm{abs}}})_{xy}|G_{y\fb}|^2\right]=\sum_{\al, x,y} \zthn(\Sele)_{\fa \al}S(0)_{\al x}  \E \left[(\cal G_0^{{\mathrm{abs}}})_{xy}|G_{y\fb}|^2\right] \\
&\qquad  \prec   \sum_{\al, x,y}B_{\fa \al}S(0)_{\al x} {\size\left[(\cal G_0)_{xy}\right]} \frac{W^{d-2}}{\langle x-y\rangle^{d-2}} \heta^{-1/2} \E  \left[\cal A_{xy}  \left|G_{y\fb}\right|^2\right] \\
&\qquad \prec  \frac{a_W}{\blam\heta^{3/2}}\sum_{\al,y}B_{\fa \al} \frac{1}{W^2\qq{\al-y}^{d-2}} \E  \sum_{x}S(0)_{\al x}\cal A_{xy}|G_{y\fb}|^2\\
&\qquad \prec  \frac{a_W}{\blam\heta}\sum_{\al,y}B_{\fa \al} \frac{1}{\qq{\al-y}^{d}} \E |G_{y\fb}|^2 \prec \frac{a_W}{\blam\heta}\sum_{y}B_{\fa y} \E|G_{y\fb}|^2\\
&\qquad \lesssim \frac{a_W}{\blam\heta}\sum_{[y_0]\in \wt\Z_n^d}B_{\fa y_0} \E  \sum_{y\in [y_0]} |G_{y\fb}|^2\lesssim \frac{a_W}{\blam\heta}\sum_{y\in \Z_L^d}B_{\fa y} \E T_{y\fb} \\
&\qquad \prec  \frac{a_W}{\blam\heta} \sum_{y\in \wt\Z_n^d}B_{\fa y} \left(B_{y\fb}+\frac{1}{N\eta}\right) \prec \frac{a_W}{\heta} \left( \frac{\blam}{W^d} + \frac{1}{N\eta}\frac{L^2}{W^2}\right)\lesssim a_W,
\end{align*}
where we abbreviated $a_W:=\size(\sum_{x} \zthn_{\fa x}(\Selek)\cal G'_{x\fb})$. In the third step, we used that $\size\left[(\cal G_0)_{xy}\right]= a_W/\blam^2$ and $\qq{\al-y}\prec \qq{x-y}$ for $\al$ and $x$ in the same molecule;   
in the fourth step, we applied a similar argument as in \eqref{eq:ES0R} to get that 
$$\frac{1}{W^2\qq{\al-y}^{d-2}}  \sum_{x}S(0)_{\al x}\cal A_{xy} \prec \frac{1}{W^2\qq{\al-y}^{d-2}}\left(B_{\al y}+\frac{1}{N\eta}\right)^{1/2}  \lesssim \frac{\heta^{1/2}}{\qq{\al-y}^{d}};$$
in the fifth step, we used that $\sum_\al B_{\fa \al} \qq{\al-y}^{-d}\prec B_{\fa y}$; in the sixth step, we decomposed the sum according to the blocks \smash{$[y_0]\in \wt\Z_n^d$}, where $y_0$ denotes their centers, and used that $B_{\fa y_0}\asymp B_{\fa y}$ for $y\in [y_0]$; in the eighth step, we applied the local law \eqref{locallaw} for $T_{x\fb}$. This concludes \eqref{eq:ThetaW2}, thereby completing the proof of \eqref{eq:ThetaW}.

Finally, with the estimate \eqref{eq:ThetaW2}, using the size conditions \eqref{eq:smallA} and \eqref{eq:smallRerr}, we also obtain that 
\begin{align*}
    \sum_x \zthn_{\fa x} (\Selek)\E  \left[ \AT_{x,\fb \fb} +(\Err_{D})_{x,\fb \fb} \right] \prec W^{-\fC}.
\end{align*}
Thus, this term is an error as long as we take $\fC>D$. This concludes \eqref{eq:Qdiff} for the Wegner orbital model. 

The proof of \eqref{eq:Qdiff} is similar for the block Anderson and Anderson orbital models except that we use the expansion \eqref{mlevelTgdef_BA} and the identity \eqref{ETab} instead. 
By \Cref{lem:propM}, we have that 
$\sum_x|M_{x\fb}|=\OO(1)$. Combined with the local law \eqref{locallawmax} for $\Gc$, it implies that 
$$\frac1{N}\sum_x (M_{x\fb}\Gc^-_{x\fb}+\Gc_{x\fb}\bar M_{x\fb}) \prec  {\heta^{1/2}}/{N}.$$
Thus, the expectation of the last three terms on the RHS of \eqref{ETab} gives
$$\frac{\im \E G_{\fb\fb}}{N\eta} + \frac{\im m}{N(\eta+\im m)} -\frac1{N}\sum_x \E (M_{x\fb}\Gc^-_{x\fb}+\Gc_{x\fb}\bar M_{x\fb}) =\frac{\im m + \OO_\prec(\eta+\heta^{1/2}) }{N\eta}. $$
If we let $\wt {\cal G}_{\fa \fb}:=\sum_{x}S_{\fa x}\E(M_{x\fb}\Gc^-_{x\fb}+\Gc_{x\fb}\bar M_{x\fb})$, then the entries of $\wt {\cal G}$ satisfies \eqref{eq:boundGalywt} by the local law \eqref{locallawmax}. Next, we take the expectation of \eqref{mlevelTgdef_BA} with $\fb_1=\fb_2=\fb$ and get that  
\begin{align}
\E\czT_{\fa\fb} + (SM^0)_{\fa\fb}&=  [\zthn(\Selek) M^0 SM^0 + SM^0]_{\fa \fb} + \sum_{x} [\zthn(\Selek) M^0 S]_{\fa x}\E\left(\Gc _{x \fb } \bar M_{ x \fb}	+M_{x \fb }\Gc^{-} _{ x \fb}\right) \nonumber\\
&+ \sum_x \zthn_{\fa x} (\Selek)\E\left[\PT_{x,\fb \fb} +  \AT_{x,\fb \fb}  + \WT_{x,\fb \fb}  +  (\Err_{D})_{x,\fb \fb}\right]\, .\label{EzTab}
\end{align}
The first term on the RHS can be written as 
\begin{align*}
\zthn(\Selek) M^0 SM^0 + SM^0 &= \frac{1}{1-\zthn \Sele}\left(\zthn  M^0S + S\right) M^0 - \left(\frac{1}{1-\zthn \Sele}\zthn \Sele\right) SM^0 \\
&= \zthn(\Selek) M^0 -  \zthn (\Sele) \Sele SM^0 .\end{align*}
Using the condition \eqref{4th_property0} for the $\self$ $\Sele$ and \Cref{lem:propM}, we obtain that 
$$(\Sele SM^0)_{xy} \prec \frac{\sizeself(\Selek)}{\blam} \frac1{\qq{x-y}^{d}}.$$
Hence, the term $- \Sele SM^0$ can be included in $\cal G$. The second term on the RHS of \eqref{EzTab} also provides a term  $\sum_\al (M^0S)_{x\al}\E(\Gc_{\al y}\bar M_{\al y}+M_{\al y}\Gc^-_{\al y}) $
to $\cal G_{xy}$, which can be bounded as  
$$\sum_\al (M^0S)_{x\al}\E\left(\Gc_{\al y}\bar M_{\al y}+M_{\al y}\Gc^-_{\al y}\right) \prec \frac{\heta^{1/2}}{W^{d}}\exp\left(-\frac{|x-y|}{CW}\right)$$
for a constant $C>0$ by \Cref{lem:propM} and the local law \eqref{locallawmax}. 
Finally, the terms 
$\E\PT_{x,\fb \fb},$ $\E\AT_{x,\fb \fb}$, $ \E\WT_{x,\fb \fb} $, and $\E(\Err_{D})_{x,\fb \fb}$ can be handled in the same manner as in the case of Wegner orbital model, so we omit the details. This concludes the proof of \eqref{eq:Qdiff} for the block Anderson and Anderson orbital models.

\section{Proof of quantum unique ergodicity}\label{sec:QUE}

Denote $\mathcal A : =\im G= ( G  - G^* )/(2\ii)$. We observe that the quantity of interest $\sum_{x \in I_N} (N|u_\alpha(x)|^2 -1)$ can be controlled by $\tr(\mathcal A \Pi \mathcal A \Pi)$, where $\Pi$ is a diagonal matrix with diagonal entries 
\be\label{eq:defPi}\Pi_{xx}=(N/|I_N|)\mathbf 1_{x\in I_N} - 1.\ee 
The next lemma follows immediately from the spectral decomposition of $\cal A$.

\begin{lemma}[Lemma 4.15 of \cite{BandIII}]
\label{uab-tr}
Let $z = E + \ii\eta$ and $\Pi = \mathrm{diag}((\Pi_x)_{x \in \Z^d_L})$ be an arbitrary real diagonal matrix. Then, for any $l\ge \eta$,
\begin{align}
	& {\sum_{\alpha, \beta: |\lambda_\alpha - E| \leq l,|\lambda_\beta - E| \leq l}}|\gE{\mathbf u_\alpha,\Pi \mathbf u_\beta}|^2  \leq \frac{4l^4}{\eta^2} \tr\left[\mathcal A(z) \Pi \mathcal A(z) \Pi\right]  .\label{eq:uab-tr-1}
\end{align}
\end{lemma}

Note the matrix $\Pi$ defined through \eqref{eq:defPi} has $\tr \Pi=0$. Then,  \Cref{thm:QUE} follows from the next lemma on high-moment bounds of $\tr\left(\mathcal A \Pi \mathcal A \Pi\right)$.

\begin{lemma}
\label{trABAB}
In the setting of \Cref{thm:QUE}, fix any $z=E+\ii \eta$ with $|E|\le 2-\kappa$ and $W^\fd \eta_*\le \eta \le 1$. Consider an arbitrary real diagonal matrix $\Pi$ such that $\tr{\Pi} = 0$ and $\Pi_{xx}$ is constant on each block, i.e., $\Pi_{xx}=\Pi_{yy}$ if $[x]=[y]$. Then, for any fixed $p \in \N$, we have that
	\begin{equation}
	\label{eq:trABAB}
		\E\left[|\tr\left[\mathcal A(z) \Pi \mathcal A(z) \Pi\right]|^{2p}\right] \prec \Big(\sum_y |\Pi_{yy}|\Big)^{2p}  \Big(\max_x \sum_{y}B_{xy}|\Pi_{yy}|\Big)^{2p} \,.
	\end{equation}
\end{lemma}

\begin{proof}[\bf Proof of Theorem \ref{thm:QUE}]
For $\Pi$ defined in \eqref{eq:defPi}, we have $\tr{\Pi}= 0$, $\Pi_{xx}$ is constant on each block, and 
\begin{equation}\label{QUE-id}
\begin{split}
	\left|\gE{\mathbf u_\alpha,\Pi \mathbf u_\alpha}\right|^2 &= \Big(\sum_x \Pi_{xx} |u_\alpha(x)|^2 \Big)^2 =  \bigg(\frac{1}{|I_N|}\sum_{x \in I_N} (N|u_\alpha(x)|^2 -1)\bigg)^2 \,.
\end{split}
\end{equation}
Next, taking $z=E+\ii \eta$ with $\eta= N^\fd \eta_*$ and applying Markov's inequality to \eqref{eq:trABAB}, we obtain that
	\begin{align*}
	\tr\left[\mathcal A(z) \Pi \mathcal A(z) \Pi\right] & \prec \blam N  \bigg( \frac{|I_N|^{2/d}}{W^2} \frac{N}{|I_N|} + \frac{L^2}{W^2}  \bigg) \lesssim \frac{\blam N^2}{ W^{2}|I_N|^{1-2/d}}\,.\end{align*}
Therefore, it follows from \eqref{eq:uab-tr-1} that 
\begin{align}\label{eq:supque} 
\sup_{\alpha: |\lambda_\alpha - E| \leq \eta}|\gE{\mathbf u_\alpha,\Pi \mathbf u_\alpha}|^2  & \prec \eta^2 \tr(\mathcal A \Pi \mathcal A \Pi) \prec \frac{\blam (N\eta)^2}{W^{2}|I_N|^{1-2/d}}\lesssim \frac{L^{10}W^{2d-12+2\fd}}{\blam|I_N|^{1-2/d}} \,,\\
 \label{eq:avque}	\frac{1}{N\eta}\sum_{\alpha: |\lambda_\alpha - E| \leq \eta}|\gE{\mathbf u_\alpha,\Pi \mathbf u_\alpha}|^2 & \prec \frac{\eta}{N} \tr(\mathcal A \Pi \mathcal A \Pi) \prec \frac{\blam N\eta}{ W^{2}|I_N|^{1-2/d}}\lesssim \frac{L^5 W^{d-7+\fd}}{|I_N|^{1-2/d}}  \,.
\end{align}
Combining \eqref{eq:supque} with the condition \eqref{eq:cond_ell1}, we conclude \eqref{eq:que} since $\fd$ can be arbitrarily small.
Moreover, from \eqref{eq:avque}, we obtain that
$$\frac{1}{N}\sum_{\alpha: |\lambda_\alpha | \leq 2-\kappa}|\gE{\mathbf u_\alpha,\Pi \mathbf u_\alpha}|^2  \prec \frac{L^5 W^{d-7+\fd}}{|I_N|^{1-2/d}} \,,$$
which, together with Markov's inequality, implies that
\begin{equation}
 \frac{1}{N}\,\left|\left\{ \alpha: |\lambda_\alpha| \leq 2 - \kappa,  |\gE{\mathbf u_\alpha,\Pi \mathbf u_\alpha}|\geq \epsilon \right\}\right| \prec  \frac{\epsilon^{-2} L^5 W^{d-7+\fd}}{|I_N|^{1-2/d}} .
\end{equation}
Combining this bound with \eqref{QUE-id}, with a union bound over $I_N \in \mathcal I$, we conclude \eqref{eq:weakque} under \eqref{eq:cond_ell2}.
\end{proof}

For the proof of \Cref{trABAB}, note that $|\tr(\mathcal A \Pi \mathcal A \Pi)|^{2p}$ is a sum of expressions of the form \be\label{eq:exp_GPi}
\sum_{\mathbf s}c(\mathbf s)\sum_{\mathbf x,\mathbf y}\mathcal G_{\mathbf x,\mathbf y}^{\mathbf s} \prod_{i=1}^{2p}\Pi_{x_i x_i}\Pi_{y_i y_i} ,\ee 
where $\mathbf s=(s_1,\ldots, s_{4p})\in \{\emptyset,\dag\}^{4p}$, $c(\{s_i\})$ denotes a deterministic coefficient of order $\OO(1)$, and $\mathcal G_{\mathbf x,\mathbf y}$  are graphs of the form
\begin{equation}\label{eq:multiG}
	\mathcal G^{\mathbf s} _{\mathbf x,\mathbf y} = \prod_{i = 1}^{2p} G^{(s_{2i-1})}_{x_i y_i} G^{(s_{2i})}_{y_i x_i}. 
\end{equation}
Now, to show \eqref{eq:trABAB}, it suffices to prove the following counterpart of Lemmas \ref{mG-fe} and \ref{lem:mG-fe-size} for
$\E \mathcal G^{\mathbf s} _{\mathbf x,\mathbf y}$, where the loop graph $\cal G_{\mathbf x}$ is replaced by a graph $\mathcal G^{\mathbf s} _{\mathbf x,\mathbf y}$ with $2p$ connected components.

\begin{lemma}\label{mGep}
In the setting of \Cref{gvalue_continuity2}, suppose \eqref{eq:cont_ini_eta} holds and we have a $\nonuni$ \eqref{mlevelTgdef weak} (for the Wegner orbital model) or \eqref{mlevelTgdef weak_BA} (for the block Anderson and Anderson orbital models). 
For the graph in \eqref{eq:multiG}, let $\Sigma_1,\ldots, \Sigma_q$ be disjoint subsets that form a partition of the set of vertices $\{x_1, \ldots, x_{2p},y_1,\ldots, y_{2p}\}$ with $1\le q\le 4p$. We identify the indices in $\Sigma_i$, $i\in \qqq{1,q}$, and denote the resulting graph by \smash{$ \mathcal G^{\mathbf s}_{\mathbf w, \mathbf \Sigma}(z)$}, where $\mathbf \Sigma=(\Sigma_1,\ldots, \Sigma_q)$ and we denote the external vertices by $\bw=(w_1,\ldots, w_q)$. 
Then, for any constant $D>0$, we have that
\begin{equation}\label{EGx_QUE}
\E[\mathcal G^{\mathbf s}_{\mathbf w, \mathbf \Sigma}] = \sum_\mu\mathcal G_{\mathbf w}^{(\mu)} + \OO(W^{-D})\,,
\end{equation}
where the RHS is a sum of $\OO(1)$ many deterministic normal graphs $\mathcal G_{\mathbf w}^{(\mu)}$ with internal vertices and without silent edges such that the following properties hold. 

\begin{itemize}
\item[(a)]  $\mathcal G_{\mathbf w}^{(\mu)}$ satisfies the properties (a) and (b) in \Cref{mG-fe} for $\mathcal G_{\mathbf x}^{(\mu)}$.

\item[(b)] If $w_i$ and $w_j$ are connected in $\mathcal G^{\mathbf s}_{\mathbf w, \mathbf \Sigma}(z)$, then they are also connected in $\mathcal G_{\mathbf w}^{(\mu)}$ through non-ghost edges. 

\item[(c)] The weak scaling size of $\mathcal G_{\mathbf w}^{(\mu)}$ satisfies that 
\be\label{mG-fe-size-w}
\wsize(\mathcal G_{\mathbf w}^{(\mu)})\lesssim \heta^{q-t},
\ee 
where $t$ denotes the number of connected components in $\mathcal G_{\mathbf w}^{(\mu)}$.
\end{itemize}
\end{lemma}
\begin{proof} 
The proof of this lemma follows the same approach as that for \Cref{mG-fe}: we will expand the graph $\mathcal G^{\mathbf s}_{\mathbf w, \mathbf \Sigma}$ by applying \Cref{strat_global_weak2} repeatedly. Then, the resulting deterministic graphs \smash{$\mathcal G_{\mathbf w}^{(\mu)}$} will satisfy the above properties (a) and (b) due to Lemma 6.2 (a)--(c) in \cite{BandIII}. The weak scaling size condition \eqref{mG-fe-size-w} can be proved in the same way as \Cref{lem:mG-fe-size}. In fact, the proof of \Cref{lem:mG-fe-size} demonstrates that the weak scaling size of each connected component in \smash{$\mathcal G_{\mathbf w}^{(\mu)}$} with $k$ external vertices can be bounded by $\OO(\heta^{k-1})$. Thus, the total weak scaling size of \smash{$\mathcal G_{\mathbf w}^{(\mu)}$} is at most $\OO(\heta^{q-t})$. We omit the details of the proof.
\end{proof}

Finally, we outline the proof of \Cref{trABAB} by using \Cref{mGep}.

\begin{proof}[\bf Proof of \Cref{trABAB}]
Under the assumptions of Theorem \ref{thm:QUE}, we know that the local law \eqref{eq:cont_ini_eta} holds by Theorem \ref{thm_locallaw}. Moreover, in the proof of Theorem \ref{thm_locallaw}, we have constructed a sequence of $T$-expansions up to arbitrarily high order $\fC$ by \Cref{completeTexp}. We can choose $\fC$ large enough so that \eqref{Lcondition1} holds. Then, we have a complete $T$-expansion by Lemma \ref{def nonuni-T}. Hence, the setting of \Cref{mGep} is satisfied, which gives the expansion \eqref{EGx_QUE}. It remains to prove that 
\be\label{eq:sum_QUE}
\sum^\star_{w_1,\ldots, w_q\in \Z_L^d} \mathcal G_{\mathbf w}^{(\mu)}\mathbf 1(\Sigma_i = w_i: i=1,\ldots, q)\prod_{i=1}^{2p}\Pi_{x_ix_i} \Pi_{y_i y_i} \prec  \Big(\sum_y |\Pi_{yy}|\Big)^{2p}  \Big(\max_x \sum_{y}B_{xy}|\Pi_{yy}|\Big)^{2p}, 
\ee
where recall that $\sum^\star$ means summation subject to the condition that $w_1,\ldots, w_q$ all take distinct values: 
$$ \sum^\star_{w_1,\ldots, w_q}=\sum_{w_1,\ldots, w_q} \prod_{i\ne j \in \qqq{q}}\mathbf 1(w_i\ne w_j). $$
We can further remove the $\star$ from the LHS of \eqref{eq:sum_QUE} by expanding each $\times$-dotted edge as $\mathbf 1(w_i\ne w_j)=1-\mathbf 1(w_i=w_j)$. Taking the product of all these decompositions, we can write the graph as a linear combination of new graphs containing dotted edges. Then, in each new graph, we merge the vertices connected through dotted edges and rename them as $w_1,\ldots, w_\ell$, where $1\le \ell \le q$ denotes the number of external vertices in the graph. In this way, we can rewrite the LHS of \eqref{eq:sum_QUE} as  
\be\nonumber
\sum_{\ell=1}^q \sum_\gamma \sum_{\mathbf w=(w_1,\ldots, w_\ell)\in (\Z_L^d)^\ell } \wt{\mathcal G}_{\mathbf w}^{(\gamma)}\mathbf 1(\wt\Sigma_i = w_i: i=1,\ldots, \ell)\prod_{i=1}^{2p}\Pi_{x_ix_i} \Pi_{y_i y_i} , 
\ee
where $\gamma$ labels the new graphs $\wt{\mathcal G}_{\mathbf w}^{(\gamma)}$ and $\wt\Sigma_i$ for $i\in \qqq{\ell}$ denote the new partitions. Thus, to prove \eqref{eq:sum_QUE}, it suffices to show that 
\be\label{eq:sum_QUE2}
\sum_{w_1,\ldots, w_\ell\in \Z_L^d } \wt{\mathcal G}_{\mathbf w}^{(\gamma)}\mathbf 1(\wt\Sigma_i = w_i: i=1,\ldots, \ell)\prod_{i=1}^{2p}\Pi_{x_ix_i} \Pi_{y_i y_i} \prec  \Big(\sum_y |\Pi_{yy}|\Big)^{2p}  \Big(\max_x \sum_{y}B_{xy}|\Pi_{yy}|\Big)^{2p}.  
\ee
Note these new graphs $\wt{\mathcal G}_{\mathbf w}^{(\gamma)}$ also satisfy the properties (a)--(c) in \Cref{mGep} for $\mathcal G_{\mathbf w}^{(\mu)}$. Here, the condition \eqref{mG-fe-size-w} should become  
\be\label{mG-fe-size-w2}
\wsize(\wt{\mathcal G}_{\mathbf w}^{(\gamma)})\lesssim \heta^{\ell-t'},
\ee
where $t'$ denotes the number of connected components in $\wt{\mathcal G}_{\mathbf w}^{(\gamma)}$. To understand why \eqref{mG-fe-size-w2} holds, we observe that whenever the count of connected components decreases by 1, it indicates the merging of external vertices connected by dotted edges, resulting in a reduction of external vertices by at least 1.


Next, a key observation is that it suffices to consider $\wt{\mathcal G}_{\mathbf w}^{(\gamma)}$ in which 
\be\label{eq:keyobs_noniso}
\text{none of the $\ell$ external vertices is isolated if we do not include free edges into the edge set}. 
\ee
Otherwise, suppose the external vertex $w_1$ is isolated without loss of generality. Since both our graphs and the matrix $\Pi$ are translation invariant on the block level,  we know that \smash{$\sum_{w_1\in [x]}\wt{\mathcal G}_{\mathbf w}^{(\gamma)}$} does not depend on $[x]\in \wt \Z_n^d$, which implies that
$$\sum_{w_1}\wt{\mathcal G}_{\mathbf w}^{(\gamma)}\Pi_{w_1w_1} = W^{-d}\sum_{w_1\in [0]}\wt{\mathcal G}_{\mathbf w}^{(\gamma)} \cdot \sum_{w_1} \Pi_{w_1w_1} = 0. $$
We remark that this is the only place where the zero trace condition for $\Pi$ is used. 

It remains to prove that \eqref{eq:sum_QUE2} holds for graphs $\wt{\mathcal G}_{\mathbf w}^{(\gamma)}$ that satisfy properties (a)--(c) in \Cref{mGep} and condition \eqref{eq:keyobs_noniso}. This proof can be completed using the arguments from the proof of Lemma 3.4 in \cite{BandIII}, followed by those in the proof of Proposition 4.16 in the same work. Since those proofs apply verbatim to our case, we omit the details here. This concludes \eqref{eq:sum_QUE2} and hence completes the proof of \Cref{trABAB}.
\end{proof}

\section{Proof of some deterministic estimates}\label{appd_det}


\subsection{Proof of \Cref{lem theta}}\label{sec:Fourier}

The proof of \eqref{thetaxy} follows from an analysis of the Fourier series representation of $\zthn$. For the block Anderson and Anderson orbital models, we have the representation \eqref{exp_Theta3}, while for the Wegner orbital model, we have 
\be\label{exp_Theta4}
\zthn=\wt P^\perp \left[1-|m|^2S^{\LK}(\lambda)\right]^{-1} \wt P^\perp\otimes\bE= \wt P^\perp \left[1-|m|^2 - |m|^2\lambda^2(2dI_n - \Delta_n)\right]^{-1} \wt P^\perp\otimes\bE.
\ee
We introduce the notion $F_n=\KM$ for the block Anderson and Anderson orbital models and $F_n=|m|^2S^{\LK}(\lambda)$ for the Wegner orbital model. We claim that for all $[x],[y]\in \wt\Z_n^d$, 
\be\label{est_M0LN0}
\left|[\wt P^\perp (1-F_n)^{-1} \wt P^\perp]_{[x][y]}\right|\lesssim  \frac{\blam\log n}{\langle [x]-[y]\rangle^{d-2}} ,
\ee
and when $|[x]-[y]|\ge W^\tau \ell_{\lambda,\eta}$, 
\be\label{est_M0LN0_small}
\left|(1-F_n)^{-1}_{[x][y]}\right|\le  \langle [x]-[y]\rangle^{-D}.
\ee
Inserting these two bounds into \eqref{exp_Theta3} or \eqref{exp_Theta4}, we conclude \eqref{thetaxy} and \eqref{thetaxy2}. 

The rest of the proof is devoted to showing \eqref{est_M0LN0} and \eqref{est_M0LN0_small}. By translation invariance, it suffices to choose $[y]=0$. For simplicity, in the following proof, we denote the lattice \smash{$\Zn$} and its vertices $[x]$ by $\Z_n^d$ and $x$ instead. Recall $J_k$ and $d_k$ defined in \eqref{eq:Jks}. By \eqref{eq:KM}, \eqref{eq:KM2}, and \eqref{exp_Theta4}, for any fixed integer $K\in \N$, we can expand $F_n$ as
\begin{align}
F_n =  a_0 + \sum_{k\ge 1}\blam^{-k} a_k \J_k, 
\label{eq:M0Ln}
\end{align}
where the coefficients $a_k(z,W,L)\ge 0$ are all of order $\OO(C^k)$ for a constant $C>0$ and $a_1\gtrsim 1$ by \eqref{eq:KM1}. (Note for the Wegner orbital model, we have $a_0=a_1=|m|^2$ and $a_k=0$ for $k\ge 2$.) From \eqref{L2M2} and \eqref{eq:a_sum}, we get that for the block Anderson and Anderson orbital models,
\be\label{eq:a0}
1-a_0=  \frac{\eta}{\im m + \eta} + \sum_{k\ge 1}  d_k  a_k\blam^{-k} , 
\ee
and for the Wegner orbital model,
\be\label{eq:a0WO}
1-a_0=  \frac{\eta}{\eta + (1+2d\lambda^2)\im m(z)} + 2d\lambda^2  |m|^2.
\ee
For simplicity of notations, in the following proof, we focus on the harder cases---the block Anderson and Anderson orbital models, where $F_n=\KM$ and $1-a_0$ satisfies \eqref{eq:a0}. The proof for the Wegner orbital model is just a special case, with $a_k=0$ for $k\ge 2$.  

With \eqref{eq:M0Ln} and \eqref{eq:a0}, we can write that 
\begin{align*}
\frac{1}{1-\KM} =&\left\{ \frac{\eta}{\im m + \eta} +  a_1\blam^{-1} \Delta^{(n)} + \sum_{k\ge 2} a_k \blam^{-k} (  d_k-J_k) \right\}^{-1}. 
\end{align*}
Note that the eigenvectors of $d_k-\J_k$ are given by the plane waves
$$v_{\mathbf p}:=n^{-d/2}\left(e^{\ii \bp\cdot x}:x\in  \Z_n^d\right), \quad \bp\in  \T_n^d:= \left(\frac{2\pi }{n}  \Z_n\right)^d,$$
with the corresponding eigenvalue 
\be\label{eq_lambdak}\lambda_k({\bp})=\sum_{\mathbf k \in  \Z_n^d:\|\mathbf k\|_1=k}(1-\cos(\bp\cdot \mathbf k))\ge 0. 
\ee
Thus, we can write that 
\be\label{eq:PKMP}
\left[\wt P^\perp (1-\KM)^{-1} \wt P^\perp\right]_{x0} = \frac1{n^d}\sum_{\bp\in  \T_n^d\setminus \{0\}} f_\eta (\bp)e^{\ii \bp\cdot x},
\ee
where $f_\eta (\bp)$ is defined as
\be\label{eq:fetap}
 f_\eta (\bp)\equiv  f_{\eta,W,L} (\bp):=\left\{\frac{\eta}{\im m + \eta} + a_1 \blam^{-1}\lambda_1(\bp) + \sum_{k\ge 2} a_k\blam^{-k}\lambda_k(\bp)  \right\}^{-1}.
\ee
To show \eqref{est_M0LN0}, it remains to bound this Fourier series as 
\be\label{est_M0LN}
 \frac1{n^d}\sum_{\bp\in  \T_n^d\setminus \{0\}} f_\eta (\bp)e^{\ii \bp\cdot x}\lesssim  \frac{\blam\log n}{\langle x\rangle^{d-2}} .
\ee

First, notice that $ \lambda_1({\bp})\gtrsim |\bp|^2$ and $\lambda_k(\bp)\lesssim k^3|\bp|^2$, from which we obtain that 
\be\label{eq:feta_est} f_\eta(\bp) \lesssim \left(\eta+\blam^{-1}|\bp|^2\right)^{-1}.\ee
Thus, when $x=0$, we have that 
$$\frac1{n^d}\sum_{\bp\in \T_n^d\setminus \{0\}} |f_\eta (\bp)|\lesssim  \frac{\blam}{n^d}\sum_{\bp\in \T_n^d\setminus \{0\}} \frac{1}{|\bp|^2}\lesssim \blam.$$
Next, we consider the case $x\ne 0$. Without loss of generality, suppose $|x_1| =\max_{i=1}^d |x_i| \ge 1$. Then, we can write that 
\be\label{eq:feta-A}  \frac1{n^d}\sum_{\bp\in \T_n^d\setminus \{0\}} f_\eta (\bp)e^{\ii \bp\cdot x}=\frac{1}{n^d}\sum_{p_1\in \T_n\setminus\{0\}} f_\eta(p_1,0)e^{\ii p_1 x_1} + \frac{1}{n^{d-1}}\sum_{\bq\in \T_n^{d-1}\setminus\{0\}} e^{\ii \bq \cdot \wh x^{(1)} }A(x_1,\bq),\ee
where we denote $\bq:=(p_2, \ldots, p_d) $, $\wh x^{(1)}:=(x_2,\cdots, x_n)$, and
\be\label{eq:Ax1q}
A(x_1 ,\bq):= \frac{1}{n}\sum_{p_1 \in  \T_n  } f_\eta (p_1,\bq) e^{\ii p_1 x_1} .\ee
Using \eqref{eq:feta_est}, we can bound the first term on the RHS of \eqref{eq:feta-A} as 
\begin{align*}
    \frac{1}{n^d}\sum_{p_1\in \T_n\setminus\{0\}}| f_\eta(p_1,0)| \lesssim \frac{1}{n^d}\sum_{p_1\in \T_n\setminus\{0\}} \frac{\blam}{|p_1|^2} \lesssim \frac{\blam}{n^{d-2}}\lesssim \frac{\blam}{\langle x\rangle^{d-2}}.
\end{align*}
It remains to control the second term on the right-hand side of \eqref{eq:feta-A}.
Denote $n_\pm:= \pm (n-1)/{2}$. 
Using summation by parts, we can write $A(x_1,\bq)$ as
$$A(x_1 ,\bq)=\frac{1}{n}\sum_{p_1 \in  \T_n}  S(x_1,p_1) \left[ f_\eta(p_1,\bq) - f_\eta\left(p_1+\frac{2\pi}{n},\bq\right)\right]  ,$$
where we used the periodic boundary condition $f_\eta\left( 2\pi (n_+ + 1)/{n},\bq\right)=f_\eta\left( {2\pi}n_-/{n},\bq\right)$, and $S(x_1,p_1)$ is a partial sum defined as
$$ S(x_1 ,p_1)=\sum_{k= n_- }^{\frac{n}{2\pi}p_1}\exp\left( \ii \frac{2\pi k}{n} x_1 \right)=\frac{\exp\left( \ii n_-\delta_n x_1\right)- \exp\left( \ii \left(p_1+\delta_n\right)x_1\right)}{1-\exp\left(\ii \delta_n x_1\right)}, \quad \delta_n:=\frac{2\pi}{n}.$$
Applying this formula to \eqref{eq:Ax1q} and using the periodic boundary condition, we obtain that  
\begin{align}
A(x_1,\bq)&=\frac{1}{n}\sum_{p_1 \in \T_n  }  \frac{f_\eta(p_1,\bq) -f_\eta\left(p_1-\delta_{n},\bq\right)}{1- \exp\left(\ii  \delta_{n} x_1\right)}e^{\ii  p_1  x_1 }  . \label{sum_parts1}
\end{align}
By the definition of $f_\eta$ in \eqref{eq:fetap}, there exists a constant $C>0$ such that for $\bq \in \T_n^{d-1}  \setminus\{0\}$, 
\begin{align}
\left|\frac{f_\eta(p_1,\bq) -f_\eta\left(p_1-\delta_{n},\bq\right)}{1- \exp\left(\ii  \delta_{n} x_1\right)}\right| &\lesssim \frac{f_\eta(p_1,\bq)f_\eta(p_1-\delta_n,\bq)}{|x_1|\delta_n} \sum_{k\ge 1} \left(\frac{C}{\blam}\right)^{k}\sum_{\mathbf k \in \Z_n^d:\|\mathbf k\|_1=k}|\cos(\bp\cdot \bk -k_1\delta_n)-\cos(\bp\cdot \bk)|  \nonumber\\
&\lesssim \left(\eta+\blam^{-1}|\bp|^2\right)^{-2}  \frac{\delta_n^2 + |\mathbf p|\delta_n}{\blam |x_1|\delta_n} \lesssim  \frac{\blam}{|\bp|^3 |x_1|}, \label{eq:bdd_1stdiff}
\end{align}
where in the second step we used the estimate \eqref{eq:feta_est}, $|p_1-\delta_n|\gtrsim |p_1|$ for $p_1 \in \T_n  \setminus\{\delta_n\}$, $|p_1-\delta_n|\le  |p_1| \le |\mathbf q|$ for $p_1=\delta_n$, and 
\begin{align*}
|\cos(\bp\cdot \bk - k_1\delta_n)-\cos(\bp\cdot \bk)| &\le |\cos(\bp\cdot \bk)| (1-\cos(k_1\delta_n)) + |\sin(\bp\cdot \bk)| |\sin(k_1\delta_n)| \\
&\lesssim (k_1\delta_n)^2 + |\bp\cdot \bk| |k_1\delta_n| .    
\end{align*}
Applying \eqref{eq:bdd_1stdiff} to \eqref{sum_parts1}, we immediately obtain that 
$$ |A(x_1 ,\bq)|\lesssim\frac{1}{n}\sum_{p_1 \in \T_n  }  \frac{\blam}{|\bp|^3 |x_1|}\lesssim \frac{\blam}{|x_1|}\frac{1}{|\bq|^2}, $$
with which we can bound the second term on the RHS of \eqref{eq:feta-A} as 
$$ \frac{1}{n^{d-1}}\sum_{\bq\in  \T_n^{d-1}\setminus\{0\}} e^{\ii \bq \cdot \wh x^{(1)} }A(x_1,\bq)\lesssim \frac{1}{n^{d-1}}\sum_{\bq\in \T_n^{d-1}\setminus\{0\}} \frac{\blam}{ |x_1|}\frac{1}{|\bq|^2} \lesssim \frac{\blam}{|x_1|}. $$

To improve the above estimate to \eqref{est_M0LN}, we only need to perform the summation by parts argument to \eqref{sum_parts1} again. For example, applying one more summation by parts to \eqref{sum_parts1}, we obtain that 
\begin{align*}
A(x_1,\bq) &=\frac{1}{n}\sum_{p_1 \in \T_n  }  \exp\left( \ii  p_1  x_1 \right)\frac{f_\eta(p_1+\delta_n,\bq) -2f_\eta\left(p_1,\bq\right)+ f_\eta\left(p_1-\delta_{n},\bq\right)}{\exp(-\ii\delta_n x_1) [1- \exp\left(\ii  \delta_{n} x_1\right)]^2} ,
\end{align*}
which involves the second-order finite difference of $f_\eta(\cdot, \bq)$. Using a similar argument as above, we can estimate $A(x_1,\bq)$ as 
$$ |A(x_1 ,\bq)|\lesssim  \frac{\blam}{|x_1|^2}\frac{1}{|\bq|^3} . $$
with which we can bound the second term on the RHS of \eqref{eq:feta-A} as 
\begin{align*}
    \frac{1}{n^{d-1}}\sum_{\bq\in \T_n^{d-1}\setminus\{0\}} e^{\ii \bq \cdot \wh x^{(1)} }A(x_1,\bq) \lesssim   \frac{\blam}{|x_1|^2} .
\end{align*}
Continuing the above arguments, after performing $(d-2)$ summations by parts, we obtain that 
$$ |A(x_1 ,\bq)|\lesssim  \frac{\blam}{|x_1|^{d-2}}\frac{1}{|\bq|^{d-1}} , $$
with which we get that 
\begin{align*}
    \frac{1}{n^{d-1}}\sum_{\bq\in \T_n^{d-1}\setminus\{0\}} e^{\ii \bq \cdot \wh x^{(1)} }A(x_1,\bq) \lesssim  \frac{1}{n^{d-1}}\sum_{\bq\in  \T_n^{d-1}\setminus\{0\}} \frac{\blam}{|x_1|^{d-2}}\frac{1}{|\bq|^{d-1}} \lesssim \frac{\blam\log n}{|x_1|^{d-2}}.
\end{align*}
This concludes \eqref{est_M0LN} (recall that we have assumed $|x_1|=\|x \|_\infty \gtrsim |x|$). 

Finally, to show \eqref{est_M0LN0_small}, we again assume that $\|x \|_\infty = |x_1|$ without loss of generality. Then, we need to prove that 
\be\label{est_M0LN_small}
 \frac1{n^d}\sum_{\bp\in  \T_n^d} f_\eta (\bp)e^{\ii \bp\cdot x}=\frac{1}{n^{d-1}}\sum_{\bq\in \T_n^{d-1} } e^{\ii \bq \cdot \wh x^{(1)} }A(x_1,\bq) \lesssim  \langle x\rangle^{-D}  
\ee
when $|x_1|\gtrsim W^\tau \ell_{\lambda,\eta}$. We can apply a similar summation by parts argument as above, with the only difference being that we need to account for the $\eta$ term when bounding the finite differences of $f_\eta(\cdot, \bq)$. For example, we would write \eqref{eq:bdd_1stdiff} as\begin{align}
\left|\frac{f_\eta(p_1,\bq) -f_\eta\left(p_1-\delta_{n},\bq\right)}{1- \exp\left(\ii  \delta_{n} x_1\right)}\right| \lesssim \left(\blam\eta+|\bp|^2\right)^{-2} \frac{\blam |\mathbf p| }{|x_1| }  \le \left(\sqrt{\blam \eta} +|\bp|\right)^{-3} \frac{\blam}{|x_1|}  . \label{eq:bdd_1stdiffnew}
\end{align} 
After applying summation by parts for $k$ times, we get that 
$$  A(x_1,\bq)\lesssim \frac{1}{n}\sum_{ p_1\in \T_n} \left( \sqrt{\blam \eta}+|\bp|\right)^{-(k+2)} \frac{\blam}{|x_1|^k} , 
$$
which leads to the following estimation for any fixed $k>d-2$, 
\begin{align*}
\frac{1}{n^{d-1}}\sum_{\bq\in  \T_n^{d-1}} e^{\ii \bq \cdot \wh x^{(1)} }A(x_1,\bq) &\lesssim \frac{1}{n^{d}}\sum_{\bp\in \T_n^{d}} \left( \sqrt{\blam \eta}+|\bp|\right)^{-(k+2)} \frac{\blam}{|x_1|^k}  \lesssim \frac{\blam}{ |x_1|^k}  [\blam\eta]^{-(k+2)/2}.
\end{align*} 
Using that $\ell_{\lambda,\eta} \le |x_1|^{1-\e} \le n$ for a small constant $\e>0$ and $|x_1|\gtrsim W^\tau$, by choosing $k$ sufficiently large depending on $\e$, $\tau$ and $D$, we can bound the RHS by $|x|^{-D}$. This concludes \eqref{est_M0LN_small}.

\subsection{Proof of \Cref{lem:label_diffusive}}

Using equations \eqref{eq:diffzthn-thn1} and \eqref{eq:diffzthn-thn2}, we can show that replacing \smash{$\zthn$} with $\thn$ leads to negligible errors in the bounds \eqref{thetaxy_renorm0} and \eqref{thetaxy_renorm} when $\eta\ge t_{Th}^{-1}$. It remains to prove \eqref{thetaxy_renorm0} and \eqref{thetaxy_renorm} for \smash{$\zthn$}. The estimate \eqref{thetaxy_renorm0} follows directly from \eqref{thetaxy} and \eqref{self_decay}.  
A key to the proof of \eqref{thetaxy_renorm} is the following lemma. 
 
\begin{lemma}\label{lem cancelTheta}
Under the setting of Lemma \ref{lem:label_diffusive}, let $g:\wt\Z_n^d \to \R$ be a symmetric function (i.e., $g(x)=g(-x)$) supported on a box $\cal B_K:=\llbracket -K,K\rrbracket^d$ of scale $K\ge 1$. Assume that $g$ satisfies the sum zero property $\sum_{x}g(x)=0.$ Then, for $x_0\in \Z_n^d$ such that $|x_0| \ge 2K$, we have that for any constants $\tau,D>0$, 
$$\Big|\sum_{x\in \wt\Z_n^d}[\wt P^\perp (1-F_n)^{-1} \wt P^\perp]_{0x} g(x-x_0)\Big| \prec  \sum_{x\in \cal B_K}\frac{x^2}{|x_0|^2}|g(x)| \left(    \frac{\blam\mathbf 1_{|x_0|\le W^\tau \ell_{\lambda, \eta}}}{\langle x_0\rangle^{d-2}} + |x_0|^{-D}\mathbf 1_{|x_0|> W^\tau \ell_{\lambda, \eta}}\right),$$
where we recall that $F_n$ was defined below \eqref{exp_Theta4}.
\end{lemma}
\begin{proof} 
For convenience, we denote $\rB:=\wt P^\perp (1-F_n)^{-1} \wt P^\perp$.  Since $g(\cdot)$ is a symmetric function and $\sum_x g(x)=0$, we can write that
\be\label{decomposezero} \sum_{x}\rB_{0 x} g(x-x_0) =\sum_{a\in\mathfrak A} g(a) \left(  \rB_{0,x_0+a}+ \rB_{0,x_0-a}- \rB_{0,x_0+y_a} -\rB_{0,x_0-y_a}\right),\ee
where $\mathfrak A$ is a subset of $\cal B_K$, and $y_a\in \cal B_K$ depends on $a$ and satisfies $|y_a|\le |a|$. To conclude the proof, it suffices to show that 
\be\label{est_second_diff:B}
\left|\rB_{0,x_0+a}+ \rB_{0,x_0-a}- \rB_{0,x_0+y_a} -\rB_{0,x_0-y_a}\right| \prec \frac{|a|^2}{|x_0|^2}  \left(    \frac{ \blam\mathbf 1_{|x_0|\le W^\tau \ell_{\lambda, \eta}}}{\langle x_0\rangle^{d-2}}  + |x_0|^{-D}\mathbf 1_{|x_0|> W^\tau \ell_{\lambda, \eta}} \right).
\ee
The proof of this estimate is similar to that for \eqref{est_M0LN} and \eqref{est_M0LN_small}. For example, for the block Anderson and Anderson orbital models, we can use \eqref{eq:PKMP} to express:
    \begin{align*}
        \rB_{0,x_0+a}+ \rB_{0,x_0-a}- \rB_{0,x_0+y_a} -\rB_{0,x_0-y_a} = \frac1{n^d}\sum_{\bp\in \T_n^d\setminus \{0\}} f_\eta (\bp)e^{\ii \bp\cdot x_0} \left[2\cos(\bp\cdot a)-2\cos(\bp\cdot y_a)\right].
    \end{align*}
Then, similar to the proof below \eqref{est_M0LN} (resp.~\eqref{est_M0LN_small}), we can apply the summation by parts argument for $d$ times (resp.~$k$ times, for a large enough $k$) to conclude \eqref{est_second_diff:B}. We omit the details. 
\end{proof}

We now complete the proof of \eqref{thetaxy_renorm} with this lemma. Using \eqref{exp_Theta3} or \eqref{exp_Theta4}, we find that $(\zthn S\cal E S)_{[x][y]}=W^{-d} (\rB \wE)_{[x][y]}$. Hence, to show \eqref{thetaxy_renorm}, it suffices to prove that for any $x\in \Z_n^d$,
\be\label{thetaxy_renorm2}
 \sum_{\al\in \Z_n^d} \rB_{0\al} \Sele_{\LK}(\al,x) \prec \frac{1}{\langle x\rangle^{d}}\left(  \psi_0+\frac{\psi}{\langle x\rangle^{2}}  \right)\min\left(\eta^{-1},\blam\langle x\rangle^{2}\right).
\ee
Here, we again abbreviate $\Zn$ and its vertices $[x]$ by $\Z_n^d$ and $x$, and by translation invariance, we have chosen $y=0$. For simplicity of presentation, we will slightly abuse the notation and denote $\Sele_{\LK}$ by $\Sele$. 
To prove \eqref{thetaxy_renorm2}, we decompose the sum over $\al$ according to the dyadic scales $ \cal I_{\ell}:=\{\al \in \Z_n^d: K_{\ell-1} \le |\al-x| \le K_{\ell}\}$, where $K_\ell$ are defined by
\be\label{defn Kn}
K_\ell:= 2^\ell  \ \ \text{for}\ \ 1\le \ell \le \log_2 n -1, \ \ \text{and}\ \  K_{0}:=0. 
\ee
If $K_\ell\ge \qq{x}/10$, then we have that  
\be\label{boundIfar} \Big| \sum_{\al\in \cal I_{\ell}}\rB_{ 0 \al} \wE_{\al x}\Big| \le \sum_{\al\in \cal I_{\ell}} \left|\rB_{ 0 \al}\right| \cdot \max_{\al\in \cal I_\ell} \left|\wE_{\al x}\right| \prec \min\left(\eta^{-1},\blam K_\ell^2\right)  \frac{\psi}{  K_{\ell}^{d+2}}\lesssim  \frac{\psi}{\langle x\rangle^{d+2}}\min\left(\eta^{-1},\blam\langle x\rangle^{2}\right),\ee
where in the second step we used \eqref{self_decay} and the following estimate for any $y\in \Z_n^d$ and $1\le K \le n$: 
\be\label{eq:sum_B}\sum_{\al:|\al-y|\le K}\rB_{ 0 \al}\prec \min\left(\eta^{-1},\blam K^2\right),\ee
which is a simple consequence of \eqref{est_M0LN0} and \eqref{est_M0LN0_small}.

It remains to bound the sum
$$\sum_{\al\in \cal I_{near}}\rB_{0 \al} \wE_{\al x},\quad \cal I_{near}:=\bigcup_{\ell: K_\ell < \langle x\rangle/10} \cal I_\ell .$$
In order for $\cal I_{near}$ to be nonempty, we assume that $\langle x\rangle > 10$. 
Using \eqref{self_decay} and \eqref{self_zero}, we obtain that 
\be\label{mean Snear}\sum_{\al\in \cal I_{near}} \wE_{\al x}=\sum_\al \wE_{\al x} - \sum_{\al\notin \cal I_{near}} \wE_{\al x} \lesssim \psi_0 + \frac{\psi}{\langle x\rangle^2}.\ee
Then, we write $ \wE_{\al x} = \overline R + \mathring R_{\al x}$ for $\al \in \cal I_{near}$, where $\overline R:= \sum_{\al\in \cal I_{near}}\wE_{\al x}/|\cal I_{near}|$ is the average of $\wE_{\al x}$ over $\cal I_{near}$.  By \eqref{self_decay} and \eqref{mean Snear}, we have that 
\be\label{bound barR} |\overline R| \lesssim \frac{\psi_0}{\langle x\rangle ^{d}}+\frac{ \psi}{\langle x\rangle ^{d+2}},\quad | \mathring R_{\al x} | \lesssim  \frac{\psi}{ \langle \al-x\rangle^{d+2}} + |\overline R| .\ee
Thus, we obtain that 
\begin{align}
\Big| \sum_{\al \in \cal I_{near}}\rB_{  0 \al} \overline R  \Big|&\prec  \left( \frac{\psi_0}{\langle x\rangle ^{d}}+\frac{ \psi}{\langle x\rangle ^{d+2}}\right)\sum_{\al\in \cal I_{near}}\rB_{0\al} \prec  \frac{1}{\langle x\rangle ^{d}}\left( {\psi_0}+\frac{ \psi}{\langle x\rangle ^{2}}\right) \min\left(\eta^{-1},\blam\langle x\rangle^2\right),\label{Rnear_sum1}
\end{align}
where in the second step we used \eqref{eq:sum_B}. 
Finally, we use Lemma \ref{lem cancelTheta} to bound the sum over $\mathring R$ as: 
\begin{align}
\Big| \sum_{\al\in \cal I_{near}}\rB_{0 \al} \mathring R_{\al x} \Big|& \prec  \left(\sum_{\al\in \cal I_{near}}\frac{|\al-x|^2}{\langle x\rangle^2}|\mathring R_{\al x}|\right) \left(    \frac{ \blam\mathbf 1_{|x|\le W^\tau \ell_{\lambda, \eta}}}{\langle x\rangle^{d-2}} + |x|^{-D}\mathbf 1_{|x|> W^\tau \ell_{\lambda, \eta}}\right)\nonumber\\
 & \prec \left(\sum_{\al\in \cal I_{near}}\frac{\psi}{\langle x\rangle^2\langle\al-x\rangle^d } + \langle x\rangle^d |\overline R|\right) \left(    \frac{ \blam\mathbf 1_{|x|\le W^\tau \ell_{\lambda, \eta}}}{\langle x\rangle^{d-2}} + |x|^{-D}\mathbf 1_{|x|> W^\tau \ell_{\lambda, \eta}}\right) \nonumber\\
& \prec \frac{1}{\langle x\rangle^d}\left(\psi_0+\frac{\psi}{\langle x\rangle^2 } \right) \left( \blam \langle x\rangle^2 \mathbf 1_{|x|\le   W^{\tau}\ell_{\lambda,\eta}}    + (W^\tau \ell_{\lambda,\eta})^{-(D-d)}\right) \nonumber\\
&\lesssim\frac{W^{2\tau}}{\langle x\rangle ^{d}}\left( {\psi_0}+\frac{ \psi}{\langle x\rangle ^{2}}\right) \min\left(\eta^{-1},\blam\langle x\rangle^2\right),\label{Rnear_sum2}
\end{align}
where in the second and third steps we used \eqref{bound barR}. 
Combining \eqref{boundIfar}, \eqref{Rnear_sum1} and \eqref{Rnear_sum2}, we conclude \eqref{thetaxy_renorm} since $\tau$ is arbitrary. 

\subsection{Proof of \Cref{lem deter}} 
Similar to \eqref{exp_Theta3} and \eqref{exp_Theta4}, we can write $S^+$ as 
\be \label{eq:S+}
S^+ 
=(1-F_n^+)^{-1}\otimes \bE.
\ee
where $F_n^+:=\KMp$ for the block Anderson and Anderson orbital models and $F_n^+:=m^2 S(\lambda)_{\LK}$ for the Wegner orbital model. With \eqref{self_mWO}, \eqref{eq:expandm} (with $m_{2k+1}=0$), \eqref{Mbound}, and \eqref{Mbound_AO}, we can show that 
\be\label{eq:KMp}
(F_n^+)_{[x][x]}= m_{sc}(z)^2 +\OO(\lambda^2),\quad \text{and}\quad 
(F_n^+)_{[x][y]} \le (C/\blam)^{|[x]-[y]|}.
\ee
Thus, we can expand $(1-F_n^+)^{-1}$ as 
\be\label{eq:expMLn}(1-F_n^+)^{-1} =\sum_{k=0}^{\infty} \left[1-(F_n^+)_{[0][0]}\right]^{-(k+1)}\left[F_n^+ -(F_n^+)_{[0][0]}\right]^{k}.
\ee
Using this expansion, with the estimate \eqref{eq:KMp} and the fact $|1-m_{sc}(z)^2|\gtrsim 1$, we can derive that 
$$ \left|(1-F_n^+)^{-1}_{[x][y]}\right| \lesssim \exp\left(-c|[x]-[y]|\right)$$
for a constant $c>0$. Plugging it into \eqref{eq:S+}, we conclude \eqref{S+xy}. The estimate \eqref{S+xy-0} is also an easy consequence of \eqref{eq:KMp}, \eqref{eq:expMLn}, \eqref{Mbound}, and \eqref{Mbound_AO}.

\subsection{Proof of \Cref{lem:estM}}
The estimate \eqref{eq:mL_minf} for the Wegner orbital model follows easily from \eqref{self_mWO}. 
For the block Anderson model, recall that $m(z)$ satisfies  equation \eqref{self_m}, which can be rewritten as follows using Fourier series:
\be\label{eq:mL} m^{(L)}(z) = \frac{1}{N}\sum_{\bp\in \T_L^d} \frac{1}{\lambda e(\bp) - z - m^{(L)}(z)},\ee
where we abbreviate $e(\bp):=\sum_{i=1}^d 2\cos p_i$. Letting $L\to\infty$, we see that $m^{(\infty)}(E)$ satisfies the equation
\be\label{eq:minf}
m^{(\infty)}(E) = \frac{1}{(2\pi)^d}\int_{\bp\in [-\pi,\pi]^d}\frac{\dd \bp}{\lambda e(\bp) - E - m^{(\infty)}(E)} .
\ee
By comparing the sum in \eqref{eq:mL} and the integral in \eqref{eq:minf}, we obtain that for any $w\in \C$ with $w=m_{sc}(z)+\oo(1)$, 
$$\frac{1}{(2\pi)^d}\int_{\bp\in [-\pi,\pi]^d}\frac{\dd \bp}{\lambda e(\bp) - E - w}-\frac{1}{N}\sum_{\bp\in \T_L^d} \frac{1}{\lambda e(\bp) - z - w} \lesssim \lambda^2/L^{2}.$$ 
Then, using the stability of the self-consistent equation \eqref{eq:mL}, we can derive that
\be\label{eq:mL-minf_appd}m^{(L)}(z)-m^{(\infty)}(E)=\OO(\lambda^2/L^2).\ee 
The existence of $M^{(\infty)}_{xy}$ follows from the observation that   
$$M^{(L)}_{xy}(z) = \frac{1}{N}\sum_{\bp\in \T_L^d} \frac{e^{\ii \bp\cdot (x-y)}}{\lambda e(\bp) - z - m^{(L)}(z)}\to M^{(\infty)}_{xy}(E)=\frac{1}{(2\pi)^d}\int_{\bp\in [-\pi,\pi]^d}\frac{e^{\ii \bp\cdot (x-y)}\dd \bp}{\lambda e(\bp) - E - m^{(\infty)}(E)}$$
as $L\to \infty$. Then, we can prove the estimate \eqref{eq:ML-Minf} by showing that for any $L'>L$ and $x,y \in \Z_L^d$,
$$|M^{(L)}_{xy}(z)-M^{(L')}_{xy}(E)|\lesssim \left(\eta + \lambda^2 L^{-2} \right)(C\lambda)^{|x-y|}+(C\lambda)^{L/2}.$$
This follows easily from the expansion \eqref{eq:expandM} and the estimate \eqref{eq:mL-minf_appd}.



For the Anderson orbital model, the proofs follow the same structure as those for the block Anderson model. The only modification is that we use Fourier series on $\wt \Z_n^d$ instead, rewriting \eqref{self_m} and \eqref{def_G0} as
\be\label{eq:AO_M}
\frac{1}{n^d}\tr   \frac{1}{\lambda (2d-\Delta_n) -z- m(z)} = m(z),\quad 
M(z):= \frac{1}{\lambda (2d-\Delta_n) -z- m(z)}\otimes I_W.
\ee
We omit the details.

\subsection{Proof of \Cref{lem:estS+}}
Recall that $S^+$ can be expressed as in \eqref{eq:S+}, where $F^+_n$ is determined by $M^{(L)}$. Using \Cref{lem:estM} and the expansion \eqref{eq:expMLn}, it is straightforward to verify (details omitted) that 
$(S_\infty^{+})_{xy}(E)$ exists and for any $L'>L$, $x\in [0]$ and $y\in \Z_L^d$, 
$$\left|(S^+_{W,L})_{xy}(z)-(S^+_{W,L'})_{xy}(E)\right|\lesssim \frac{\eta+\lambda^2\al(W)/ L^{2}}{W^{d}}e^{-{|[x]-[y]|_n}/{C}} + e^{-n/C}.$$
Taking $L'\to \infty$ concludes \eqref{eq:S+-Sinf}.

\subsection{Proof of \Cref{lem esthatTheta}}
Since all the arguments in the proof are basic, we will outline the proof without providing all the details. We first claim the following result.

\begin{lemma}\label{lem:estM_appd}
 In the setting of \Cref{lem:estM}, for any $K\in \N$, there exist two sequences of analytic functions \smash{$\{m_i^{(L)}:i=1,2,\ldots,K\}$ and $\{m_i^{(\infty)}:i=1,2,\ldots,K\}$} on the upper half complex plane such that the following estimates hold: 
\begin{align}
&m^{(L)}(z)=m_{sc}(z)+\sum_{i=1}^K m_{2i}^{(L)}(z)\lambda^{2i} + \OO(\lambda^{2K+2}),\label{eq:exp_mm0}\\
&m^{(\infty)}(E)=m_{sc}(E)+\sum_{i=1}^K m_{2i}^{(\infty)}(E)\lambda^{2i} + \OO(\lambda^{2K+2}),\label{eq:exp_mm}\\
& |m_2^{(L)}(z)-m_2^{(\infty)}(E)|\lesssim \eta,\quad \max_{i=2}^K|m_{2i}^{(L)}(z)-m_{2i}^{(\infty)}(E)|\lesssim \eta + \al(W) /L^{2}.\label{eq:mL-minf}
\end{align}
\end{lemma}
\begin{proof}
The estimates \eqref{eq:exp_mm0}--\eqref{eq:mL-minf} for the Wegner orbital model follow easily from the Taylor expansion of  \eqref{self_mWO}. For the block Anderson and Anderson orbital models, the estimate \eqref{eq:exp_mm0} follows from an asymptotic expansion of the self-consistent equation \eqref{self_m} in terms of $\lambda$. 
It remains to show that \smash{$m_{2i}^{(\infty)}(E)=\lim_{L\to\infty}m_{2i}^{(L)}(E+\ii 0_+)$} exists and satisfies the estimate \eqref{eq:mL-minf}.
By \eqref{eq_m1}, we have that 
$$m_2^{(L)}(z)=\frac{2dm_{sc}^2(z)}{1-m_{sc}^2(z)},$$ 
so $m_2^{(\infty)}(E)$ clearly exists and
$|m_2^{(L)}(z)-m_2^{(\infty)}(E)|\lesssim \eta$. In general, for $2\le i \le K$, we can derive that every $m_{2i}^{(L)}(z)$ is a polynomial of $$I_{k}^{(L)}:=\frac{1}{N}\tr (\Psi^k) ,\quad 2\le k\le K,$$ with coefficients being analytic functions of $m_{sc}(z)$. For the block Anderson model, we can express these functions as the following Fourier series:  
\be\label{eq:IkL_inf} I_{k}^{(L)} = \frac{1}{N}\sum_{\bp\in \T_L^d}e(\bp)^k \stackrel{L\to \infty}{\longrightarrow} I_{k}^{(\infty)}:=\frac{1}{(2\pi)^d}\int_{\bp\in [-\pi,\pi]^d}e(\bp)^k\dd \mathbf p \, ,
\ee
where we recall that $e(\bp)$ is defined as $e(\bp):=\sum_{i=1}^d 2\cos p_i$. Therefore, $m_{2i}^{(\infty)}(E)$ in \eqref{eq:exp_mm} exists and is  obtained by replacing $m_{sc}(z)$ and \smash{$I_{k}^{(L)}$} in the polynomial representation of \smash{$m_{2i}^{(L)}(z)$} by $m_{sc}(E)$ and \smash{$I_{k}^{(\infty)}$}, respectively. 
Furthermore, by comparing the sum and the integral in \eqref{eq:IkL_inf}, we find that 
$$|I_{k}^{(L)}-I_{k}^{(\infty)}|\lesssim L^{-2},$$ 
which implies \eqref{eq:mL-minf}. The proof for the Anderson orbital model is similar by using the Fourier series representation of \eqref{eq:AO_M}.
\end{proof}

As in the proof of \Cref{lem theta}, for the sake of clarity, we will focus on the case of block Anderson and Anderson orbital models, noting that the proof for the Wegner orbital model is simpler. 
For block Anderson and Anderson orbital models, $\zthn^{(W,L)}(z)$ takes the form \eqref{exp_Theta3}, where $\wt P^\perp (1-\KM(z))^{-1} \wt P^\perp$ can be written as \eqref{eq:PKMP}. Moreover, $f_{\eta,W,L} (\bp)$ in \eqref{eq:fetap} writes
\be\label{eq:fetap2}
f_{\eta,W,L} (\bp)=\bigg\{\frac{\eta}{\im m^{(L)} + \eta} + a_1^{(W,L)}(z) \blam^{-1}\lambda_1(\bp) + \sum_{k\ge 2} a_k^{(W,L)}(z)\blam^{-k} \lambda_k(\bp) \bigg\}^{-1},
\ee
where $\lambda_k(\bp)\equiv \lambda_k(\bp,n)$ was defined in \eqref{eq_lambdak} and there is an constant $C>0$ depending only on $\kappa$ such that 
\be\label{eq:absak}
|a_k^{(W,L)}(z)|\le C^k \quad \text{for all $k$}. 
\ee
By \Cref{lem:estM_appd}, it is straightforward to check by definition that $a_k^{(W,L)}(z)$ converges as $L \to \infty$ and $\eta\to 0$ (with $W$ fixed) for every fixed $k\in \N$. Together with \eqref{eq:absak}, we see that for sufficiently large $W$, $f_{\eta,W,L} (\bp)$ converges to the limit 
$$f_{0,W,\infty} (\bp)=\bigg\{\sum_{k=1}^\infty a_k^{(W,\infty)}\blam^{-k}\lambda_k(\mathbf p,\infty)\bigg\}^{-1}$$ 
uniformly in $\mathbf p \in [-\pi,\pi]^d$, where for all $k$,
\be\label{eq:absak2}
|a_k^{(W,\infty)}(z)|\le C^k, \quad  \lambda_k({\bp},\infty)=\sum_{\mathbf k \in  \Z^d:|\mathbf k|=k}(1-\cos(\bp\cdot \mathbf k)). 
\ee
Then, we can derive that for any $x\in \Z^d$, 
\be\label{eq:I0W1}I_{0,W,L}(x):=\frac{1}{n^{d}}\sum_{\bp\in  \T_n^d\setminus \{0\}} f_{0,W,L} (\bp)e^{\ii \bp\cdot x}\xrightarrow{n \to \infty} I_{0,W,\infty}(x):=\frac{1}{(2\pi)^d}\int_{[-\pi,\pi]^d}f_{0,W,\infty}(\bp)e^{\ii \bp\cdot x}\dd \bp.
\ee
This proves the existence of $\zthn^{(W,\infty)}_{xy}$ for any $x\in [0]$ and $y\in \Z^d$ since $\zthn^{(W,\infty)}_{xy} = W^{-d}I_{0,W,\infty}([y])$. 

Now, to show \eqref{Theta-wh1}, it suffices to prove that  
\be\label{eq:I0W}
I_{0,W,\infty}(x)\lesssim  \frac{\blam\log(|x|+2)}{(|x|+1)^{d-2}},\quad \forall x\in \Z^d.
\ee
At $x=0$, we trivially have $I_{0,W,\infty}(0)=\OO(1)$. For $x\in \Z^d\setminus \{0\}$, suppose $|x_1| =\max_{i=1}^d |x_i| \ge 1$ without loss of generality. 
Applying integration by parts with respect to $p_1$ for $(d-3)$ times, we get that 
$$I_{0,W,\infty}(x):=\frac{1}{(2\pi)^d}\left(\frac{\ii}{x_1}\right)^{d-3}\int_{ [-\pi,\pi]^d}\left[\partial_{p_1}^{d-3}f_{0,W,\infty}(\bp)\right]e^{\ii \bp\cdot x}\dd \bp=I_{in}+I_{out}.$$
Then, we divide this expression into two parts $I_{in}$ and $I_{out}$ according to whether $|p_1| \le |x_1|^{-1}$ or not. For $I_{out}$,  applying another integration by part with respect to $p_1$, we get that  
$$ I_{out} \lesssim \frac{\blam}{|x_1|^{d-2}}\left[\int_{[-\pi,\pi]^{d-1}} \frac{\dd \bq}{|x_1|^{-(d-1)}+|\bq|^{d-1}} + \int_{  [-\pi,\pi]^{d}} \frac{\mathbf 1(|p_1|\ge |x_1|^{-1}) \dd \bp}{|x_1|^{-d}+|\bq|^d}\right] \lesssim \frac{\blam(1+\log |x_1|)}{ |x_1|^{d-2}},$$
where $\bq:=(p_2, \ldots, p_d) $, $\wh x^{(1)}:=(x_2,\cdots, x_n)$, and we have used that $\partial_{p_1}^{k}f_{0,W,\infty}(\bp)\lesssim \blam |\bp|^{-(k+2)}$ for any fixed $k\in \N$. For $I_{in}$, we bound it directly as  
$$ I_{in}\lesssim \frac{\blam}{|x_1|^{d-3}}\left[\int_{|p_1|\le |x_1|^{-1},|\bq|\le |x_1|^{-1}} \frac{\dd \bp}{|\bp|^{d-1}}+ \int_{|p_1|\le |x_1|^{-1},|\bq|> |x_1|^{-1}} \frac{\dd \bp}{|\bp|^{d-1}}\right]\lesssim \frac{\blam(1+\log |x_1|)}{|x_1|^{d-2}}.$$
Combining the above two estimates yields \eqref{eq:I0W}.

To show \eqref{Theta-wh}, for $x\in \Z_n^d$, we need to bound 
\begin{align*}
    &I_{0,W,\infty}(x)-I_{\eta,W,L}(x) = I_1+I_2+I_3+I_4,  
\end{align*}
where the four terms are defined as follows:  
\begin{align*}
    &  I_1:=\frac{1}{(2\pi)^d}\sum_{\bp_0\in  \T_n^d\setminus \{0\}} e^{\ii \bp_0\cdot x} \int_{\bp\in \cal O_n(p_0)}\left[\wt f_{0,W,\infty} (\bp)e^{\ii (\bp-\bp_0)\cdot x}- \wt f_{0,W,\infty} (\bp_0)\right]\dd \bp, \\
    & I_2:= \frac{1}{(2\pi)^d}\int_{\bp\in [-\pi,\pi]^d\setminus \cal O_n(0)}\left[f_{0,W,\infty}(\bp)-\wt f_{0,W,\infty}(\bp)\right]e^{\ii \bp\cdot x}\dd \bp , \\
    & I_3:= \frac{1}{n^{d}}\sum_{\bp_0\in  \T_n^d\setminus \{0\}} \left[\wt f_{0,W,\infty} (\bp_0)-f_{\eta,W,L} (\bp_0)\right]e^{\ii \bp_0\cdot x} ,\\
    & I_4:=\frac{1}{(2\pi)^d}\int_{\bp\in \cal O_n(0)}f_{0,W,\infty} (\bp)e^{\ii \bp\cdot x}\dd \bp.
\end{align*}
Here, $\cal O_n(\bp_0)$ denotes the box centered at $\bp_0$ and with side length $2\pi/n$, and $\wt f_{0,W,\infty} (\bp)$ is defined as 
\be\label{eq:fetap3}
\wt f_{0,W,\infty} (\bp):=\bigg\{   a_1^{(W,\infty)}(E)\blam^{-1} \lambda_1(\bp,\infty) + \sum_{k=2}^K  a_k^{(W,\infty)}(E)\blam^{-k}\lambda_k(\bp,\infty)\bigg\}^{-1}
\ee
for a large integer $K\ge 2$. First, it is trivial to see that 
\be\label{eq:I4}|I_4|\lesssim \blam n^{-(d-2)}.\ee 
For $I_2$, there exists a constant $C>0$ such that 
$$f_{0,W,\infty}(\bp)-\wt f_{0,W,\infty}(\bp)\le (C/\blam)^{K-1} |\bp|^{-2}.$$ 
With an integration by parts argument as that in the proof of \Cref{lem theta}, we obtain that 
\be\label{eq:I2}|I_2|\lesssim (C/\blam)^{K-1} (|x|+1)^{-(d-2)}\log n .\ee
Using \Cref{lem:estM_appd}, we can check that 
$$|a_k^{(W,\infty)}(E) -a_k^{(W,L)}(z) |\lesssim \eta + \lambda^2 \al(W)/L^{2},\quad 1\le k \le K.$$
With this estimate, we can get that  
$$\wt f_{0,W,\infty}(\bp_0) - f_{\eta,W,L}(\bp_0) \lesssim \blam^2\frac{\eta}{|\bp_0|^4} + \blam\frac{\lambda^2 \al(W)/L^{2}}{|\bp_0|^2}.$$
Again, using a similar summation by parts argument, we can obtain that
\be\label{eq:I3}|I_3|\lesssim  \blam^2 \frac{\eta\log n}{ (|x|+1)^{d-4}}
+ \blam\frac{ (\lambda^2 \al(w)/L^{2}) \log n}{ (|x|+1)^{d-2}} .\ee
Finally, for $I_1$, each integral inside the summation writes  
$$g(\bp_0):= \int_{\bp\in \cal O_n(p_0)}\left[\wt f_{0,W,\infty} (\bp)e^{\ii (\bp-\bp_0)\cdot x}- \wt f_{0,W,\infty} (\bp_0)-\left(\nabla_{\bp}\wt f_{0,W,\infty}(\bp_0)+\ii x\right)\cdot (\bp-\bp_0)\right]\dd\bp.$$
It is straightforward to check that 
$$ g(\bp_0) \lesssim \frac{\blam}{n^{d+2}} \left(\frac{|x|^2}{|\bp_0|^2}+\frac{1}{|\bp_0|^4}\right).$$
Again, applying the summation by parts argument to $I_1$, we obtain that
\be\label{eq:I1} I_1 \lesssim \frac{\blam n^{-2}\log n}{(|x|+1)^{d-4}}.\ee
Combining the above estimates \eqref{eq:I4}--\eqref{eq:I1}, we obtain that 
\begin{align*}
    &I_{0,W,\infty}(x)-I_{\eta,W,L}(x) \lesssim \left(\blam\eta + n^{-2}\right)\frac{\blam\log n}{ (|x|+1)^{d-4}}.  
\end{align*}
This concludes \eqref{Theta-wh} together with the fact $\zthn^{(W,\infty)}_{xy}(E)-\zthn^{(W,L)}_{xy}(z)= W^{-d} [I_{0,W,\infty}([y])-I_{\eta,W,L}([y])] $.

\subsection{Proof of \Cref{lem redundantagain2}}
To show \eqref{redundant again2}, it suffices to prove that (recall $I_{0,W,\infty}$ defined in \eqref{eq:I0W1})
$$ \sum_{x\in \Z^d}I_{0,W,\infty}(x)\Sele({x,y}) \lesssim \blam \frac{\psi  \log (|y|+2)}{(|y|+1)^d},\quad \forall y \in \Z^d.$$
With \eqref{self_sym1}, \eqref{self_decay}, and \eqref{weaz}, we can establish this estimate with very similar arguments to those used in the proofs of \Cref{lem:label_diffusive} and \Cref{lem cancelTheta}. Therefore, we omit the details. 
\subsection{Proof of \Cref{lem redundantagain}}\label{sec:pf_redundantagain}
From the properties \eqref{self_sym}, \eqref{4th_property0}, and \eqref{3rd_property0}, we see that \eqref{self_sym1}--\eqref{self_zero} hold with $\psi=\sizeself(\Sele_i)/\blam$ and $\psi_0=\sizeself(\Sele_i)\left(  \eta  + t_{Th}^{-1}  \right)$ for $i=0,1,\ldots,k$. Hence, by \eqref{thetaxy_renorm}, we have that 
\be\label{eq:ThetaE:proj} (\zthn^{\LK} \cal E_{\LK})_{[x][y]} \prec \psi(\Sele)\qq{[x]-[y]}^{-d},\quad \text{for}\  \ \Sele\in \{\Sele_0,\Sele_1,\ldots,\Sele_k\}.\ee
By property (iii) of \Cref{collection elements}, we have $\zthn \Sele = (\zthn^{\LK} \cal E_{\LK})\otimes \mathbf E$ for $\WO$, which implies that
\be\label{eq:ThetaE:proj2}  \zthn\Sele_1\zthn \cdots \zthn \Sele_k \zthn = \left(\zthn^{\LK}\Sele_1^{\LK}\zthn^{\LK} \cdots \zthn^{\LK}\Sele_k^{\LK}\zthn^{\LK} \right)\otimes \mathbf E.\ee
On the other hand, using the identity $S(0)^2=S(0)$, we get that for $\BA$ and $\AO$,
$$\zthn\Sele_1\zthn\Sele_2\zthn \cdots \zthn \Sele_k \zthn = \zthn [S(0)\Sele_1 S(0)]\zthn [S(0)\Sele_2 S(0)]\zthn \cdots \zthn  [S(0)\Sele_k S(0)] \zthn ,$$
which also gives \eqref{eq:ThetaE:proj2}. 
With \eqref{thetaxy}, \eqref{eq:ThetaE:proj} and \eqref{eq:ThetaE:proj2}, we readily conclude the second estimate in \eqref{BRB}. 
To show the first estimate in \eqref{BRB}, we write that 
$$ \zthn(\Sele)=\left[(1-\zthn^{\LK} \cal E_{\LK})^{-1}\zthn^{\LK}\right]\otimes \mathbf E.$$
Notice that \eqref{eq:ThetaE:proj} implies
\be\label{eq:ThetaE:linf}\|\zthn^{\LK} \cal E_{\LK}\|_{\ell^\infty(\wt\Z_n^d)\to \ell^\infty(\wt\Z_n^d)}\prec \psi(\Sele) \le W^{-\fd}. \ee
Then, we use the following expansion of $\zthn(\Sele)$ for any fixed $K\in \N$:
\be\label{eq:TaylorSele}\zthn(\Sele)= \frac{1}{1-(\zthn\Sele)^{K+1}}
\sum_{k=0}^K (\zthn\Sele)^k \zthn.
\ee
Using \eqref{thetaxy}, \eqref{eq:ThetaE:proj}, and \eqref{eq:ThetaE:linf}, we get from this expansion that 
$$ \zthn_{xy}(\Sele) = \sum_{k=0}^K [(\zthn\Sele)^k \zthn]_{xy} + \OO_\prec (W^{-K\fd}) \prec B_{xy},$$
as long as we take $K$ sufficiently large so that $W^{-K\fd} \le \blam/(W^2L^{d-2})$. 

Finally, recall that when $\eta\ge t_{Th}^{-1}$, the bounds \eqref{thetaxy} and \eqref{thetaxy_renorm} still hold if we replaced $\zthn$ with $\thn$. Then, using a similar argument as above, we conclude \eqref{BRB} for $\thn(\Sele_0)$ and the labeled diffusive edge in \eqref{eq:labeldiff}. 

\subsection{Proof of \Cref{lem:thn-zthn}}\label{subsect:pfthn-zthn}
Suppose we have replaced a $\zthn_{\al\beta}$ edge with a $\thn_{\al\beta}$ edge in a graph $\cal G$ of $\Sele$ (the proof is the same if we have replaced $\thn_{\al\beta}$ with \smash{$\zthn_{\al\beta}$}). Denoting the new graph by $\cal G'$, we can write $\Sele^{new}=\Sele-\cal G+ \cal G'$. It is easy to see that (i) and (iii) of \Cref{collection elements} still hold for $\Sele^{new}$. Both $\thn_{xy}$ and \smash{$\zthn_{xy}$} satisfy \eqref{self_sym}, so the new term $\Sele^{new}$ still satisfies \eqref{self_sym} (recall the explanation below \eqref{eq:sizeselfG}). Furthermore, by \Cref{dG-bd}, $\Sele^{new}$ still satisfies \eqref{4th_property0} with $\sizeself(\Sele^{new})=\sizeself(\Sele)$. It remains to prove the property \eqref{3rd_property0} for $\Sele^{new}$, which can be reduced to showing 
\be\label{eq:G-G'}
\sum_{x} \left(\cal G_{xy}-\cal G'_{xy}\right) \prec  \sizeself(\Sele) \eta \quad \forall \eta\ge t_{Th}^{-1},\ y\in [0].
\ee

First, we assume that $\cal G$ contains no labeled diffusive edge. Let $\cal G_{xy;\al\beta}$ be the graph obtained by setting $\al,\beta$ to be external vertices and picking out the \smash{$\zthn_{\al\beta}$} edge, and let \smash{$\wt{\cal G}_{xy}$} be a graph obtained by setting $\al,\beta$ to be internal vertices again in $\cal G_{xy;\al\beta}$. By (i) of \cref{collection elements},  \smash{$\wt{\cal G}_{xy}$} is still a doubly connected graph since \smash{$\zthn_{\al\beta}$} is redundant. Recalling \eqref{eq:diffzthn-thn1} or \eqref{eq:diffzthn-thn2}, we denote 
$$\al_\eta:=\zthn_{\al\beta}(z)-\thn_{\al\beta}(z) \lesssim (N\eta)^{-1}.$$
Then, applying \Cref{dG-bd} to $|\wt{\cal G}_{\al\beta}|$, we get that
\begin{align}\label{eq:E-E'}
\cal G_{xy}-\cal G'_{xy} &=\al_\eta \wt{\cal G}_{xy} \prec \al_\eta \frac{W^{2d-4}\size(\wt{\cal G}_{xy})}{\qq{x-y}^{2d-4}}  \le \frac{\sizeself(\Sele)W^d}{\blam^2 N\eta} \frac{W^{d-4}}{\qq{x-y}^{2d-4}} \le  \sizeself(\Sele)\eta\cdot \frac{W^{d-4}}{\qq{x-y}^{2d-4}},
\end{align}
where we used $\size(\wt{\cal G}_{\al\beta})=\size({\cal G}_{xy})/\heta$ and $\sizeself(\Sele)\ge \blam W^d\size({\cal G}_{xy})$ in the second step. Summing the above equation over $x$ conclude \eqref{eq:G-G'}, 
which also shows that $\Sele^{new}$ is indeed a $\self$. In addition, by \eqref{eq:E-E'}, $\Sele-\Sele^{new}$ satisfies that
\be\label{eq:induc_E}
(\Sele-\Sele^{new})_{xy} \prec \frac{\sizeself(\Sele)W^d}{\blam^2 N\eta} \cdot \frac{W^{d-4}}{\qq{x-y}^{2d-4}}.
\ee

Second, suppose $\Sele$ contains labeled diffusive edges formed from $\selfs$, say $\Sele_1,\ldots, \Sele_k$. As an induction hypothesis, suppose we have proved that replacing an arbitrary  \smash{$\zthn$} edge in $\Sele_i$ by a $\thn$ edge yields an expression in $\Sele_i^{new}$ that is also a $\self$ and satisfies that (owing to \eqref{eq:induc_E})
\be\label{eq:induc_EEE}
(\Sele_i)_{xy}-(\Sele_i^{new})_{xy} \prec \frac{\sizeself(\Sele_i)W^d}{\blam^2 N\eta}\cdot \frac{W^{d-4}}{\qq{x-y}^{2d-4}},\quad \forall i\in \qqq{k}, \ x,y \in \Z_L^d. 
\ee
If we have modified a $\zthn_{\al\beta}$ edge that does not belong to any labeled diffusive edge, then the proof of \eqref{eq:G-G'} and \eqref{eq:induc_E} is exactly the same as above. Otherwise, assume that \smash{$\zthn_{\al\beta}$} belongs to an $(s,\Sele_i)$-labeled diffusive edge for $s\in \{\emptyset,\circ\}$ and $i\in \qqq{k}$, or an $[s_1,\Sele_{i_1},s_2,\Sele_{i_2}, \ldots, s_\ell,\Sele_{i_\ell},s_{\ell+1}]$-labeled diffusive edge for $(s_1,\ldots, s_{\ell+1})\in \{\emptyset,\circ\}^{\ell+1}$ and $  (i_1,\ldots, i_{\ell})\in \qqq{k}^{\ell}$, denoted by $\Theta_{ab}$. Then, using \eqref{eq:induc_EEE}, \Cref{lem:label_diffusive}, and the arguments in the proof of \Cref{lem redundantagain}, we can show that: for the $(s,\Sele_i)$-labeled diffusive edge $\Theta_{ab}$, the new labeled diffusive edge $\Theta_{ab}'$ satisfies  
\be\nonumber
\max_{a,b}\left|\Theta_{ab}-\Theta_{ab}'\right| \prec (N\eta)^{-1}; 
\ee
for the $[s_1,\Sele_{i_1},s_2,\Sele_{i_2}, \ldots, s_\ell,\Sele_{i_\ell},s_{\ell+1}]$-labeled diffusive edge $\Theta_{ab}$, the new labeled diffusive edge $\Theta_{ab}'$ satisfies  
\be\nonumber
\max_{a,b}\left|\Theta_{ab}-\Theta_{ab}'\right| \prec \frac{1}{N\eta}\prod_{j=1}^\ell \sizeself(\Sele_{i_j})\, . 
\ee
Since the proof of these two estimates is straightforward, we omit the details. Then, with the same argument as above, we obtain that $\cal G-\cal G'$ satisfies \eqref{eq:E-E'}, which further implies the estimate \eqref{eq:G-G'}. By induction, we conclude the proof of this lemma.  

\section{Proof of \Cref{lem G<T}}\label{appd:MDE}

For the Wegner orbital model, the estimate \eqref{diagG largedev} has been proved in Lemma 5.3 of \cite{delocal}. 
It remains to consider the block Anderson and Anderson orbital models. We will use the simplified notation of generalized matrix entries: given a matrix $\cal A$ and any vectors $\mathbf u,\bv$, we denote $ \cal A_{\mathbf u\mathbf v}:= \bu^* \cal A\bv$ and $\cal A_{x\mathbf v}:= \mathbf e_x^* \cal A\bv$, where $\mathbf e_x$ is the standard basis unit vector along the $x$-th direction. 

We will use an argument based on the matrix Dyson equation, as developed in \cite{He2018}. We define
$$ \Pi(G):=I+zG+\cal S(G) G- AG,$$
where $\cal S(G)$ is a diagonal matrix defined as $\cal S(G)_{xy}:=\delta_{xy}\sum_{\al}S_{x\al}G_{\al\al}$ and we abbreviated $A= \lambda\Psi$ for $\Psi\in \{\Psi^{\BA},\Psi^{\AO}\}$. For any $p\in \N$, we bound the moments
$\E |\Pi(G)_{\bv \bw}|^{2p},$ 
where $\bv,\bw\in \C^N$ are deterministic unit vectors. 
With $ I+zG-AG=VG$ and Gaussian integration by parts, we obtain that 
\begin{align*} \E |\Pi(G)_{\bv \bw}|^{2p}&= \E\left[\bv^* VG \bw\Pi(G)_{\bv \bw}^{p-1} \overline \Pi(G)_{\bv \bw}^{p}\right]+ \E  \left[\bv^*\cal S(G)G \bw\Pi(G)_{\bv \bw}^{p-1} \overline \Pi(G)_{\bv \bw}^{p}\right]\\
    &=\sum_{x,\al}\E\left[\overline v(x)V_{x\al}G_{\al \bw} \Pi(G)_{\bv \bw}^{p-1} \overline \Pi(G)_{\bv \bw}^{p}\right] + \E  \left[\bv^*\cal S(G)G \bw\Pi(G)_{\bv \bw}^{p-1} \overline \Pi(G)_{\bv \bw}^{p}\right]\\
    &= \sum_{x,\al}\E\left[\overline v(x)S_{x\al}G_{\al \bw} \partial_{h_{\al x}}\left(\Pi(G)_{\bv \bw}^{p-1} \overline \Pi(G)_{\bv \bw}^{p}\right)\right].
\end{align*}
To simplify notations, we drop the complex conjugates of $\Pi(G)$, which play no role in the following proof, and estimate 
\begin{align}\label{eq:simple_cumu}
    &\sum_{x,\al}\E\left[\overline v(x)S_{x\al}G_{\al \bw} \partial_{h_{\al x}}\left(\Pi(G)_{\bv \bw}^{2p-1}  \right)\right]=(2p-1)\sum_{x,\al} \E\left[\overline v(x)S_{x\al}G_{\al \bw} \left(\partial_{h_{\al x}}\Pi(G)_{\bv \bw}\right)\Pi(G)_{\bv \bw}^{2p-2}\right] .
 \end{align}
Direct calculations yield that 
 $$\partial_{h_{\al x}}\Pi(G)_{\bv \bw}=-\Pi(G)_{\bv\al} G_{x\bw} + \bar v(\al) G_{x\bw} - \sum_y \overline{v}(y) \sum_\beta S_{y\beta}G_{\beta \al}G_{x\beta} G_{y\bw}.$$
Thus, to control \eqref{eq:simple_cumu}, we need to bound
    \begin{align}
   \sum_{x,\al}\E \bigg[\overline v(x)S_{x\al} G_{\al \bw} \bigg( \Pi(G)_{\bv\al} G_{x\bw} -\bar v(\al) G_{x\bw} + \sum_y \overline{v}(y) \sum_\beta S_{y\beta}G_{\beta \al}G_{x\beta} G_{y\bw}\bigg)  \Pi(G)_{\bv \bw}^{2p-2}\bigg].\label{eq:MDE_three}
\end{align}

For the proof of \Cref{lem G<T}, it suffices to take $\bv=\mathbf e_{x_0}$ and $\bw=\mathbf e_{y_0}$ for some $x_0,y_0\in \Z_L^d$. We assume that $\|\Pi(G)\|_{\max}\prec \Psi$ for some deterministic parameter $\Psi>0$ (noting that a priori, this estimate holds for $\Psi=\eta^{-2}$). Then, using \eqref{initialGT} and Cauchy-Schwarz inequality, we can bound the three terms in \eqref{eq:MDE_three} as 
\begin{align*}
\sum_{\al}\E \left|S_{x_0\al} G_{\al y_0} G_{x_0y_0}  \Pi(G)_{x_0\al}  \Pi(G)_{x_0y_0}^{2p-2}\right| & \prec \Psi \sum_{\al} \E \left|S_{x_0\al} G_{\al y_0}\right|\cdot \left| \Pi(G)_{x_0y_0}^{2p-2}\right| 
 \prec \Psi\Phi \cdot \E \left| \Pi(G)_{x_0y_0}\right|^{2p-2},\\
\E \left|S_{x_0x_0} (G_{x_0y_0})^2 \Pi(G)_{x_0y_0}^{2p-2}\right| & \prec W^{-d}  \E \left| \Pi(G)_{x_0y_0}\right|^{2p-2},\\
\sum_{\al,\beta}\E \left|S_{x_0\al} S_{x_0\beta} G_{\al y_0}  G_{\beta \al}G_{x_0\beta} G_{x_0y_0}  \Pi(G)_{x_0y_0}^{2p-2}\right| &\prec \Phi^2  \sum_{\al}\E \left|S_{x_0\al}  G_{\al y_0}  \Pi(G)_{x_0y_0}^{2p-2}\right| \prec \Phi^3 \E \left| \Pi(G)_{x_0y_0}\right|^{2p-2}.
\end{align*}
Applying these estimates to \eqref{eq:MDE_three}, we obtain that 
$$ \E |\Pi(G)_{x_0y_0}|^{2p} \prec \left(\Psi \Phi + W^{-d} +\Phi^3\right)\E \left| \Pi(G)_{x_0y_0}\right|^{2p-2}.$$
Using H{\"o}lder's and Young's inequalities, we obtain from the above estimate that
$$ \E |\Pi(G)_{x_0y_0}|^{2p} \prec \left(\Psi \Phi + W^{-d} +\Phi^3\right)^p $$ 
for any fixed $p\in\N$. Then, applying Markov's inequality and taking a union bound over all $x_0,y_0\in \Z_L^d$, we get a self-improving estimate:
$$\|\Pi(G)\|_{\max}\prec \Psi \ \Rightarrow \ \|\Pi(G)\|_{\max} \prec \left(\Psi \Phi\right)^{1/2} + W^{-d/2} +\Phi^{3/2}.$$
Starting from $\Psi=\eta^{-2}$, iterating this estimate for $\OO(1)$ many times, we obtain that 
\be\label{eq:max_Pi}
\|\Pi(G)\|_{\max} \prec \Phi.
\ee

Next, with \eqref{def_G0} and the definition of $\Pi(G)$, we write that 
\be\label{eq:G-M} 
G - M= - M\left[ \Pi(G) -(\cal S(G)-m)G\right] . 
\ee
With the simple fact $\|M\|_{\ell^\infty(\Z_L^d) \to \ell^\infty(\Z_L^d)}=\max_{x}\sum_y |M_{xy}|\lesssim 1$ due to \eqref{Mbound} and \eqref{Mbound_AO},
applying \eqref{initialGT} and \eqref{eq:max_Pi} to \eqref{eq:G-M}, we obtain that
\begin{align}
    G_{xy} - M_{xy}&= \OO_\prec(\Phi) + [M  (\cal S(G)-m)(G-M)]_{xy}+  [M(\cal S(G)-m)M]_{xy} \nonumber\\
    &= \OO_\prec(\Phi + W^{-\delta_0} \|\cal S(G)-m\|_{\max})+  [M(\cal S(G)-m)M]_{xy}, \quad \forall x,y\in\Z_L^d.\label{eq:MDE}
\end{align}
Applying $\cal S$ to the above equation, we obtain that
\begin{align*}
    \cal S(G)_{xx}-m =\sum_\al S_{x\al}(G-M)_{\al\al} =\OO_\prec(\Phi + W^{-\delta_0} \|\cal S(G)-m\|_{\max})+  \sum_{\al,\beta}S_{x\al}(M_{\al \beta})^2(\cal S(G)-m)_{\beta\beta}.
\end{align*} 
This shows that the vector ${\Lambda}:=(\cal S(G)_{xx}-m)_{x\in \Z_L^d}\in \C^N$ satisfies
$$ (1-SM^+){\Lambda} =\OO_\prec(\Phi + W^{-\delta_0} \|\Lambda\|_{\max}).$$
Using $\|(1-SM^+)^{-1}\|_{\ell^\infty(\Z_L^d) \to \ell^\infty(\Z_L^d)}\lesssim 1$ (because $(1-SM^+)^{-1}=1+S^+M^+$ and $\|S^+M^+\|_{\ell^\infty \to \ell^\infty}\lesssim 1$ by \eqref{Mbound}, \eqref{Mbound_AO}, and \eqref{S+xy}), we obtain from the above equation that 
$$ \|\Lambda\|_{\max} \prec \Phi + W^{-\delta_0} \|\Lambda\|_{\max} \ \Rightarrow \ \|\Lambda\|_{\max} \prec \Phi.$$
Plugging it back into \eqref{eq:MDE}, we conclude \eqref{diagG largedev}.

\section{Examples of self-energy and vertex renormalizations}\label{sec:examples}

\subsection{Example of self-energy renormalization}

To help the reader understand the sum zero property \eqref{4th_property0}, we provide examples of $\selfs$ for the block Anderson model, whose $T$-expansion takes the most intricate form among the three models. 
As discussed in \cite[Section 4]{BandI}, the fourth order $\self$ (corresponding to the $T$-expansion up to order $2d$ under the notations of \Cref{defn genuni_BA}) vanishes for random band matrices in the Gaussian case with zero fourth cumulant, while the next order---the sixth order $\self$---is too complicated to be calculated explicitly. In fact, this situation also occurs for our RBSO in \Cref{def: BM}. 
To provide an explicit example that illustrates the sum zero property of $\selfs$, we consider a more general distribution for the entries of $V$ that has a \emph{non-zero fourth cumulant} and a \emph{symmetric distribution} (ensuring that the third cumulant vanishes). Define the following notations:
$$
 \kappa_4:=W^{2d}\E \left|h_{11}\right|^4 -2,\quad S^{(4)}_{xy}:=\kappa_4\cdot (S_{xy})^2 = \kappa_4 W^{-d}S_{xy}.
$$
We will see that the ``fourth order" $\self$ $\Sele$ (with $\psi(\Sele)=\OO(\blam/W^d)$) for the $T$-expansion \eqref{mlevelTgdef_BA} of the block Anderson model is proportional to $\kappa_4$.  
This setting extends beyond that of \Cref{def: BM}, but as mentioned in \Cref{rmk:key}, we expect all our results to extend universally to other distributions satisfying certain tail conditions. Thus, this example will both illustrate the ``magic" cancellation in the $\self$ and serve as a useful example to help the reader understand the structures of $\selfs$. We also remark that since the example is not used in the main proof, we will not provide the full details of the derivation. 


First, we can replace the Gaussian integration by parts argument in the proof of Lemma \ref{2ndExp0} with a more general cumulant expansion formula. For example, we can use the following form of complex cumulant expansion formula stated in \cite[Lemma 7.1]{He_Mesoscopic}: 
let $h$ be a complex random variable for which all moments exist. The $(p,q)$-cumulant of $h$ is defined as
$$
\mathcal{C}^{(p,q)}(h)\deq (-\ii)^{p+q} \cdot \left(\frac{\partial^{p+q}}{\partial {s^p} \partial {t^q}} \log \E e^{\mathrm{i}sh+\mathrm{i}t\bar{h}}\right) \bigg{|}_{s=t=0}\,.
$$
Let $f:\mathbb C^2 \to \C$ be a smooth function, and we denote its holomorphic  derivatives by
$$f^{(p,q)}(z_1,z_2)\deq \frac{\partial^{p+q}}{\partial z_1^p \partial z_2^q} f(z_1,z_2)\,.$$ 
Then, for any fixed $l \in \N$, we have
\begin{equation} \label{5.16}
\E f(h,\bar{h})\bar{h}=\sum\limits_{p+q=0}^l \frac{1}{p!\,q!}\mathcal{C}^{(p,q+1)}(h)\E f^{(p,q)}(h,\bar{h}) + R_{l+1}\,,
\end{equation}
given that all integrals in (\ref{5.16}) exist. Here, $R_{l+1}$ is the remainder term depending on $f$ and $h$, which satisfies an error bound that is not important for our derivations here. 
Then, applying \eqref{5.16} to the derivation of \eqref{2ndExp0}, we obtain the following extension of \eqref{eq:2nd0}:
\begin{align}  
	 \czT_{\fa,\fb_1\fb_2}&=  \sum_x (\zthn M^0 S)_{\fa x }   \left(M _{x \fb_1 }\bar M_{ x \fb_2} +\Gc _{x \fb_1 }\bar M_{ x \fb_2}
	+M_{x \fb_1 }\Gc^{-} _{ x \fb_2}\right) \nonumber
	\\ 
	&  +\sum_{x,y,\beta} \zthn_{\fa x}  M_{xy} S_{y \beta}\left( \Gc^-_{xy} \Gc _{\beta \fb_1}\Gc^-_{\beta \fb_2} +  \Gc_{\beta\beta} \Gc _{y \fb_1} \Gc^- _{x\fb_2} \right)   \label{eq:self4-1}\\
	& - \frac{1}{2}\sum_{x,y,\beta} \zthn_{\fa x} M_{xy} S^{(4)}_{ y \beta}\partial_{h_{\beta y}}^2 \partial_{h_{ y\beta}} \left(  G _{\beta \fb_1}  \Gc^-_{x\fb_2} \right)+ \cal R_{\fa,\fb_1\fb_2}+ {\cal A}_{\fa,\fb_1\fb_2}+ \cal Q_{\fa,\fb_1\fb_2} . \label{eq:self4-2}
\end{align}
Here, $\cal R_{\fa,\fb_1\fb_2}$ and $\cal Q_{\fa,\fb_1\fb_2}$ represent $\{\fb_1,\fb_2\}$-recollision graphs and $Q$-graphs, respectively, and 
$\cal A_{\fa,\fb_1\fb_2}$ is a sum of higher order graphs that come from five or higher order cumulant terms. Their exact forms are irrelevant for the $\self$ we will derive. 



It remains to expand the two sums in \eqref{eq:self4-1} and \eqref{eq:self4-2}. We first consider the term in \eqref{eq:self4-2}:
\begin{align}  
&\quad \, - \partial_{h_{ \beta y}}^2  \partial_{h_{  y\beta}} \big( G _{\beta \fb_1}\Gc^-_{x \fb_2}\big) = 
 \partial_{h_{ \beta y}} \partial_{h_{  y\beta}}  \left(
  G _{\beta \beta}G_{ y \fb_1}\Gc^-_{x \fb_2}
  +
  G _{\beta \fb_1}G^-_{x  y }G^-_{\beta \fb_2}
  \right)\nonumber\\
 & =- \partial_{h_{ \beta y}} \left(G _{\beta y}G_{\beta \beta}G_{ y \fb_1}\Gc^-_{x \fb_2}
+ G _{\beta \beta}G_{yy}G_{\beta \fb_1}\Gc^-_{x \fb_2}
+ G _{\beta \beta}G_{y \fb_1}G^-_{x \beta}G^-_{y \fb_2}\right) 
  \nonumber \\\nonumber   &
  \quad\, - \partial_{h_{ \beta y}} \left(  G _{\beta  y}G_{\beta \fb_1}G^-_{x y }G^-_{\beta \fb_2}
   +  G _{\beta \fb_1}G^-_{x\beta}G^-_{ y y}G^-_{\beta \fb_2}
     +  G _{\beta \fb_1}G^-_{x y }G^-_{\beta\beta}G^-_{ y \fb_2} \right) \nonumber\\
    &=2m^3 \Gc_{y \fb_1}\Gc^-_{x\fb_2} + 2(m^2 +\overline m^2) \overline{M}_{xy}\Gc_{\beta \fb_1}\Gc^-_{\beta \fb_2} + 2|m|^2 \overline{M}_{xy}\Gc_{y \fb_1}\Gc^-_{y\fb_2}  + \text{(others)}.\label{eq:doublepartial}
   \end{align} 
Hereafter, ``(others)" denotes certain $\{\fb_1,\fb_2\}$-recollision, higher order, or $Q$-graphs, whose exact expressions are irrelevant to the derivation of the $\self$ and will change from one equation to another.  
Applying \eqref{5.16} to the proof of \Cref{lanlw}, we obtain the following extension of \eqref{eq:BE} with fourth-order cumulants: 
 \begin{align}\label{eq:BE_fourth}
    \Gc_{xy}\Gamma=\sum_{\al,\beta}  M_{x\al} \left(
  S_{\al \beta}\Gc_{\beta\beta} G _{\al y} \bigGamma
  -   S_{\al \beta}G _{\beta y} 
  \partial_{h_{ \beta\al}} \bigGamma - \frac{1}{2} S_{\al\beta}^{(4)} \partial_{h_{\beta\al}}^2 \partial_{h_{\al\beta}}(G_{\beta y}\Gamma) \right) +\text{(others)}.
\end{align}
Then, we apply the expansion \eqref{eq:BE_fourth} to the three terms on the RHS of \eqref{eq:doublepartial} and get that 
\begin{align*}
    \Gc_{y \fb_1}\Gc^-_{y \fb_2}&=\sum_{\gamma,\gamma'}|M_{y\gamma}|^2 S_{\gamma \gamma'}\Gc_{\gamma' \fb_1}\Gc^-_{\gamma' \fb_2}+\text{(others)}    =\sum_{\gamma}M^0_{y\gamma}\cT_{\gamma,\fb_1\fb_2}+\text{(others)},\\
    \Gc_{\beta \fb_1}\Gc^-_{\beta \fb_2}&=\sum_{\gamma}M^0_{\beta\gamma}  \cT_{\gamma,\fb_1\fb_2}+\text{(others)},\\
    \Gc_{y \fb_1}\Gc^-_{x\fb_2} &= \sum_{\gamma,\gamma'} M_{y\gamma}\overline M_{x\gamma} S_{\gamma\gamma'} \Gc_{\gamma' \fb_1 }\Gc^-_{\gamma' \fb_2}+\text{(others)} = \sum_{\gamma} M_{y\gamma}\overline M_{x\gamma}\cT_{\gamma,\fb_1 \fb_2}+\text{(others)},
\end{align*}
where we used the definition of $\cT$ in \eqref{general_Tc}. 
In sum, we see that the expansions of \eqref{eq:self4-2} lead to the following expression that contains the $\selfs$ we are interested in: 
\begin{align}
     &\sum_{x,y,\beta,\gamma} \zthn_{\fa x} M_{x y} S^{(4)}_{y \beta}\left(m^3M_{y\gamma}\overline M_{x\gamma}  + 2(\re m^2)  \overline{M}_{xy}M_{\beta\gamma}^0 + |m|^2 \overline{M}_{xy}M_{y\gamma}^0 \right) \cT_{\gamma,\fb_1\fb_2} \nonumber\\
     =&\sum_{x,y,\beta,\gamma} \zthn_{\fa x} S^{(4)}_{y \beta}\left(m^3 M_{xy} M_{y\gamma}\overline M_{x\gamma}  + 2(\re m^2) {M}_{xy}^0M_{\beta\gamma}^0 + |m|^2 {M}_{xy}^0M_{y\gamma}^0 \right) \cT_{\gamma,\fb_1\fb_2}. \label{eq:4self1} 
\end{align}
  

Next, we expand the terms in \eqref{eq:self4-1}. 
In the proof of \Cref{lem_lweight}, by replacing the Gaussian integration by parts with the cumulant expansion formula \eqref{5.16}, we can derive the following extension of \eqref{eq:LW} with fourth-order cumulants: 
\begin{align} 
\Gc _{xx} \Gamma  
= &~\sum_{y,\al,\beta}(1+M^+S^+)_{xy}M_{y\al}
 \left(  S_{\al \beta} \Gc_{ \al y} 
 \Gc_{\beta\beta}  \bigGamma
  -  S_{\al \beta} G _{\beta y} 
   \partial_{h_{ \beta\al}} \bigGamma - \frac{1}{2} S_{\al\beta}^{(4)} \partial_{h_{\beta\al}}^2 \partial_{h_{\al\beta}}(G_{\beta y}\Gamma)\right) \nonumber\\
   &~+\text{(others)}.\label{eq:LW_BA} 
 \end{align}
Applying it to the term $\sum_{x,y,\beta} \zthn_{\fa x}  M_{xy} S_{y \beta} \Gc_{\beta\beta} \Gc _{y \fb_1} \Gc^- _{x\fb_2} $ in \eqref{eq:self4-1}, we get that
\begin{align}
    \Gc_{\beta\beta} \Gc_{y \fb_1} \Gc^- _{x\fb_2} =&~ - \sum_{\gamma,\gamma_1,\gamma_2}(1+M^+S^+)_{\beta \gamma}M_{\gamma \gamma_1} S_{\gamma_1\gamma_2} G _{\gamma_2 \gamma}  
   \partial_{h_{ \gamma_2 \gamma_1}} \left(\Gc_{y \fb_1} \Gc^- _{x\fb_2}\right)  \nonumber\\
   &~  - \frac 12\sum_{\gamma,\gamma_1,\gamma_2}(1+M^+S^+)_{\beta \gamma}M_{\gamma \gamma_1} S^{(4)}_{\gamma_1\gamma_2} \partial^2_{h_{ \gamma_2 \gamma_1}}\partial_{h_{ \gamma_1 \gamma_2}} \left(G _{\gamma_2 \gamma}  \Gc_{y \fb_1} \Gc^- _{x\fb_2}\right) +\text{(others)} \nonumber\\
   =&~ - \frac 12\sum_{\gamma,\gamma_1,\gamma_2}(1+M^+S^+)_{\beta \gamma}M_{\gamma \gamma_1} S^{(4)}_{\gamma_1\gamma_2} \partial^2_{h_{ \gamma_2 \gamma_1}}\partial_{h_{ \gamma_1 \gamma_2}} \left(G _{\gamma_2 \gamma}  \Gc_{y \fb_1} \Gc^- _{x\fb_2}\right) +\text{(others)} \nonumber\\
   =&~  \sum_{\gamma,\gamma_1,\gamma_2}(1+M^+S^+)_{\beta \gamma}M_{\gamma \gamma_1} S^{(4)}_{\gamma_1\gamma_2}  m^3 M_{\gamma_1\gamma}\Gc_{y \fb_1} \Gc^- _{x\fb_2}  +\text{(others)} \nonumber\\
   =&~  \sum_{ \gamma_1,\gamma_2} (M^+S^+)_{\beta \gamma_1} S^{(4)}_{\gamma_1\gamma_2}  m^3 \Gc_{y \fb_1} \Gc^- _{x\fb_2}  +\text{(others)},\label{eq:LW_BA2} 
\end{align}
where in the last step we used that $(1+M^+S^+)M^+S = M^+ S^+$ and $S^{(4)}S=S^{(4)}$. 
Then, we apply \eqref{eq:BE_fourth} to $\Gc_{y \fb_1} \Gc^- _{x\fb_2}$ and obtain that 
\begin{align}
    \Gc_{\beta\beta} \Gc_{y \fb_1} \Gc^- _{x\fb_2}&=  \sum_{ \gamma_1,\gamma_2,\gamma,\gamma'} (M^+S^+)_{\beta \gamma_1} S^{(4)}_{\gamma_1\gamma_2}  m^3 M_{y\gamma}\overline M_{x\gamma}S_{\gamma\gamma'} \Gc_{\gamma' \fb_1} \Gc^- _{\gamma' \fb_2}  +\text{(others)}\nonumber\\
    &=  \sum_{ \gamma_1,\gamma_2,\gamma}  m^3(M^+S^+)_{\beta\gamma_1}S^{(4)}_{\gamma_1 \gamma_2}  M_{y\gamma}\overline M_{x\gamma} \cT_{\gamma, \fb_1 \fb_2}  +\text{(others)}.\label{eq:4self2.1}
\end{align}

Finally, we expand $\sum_{x,y,\beta} \zthn_{\fa x}  M_{xy} S_{y \beta}\Gc^-_{xy} \Gc _{\beta \fb_1}\Gc^-_{\beta \fb_2}$ in \eqref{eq:self4-1}. Applying \eqref{eq:BE_fourth} to it, we obtain that 
\begin{align}
    \Gc^-_{xy} \Gc _{\beta \fb_1}\Gc^-_{\beta \fb_2}  &=\sum_{\gamma,\gamma'}   \overline M_{x\gamma}S_{\gamma\gamma'}\left(G^-_{\gamma y} \Gc^-_{\gamma'\gamma'}\Gc _{\beta \fb_1}\Gc^-_{\beta \fb_2} + G^-_{\gamma'y} G_{\beta \gamma}G_{\gamma' \fb_1}\Gc^-_{\beta \fb_2} + G^-_{\gamma'y} \Gc_{\beta \fb_1}G^-_{\beta \gamma'}G^-_{\gamma \fb_2} \right)\nonumber\\
    &\quad - \frac 12\sum_{ \gamma,\gamma'}  \overline M_{x\gamma} S^{(4)}_{\gamma\gamma'}\partial_{h_{\gamma\gamma'}}^2 \partial_{h_{\gamma'\gamma}} \left(G^-_{\gamma'y}  \Gc _{\beta \fb_1}\Gc^-_{\beta \fb_2} \right) +\text{(others)}\nonumber\\
    &=\sum_{\gamma,\gamma'}   \overline M_{x\gamma}\overline M_{\gamma  y}  S_{\gamma\gamma'} \Gc^-_{\gamma'\gamma'}\Gc _{\beta \fb_1}\Gc^-_{\beta \fb_2}    + \sum_{ \gamma,\gamma'}  \overline m^3 \overline M_{x\gamma}\overline M_{\gamma y} S^{(4)}_{\gamma\gamma'} \Gc_{\beta \fb_1}\Gc^-_{\beta \fb_2} +\text{(others)}.\label{eq:4thselfderivation}
\end{align}
Applying the expansion \eqref{eq:LW_BA} again to the first term on the RHS, with a similar calculation as in \eqref{eq:LW_BA2}, we obtain that 
\begin{align*}
    \Gc^-_{\gamma'\gamma'}\Gc _{\beta \fb_1}\Gc^-_{\beta \fb_2}= \sum_{\gamma_1,\gamma_2} (M^+S^+)^-_{\gamma' \gamma_1}S^{(4)}_{\gamma_1\gamma_2} \overline m^3 \Gc _{\beta \fb_1}\Gc^-_{\beta \fb_2}+\text{(others)}.
\end{align*}
Plugging it into the equation \eqref{eq:4thselfderivation}, we get that
\begin{align}
    &\Gc^-_{xy} \Gc _{\beta \fb_1}\Gc^-_{\beta \fb_2}  =\sum_{\gamma,\gamma_1}  \overline m^3   \overline M_{x\gamma}\overline M_{\gamma y} [(1+S (M^+S^+)^-)S^{(4)}]_{\gamma\gamma_1}  \Gc _{\beta \fb_1}\Gc^*_{\beta \fb_2}  +\text{(others)}\nonumber\\
    &=\sum_{\gamma,\gamma_1}  \overline m^3   \overline M_{x\gamma}\overline M_{\gamma y} [(S+S(M^+ S^+)^-)S^{(4)}]_{\gamma\gamma_1} \Gc _{\beta \fb_1}\Gc^-_{\beta \fb_2}  +\text{(others)}\nonumber\\
    &=\sum_{\gamma,\gamma_1}  \overline m^3   \overline M_{x\gamma}\overline M_{\gamma y} (S^-S^{(4)})_{\gamma\gamma_1} \Gc _{\beta \fb_1}\Gc^-_{\beta \fb_2} +\text{(others)},\label{eq:4self2.2} 
\end{align}
where we used that $SS^{(4)}=S^{(4)}$ and  $S+S(M^+S^+)^- = S^-$ in the derivation. Combining \eqref{eq:4self2.1} and \eqref{eq:4self2.2}, we see that the expansions of \eqref{eq:self4-1} lead to the following expression that contains the $\selfs$ we are interested in: 
\begin{align}
    &~\sum_{x,y,\beta,\gamma_1,\gamma_2,\gamma} \zthn_{\fa x} M_{xy} S_{y \beta} m^3 (M^+S^+)_{\beta \gamma_1}S^{(4)}_{\gamma_1 \gamma_2}  M_{y\gamma}\overline M_{x\gamma} \cT_{\gamma, \fb_1 \fb_2} \nonumber\\
    + &~\sum_{x,y,\beta,\gamma_1} \zthn_{\fa x}  M_{xy} S_{y \beta}\overline m^3   \overline M_{x\gamma}\overline M_{\gamma y} (S^-S^{(4)})_{\gamma\gamma_1} \Gc _{\beta \fb_1}\Gc^-_{\beta \fb_2}\nonumber\\
    =&~\sum_{x,y,\beta,\gamma_1,\gamma_2,\gamma} \zthn_{\fa x} M_{xy} S_{y \beta} \left[m^3 (M^+S^+)_{\beta \gamma_1}S^{(4)}_{\gamma_1 \gamma_2}  M_{y\gamma}\overline M_{x\gamma} + \overline m^3   \overline M_{x\gamma_1}\overline M_{\gamma_1 y} (S^-S^{(4)})_{\gamma_1\gamma_2} \delta_{\beta\gamma}\right] \cT_{\gamma, \fb_1 \fb_2} ,\label{eq:4self2}
\end{align}
where we also used that $\sum_\beta S_{y \beta} \Gc _{\beta \fb_1}\Gc^-_{\beta \fb_2} = \sum_{\beta}S_{y\beta}\cT_{\beta,\fb_1\fb_2}$ due to the identity $S=S^2$.

Now, combining \eqref{eq:4self1} and \eqref{eq:4self2}, we obtain the following $\self$ term in the $T$-expansion of $\czT_{\fa,\fb_1\fb_2}$:
$$\sum_{x,y,\beta,\gamma} \zthn_{\fa x} \Sele^{(4)}_{x\gamma}\zT_{\gamma',\fb_1\fb_2},$$
where $\Sele^{(4)}$ is defined as
\begin{align*}
\Sele^{(4)}_{x\gamma}&=\sum_{y,\beta,\gamma} S^{(4)}_{y \beta}\left(m^3 M_{xy} M_{y\gamma}\overline M_{x\gamma}  + 2(\re m^2) {M}_{xy}^0M_{\beta\gamma}^0 + |m|^2 {M}_{xy}^0M_{y\gamma}^0 \right) \\
&+\sum_{y,\beta,\gamma_1,\gamma_2,\gamma} M_{xy} S_{y \beta} \left[m^3 (M^+S^+)_{\beta \gamma_1}S^{(4)}_{\gamma_1 \gamma_2}  M_{y\gamma}\overline M_{x\gamma} + \overline m^3   \overline M_{x\gamma_1}\overline M_{\gamma_1 y} (S^-S^{(4)})_{\gamma_1\gamma_2} \delta_{\beta\gamma}\right].
\end{align*}
We now check the sum zero property of this $\self$. 
Denote 
$$ t_0:=\sum_y M_{xy}^0,\quad t_+:=\sum_y M_{xy}^+,\quad t_-:=\overline t_+,\quad A_4:=\sum_{ \beta} S^{(4)}_{xy} =\kappa_4 W^{-d}.$$
Using these notations and that $\sum_{y} S_{xy}=1$, we can write the entries of $\cal E^{(4)}$ as
\begin{align*}
     \cal E^{(4)}_{xy}=&~ A_4 m^3 (M^2)_{x y}\overline M_{xy}  +2(\re m^2) (M^0S^{(4)}M^0)_{xy} + A_4 |m|^2 (M^0)^2_{xy} \\
     &~+A_4 m^3\frac{t_+}{1-t_+}   (M^2)_{x y}\overline M_{x y}  + A_4 \frac{\overline m^3 }{1-t_-}\sum_{\al}    (\overline M^2)_{x \al} M_{x\al} (SM^0)_{\al y}\\
     =&~ \frac{A_4 m^3}{1-t_+} (M^2)_{x y}\overline M_{xy} +  \frac{A_4\overline m^3  }{1-t_-}\sum_{\al}   (\overline M^2)_{x \al} M_{x\al} S_{\al y} + 2(\re m^2) (M^0S^{(4)}M^0)_{xy} + A_4 |m|^2 (M^0)^2_{xy} .
\end{align*}
Summing over $t$, we get 
\begin{align}\label{eq:sumE4}
     \sum_y \cal E^{(4)}_{xy}&= \frac{A_4 m^3}{1-t_+} \sum_y (M^2)_{x y}\overline M_{xy} + \frac{A_4 \overline m^3}{1-t_-}\sum_{\al}   (\overline M^2)_{xy} M_{xy} +2 A_4 (\re m^2)t_0^2   + A_4 |m|^2 t_0^2 .
\end{align} 
By \eqref{L2M}, we have 
$$t_0= \frac{\im m}{\eta +\im m} =1+\OO(\eta).$$
Using Ward's identity \eqref{eq_Ward}, we get that 
$$ \sum_{y} (M^2)_{x y}\overline M_{x y} = \left(M^2M^\dag \right)_{xx} =\frac{(M^2)_{xx}-(M M^\dag)_{xx}}{2\ii (\eta + \im m)}=   \frac{t_+ - t_0}{2\ii (\eta + \im m)}.
$$
Plugging these two identities into \eqref{eq:sumE4} yields that 
\begin{align*}
     \sum_y \cal E^{(4)}_{xy}&= \frac{A_4 m^3}{1-t_+} \frac{t_+-t_0}{2\ii (\eta +\im m)}+ \frac{A_4 \overline m^3}{1-t_-} \frac{t_--t_0}{-2\ii (\eta +\im m)} t_0 + 2 A_4 (\re m^2)t_0^2   + A_4 |m|^2 t_0^2 \\
     &=A_4\left[-  \frac{m^3-\overline m^3}{2\ii \im m} + 2\re m^2 + |m|^2 +\OO(\eta)\right]=\OO(\eta W^{-d}).
\end{align*}
This shows the desired sum zero property for $\Sele^{(4)}$.

\subsection{Example of 10-vertex renormalization}\label{subsec:6v}

In this subsection, we derive the values of $\Delta(\Pi_1)$ and $\Delta(\Pi_2)$ given in \eqref{DeltaPi1} (for the Wegner orbital model). 
We divide the derivation into 5 cases according to the five lines on the RHS of \eqref{eq:Gammai}. We will discard all terms that do not lead to the pairings \eqref{eq:6-loop} and \eqref{eq:4+2-loop}. 

\medskip
\noindent{\bf Case 1:} Corresponding to the last term on the RHS of \eqref{eq:Gammai}, we need to expand 
$$\left( G_{x b_0} \overline {G}_{x \bar b_0} -|m|^2 T_{x,b_0\bar b_0}\right) |m|^{-4}G_{a_1 x}G_{x b_1} \overline G_{\bar a_1 x}\overline G_{x\bar b_1}G_{a_2 x} G_{x b_2}\overline G_{\bar a_2 x}\overline G_{x\bar b_2}.$$
Applying the edge expansion with respect to $G_{x b_0}$, only the following graph is relevant for the pairings in \eqref{eq:6-loop} and \eqref{eq:4+2-loop}:
\begin{align*}
   &- |m|^{-2} \sum_y S_{xy} (G_{y b_0}\overline G_{y\bar b_1}) \overline {G}_{x \bar b_0} G_{a_1 x}G_{x b_1} \overline G_{\bar a_1 x}G_{a_2 x} G_{x b_2}\overline G_{\bar a_2 x}\overline G_{x\bar b_2} .
\end{align*}
Next, applying the edge expansion with respect to $\overline {G}_{x \bar b_0}$, only the following graph is relevant for \eqref{eq:6-loop} and \eqref{eq:4+2-loop}:
\begin{align*}
   &- \sum_{y_1,y_2} S_{xy_1}S_{xy_2} (G_{y_1 b_0}\overline G_{y_1\bar b_1}) (\overline {G}_{y_2 \bar b_0}G_{y_2 b_1}) G_{a_1 x} \overline G_{\bar a_1 x}G_{a_2 x} G_{x b_2}\overline G_{\bar a_2 x}\overline G_{x\bar b_2}.
\end{align*}
Continuing the edge expansions, we obtain that 
\begin{align*}
   &- |m|^6\sum_{y_1,y_2,y_3,y_4,y_5} S_{xy_1}S_{xy_2}S_{xy_3} S_{xy_4} S_{xy_5}(G_{y_1 b_0}\overline G_{y_1\bar b_1}) (\overline {G}_{y_2 \bar b_0}G_{y_2 b_1}) (G_{a_1 y_3} \overline G_{\bar a_1 y_3}) (G_{a_2 y_4}\overline G_{\bar a_2 y_4}) (G_{y_5 b_2}\overline G_{y_5\bar b_2})\\
   & -|m|^6 \sum_{y_1,y_2,y_3,y_4,y_5} S_{xy_1}S_{xy_2}S_{xy_3} S_{xy_4} S_{xy_5} (G_{y_1 b_0}\overline G_{y_1\bar b_1}) (\overline {G}_{y_2 \bar b_0}G_{y_2 b_1}) (G_{a_1 y_3}\overline G_{\bar a_2 y_3}) (\overline G_{\bar a_1 y_4}G_{a_2 y_4})(G_{y_5 b_2}\overline G_{y_5\bar b_2}).
\end{align*}
The first and second terms provide a factor $-|m|^6 $ to $\Delta(\Pi_2)$ and $\Delta(\Pi_1)$, respectively:
\be\label{Delta1}
\Delta_1(\Pi_1)=-|m|^6,\quad \Delta_1(\Pi_2)=-|m|^6.
\ee

\medskip
\noindent{\bf Case 2:} Corresponding to the fourth line on the RHS \eqref{eq:Gammai}, we need to expand 
\begin{align}
&\sum_{i=1}^2\left( G_{x b_0} \overline {G}_{x \bar b_0} -|m|^2T_{x,b_0\bar b_0}\right)\cdot |m|^{-2}m^{-1} G_{a_i x}G_{x b_i} \overline G_{\bar a_i x}\overline G_{x\bar b_i}G_{a_j x}G_{x b_j}\overline G_{\bar a_j \bar b_j} \label{eq:case21}\\
+&\sum_{i=1}^2\left( G_{x b_0} \overline {G}_{x \bar b_0} -|m|^2T_{x,b_0\bar b_0}\right)\cdot |m|^{-2}\bar m^{-1}G_{a_i x}G_{x b_i} \overline G_{\bar a_i x}\overline G_{x\bar b_i}G_{a_j b_j}\overline G_{\bar a_j x}\overline G_{x\bar b_j}  .  \label{eq:case22}
\end{align}
We will focus on the contribution from the term \eqref{eq:case21} to $\Delta(\Pi_1)$ and $\Delta(\Pi_2)$, while the contribution from the term \eqref{eq:case22} can be obtained by taking the complex conjugate. By applying edge expansions to \eqref{eq:case21} with respect to $G_{x b_0}$, we obtain the following (without the summation notation over $i$):
\begin{align*}
    &\bar m^{-1}\sum_{y_1}S_{xy_1}(G_{a_i y_1}G_{y_1 b_0}) \overline {G}_{x \bar b_0}  G_{x b_i} \overline G_{\bar a_i x}\overline G_{x\bar b_i}G_{a_j x}G_{x b_j}\overline G_{\bar a_j \bar b_j} \\
    + &m^{-1}\sum_{y_1}S_{xy_1}(G_{y_1 b_0}\overline G_{y_1\bar b_i}) \overline {G}_{x \bar b_0}  G_{a_i x}G_{x b_i} \overline G_{\bar a_i x}G_{a_j x}G_{x b_j}\overline G_{\bar a_j \bar b_j} \\
    +&\bar m^{-1}\sum_{y_1}S_{xy_1}(G_{a_j y_1}G_{y_1 b_0}) \overline {G}_{x \bar b_0} G_{a_i x} G_{x b_i} \overline G_{\bar a_i x}\overline G_{x\bar b_i}G_{x b_j}\overline G_{\bar a_j \bar b_j}\\
    + &|m|^{-2}\sum_{y_1}S_{xy_1}(G_{y_1 b_0}\overline G_{y_1 \bar b_j}) \overline {G}_{x \bar b_0}  G_{a_i x}G_{x b_i} \overline G_{\bar a_i x}\overline G_{x\bar b_i}G_{a_j x}G_{x b_j}\overline G_{\bar a_j x} .
\end{align*}
Continuing the edge expansion with respect to $\overline {G}_{x \bar b_0}$, we get the following relevant graphs:
\begin{align*}
    &m\sum_{y_1,y_2}S_{xy_1}S_{xy_2}(G_{a_i y_1}G_{y_1 b_0}) (\overline {G}_{y_2 \bar b_0}  G_{y_2 b_i}) \overline G_{\bar a_i x}\overline G_{x\bar b_i}G_{a_j x}G_{x b_j}\overline G_{\bar a_j \bar b_j} \\
    +&m\sum_{y_1,y_2}S_{xy_1}S_{xy_2}(G_{a_i y_1}G_{y_1 b_0}) (\overline {G}_{y_2 \bar b_0}  G_{y_2 b_j}) G_{x b_i} \overline G_{\bar a_i x}\overline G_{x\bar b_i}G_{a_j x}\overline G_{\bar a_j \bar b_j}\\
    +& \bar m\sum_{y_1,y_2}S_{xy_1}S_{xy_2}(G_{y_1 b_0}\overline G_{y_1\bar b_i}) (\overline {G}_{y_2 \bar b_0} G_{y_2 b_i}) G_{a_i x} \overline G_{\bar a_i x}G_{a_j x}G_{x b_j}\overline G_{\bar a_j \bar b_j} \\
    +&(\bar m^2/m)\sum_{y_1,y_2}S_{xy_1}S_{xy_2}(G_{y_1 b_0}\overline G_{y_1\bar b_i}) (\overline G_{\bar a_i y_2}\overline {G}_{y_2 \bar b_0}) G_{a_i x}G_{x b_i} G_{a_j x}G_{x b_j}\overline G_{\bar a_j \bar b_j}\\
    +& \bar m\sum_{y_1,y_2}S_{xy_1}S_{xy_2}(G_{y_1 b_0}\overline G_{y_1\bar b_i}) (\overline {G}_{y_2 \bar b_0} G_{y_2 b_j}) G_{a_i x}G_{x b_i} \overline G_{\bar a_i x}G_{a_j x}\overline G_{\bar a_j \bar b_j} \\
    +& m \sum_{y_1,y_2}S_{xy_1}S_{xy_2}(G_{a_j y_1}G_{y_1 b_0}) (\overline {G}_{y_2 \bar b_0}G_{y_2 b_i}) G_{a_i x}  \overline G_{\bar a_i x}\overline G_{x\bar b_i}G_{x b_j}\overline G_{\bar a_j \bar b_j}\\
    +& m \sum_{y_1,y_2}S_{xy_1}S_{xy_2}(G_{a_j y_1}G_{y_1 b_0}) (\overline {G}_{y_2 \bar b_0}G_{y_2 b_j}) G_{a_i x}G_{x b_i}  \overline G_{\bar a_i x}\overline G_{x\bar b_i}\overline G_{\bar a_j \bar b_j}\\
    + &\sum_{y_1,y_2}S_{xy_1}S_{xy_2}(G_{y_1 b_0}\overline G_{y_1 \bar b_j}) (\overline {G}_{y_2 \bar b_0}G_{y_2 b_i})  G_{a_i x} \overline G_{\bar a_i x}\overline G_{x\bar b_i}G_{a_j x}G_{x b_j}\overline G_{\bar a_j x} \\
    + &\sum_{y_1,y_2}S_{xy_1}S_{xy_2}(G_{y_1 b_0}\overline G_{y_1 \bar b_j}) (\overline {G}_{y_2 \bar b_0} G_{y_2 b_j}) G_{a_i x}G_{x b_i} \overline G_{\bar a_i x}\overline G_{x\bar b_i}G_{a_j x}\overline G_{\bar a_j x} .
\end{align*}
Continuing the edge expansion, we obtain that 
\begin{align}
    &  m|m|^4\sum_{y_1,y_2,y_3,y_4}S_{xy_1}S_{xy_2}S_{xy_3}S_{xy_4}(G_{a_i y_1}G_{y_1 b_0}) (\overline {G}_{y_2 \bar b_0}  G_{y_2 b_i}) (\overline G_{\bar a_i y_3}G_{a_j y_3})(\overline G_{y_4\bar b_i}G_{y_4 b_j})\overline G_{\bar a_j \bar b_j} \label{eq:case211}\\
    +& m|m|^4\sum_{y_1,y_2,y_3,y_4}S_{xy_1}S_{xy_2}S_{xy_3}S_{xy_4}(G_{a_i y_1}G_{y_1 b_0}) (\overline {G}_{y_2 \bar b_0}  G_{y_2 b_j})( G_{y_3 b_i}\overline G_{y_3\bar b_i}) (\overline G_{\bar a_i y_4}G_{a_j y_4})\overline G_{\bar a_j \bar b_j}\label{eq:case212}\\
    +&  m|m|^4\sum_{y_1,y_2,y_3,y_4}S_{xy_1}S_{xy_2}S_{xy_3}S_{xy_4}(G_{y_1 b_0}\overline G_{y_1\bar b_i}) (\overline {G}_{y_2 \bar b_0} G_{y_2 b_i}) (G_{a_i y_3} \overline G_{\bar a_i y_3})(G_{a_j y_4}G_{y_4 b_j})\overline G_{\bar a_j \bar b_j} \nonumber\\
    +&|m|^6\sum_{y_1,y_2,y_3,y_4,y_5}S_{xy_1}S_{xy_2}S_{xy_3}S_{xy_4}S_{xy_5}(G_{y_1 b_0}\overline G_{y_1\bar b_i}) (\overline {G}_{y_2 \bar b_0} G_{y_2 b_i}) (G_{a_i y_3} \overline G_{\bar a_i y_3})(G_{a_j y_4}\overline G_{\bar a_j y_4})(G_{y_5 b_j}\overline G_{y_5 \bar b_j}) \nonumber\\
    +&m|m|^4\sum_{y_1,y_2,y_3,y_4}S_{xy_1}S_{xy_2}S_{xy_3}S_{xy_4}(G_{y_1 b_0}\overline G_{y_1\bar b_i}) (\overline {G}_{y_2 \bar b_0} G_{y_2 b_i}) (G_{a_i y_3}G_{y_3 b_j}) (\overline G_{\bar a_i y_4}G_{a_j y_4})\overline G_{\bar a_j \bar b_j} \nonumber\\
    +&|m|^6\sum_{y_1,y_2,y_3,y_4,y_5}S_{xy_1}S_{xy_2}S_{xy_3}S_{xy_4}S_{xy_5}(G_{y_1 b_0}\overline G_{y_1\bar b_i}) (\overline {G}_{y_2 \bar b_0} G_{y_2 b_i}) (G_{a_i y_3}\overline G_{\bar a_j y_3}) (\overline G_{\bar a_i y_4}G_{a_j y_4})(G_{y_5 b_j}\overline G_{y_5 \bar b_j}) \nonumber\\
    +&m|m|^4\sum_{y_1,y_2,y_3,y_4}S_{xy_1}S_{xy_2}S_{xy_3}S_{xy_4}(G_{y_1 b_0}\overline G_{y_1\bar b_i}) (\overline {G}_{y_2 \bar b_0} G_{y_2 b_j}) (G_{a_i y_3}\overline G_{\bar a_i y_3}) (G_{y_4 b_i} G_{a_j y_4})\overline G_{\bar a_j \bar b_j} \nonumber\\
    +& |m|^6\sum_{y_1,y_2,y_3,y_4,y_5}S_{xy_1}S_{xy_2}S_{xy_3}S_{xy_4}S_{xy_5}(G_{y_1 b_0}\overline G_{y_1\bar b_i}) (\overline {G}_{y_2 \bar b_0} G_{y_2 b_j}) (G_{a_i y_3}\overline G_{\bar a_i y_3}) (G_{y_4 b_i}\overline G_{y_4\bar b_j}) (G_{a_j y_5}\overline G_{\bar a_j y_5}) \nonumber\\
      +&  m|m|^4\sum_{y_1,y_2,y_3,y_4}S_{xy_1}S_{xy_2}S_{xy_3}S_{xy_4}(G_{y_1 b_0}\overline G_{y_1\bar b_i}) (\overline {G}_{y_2 \bar b_0} G_{y_2 b_j}) (G_{a_i y_3}G_{y_3 b_i}) (\overline G_{\bar a_i y_4}G_{a_j y_4})\overline G_{\bar a_j \bar b_j} \nonumber\\
      +& |m|^6\sum_{y_1,y_2,y_3,y_4,y_5}S_{xy_1}S_{xy_2}S_{xy_3}S_{xy_4}S_{xy_5}(G_{y_1 b_0}\overline G_{y_1\bar b_i}) (\overline {G}_{y_2 \bar b_0} G_{y_2 b_j}) (\overline G_{\bar a_j y_3}G_{a_i y_3})(\overline G_{\bar a_i y_4}G_{a_j y_4})(G_{y_5 b_i} \overline G_{y_5 \bar b_j}) \nonumber\\
    +&  m|m|^4 \sum_{y_1,y_2,y_3,y_4}S_{xy_1}S_{xy_2}S_{xy_3}S_{xy_4}(G_{a_j y_1}G_{y_1 b_0}) (\overline {G}_{y_2 \bar b_0}G_{y_2 b_i}) (G_{a_i y_3}  \overline G_{\bar a_i y_3})(\overline G_{y_4\bar b_i}G_{y_4 b_j})\overline G_{\bar a_j \bar b_j}\nonumber\\
    +& m|m|^4 \sum_{y_1,y_2,y_3,y_4}S_{xy_1}S_{xy_2}S_{xy_3}S_{xy_4}(G_{a_j y_1}G_{y_1 b_0}) (\overline {G}_{y_2 \bar b_0}G_{y_2 b_j}) (G_{a_i y_3}\overline G_{\bar a_i y_3})(G_{y_4 b_i}  \overline G_{y_4\bar b_i})\overline G_{\bar a_j \bar b_j}\nonumber\\
    + & |m|^6 \sum_{y_1,y_2,y_3,y_4,y_5}S_{xy_1}S_{xy_2}S_{xy_3}S_{xy_4}S_{xy_5}(G_{y_1 b_0}\overline G_{y_1 \bar b_j}) (\overline {G}_{y_2 \bar b_0}G_{y_2 b_i})  (G_{a_i y_3} \overline G_{\bar a_i y_3}) (\overline G_{y_4\bar b_i}G_{y_4 b_j})(G_{a_j y_5}\overline G_{\bar a_j y_5}) \nonumber\\
    + & |m|^6 \sum_{y_1,y_2,y_3,y_4,y_5}S_{xy_1}S_{xy_2}S_{xy_3}S_{xy_4}S_{xy_5}(G_{y_1 b_0}\overline G_{y_1 \bar b_j}) (\overline {G}_{y_2 \bar b_0}G_{y_2 b_i})  (G_{a_i y_3}\overline G_{\bar a_j y_3}) (\overline G_{\bar a_i y_4}G_{a_j y_4})(\overline G_{y_5\bar b_i}G_{y_5 b_j}) \nonumber\\
    + &|m|^6 \sum_{y_1,y_2,y_3,y_4,y_5}S_{xy_1}S_{xy_2}S_{xy_3}S_{xy_4}S_{xy_5}(G_{y_1 b_0}\overline G_{y_1 \bar b_j}) (\overline {G}_{y_2 \bar b_0} G_{y_2 b_j}) (G_{a_i y_3}\overline G_{\bar a_i y_3})(G_{y_4 b_i} \overline G_{y_4\bar b_i})(G_{a_j y_5}\overline G_{\bar a_j y_5}) \nonumber\\
    + &|m|^6 \sum_{y_1,y_2,y_3,y_4,y_5}S_{xy_1}S_{xy_2}S_{xy_3}S_{xy_4}S_{xy_5}(G_{y_1 b_0}\overline G_{y_1 \bar b_j}) (\overline {G}_{y_2 \bar b_0} G_{y_2 b_j}) (G_{a_i y_3}\overline G_{\bar a_j y_3} )(G_{y_4 b_i} \overline G_{y_4\bar b_i})(\overline G_{\bar a_i y_5}G_{a_j y_5}).\nonumber
\end{align}
For the 8 terms with coefficient $|m|^6$, 4 of them involve pairings $\{b_0,\bar b_i\}, \{\bar b_0,b_j\}$ with $i=2-j\in \{1,2\}$. They are not relevant for \eqref{eq:6-loop} and \eqref{eq:4+2-loop}. For the remaining 4 terms that involve pairings  $\{b_0,\bar b_i\}, \{\bar b_0,b_i\}$ with $i=1$ or $\{b_0,\bar b_j\}, \{\bar b_0,b_j\}$ with $j=1$ provide a factor $2|m|^6$ to $\Delta(\Pi_1)$ and a factor $2|m|^6$ to $\Delta(\Pi_2)$. 
For the 8 terms with coefficient $m|m|^4$, we need to apply $GG$-expansions to them. For example, the term in \eqref{eq:case211} gives that 
 \begin{align*}
    & m^2|m|^4\sum_{y_1,y_2,y_3,y_4,y_5}S_{xy_1}S_{xy_2}S_{xy_3}S_{xy_4}S_{y_1 y_5}(\overline G_{y_5 \bar b_j} G_{y_5 b_0})(\overline {G}_{y_2 \bar b_0}  G_{y_2 b_i}) (\overline G_{\bar a_j y_1}G_{a_i y_1}) (\overline G_{\bar a_i y_3}G_{a_j y_3})(\overline G_{y_4\bar b_i}G_{y_4 b_j})\\
    +& m^4|m|^4\sum_{y_1,y_2,y_3,y_4,y_5,y_6}S_{xy_1}S_{xy_2}S_{xy_3}S_{xy_4}S^+_{y_1y_5}S_{y_5y_6}(\overline G_{y_6 \bar b_j}G_{y_6 b_0}) (\overline {G}_{y_2 \bar b_0}  G_{y_2 b_i}) (\overline G_{\bar a_i y_3}G_{a_j y_3})(\overline G_{y_4\bar b_i}G_{y_4 b_j})(\overline G_{\bar a_j y_5}G_{a_i y_5}).
\end{align*}
They are not relevant for \eqref{eq:6-loop} and \eqref{eq:4+2-loop}, since they lead to parings $\{b_0,\bar b_i\}, \{\bar b_0,b_j\}$ with $i=2-j\in \{1,2\}$. The $GG$ expansion of \eqref{eq:case212} gives that 
\begin{align*}
&m^2 |m|^4\sum_{y_1,y_2,y_3,y_4,y_5}S_{xy_1}S_{xy_2}S_{xy_3}S_{xy_4}S_{y_1y_5}(\overline {G}_{y_2 \bar b_0}  G_{y_2 b_j})(\overline G_{y_5 \bar b_j}G_{y_5 b_0})(G_{a_i y_1}\overline G_{\bar a_j y_1}) (\overline G_{\bar a_i y_4}G_{a_j y_4}) ( G_{y_3 b_i}\overline G_{y_3\bar b_i}) \\
+&m^4 |m|^4\sum_{y_1,y_2,y_3,y_4,y_5,y_6}S_{xy_1}S_{xy_2}S_{xy_3}S_{xy_4}S_{y_1y_5}S^+_{y_5y_6}(\overline {G}_{y_2 \bar b_0}  G_{y_2 b_j})(\overline G_{y_6\bar b_j} G_{y_6 b_0})(G_{a_i y_5}\overline G_{\bar a_j y_5})  (\overline G_{\bar a_i y_4}G_{a_j y_4})( G_{y_3 b_i}\overline G_{y_3\bar b_i}).
\end{align*}
Performing another edge expansion at the vertices $y_1$ and $y_5$, respectively, we get two graphs that contribute a factor 
$$m^2|m|^6\left(1+\frac{m^2}{1-m^2}\right)=|m|^6\iota$$
to $\Delta(\Pi_1)$. For the rest of the 6 terms with coefficient $m|m|^4$, after applying $GG$-expansions to them, we see that two of them are not relevant for \eqref{eq:6-loop} and \eqref{eq:4+2-loop}; one of them contributes a factor $|m|^6\iota$ to $\Delta(\Pi_1)$; two of them contribute a factor $|m|^6\iota$ to $\Delta(\Pi_2)$. 
Together with \eqref{Delta1}, the above calculations show that the first two cases give 
\be\label{Delta2}
\Delta_2(\Pi_1)=\Delta_2(\Pi_2)=-|m|^6+\left(2|m|^6+ 2|m|^6\iota+c.c.\right)=|m|^6\left(3 + 2\iota+ 2\bar \iota\right),
\ee
where $c.c.$ denotes the complex conjugate of the terms in the bracket. 
\medskip
\noindent{\bf Case 3:} Corresponding to the third line on the RHS of \eqref{eq:Gammai}, we need to expand
\begin{align}
&    - \left( G_{x b_0} \overline {G}_{x \bar b_0} -|m|^2T_{x,b_0\bar b_0}\right)\cdot m^{-2}G_{a_1 x}G_{x b_1} G_{a_2 x} G_{x b_2}\overline G_{\bar a_1 \bar b_1}\overline G_{\bar a_2 \bar b_2} \label{eq:case31}\\
&    - \left( G_{x b_0} \overline {G}_{x \bar b_0} -|m|^2T_{x,b_0\bar b_0}\right)\cdot\bar m^{-2}G_{a_1 b_1} \overline G_{\bar a_1 x}\overline G_{x\bar b_1}G_{a_2 b_2}\overline G_{\bar a_2 x}\overline G_{x\bar b_2}. \label{eq:case32}
\end{align}
We will focus on the contribution from the term \eqref{eq:case31} to $\Delta(\Pi_1)$ and $\Delta(\Pi_2)$, while the contribution from the term \eqref{eq:case32} can be obtained by taking the complex conjugate. Applying edge expansions to \eqref{eq:case31} with respect to $G_{x b_0}$, we obtain that 
\begin{align*}
&    -\sum_{y_1} S_{xy_1}  (G_{y_1 b_0}G_{a_1 y_1}) \overline {G}_{x \bar b_0}  G_{x b_1} G_{a_2 x} G_{x b_2}\overline G_{\bar a_1 \bar b_1}\overline G_{\bar a_2 \bar b_2}\\
&    -\sum_{y_1} S_{xy_1}  (G_{y_1 b_0}G_{a_2 y_1})  \overline {G}_{x \bar b_0}  G_{a_1 x}G_{x b_1} G_{x b_2}\overline G_{\bar a_1 \bar b_1}\overline G_{\bar a_2 \bar b_2}\\
&    -m^{-1}\sum_{y_1} S_{xy_1} ( G_{y_1 b_0}\overline G_{y_1 \bar b_1}) \overline {G}_{x \bar b_0}  G_{a_1 x}G_{x b_1} G_{a_2 x} G_{x b_2}\overline G_{\bar a_1 x}\overline G_{\bar a_2 \bar b_2}\\
&    -m^{-1}\sum_{y_1} S_{xy_1} ( G_{y_1 b_0}\overline G_{y_1 \bar b_2}) \overline {G}_{x \bar b_0}  G_{a_1 x}G_{x b_1} G_{a_2 x} G_{x b_2}\overline G_{\bar a_2 x}\overline G_{\bar a_1 \bar b_1}.
\end{align*}
The last term involves pairing $(b_0,\bar b_2)$, and hence is not relevant for \eqref{eq:6-loop} and \eqref{eq:4+2-loop}. For the first three terms, applying the edge expansions and keeping only the relevant graphs for the pairings \eqref{eq:6-loop} and \eqref{eq:4+2-loop}, we obtain that 
\begin{align*}
&   -|m|^2\sum_{y_1,y_2} S_{xy_1} S_{xy_2} (G_{y_1 b_0}G_{a_1 y_1}) (\overline {G}_{y_2 \bar b_0}  G_{y_2 b_1}) G_{a_2 x} G_{x b_2}\overline G_{\bar a_1 \bar b_1}\overline G_{\bar a_2 \bar b_2}\\
&   -|m|^2\sum_{y_1,y_2} S_{xy_1} S_{xy_2} (G_{y_1 b_0}G_{a_2 y_1})  (\overline {G}_{y_2 \bar b_0}G_{y_2 b_1} )  G_{a_1 x}G_{x b_2}\overline G_{\bar a_1 \bar b_1}\overline G_{\bar a_2 \bar b_2}\\
& -\bar m\sum_{y_1,y_2} S_{xy_1} S_{xy_2} ( G_{y_1 b_0}\overline G_{y_1 \bar b_1}) (\overline {G}_{y_2 \bar b_0}G_{y_2 b_1})  G_{a_1 x} G_{a_2 x} G_{x b_2}\overline G_{\bar a_1 x}\overline G_{\bar a_2 \bar b_2}.
\end{align*}
Continuing the edge expansions, we get the following relevant graphs:
\begin{align}
&   -m^2|m|^2\sum_{y_1,y_2,y_3} S_{xy_1} S_{xy_2}S_{xy_3} (G_{y_1 b_0}G_{a_1 y_1}) (\overline {G}_{y_2 \bar b_0}  G_{y_2 b_1}) ( G_{y_3 b_2}G_{a_2 y_3})\overline G_{\bar a_1 \bar b_1}\overline G_{\bar a_2 \bar b_2}\label{eq:case311}\\
&   -m|m|^4\sum_{y_1,y_2,y_3,y_4} S_{xy_1} S_{xy_2}S_{xy_3}S_{xy_4} (G_{y_1 b_0}G_{a_1 y_1}) (\overline {G}_{y_2 \bar b_0}  G_{y_2 b_1}) (G_{a_2 y_3}\overline G_{\bar a_1y_3} ) (G_{y_4 b_2}\overline G_{y_4 \bar b_1})\overline G_{\bar a_2 \bar b_2}\label{eq:case312}\\
&   -m|m|^4\sum_{y_1,y_2,y_3,y_4} S_{xy_1} S_{xy_2}S_{xy_3}S_{xy_4} (G_{y_1 b_0}G_{a_1 y_1}) (\overline {G}_{y_2 \bar b_0}  G_{y_2 b_1}) (G_{a_2 y_3}\overline G_{\bar a_2 y_3}) (G_{y_4 b_2}\overline G_{y_4 \bar b_2})\overline G_{\bar a_1 \bar b_1}\label{eq:case313}\\
&  -m^2|m|^2\sum_{y_1,y_2,y_3} S_{xy_1} S_{xy_2}S_{xy_3} (G_{y_1 b_0}G_{a_2 y_1})  (\overline {G}_{y_2 \bar b_0}G_{y_2 b_1} )  (G_{a_1 y_3}G_{y_3 b_2})\overline G_{\bar a_1 \bar b_1}\overline G_{\bar a_2 \bar b_2}\label{eq:case314}\\
&  -m|m|^4\sum_{y_1,y_2,y_3,y_4} S_{xy_1} S_{xy_2}S_{xy_3}S_{xy_4} (G_{y_1 b_0}G_{a_2 y_1})  (\overline {G}_{y_2 \bar b_0}G_{y_2 b_1} ) ( G_{a_1 y_3}\overline G_{\bar a_1 y_3} )(G_{y_4 b_2}\overline G_{y_4 \bar b_1})\overline G_{\bar a_2 \bar b_2}\label{eq:case315}\\
&  -m|m|^4\sum_{y_1,y_2,y_3,y_4} S_{xy_1} S_{xy_2}S_{xy_3}S_{xy_4} (G_{y_1 b_0}G_{a_2 y_1})  (\overline {G}_{y_2 \bar b_0}G_{y_2 b_1} ) ( G_{a_1 y_3}\overline G_{\bar a_2 y_3} )(G_{y_4 b_2}\overline G_{y_4 \bar b_2})\overline G_{\bar a_1 \bar b_1} \label{eq:case316}\\
& - m|m|^4\sum_{y_1,y_2,y_3,y_4} S_{xy_1} S_{xy_2}S_{xy_3}S_{xy_4} ( G_{y_1 b_0}\overline G_{y_1 \bar b_1}) (\overline {G}_{y_2 \bar b_0}G_{y_2 b_1})  (G_{a_1 y_3}G_{y_3 b_2}) (G_{a_2 y_4} \overline G_{\bar a_1 y_4}) \overline G_{\bar a_2 \bar b_2}\label{eq:case317}\\
& -\bar m|m|^2\sum_{y_1,y_2,y_3} S_{xy_1} S_{xy_2}S_{xy_3} ( G_{y_1 b_0}\overline G_{y_1 \bar b_1}) (\overline {G}_{y_2 \bar b_0}G_{y_2 b_1})  (G_{a_1 y_3} \overline G_{\bar a_1 y_3})G_{a_2 x} G_{x b_2}\overline G_{\bar a_2 \bar b_2}\label{eq:case319}\\
&-|m|^2\sum_{y_1,y_2,y_3} S_{xy_1} S_{xy_2}S_{xy_3} ( G_{y_1 b_0}\overline G_{y_1 \bar b_1}) (\overline {G}_{y_2 \bar b_0}G_{y_2 b_1}) ( G_{a_1 y_3} \overline G_{\bar a_2 y_3})G_{a_2 x} G_{x b_2}\overline G_{\bar a_1 x}\overline G_{x \bar b_2}.\label{eq:case318}
\end{align}
Applying the edge and $GG$ expansions to them, we obtain their respective contributions to $\Delta(\Pi_1)$ and $\Delta(\Pi_2)$ as follows: \eqref{eq:case311} contributes $-|m|^6 \iota^2$  to $\Delta(\Pi_2)$; \eqref{eq:case312}'s contribution is 0 since it involves a $(b_0,\bar b_2)$ pairing; \eqref{eq:case313} contributes $-|m|^6 \iota$  to $\Delta(\Pi_2)$; \eqref{eq:case314} contributes $-|m|^6 \iota^2$  to $\Delta(\Pi_1)$; \eqref{eq:case315}'s contribution is 0 since it involves a $(b_0,\bar b_2)$ pairing; \eqref{eq:case316} contributes $-|m|^6 \iota$  to $\Delta(\Pi_1)$; \eqref{eq:case317} contributes $-|m|^6 \iota$  to $\Delta(\Pi_1)$; \eqref{eq:case318} contributes $-|m|^6$  to $\Delta(\Pi_1)$. For \eqref{eq:case319}, applying an edge expansion with respect to $G_{a_2 x}$ gives 
\begin{align*}
&- m|m|^4\sum_{y_1,y_2,y_3,y_4} S_{xy_1} S_{xy_2}S_{xy_3}S_{xy_4} ( G_{y_1 b_0}\overline G_{y_1 \bar b_1}) (\overline {G}_{y_2 \bar b_0}G_{y_2 b_1})  (G_{a_1 y_3} \overline G_{\bar a_1 y_3})(G_{a_2 y_4} G_{y_4 b_2})\overline G_{\bar a_2 \bar b_2} \\
&-|m|^6\sum_{y_1,y_2,y_3,y_4,y_5} S_{xy_1} S_{xy_2}S_{xy_3}S_{xy_4}S_{xy_5} ( G_{y_1 b_0}\overline G_{y_1 \bar b_1}) (\overline {G}_{y_2 \bar b_0}G_{y_2 b_1})  (G_{a_1 y_3} \overline G_{\bar a_1 y_3})(G_{a_2 y_4}\overline G_{\bar a_2 y_4}) (G_{y_5 b_2}\overline G_{y_5 \bar b_2}) .
\end{align*}
The first term contributes $-|m|^6 \iota$  to $\Delta(\Pi_2)$ after a $GG$-expansion, and the second term contributes $-|m|^6$  to $\Delta(\Pi_2)$. 
Together with \eqref{Delta2}, the above calculations show that the first three cases give 
\begin{align}
\Delta_3(\Pi_1)=\Delta_3(\Pi_2)&= |m|^6\left(3 + 2\iota+ 2\bar \iota \right) - |m|^6 \left(1+ 2\iota+\iota^2 +c.c.\right) = |m|^6\left(1-\iota^2-\bar \iota^2\right). \label{Delta3}
\end{align}


\medskip
\noindent{\bf Case 4:} Corresponding to the second line on the RHS of \eqref{eq:Gammai}, we need to expand
\begin{align}
&    - \sum_{i=1}^2 \left[\left( G_{x b_0} \overline {G}_{x \bar b_0} -|m|^2T_{x,b_0\bar b_0}\right)\cdot |m|^{-2}G_{a_i x}G_{x b_i} \overline G_{\bar a_i x}\overline G_{x\bar b_i}G_{a_j b_j}\overline G_{\bar a_j \bar b_j} + (\bar a_i,\bar b_i)\leftrightarrow (\bar a_j,\bar b_j)\right] \label{eq:case41},
\end{align}
where $(\bar a_i,\bar b_i)\leftrightarrow (\bar a_j,\bar b_j)$ denotes an expression that is obtained by switching $(\bar a_i,\bar b_i)$ and $(\bar a_j,\bar b_j)$ in all the expressions before it. First, applying edge expansions with respect to $G_{xb_0}$, we obtain that 
\begin{align*}
    &-(m/\bar m) \sum_{y_1}S_{xy_1}( G_{y_1 b_0} G_{a_i y_1})\overline {G}_{x \bar b_0}G_{x b_i} \overline G_{\bar a_i x}\overline G_{x\bar b_i}G_{a_j b_j}\overline G_{\bar a_j \bar b_j}\\
    &-\sum_{y_1}S_{xy_1}( G_{y_1 b_0} \overline G_{y_1\bar b_i})\overline {G}_{x \bar b_0}G_{a_i x}G_{x b_i} \overline G_{\bar a_i x}G_{a_j b_j}\overline G_{\bar a_j \bar b_j}\\
    &-\bar m^{-1}\sum_{y_1}S_{xy_1}( G_{y_1 b_0} G_{a_j y_1})\overline {G}_{x \bar b_0}G_{a_i x}G_{x b_i} \overline G_{\bar a_i x}\overline G_{x\bar b_i}G_{x b_j}\overline G_{\bar a_j \bar b_j}\\
    &-\bar m^{-1}\sum_{y_1}S_{xy_1}( G_{y_1 b_0} \overline G_{y_1 \bar b_j})\overline {G}_{x \bar b_0}G_{a_i x}G_{x b_i} \overline G_{\bar a_i x}\overline G_{x\bar b_i}\overline G_{\bar a_j x}G_{a_j b_j} + (\bar a_i,\bar b_i)\leftrightarrow (\bar a_j,\bar b_j). 
\end{align*}
Next, applying the edge expansions with respect to $\bar G_{x\bar b_0}$, we get the following graphs:
\begin{align*}
&-m^2\sum_{y_1,y_2}S_{xy_1}S_{xy_2}( G_{y_1 b_0} G_{a_i y_1})(\overline {G}_{y_2 \bar b_0}G_{y_2 b_i}) \overline G_{\bar a_i x}\overline G_{x\bar b_i}G_{a_j b_j}\overline G_{\bar a_j \bar b_j}\\
&-|m|^{2}\sum_{y_1,y_2}S_{xy_1}S_{xy_2}( G_{y_1 b_0} G_{a_i y_1})(\overline {G}_{y_2 \bar b_0} \overline G_{\bar a_i y_2})G_{x b_i}\overline G_{x\bar b_i}G_{a_j b_j}\overline G_{\bar a_j \bar b_j}\\
&-m\sum_{y_1,y_2}S_{xy_1}S_{xy_2}( G_{y_1 b_0} G_{a_i y_1})(\overline {G}_{y_2 \bar b_0}G_{y_2 b_j})G_{x b_i} \overline G_{\bar a_i x}\overline G_{x\bar b_i}G_{a_j x}\overline G_{\bar a_j \bar b_j}\\
&-|m|^2\sum_{y_1,y_2}S_{xy_1}S_{xy_2}( G_{y_1 b_0} \overline G_{y_1\bar b_i})(\overline {G}_{y_2 \bar b_0}G_{y_2 b_i})G_{a_i x} \overline G_{\bar a_i x}G_{a_j b_j}\overline G_{\bar a_j \bar b_j}\\
&-\bar m^2\sum_{y_1,y_2}S_{xy_1}S_{xy_2}( G_{y_1 b_0} \overline G_{y_1\bar b_i})(\overline {G}_{y_2 \bar b_0}\overline G_{\bar a_i y_2})G_{a_i x}G_{x b_i} G_{a_j b_j}\overline G_{\bar a_j \bar b_j}\\
&-\bar m\sum_{y_1,y_2}S_{xy_1}S_{xy_2}( G_{y_1 b_0} \overline G_{y_1\bar b_i})(\overline {G}_{y_2 \bar b_0}G_{y_2 b_j})G_{a_i x}G_{x b_i} \overline G_{\bar a_i x}G_{a_j x}\overline G_{\bar a_j \bar b_j}\\
&-\bar m\sum_{y_1,y_2}S_{xy_1}S_{xy_2}( G_{y_1 b_0} \overline G_{y_1\bar b_i})(\overline {G}_{y_2 \bar b_0}\overline G_{\bar a_jy_2})G_{a_i x}G_{x b_i} \overline G_{\bar a_i x}\overline G_{x\bar b_j}G_{a_j b_j}\\
&-m\sum_{y_1,y_2}S_{xy_1}S_{xy_2}( G_{y_1 b_0} G_{a_j y_1})(\overline {G}_{y_2 \bar b_0}G_{y_2 b_i})G_{a_i x} \overline G_{\bar a_i x}\overline G_{x\bar b_i}G_{x b_j}\overline G_{\bar a_j \bar b_j}\\
&-m\sum_{y_1,y_2}S_{xy_1}S_{xy_2}( G_{y_1 b_0} G_{a_j y_1})(\overline {G}_{y_2 \bar b_0}G_{y_2 b_j})G_{a_i x}G_{x b_i} \overline G_{\bar a_i x}\overline G_{x\bar b_i}\overline G_{\bar a_j \bar b_j}\\
&-m\sum_{y_1,y_2}S_{xy_1}S_{xy_2}( G_{y_1 b_0} \overline G_{y_1 \bar b_j})(\overline {G}_{y_2 \bar b_0}G_{y_2 b_i})G_{a_i x} \overline G_{\bar a_i x}\overline G_{x\bar b_i}\overline G_{\bar a_j x}G_{a_j b_j} \\
&-\bar m\sum_{y_1,y_2}S_{xy_1}S_{xy_2}( G_{y_1 b_0} \overline G_{y_1 \bar b_j})(\overline {G}_{y_2 \bar b_0}\overline G_{\bar a_i y_2})G_{a_i x}G_{x b_i} \overline G_{x\bar b_i}\overline G_{\bar a_j x}G_{a_j b_j} \\
&-\bar m\sum_{y_1,y_2}S_{xy_1}S_{xy_2}( G_{y_1 b_0} \overline G_{y_1 \bar b_j})(\overline {G}_{y_2 \bar b_0}\overline G_{\bar a_j y_2})G_{a_i x}G_{x b_i} \overline G_{\bar a_i x}\overline G_{x\bar b_i}G_{a_j b_j} \\
&-\sum_{y_1,y_2}S_{xy_1}S_{xy_2}( G_{y_1 b_0} \overline G_{y_1 \bar b_j})(\overline {G}_{y_2 \bar b_0}G_{y_2 b_j})G_{a_i x}G_{x b_i} \overline G_{\bar a_i x}\overline G_{x\bar b_i}\overline G_{\bar a_j x}G_{a_j x} + (\bar a_i,\bar b_i)\leftrightarrow (\bar a_j,\bar b_j).
\end{align*}
Continuing the edge expansions, we get the following graphs that are relevant for \eqref{eq:6-loop} and \eqref{eq:4+2-loop}:
\begin{align}
&-|m|^4\sum_{y_1,y_2,y_3}S_{xy_1}S_{xy_2}S_{xy_3}( G_{y_1 b_0} G_{a_i y_1})(\overline {G}_{y_2 \bar b_0}G_{y_2 b_i}) (\overline G_{\bar a_i y_3}\overline G_{y_3\bar b_i})G_{a_j b_j}\overline G_{\bar a_j \bar b_j} \label{eq:case411}\\
&-m |m|^4\sum_{y_1,y_2,y_3,y_4}S_{xy_1}S_{xy_2}S_{xy_3}S_{xy_4}( G_{y_1 b_0} G_{a_i y_1})(\overline {G}_{y_2 \bar b_0}G_{y_2 b_i}) (\overline G_{\bar a_i y_3}G_{a_j y_3})(\overline G_{y_4\bar b_i}G_{y_4 b_j})\overline G_{\bar a_j \bar b_j} \label{eq:case412}\\
& -|m|^{4}\sum_{y_1,y_2,y_3}S_{xy_1}S_{xy_2}S_{xy_3}( G_{y_1 b_0} G_{a_i y_1})(\overline {G}_{y_2 \bar b_0} \overline G_{\bar a_i y_2})(G_{y_3 b_i}\overline G_{y_3\bar b_i})G_{a_j b_j}\overline G_{\bar a_j \bar b_j}\label{eq:case413}\\
& -m|m|^4\sum_{y_1,y_2,y_3,y_4}S_{xy_1}S_{xy_2}S_{xy_3}S_{xy_4}( G_{y_1 b_0} G_{a_i y_1})(\overline {G}_{y_2 \bar b_0}G_{y_2 b_j})(G_{y_3 b_i}\overline G_{y_3\bar b_i} )(\overline G_{\bar a_i y_4}G_{a_j y_4})\overline G_{\bar a_j \bar b_j} \label{eq:case414}\\
&-|m|^4\sum_{y_1,y_2,y_3}S_{xy_1}S_{xy_2}S_{xy_3}( G_{y_1 b_0} \overline G_{y_1\bar b_i})(\overline {G}_{y_2 \bar b_0}G_{y_2 b_i})(G_{a_i y_3} \overline G_{\bar a_i y_3})G_{a_j b_j}\overline G_{\bar a_j \bar b_j} \label{eq:case415} \\
& -m|m|^4\sum_{y_1,y_2,y_3,y_4}S_{xy_1}S_{xy_2}S_{xy_3}S_{xy_4}( G_{y_1 b_0} \overline G_{y_1\bar b_i})(\overline {G}_{y_2 \bar b_0}G_{y_2 b_i})(G_{a_i y_3}G_{y_3 b_j}) (\overline G_{\bar a_i y_4}G_{a_j y_4})\overline G_{\bar a_j \bar b_j} \label{eq:case416} \\
&-\bar m|m|^4\sum_{y_1,y_2,y_3,y_4}S_{xy_1}S_{xy_2}S_{xy_3}S_{xy_4}( G_{y_1 b_0} \overline G_{y_1\bar b_i})(\overline {G}_{y_2 \bar b_0}G_{y_2 b_i})(G_{a_i y_3}\overline G_{\bar a_j y_3} ) (\overline G_{\bar a_i y_4}\overline G_{y_4 \bar b_j})G_{a_j b_j}  \label{eq:case417}\\
& - |m|^4\sum_{y_1,y_2,y_3,y_4}S_{xy_1}S_{xy_2}S_{xy_3}S_{xy_4}( G_{y_1 b_0} \overline G_{y_1\bar b_i})(\overline {G}_{y_2 \bar b_0}G_{y_2 b_i})(G_{a_i y_3}\overline G_{\bar a_j y_3} ) (\overline G_{\bar a_i y_4}G_{a_jy_4})\overline G_{x \bar b_j}G_{x b_j}\label{eq:case418}\\
&-|m|^4\sum_{y_1,y_2,y_3}S_{xy_1}S_{xy_2}S_{xy_3}( G_{y_1 b_0} \overline G_{y_1\bar b_i})(\overline {G}_{y_2 \bar b_0}\overline G_{\bar a_i y_2})(G_{a_i y_3}G_{y_3 b_i}) G_{a_j b_j}\overline G_{\bar a_j \bar b_j} \label{eq:case419}\\
& -\bar m |m|^4\sum_{y_1,y_2,y_3,y_4}S_{xy_1}S_{xy_2}S_{xy_3}S_{xy_4}( G_{y_1 b_0} \overline G_{y_1\bar b_i})(\overline {G}_{y_2 \bar b_0}\overline G_{\bar a_i y_2})(G_{a_i y_3}\overline G_{\bar a_j y_3})(G_{y_4 b_i}\overline G_{y_4 \bar b_j}) G_{a_j b_j} \label{eq:case4110} \\
&- m|m|^4\sum_{y_1,y_2,y_3,y_4}S_{xy_1}S_{xy_2}S_{xy_3}S_{xy_4}( G_{y_1 b_0} \overline G_{y_1\bar b_i})(\overline {G}_{y_2 \bar b_0}G_{y_2 b_j})(G_{a_i y_3}G_{y_3 b_i}) (\overline G_{\bar a_i y_4}G_{a_j y_4})\overline G_{\bar a_j \bar b_j}\label{eq:case4111}\\
&-|m|^6\sum_{\substack{y_1,y_2,y_3\\ y_4,y_5}}S_{xy_1}S_{xy_2}S_{xy_3}S_{xy_4}S_{xy_5}( G_{y_1 b_0} \overline G_{y_1\bar b_i})(\overline {G}_{y_2 \bar b_0}G_{y_2 b_j})(G_{a_i y_3}\overline G_{\bar a_jy_3})(G_{y_4 b_i} \overline G_{y_4 \bar b_j})\overline G_{\bar a_i y_5}G_{a_j y_5} \label{eq:case4112}\\
& - m|m|^4\sum_{y_1,y_2,y_3,y_4}S_{xy_4}S_{xy_1}S_{xy_2}S_{xy_3}( G_{y_1 b_0} \overline G_{y_1\bar b_i})(\overline {G}_{y_2 \bar b_0}G_{y_2 b_j})(G_{a_i y_3}\overline G_{\bar a_i y_3})(G_{y_4 b_i} G_{a_j y_4})\overline G_{\bar a_j \bar b_j}\label{eq:case4113}\\
& - |m|^6\sum_{\substack{y_1,y_2,y_3\\ y_4,y_5}}S_{xy_1}S_{xy_2}S_{xy_3}S_{xy_4}S_{xy_5}( G_{y_1 b_0} \overline G_{y_1\bar b_i})(\overline {G}_{y_2 \bar b_0}G_{y_2 b_j})(G_{a_i y_3}\overline G_{\bar a_i y_3})(G_{y_4 b_i}\overline G_{y_4 \bar b_j}) G_{a_j y_5}\overline G_{\bar a_j y_5} \label{eq:case4114}\\
&-\bar m|m|^4\sum_{y_1,y_2,y_3,y_4}S_{xy_1}S_{xy_2}S_{xy_3}S_{xy_4}( G_{y_1 b_0} \overline G_{y_1\bar b_i})(\overline {G}_{y_2 \bar b_0}\overline G_{\bar a_jy_2})(G_{a_i y_3}\overline G_{\bar a_i y_3})(G_{y_4 b_i} \overline G_{y_4\bar b_j})G_{a_j b_j}\label{eq:case4115}\\
&-m|m|^4\sum_{y_1,y_2,y_3,y_4}S_{xy_1}S_{xy_2}S_{xy_3}S_{xy_4}( G_{y_1 b_0} G_{a_j y_1})(\overline {G}_{y_2 \bar b_0}G_{y_2 b_i})(G_{a_i y_3} \overline G_{\bar a_i y_3})(\overline G_{y_4\bar b_i}G_{y_4 b_j})\overline G_{\bar a_j \bar b_j}\label{eq:case4116}\\
& -m|m|^4\sum_{y_1,y_2,y_3,y_4}S_{xy_1}S_{xy_2}S_{xy_3}S_{xy_4}( G_{y_1 b_0} G_{a_j y_1})(\overline {G}_{y_2 \bar b_0}G_{y_2 b_j})(G_{a_i y_3}\overline G_{\bar a_i y_3})(G_{y_4 b_i} \overline G_{y_4\bar b_i})\overline G_{\bar a_j \bar b_j} \label{eq:case4117}\\
& -\bar m |m|^4\sum_{y_1,y_2,y_3,y_4}S_{xy_1}S_{xy_2}S_{xy_3}S_{xy_4}( G_{y_1 b_0} \overline G_{y_1 \bar b_j})(\overline {G}_{y_2 \bar b_0}G_{y_2 b_i})(G_{a_i y_3} \overline G_{\bar a_i y_3})(\overline G_{y_4\bar b_i}\overline G_{\bar a_j y_4})G_{a_j b_j} \label{eq:case4118}\\
& - |m|^6\sum_{\substack{y_1,y_2,y_3\\ y_4,y_5}}S_{xy_1}S_{xy_2}S_{xy_3}S_{xy_4}S_{xy_5}( G_{y_1 b_0} \overline G_{y_1 \bar b_j})(\overline {G}_{y_2 \bar b_0}G_{y_2 b_i})(G_{a_i y_3} \overline G_{\bar a_i y_3})(\overline G_{y_4\bar b_i}G_{y_4 b_j})\overline G_{\bar a_j y_5}G_{a_jy_5} \label{eq:case4119}\\
& -\bar m|m|^4\sum_{y_1,y_2,y_3,y_4}S_{xy_1}S_{xy_2}S_{xy_3}S_{xy_4}( G_{y_1 b_0} \overline G_{y_1 \bar b_j})(\overline {G}_{y_2 \bar b_0}G_{y_2 b_i})(G_{a_i y_3} \overline G_{\bar a_j y_3})(\overline G_{\bar a_i y_4}\overline G_{y_4\bar b_i})G_{a_j b_j} \label{eq:case4120} \\
& - |m|^6\sum_{\substack{y_1,y_2,y_3\\ y_4,y_5}}S_{xy_1}S_{xy_2}S_{xy_3}S_{xy_4}S_{xy_5}( G_{y_1 b_0} \overline G_{y_1 \bar b_j})(\overline {G}_{y_2 \bar b_0}G_{y_2 b_i})(G_{a_i y_3} \overline G_{\bar a_j y_3})(\overline G_{\bar a_i y_4}G_{a_j y_4})\overline G_{y_5\bar b_i}G_{y_5 b_j}  \label{eq:case4121}\\
&-\bar m|m|^4\sum_{y_1,y_2,y_3,y_4}S_{xy_1}S_{xy_2}S_{xy_3}S_{xy_4}( G_{y_1 b_0} \overline G_{y_1 \bar b_j})(\overline {G}_{y_2 \bar b_0}\overline G_{\bar a_i y_2})(G_{a_i y_3}\overline G_{\bar a_j y_3})(G_{y_4 b_i} \overline G_{y_4\bar b_i})G_{a_j b_j} \label{eq:case4122} \\
& -\bar m|m|^4\sum_{y_1,y_2,y_3,y_4}S_{xy_1}S_{xy_2}S_{xy_3}S_{xy_4}( G_{y_1 b_0} \overline G_{y_1 \bar b_j})(\overline {G}_{y_2 \bar b_0}\overline G_{\bar a_j y_2})(G_{a_i y_3}\overline G_{\bar a_i y_3})(G_{y_4 b_i} \overline G_{y_4\bar b_i})G_{a_j b_j} \label{eq:case4123}\\
&-|m|^6\sum_{\substack{y_1,y_2,y_3\\ y_4,y_5}}S_{xy_1}S_{xy_2}S_{xy_3}S_{xy_4}S_{xy_5}( G_{y_1 b_0} \overline G_{y_1 \bar b_j})(\overline {G}_{y_2 \bar b_0}G_{y_2 b_j})(G_{a_i y_3}\overline G_{\bar a_i y_3})(G_{y_4 b_i} \overline G_{y_4\bar b_i})\overline G_{\bar a_j y_5}G_{a_j y_5} \label{eq:case4124} \\
& -|m|^6\sum_{\substack{y_1,y_2,y_3\\ y_4,y_5}}S_{xy_1}S_{xy_2}S_{xy_3}S_{xy_4}S_{xy_5} ( G_{y_1 b_0} \overline G_{y_1 \bar b_j})(\overline {G}_{y_2 \bar b_0}G_{y_2 b_j})(G_{a_i y_3}\overline G_{\bar a_j y_3})(G_{y_4 b_i}\overline G_{y_4\bar b_i}) \overline G_{\bar a_i y_5}G_{a_j y_5} \label{eq:case4125} \\
&+  (\bar a_i,\bar b_i)\leftrightarrow (\bar a_j,\bar b_j).\nonumber
\end{align}
Applying $GG$ expansions to these graphs, we obtain their respective contributions to $\Delta(\Pi_1)$ and $\Delta(\Pi_2)$: (1) \eqref{eq:case411}'s contribution is 0 (since it involves  $(b_0,\bar b_j), (\bar b_0,b_i)$ pairings), but its $ (\bar a_i,\bar b_i)\leftrightarrow (\bar a_j,\bar b_j)$ counterpart contributes $-|m|^6 |\iota|^2$ to $\Delta(\Pi_2)$; (2) \eqref{eq:case412}'s contribution is 0, but its $ (\bar a_i,\bar b_i)\leftrightarrow (\bar a_j,\bar b_j)$ counterpart contributes $-|m|^6 \iota$ to $\Delta(\Pi_2)$; (3) \eqref{eq:case413}
contributes $-|m|^6 |\iota|^2$ to $\Delta(\Pi_1)$, and its $ (\bar a_i,\bar b_i)\leftrightarrow (\bar a_j,\bar b_j)$ counterpart's contribution is 0; (4) \eqref{eq:case414}
contributes $-|m|^6 \iota$ to $\Delta(\Pi_1)$, and its $ (\bar a_i,\bar b_i)\leftrightarrow (\bar a_j,\bar b_j)$ counterpart's contribution is 0; (5) \eqref{eq:case415} is a term corresponding to the $p=1$ case and does not contribute to $\Delta(\Pi_1)$ and $\Delta(\Pi_2)$; (6) \eqref{eq:case416}
contributes $-|m|^6 \iota$ to $\Delta(\Pi_1)$, and its $ (\bar a_i,\bar b_i)\leftrightarrow (\bar a_j,\bar b_j)$ counterpart's contribution is 0; (7) \eqref{eq:case417}
contributes $-|m|^6 \bar \iota$ to $\Delta(\Pi_1)$, and its $ (\bar a_i,\bar b_i)\leftrightarrow (\bar a_j,\bar b_j)$ counterpart's contribution is 0; (8) \eqref{eq:case418}
contributes $-|m|^6$ to $\Delta(\Pi_1)$, and its $ (\bar a_i,\bar b_i)\leftrightarrow (\bar a_j,\bar b_j)$ counterpart's contribution is 0; (9) \eqref{eq:case419}'s contribution is 0, but its $ (\bar a_i,\bar b_i)\leftrightarrow (\bar a_j,\bar b_j)$ counterpart contributes $-|m|^6 |\iota|^2$ to $\Delta(\Pi_2)$;  (10) \eqref{eq:case4110}'s contribution is 0, but its $ (\bar a_i,\bar b_i)\leftrightarrow (\bar a_j,\bar b_j)$ counterpart contributes $-|m|^6 \bar \iota $ to $\Delta(\Pi_2)$; (11) \eqref{eq:case4111}'s contribution is 0, but its $ (\bar a_i,\bar b_i)\leftrightarrow (\bar a_j,\bar b_j)$ counterpart contributes $-|m|^6 \iota $ to $\Delta(\Pi_2)$; (12) \eqref{eq:case4112}'s contribution is 0, but its $ (\bar a_i,\bar b_i)\leftrightarrow (\bar a_j,\bar b_j)$ counterpart contributes $-|m|^6 $ to $\Delta(\Pi_2)$; (13) \eqref{eq:case4113}'s contribution is 0, but its $ (\bar a_i,\bar b_i)\leftrightarrow (\bar a_j,\bar b_j)$ counterpart contributes $-|m|^6\iota $ to $\Delta(\Pi_1)$; (14) \eqref{eq:case4114}'s contribution is 0, but its $ (\bar a_i,\bar b_i)\leftrightarrow (\bar a_j,\bar b_j)$ counterpart contributes $-|m|^6$ to $\Delta(\Pi_1)$; (15) \eqref{eq:case4115}'s contribution is 0, but its $ (\bar a_i,\bar b_i)\leftrightarrow (\bar a_j,\bar b_j)$ counterpart contributes $-|m|^6\bar \iota $ to $\Delta(\Pi_1)$; (16) \eqref{eq:case4116}'s contribution is 0, but its $ (\bar a_i,\bar b_i)\leftrightarrow (\bar a_j,\bar b_j)$ counterpart contributes $-|m|^6 \iota $ to $\Delta(\Pi_1)$; (17) \eqref{eq:case4117} contributes $-|m|^6 \iota$ to $\Delta(\Pi_2)$, and its $ (\bar a_i,\bar b_i)\leftrightarrow (\bar a_j,\bar b_j)$ counterpart's contribution is 0; (18) \eqref{eq:case4118}'s contribution is 0, but its $ (\bar a_i,\bar b_i)\leftrightarrow (\bar a_j,\bar b_j)$ counterpart contributes $-|m|^6 \bar \iota $ to $\Delta(\Pi_1)$; (19) \eqref{eq:case4119}'s contribution is 0, but its $ (\bar a_i,\bar b_i)\leftrightarrow (\bar a_j,\bar b_j)$ counterpart contributes $-|m|^6$ to $\Delta(\Pi_1)$; (20) \eqref{eq:case4120}'s contribution is 0, but its $ (\bar a_i,\bar b_i)\leftrightarrow (\bar a_j,\bar b_j)$ counterpart contributes $-|m|^6\bar \iota$ to $\Delta(\Pi_2)$; (21) \eqref{eq:case4121}'s contribution is 0, but its $ (\bar a_i,\bar b_i)\leftrightarrow (\bar a_j,\bar b_j)$ counterpart contributes $-|m|^6$ to $\Delta(\Pi_2)$;  (22) \eqref{eq:case4122} contributes $-|m|^6 \bar \iota$ to $\Delta(\Pi_1)$, and its $ (\bar a_i,\bar b_i)\leftrightarrow (\bar a_j,\bar b_j)$ counterpart's contribution is 0; (23) \eqref{eq:case4123} contributes $-|m|^6 \bar \iota$ to $\Delta(\Pi_2)$, and its $ (\bar a_i,\bar b_i)\leftrightarrow (\bar a_j,\bar b_j)$ counterpart's contribution is 0; (24) \eqref{eq:case4124} contributes $-|m|^6$ to $\Delta(\Pi_2)$, and its $ (\bar a_i,\bar b_i)\leftrightarrow (\bar a_j,\bar b_j)$ counterpart's contribution is 0; (25) \eqref{eq:case4125} contributes $-|m|^6$ to $\Delta(\Pi_1)$, and its $ (\bar a_i,\bar b_i)\leftrightarrow (\bar a_j,\bar b_j)$ counterpart's contribution is 0. 

Together with \eqref{Delta3}, the above calculations show that the first four cases give the following contributions to $\Delta(\Pi_1)$ and $\Delta(\Pi_2)$: 
\begin{align}
&\Delta_4(\Pi_1)= |m|^6\left(1-\iota^2-\bar \iota^2\right)-|m|^6\left(4+4\iota + 4\bar \iota+ |\iota|^2 \right)=-|m|^6\left(3+4\iota + 4\bar \iota+ |\iota|^2 +\iota^2 +\bar \iota^2 \right), \label{Delta41}\\
&\Delta_4(\Pi_2)= |m|^6\left(1-\iota^2-\bar \iota^2\right) -|m|^6\left(3+ 3\iota+ 3\bar \iota+2|\iota|^2\right)=-|m|^6\left(2+ 3\iota+ 3\bar \iota+2|\iota|^2+\iota^2+\bar \iota^2\right). \label{Delta42}
\end{align}

\medskip
 \noindent{\bf Case 5:} Corresponding to the first line on the RHS of \eqref{eq:Gammai}, we need to expand
\begin{align}
&   \sum_{i=1}^2\left( G_{x b_0} \overline {G}_{x \bar b_0} -|m|^2 T_{x,b_0\bar b_0}\right)\cdot m^{-1}G_{a_i x}G_{x b_i} \overline G_{\bar a_i \bar b_i}G_{a_j b_j}\overline G_{\bar a_j \bar b_j} \label{eq:case51}\\
+&   \sum_{i=1}^2 \left( G_{x b_0} \overline {G}_{x \bar b_0} -|m|^2 T_{x,b_0\bar b_0}\right)\cdot\bar m^{-1}G_{a_i b_i} \overline G_{\bar a_i x}\overline G_{x\bar b_i}G_{a_j b_j}\overline G_{\bar a_j \bar b_j}. \label{eq:case52}
\end{align}
We will focus on the contribution from the term \eqref{eq:case51} to $\Delta(\Pi_1)$ and $\Delta(\Pi_2)$, while the contribution from the term \eqref{eq:case52} can be obtained by taking the complex conjugate. Applying edge expansions to \eqref{eq:case51} with respect to $G_{x b_0}$, we obtain the following relevant graphs (without the summation notation over $i$): 
\begin{align*}
	& m\sum_{y_1} S_{xy_1}(G_{y_1 b_0}G_{a_i y_1}) (\overline {G}_{x \bar b_0}G_{x b_i}) \overline G_{\bar a_i \bar b_i}G_{a_j b_j}\overline G_{\bar a_j \bar b_j} + \sum_{y_1} S_{xy_1}(G_{y_1 b_0}\overline G_{y_1 \bar b_i})( \overline {G}_{x \bar b_0} G_{a_i x}G_{x b_i} \overline G_{\bar a_i x})G_{a_j b_j}\overline G_{\bar a_j \bar b_j}   \\
 	+& \sum_{y_1} S_{xy_1}(G_{y_1 b_0}G_{a_j y_1} ) (\overline {G}_{x \bar b_0} G_{a_i x}G_{x b_i}  G_{xb_j})\overline G_{\bar a_i \bar b_i}\overline G_{\bar a_j \bar b_j} + \sum_{y_1} S_{xy_1}(G_{y_1 b_0}\overline G_{y_1 \bar b_j}) (\overline {G}_{x \bar b_0} G_{a_i x}G_{x b_i}\overline G_{\bar a_j x}) \overline G_{\bar a_i \bar b_i}G_{a_j b_j}  .
\end{align*}
Next, applying edge expansions with respect to $\overline {G}_{x \bar b_0}$, we get the following graphs: 
\begin{align}
 &~  m|m|^2\sum_{y_1,y_2} S_{xy_1}S_{xy_2}(G_{y_1 b_0}G_{a_i y_1}) (\overline {G}_{y_2 \bar b_0}G_{y_2 b_i}) \overline G_{\bar a_i \bar b_i}G_{a_j b_j}\overline G_{\bar a_j \bar b_j} \label{eq:case510}\\
 +&~  |m|^2\sum_{y_1,y_2} S_{xy_1}S_{xy_2}(G_{y_1 b_0}G_{a_i y_1}) (\overline {G}_{y_2 \bar b_0}\overline G_{\bar a_i y_2})G_{x b_i} \overline G_{x \bar b_i}G_{a_j b_j}\overline G_{\bar a_j \bar b_j} \nonumber\\
 +&~  |m|^2\sum_{y_1,y_2} S_{xy_1}S_{xy_2}(G_{y_1 b_0}G_{a_i y_1}) (\overline {G}_{y_2 \bar b_0}G_{y_2 b_j})G_{x b_i}G_{a_j x} \overline G_{\bar a_i \bar b_i}\overline G_{\bar a_j \bar b_j} \nonumber\\
 +&~  |m|^2\sum_{y_1,y_2} S_{xy_1}S_{xy_2}(G_{y_1 b_0}G_{a_i y_1}) (\overline {G}_{y_2 \bar b_0}\overline G_{\bar a_j y_2}) G_{x b_i}\overline G_{x\bar b_j} \overline G_{\bar a_i \bar b_i}G_{a_j b_j}\nonumber\\
 +&~  |m|^2\sum_{y_1,y_2} S_{xy_1}S_{xy_2}(G_{y_1 b_0}\overline G_{y_1 \bar b_i})( \overline {G}_{y_2 \bar b_0}G_{y_2 b_i}) G_{a_i x} \overline G_{\bar a_i x}G_{a_j b_j}\overline G_{\bar a_j \bar b_j}  \nonumber \\
 +&~ \bar m^2 \sum_{y_1,y_2} S_{xy_1}S_{xy_2}(G_{y_1 b_0}\overline G_{y_1 \bar b_i})( \overline {G}_{y_2 \bar b_0}\overline G_{\bar a_i y_2}) G_{a_i x}G_{x b_i} G_{a_j b_j}\overline G_{\bar a_j \bar b_j}  \nonumber \\
+&~\bar m \sum_{y_1,y_2} S_{xy_1}S_{xy_2}(G_{y_1 b_0}\overline G_{y_1 \bar b_i})( \overline {G}_{y_2 \bar b_0}\overline G_{\bar a_j y_2}) G_{a_i x}G_{x b_i} \overline G_{\bar a_i x}\overline G_{x\bar b_j} G_{a_j b_j} \nonumber \\
 +&~ |m|^2 \sum_{y_1,y_2} S_{xy_1}S_{xy_2}(G_{y_1 b_0}G_{a_j y_1} ) (\overline {G}_{y_2 \bar b_0}G_{y_2 b_i}) G_{a_i x}  G_{xb_j}\overline G_{\bar a_i \bar b_i}\overline G_{\bar a_j \bar b_j} \nonumber\\
 +&~  |m|^2\sum_{y_1,y_2} S_{xy_1}S_{xy_2}(G_{y_1 b_0}G_{a_j y_1} ) (\overline {G}_{y_2  \bar b_0} G_{y_2 b_j}) G_{a_i x}G_{x b_i}  \overline G_{\bar a_i \bar b_i}\overline G_{\bar a_j \bar b_j} \nonumber\\
+&~|m|^2 \sum_{y_1,y_2} S_{xy_1}S_{xy_2}(G_{y_1 b_0}\overline G_{y_1 \bar b_j}) (\overline {G}_{y_2 \bar b_0}G_{y_2 b_i}) G_{a_i x}\overline G_{\bar a_j x} \overline G_{\bar a_i \bar b_i}G_{a_j b_j} \nonumber \\
+&~\bar m^2 \sum_{y_1,y_2} S_{xy_1}S_{xy_2}(G_{y_1 b_0}\overline G_{y_1 \bar b_j}) (\overline {G}_{y_2 \bar b_0}\overline G_{\bar a_j y_2}) G_{a_i x}G_{x b_i} \overline G_{\bar a_i \bar b_i}G_{a_j b_j} \nonumber \\
+&~\bar m \sum_{y_1,y_2} S_{xy_1}S_{xy_2}(G_{y_1 b_0}\overline G_{y_1 \bar b_j}) (\overline {G}_{y_2 \bar b_0}\overline G_{\bar a_i y_2}) G_{a_i x}G_{x b_i}\overline G_{\bar a_j x} \overline G_{x \bar b_i}G_{a_j b_j}\nonumber  \\
+&~\bar m \sum_{y_1,y_2} S_{xy_1}S_{xy_2}(G_{y_1 b_0}\overline G_{y_1 \bar b_j}) (\overline {G}_{y_2 \bar b_0}G_{y_2 b_j}) G_{a_i x}G_{x b_i}\overline G_{\bar a_j x}G_{a_j x} \overline G_{\bar a_i \bar b_i}.  \nonumber
\end{align}
Applying the $GG$-expansion to \eqref{eq:case510}, we get some graphs corresponding to the $p=1$ case and two graphs
\begin{align*}
 &~m^2|m|^2\sum_{y_1,y_2,y_3} S_{xy_1}S_{xy_2}S_{y_1y_3} (G_{a_i y_1}G_{y_1 b_j}) (\overline {G}_{y_2 \bar b_0}G_{y_2 b_i}) (G_{y_3 b_0}G_{a_j y_3})\overline G_{\bar a_i \bar b_i}\overline G_{\bar a_j \bar b_j} \\ 
+&~ m^4|m|^2\sum_{y_1,y_2,y_3,y_4} S_{xy_1}S_{xy_2}S^+_{y_1y_3}S_{y_3y_4} (\overline {G}_{y_2 \bar b_0}G_{y_2 b_i})(G_{a_i y_3}G_{y_3 b_j}) (G_{y_4 b_0}G_{a_j y_4}) \overline G_{\bar a_i \bar b_i}\overline G_{\bar a_j \bar b_j}, 
\end{align*}
that contribute $|m|^6 \iota^2+|m|^6 \iota^3$ to $\Delta(\Pi_1)$. For the remaining graphs, applying the edge expansions, we get the following graphs that are relevant for \eqref{eq:6-loop} and \eqref{eq:4+2-loop}: 
\begin{align}
 & \quad\,  |m|^4\sum_{y_1,y_2,y_3} S_{xy_1}S_{xy_2}S_{xy_3}(G_{y_1 b_0}G_{a_i y_1}) (\overline {G}_{y_2 \bar b_0}\overline G_{\bar a_i y_2})(G_{y_3 b_i} \overline G_{y_3 \bar b_i})G_{a_j b_j}\overline G_{\bar a_j \bar b_j} \label{eq:case511} \\
 &+ m^2 |m|^2\sum_{y_1,y_2,y_3} S_{xy_1}S_{xy_2}S_{xy_3}(G_{y_1 b_0}G_{a_i y_1}) (\overline {G}_{y_2 \bar b_0}G_{y_2 b_j})(G_{y_3 b_i}G_{a_j y_3}) \overline G_{\bar a_i \bar b_i}\overline G_{\bar a_j \bar b_j} \label{eq:case512}\\
 & +  m|m|^4\sum_{y_1,y_2,y_3,y_4} S_{xy_1}S_{xy_2}S_{xy_3}S_{xy_4}(G_{y_1 b_0}G_{a_i y_1}) (\overline {G}_{y_2 \bar b_0}G_{y_2 b_j})(G_{y_3 b_i}\overline G_{y_3 \bar b_i})(G_{a_j y_4} \overline G_{\bar a_i y_4})\overline G_{\bar a_j \bar b_j} \label{eq:case513} \\
 & +  m|m|^4\sum_{y_1,y_2,y_3,y_4} S_{xy_1}S_{xy_2}S_{xy_3}S_{xy_4}(G_{y_1 b_0}G_{a_i y_1}) (\overline {G}_{y_2 \bar b_0}G_{y_2 b_j})(G_{y_3 b_i}\overline G_{y_3 \bar b_j})(G_{a_j y_4} \overline G_{\bar a_j y_4})\overline G_{\bar a_i \bar b_i} \label{eq:case514}\\
 &+  |m|^4\sum_{y_1,y_2,y_3} S_{xy_1}S_{xy_2}S_{xy_3}(G_{y_1 b_0}G_{a_i y_1}) (\overline {G}_{y_2 \bar b_0}\overline G_{\bar a_j y_2}) (G_{y_3 b_i}\overline G_{y_3\bar b_j}) \overline G_{\bar a_i \bar b_i}G_{a_j b_j} \label{eq:case515} \\
  &+  |m|^4\sum_{y_1,y_2,y_3} S_{xy_1}S_{xy_2}S_{xy_3}(G_{y_1 b_0}\overline G_{y_1 \bar b_i})( \overline {G}_{y_2 \bar b_0}G_{y_2 b_i}) (G_{a_i y_3} \overline G_{\bar a_i y_3})G_{a_j b_j}\overline G_{\bar a_j \bar b_j}  \label{eq:case516}  \\
   &+ m |m|^4\sum_{y_1,y_2,y_3,y_4} S_{xy_1}S_{xy_2}S_{xy_3}S_{xy_4}(G_{y_1 b_0}\overline G_{y_1 \bar b_i})( \overline {G}_{y_2 \bar b_0}G_{y_2 b_i}) (G_{a_i y_3}G_{y_3 b_j} )(\overline G_{\bar a_i y_4}G_{a_j y_4})\overline G_{\bar a_j \bar b_j}\label{eq:case517}  \\
&+ \bar m |m|^4\sum_{y_1,y_2,y_3,y_4} S_{xy_1}S_{xy_2}S_{xy_3}S_{xy_4}(G_{y_1 b_0}\overline G_{y_1 \bar b_i})( \overline {G}_{y_2 \bar b_0}G_{y_2 b_i}) (G_{a_i y_3}\overline G_{\bar a_j y_3}) (\overline G_{\bar a_i y_4}\overline G_{y_4 \bar b_j}) G_{a_j b_j} \label{eq:case518}  \\
&+ |m|^4\sum_{y_1,y_2,y_3,y_4} S_{xy_1}S_{xy_2}S_{xy_3}S_{xy_4}(G_{y_1 b_0}\overline G_{y_1 \bar b_i})( \overline {G}_{y_2 \bar b_0}G_{y_2 b_i}) (G_{a_i y_3}\overline G_{\bar a_j y_3}) (\overline G_{\bar a_i y_4}G_{a_j y_4})\overline G_{x \bar b_j} G_{x b_j} \label{eq:case519}\\
 &+ |m|^4 \sum_{y_1,y_2,y_3} S_{xy_1}S_{xy_2}S_{xy_3}(G_{y_1 b_0}\overline G_{y_1 \bar b_i})( \overline {G}_{y_2 \bar b_0}\overline G_{\bar a_i y_2}) (G_{a_i y_3}G_{y_3 b_i}) G_{a_j b_j}\overline G_{\bar a_j \bar b_j}  \label{eq:case5110}\\
 &+ \bar m|m|^4 \sum_{y_1,y_2,y_3,y_4} S_{xy_1}S_{xy_2}S_{xy_3}S_{xy_4}(G_{y_1 b_0}\overline G_{y_1 \bar b_i})( \overline {G}_{y_2 \bar b_0}\overline G_{\bar a_i y_2}) (G_{a_i y_3}\overline G_{\bar a_j y_3})(G_{y_4 b_i}\overline G_{y_4 \bar b_j})  G_{a_j b_j} \label{eq:case5111}\\
 &+\bar m |m|^4\sum_{y_1,y_2,y_3,y_4} S_{xy_1}S_{xy_2}S_{xy_4}S_{xy_4}(G_{y_1 b_0}\overline G_{y_1 \bar b_i})( \overline {G}_{y_2 \bar b_0}\overline G_{\bar a_j y_2}) (G_{a_i y_3}\overline G_{\bar a_i y_3})(G_{y_4 b_i} \overline G_{y_4\bar b_j}) G_{a_j b_j}  \label{eq:case5112}\\
 &+ m^2|m|^2 \sum_{y_1,y_2,y_3} S_{xy_1}S_{xy_2}S_{xy_3}(G_{y_1 b_0}G_{a_j y_1} ) (\overline {G}_{y_2 \bar b_0}G_{y_2 b_i}) (G_{a_i y_3}  G_{y_3b_j})\overline G_{\bar a_i \bar b_i}\overline G_{\bar a_j \bar b_j} \label{eq:case5113} \\
& +m |m|^4 \sum_{y_1,y_2,y_3,y_4} S_{xy_1}S_{xy_2}S_{xy_3}S_{xy_4}(G_{y_1 b_0}G_{a_j y_1} ) (\overline {G}_{y_2 \bar b_0}G_{y_2 b_i}) (G_{a_i y_3}\overline G_{\bar a_jy_3})  (G_{y_4b_j}\overline G_{y_4 \bar b_j}) \overline G_{\bar a_i \bar b_i} \label{eq:case5114} \\
& +  m^2|m|^2\sum_{y_1,y_2,y_3} S_{xy_1}S_{xy_2}S_{xy_3}(G_{y_1 b_0}G_{a_j y_1} ) (\overline {G}_{y_2  \bar b_0} G_{y_2 b_j}) (G_{a_i y_3}G_{y_3 b_i} ) \overline G_{\bar a_i \bar b_i}\overline G_{\bar a_j \bar b_j} \label{eq:case5115} \\
&+ m |m|^4\sum_{y_1,y_2,y_3,y_4} S_{xy_1}S_{xy_2}S_{xy_3}S_{xy_4}(G_{y_1 b_0}G_{a_j y_1} ) (\overline {G}_{y_2  \bar b_0} G_{y_2 b_j}) (G_{a_i y_3} \overline G_{\bar a_i y_3})(G_{y_4 b_i}  \overline G_{y_4 \bar b_i})\overline G_{\bar a_j \bar b_j} \label{eq:case5116}\\
& + m |m|^2\sum_{y_1,y_2,y_3} S_{xy_1}S_{xy_2}S_{xy_3}(G_{y_1 b_0}G_{a_j y_1} ) (\overline {G}_{y_2  \bar b_0} G_{y_2 b_j}) (G_{a_i y_3} \overline G_{\bar a_j y_3})G_{x b_i}  \overline G_{x \bar b_j}\overline G_{\bar a_i \bar b_i}  \label{eq:case5117}\\
&+|m|^2 \sum_{y_1,y_2,y_3} S_{xy_1}S_{xy_2}S_{xy_3}(G_{y_1 b_0}\overline G_{y_1 \bar b_j}) (\overline {G}_{y_2 \bar b_0}G_{y_2 b_i}) (G_{a_i y_3}\overline G_{\bar a_iy_3})\overline G_{\bar a_j x} \overline G_{x\bar b_i}G_{a_j b_j}  \label{eq:case5118}\\
&+|m|^4 \sum_{y_1,y_2,y_3,y_4} S_{xy_1}S_{xy_2}S_{xy_3}S_{xy_4}(G_{y_1 b_0}\overline G_{y_1 \bar b_j}) (\overline {G}_{y_2 \bar b_0}G_{y_2 b_i}) (G_{a_i y_3}G_{y_3 b_j}) (\overline G_{\bar a_j y_4} G_{a_j y_4})\overline G_{\bar a_i \bar b_i}  \label{eq:case5119}\\
&+|m|^4 \sum_{y_1,y_2,y_3} S_{xy_1}S_{xy_2}S_{xy_3}(G_{y_1 b_0}\overline G_{y_1 \bar b_j}) (\overline {G}_{y_2 \bar b_0}\overline G_{\bar a_j y_2}) (G_{a_i y_3}G_{y_3 b_i}) \overline G_{\bar a_i \bar b_i}G_{a_j b_j} \label{eq:case5120} \\
&+ \bar m|m|^4 \sum_{y_1,y_2,y_3,y_4} S_{xy_1}S_{xy_2}S_{xy_3}S_{xy_4}(G_{y_1 b_0}\overline G_{y_1 \bar b_j}) (\overline {G}_{y_2 \bar b_0}\overline G_{\bar a_j y_2}) (G_{a_i y_3}\overline G_{\bar a_i y_3})(G_{y_4 b_i} \overline G_{y_4\bar b_i})G_{a_j b_j} \label{eq:case5121}  \\
&+\bar m |m|^4 \sum_{y_1,y_2,y_3,y_4} S_{xy_1}S_{xy_2}S_{xy_3}S_{xy_4}(G_{y_1 b_0}\overline G_{y_1 \bar b_j}) (\overline {G}_{y_2 \bar b_0}\overline G_{\bar a_i y_2}) (G_{a_i y_3}\overline G_{\bar a_j y_3}) (G_{y_4 b_i}\overline G_{y_4 \bar b_i})G_{a_j b_j}  \label{eq:case5122}\\
&+m|m|^4 \sum_{y_1,y_2,y_3,y_4} S_{xy_1}S_{xy_2}S_{xy_3}S_{xy_4}(G_{y_1 b_0}\overline G_{y_1 \bar b_j}) (\overline {G}_{y_2 \bar b_0}G_{y_2 b_j}) (G_{a_i y_3}G_{y_3 b_i})(\overline G_{\bar a_j y_4}G_{a_j y_4}) \overline G_{\bar a_i \bar b_i}  \label{eq:case5123}\\
&+|m|^2 \sum_{y_1,y_2,y_3} S_{xy_1}S_{xy_2}S_{xy_3}(G_{y_1 b_0}\overline G_{y_1 \bar b_j}) (\overline {G}_{y_2 \bar b_0}G_{y_2 b_j}) (G_{a_i y_3}\overline G_{\bar a_iy_3})G_{x b_i}\overline G_{\bar a_j x} \overline G_{x \bar b_i}G_{a_j x} \label{eq:case5124}\\
& +m |m|^4\sum_{y_1,y_2,y_3,y_4} S_{xy_1}S_{xy_2}S_{xy_3}S_{xy_4}(G_{y_1 b_0}\overline G_{y_1 \bar b_j}) (\overline {G}_{y_2 \bar b_0}G_{y_2 b_j})( G_{a_i y_3}\overline G_{\bar a_j y_3})(G_{y_4 b_i}G_{a_j y_4}) \overline G_{\bar a_i \bar b_i} \label{eq:case5125} \\
& + |m|^4\sum_{y_1,y_2,y_3,y_4} S_{xy_1}S_{xy_2}S_{xy_3}S_{xy_4}(G_{y_1 b_0}\overline G_{y_1 \bar b_j}) (\overline {G}_{y_2 \bar b_0}G_{y_2 b_j})( G_{a_i y_3}\overline G_{\bar a_j y_3})(G_{y_4 b_i} \overline G_{y_4 \bar b_i})G_{a_j x}\overline G_{\bar a_i x}.\label{eq:case5126} 
\end{align}
Applying $GG$ expansions to these graphs, we obtain their respective contributions to $\Delta(\Pi_1)$ and $\Delta(\Pi_2)$: 
(1) \eqref{eq:case511} contributes $|m|^6|\iota|^2$ to $\Delta(\Pi_1)$; (2) \eqref{eq:case512} contributes $|m|^6\iota^2 $ to $\Delta(\Pi_1)$; (3) \eqref{eq:case513} contributes $|m|^6\iota$ to $\Delta(\Pi_1)$; (4) \eqref{eq:case514} and \eqref{eq:case515} both involve $(b_0,\bar b_i), (\bar b_0,b_j)$ pairings, so their contributions are 0; (5) \eqref{eq:case516} is a term corresponding to the $p=1$ case and does not contribute to $\Delta(\Pi_1)$ and $\Delta(\Pi_2)$; (6) \eqref{eq:case517} contributes $|m|^6\iota$ to $\Delta(\Pi_1)$; (7) \eqref{eq:case518} contributes $|m|^6\bar \iota$ to $\Delta(\Pi_1)$; (8) \eqref{eq:case519} contributes $|m|^6 $ to $\Delta(\Pi_1)$; (9) \eqref{eq:case5110}--\eqref{eq:case5112} all involve $(b_0,\bar b_i), (\bar b_0,b_j)$ pairings, so their contributions are 0; (10) \eqref{eq:case5113} contributes $|m|^6 \iota^2$ to $\Delta(\Pi_1)$; (11) \eqref{eq:case5114} contributes $|m|^6 \iota$ to $\Delta(\Pi_1)$; (12) \eqref{eq:case5115} contributes $|m|^6 \iota^2$ to $\Delta(\Pi_2)$; (13) \eqref{eq:case5116} contributes $|m|^6 \iota$ to $\Delta(\Pi_2)$; (14) \eqref{eq:case5117}--\eqref{eq:case5119} involve $(b_0,\bar b_i), (\bar b_0,b_j)$ or $(b_0,\bar b_j), (\bar b_0,b_i)$ pairings, so their contributions are 0; (15) \eqref{eq:case5120} contributes $|m|^6 |\iota|^2$ to $\Delta(\Pi_2)$; (16) \eqref{eq:case5121} contributes $|m|^6 \bar \iota$ to $\Delta(\Pi_2)$; (17) \eqref{eq:case5122} contributes $|m|^6 \bar \iota$ to $\Delta(\Pi_1)$; (18) \eqref{eq:case5123} contributes $|m|^6 \iota$ to $\Delta(\Pi_2)$; (19) \eqref{eq:case5124} contributes $|m|^6$ to $\Delta(\Pi_2)$; (20) \eqref{eq:case5125} contributes $|m|^6 \iota$ to $\Delta(\Pi_1)$;  (21) \eqref{eq:case5126} contributes $|m|^6$ to $\Delta(\Pi_1)$. 

Together with \eqref{Delta41} and \eqref{Delta42}, the above calculations show that $\Delta(\Pi_1)$ and $\Delta(\Pi_2)$ are equal to 
\begin{align}
\Delta(\Pi_1) &= -|m|^6\left(3+4\iota + 4\bar \iota+ |\iota|^2 +\iota^2 +\bar \iota^2 \right)+|m|^6\left( 2+4\iota +2\bar \iota + |\iota|^2+ 3\iota^2 +\iota^3+c.c.\right) \nonumber\\
&= |m|^6\left(1+2 \iota+ 2 \bar \iota+\left|\iota\right|^2 +2\iota^2+2\bar \iota^2 + \iota^3+ \bar \iota^3\right), \label{Delta51}\\
 \Delta(\Pi_2)&=-|m|^6\left(2+ 3\iota+ 3\bar \iota+2|\iota|^2+\iota^2+\bar \iota^2\right)+|m|^6\left(1+2\iota +\bar \iota +|\iota|^2 + \iota^2+c.c. \right)=0.
\label{Delta52}
\end{align}
This concludes the derivation of \eqref{DeltaPi1}.

\end{document}